\documentclass[11pt,twoside]{article}
\usepackage[T1]{fontenc}
\usepackage{lmodern}
\usepackage{amsmath,amsthm,amssymb}
\usepackage{enumerate}
\usepackage{graphicx}
\usepackage{cite}
\usepackage{comment}
\usepackage{oands}
\usepackage{tikz}
\usepackage{changepage}
\usepackage{bbm}
\usepackage{mathtools}
\usepackage[margin=1in]{geometry}
\usepackage[pagewise,mathlines]{lineno}
\usepackage{appendix}
\usepackage{stmaryrd}
\usepackage{microtype}
\usepackage[colorinlistoftodos]{todonotes}
\usepackage{dsfont}
\usepackage{mathrsfs}  
\usepackage{multirow}
\usepackage{graphicx,subfigure}
\usepackage{array}
\usepackage{bm}
\newcolumntype{P}[1]{>{\centering\arraybackslash}p{#1}}
\usepackage{subcaption}
\usepackage{ctable}
\usepackage{longtable}
\usepackage{float}
\usepackage{afterpage}

\usepackage{subcaption}

\definecolor{red}{HTML}{f54b1a}
\definecolor{pink}{HTML}{d19eb1}
\definecolor{orange}{HTML}{d3772e}
\definecolor{yellow}{HTML}{ebe85d}
\definecolor{green}{HTML}{0f6852}
\definecolor{lightblue}{HTML}{01abe9}
\definecolor{darkblue}{HTML}{1b346c}
\definecolor{tan}{HTML}{e5c39e}
\definecolor{darktan}{HTML}{af9e73}
\definecolor{grey}{HTML}{c3ced0}
\definecolor{darkgrey}{HTML}{9dadc4}
\definecolor{black}{HTML}{110d1b}
\definecolor{white}{HTML}{f1f8f1}

\usepackage[pdftitle={Surface sums in 2d large-$N$ lattice Yang--Mills: Cancellations and explicit computations}, pdfauthor={Jacopo Borga, Sky Cao, and Jasper Shogren-Knaak},
colorlinks=true,linkcolor=lightblue,urlcolor=black,citecolor=darkblue,bookmarks=true,bookmarksopen=true,bookmarksopenlevel=2,unicode=true,linktocpage]{hyperref}
\usepackage[capitalise,noabbrev]{cleveref}

\usepackage{libertine}

\usepackage{fancyhdr}
\theoremstyle{plain}
\newtheorem{thm}{Theorem}[section]
\newtheorem{cor}[thm]{Corollary}
\newtheorem{lem}[thm]{Lemma}
\newtheorem{prop}[thm]{Proposition}

\newtheorem{obs}[thm]{Observation}
\newtheorem{claim}[thm]{Claim}

\newtheorem*{claim*}{Claim}

\def\@rst #1 #2other{#1}
\newcommand\MR[1]{\relax\ifhmode\unskip\spacefactor3000 \space\fi
  \MRhref{\expandafter\@rst #1 other}{#1}}
\newcommand{\MRhref}[2]{\href{http://www.ams.org/mathscinet-getitem?mr=#1}{MR#2}}

\theoremstyle{definition}
\newtheorem{defn}[thm]{Definition}
\newtheorem{remark}[thm]{Remark}

\numberwithin{equation}{section}

\newcommand{\dsb}{\begin{adjustwidth}{2.5em}{0pt}
\begin{footnotesize}}
\newcommand{\dse}{\end{footnotesize}
\end{adjustwidth}}

\newcommand{\ssb}{\begin{adjustwidth}{2.5em}{0pt}}
\newcommand{\sse}{\end{adjustwidth}}

\newcommand{\aryb}{\begin{eqnarray*}}
\newcommand{\arye}{\end{eqnarray*}}
\def\alb#1\ale{\begin{align*}#1\end{align*}}
\def\allb#1\alle{\begin{align}#1\end{align}}
\newcommand{\eqb}{\begin{equation}}
\newcommand{\eqe}{\end{equation}}
\newcommand{\eqbn}{\begin{equation*}}
\newcommand{\eqen}{\end{equation*}}

\newcommand{\mcl}{\mathcal}

\newcommand{\multK}{\mathcal{K}}
\newcommand{\tree}{T^*}
\newcommand{\spt}{\mathbb{T}^*}
\newcommand{\spthat}{\hat{\mathbb{T}}^*}
\newcommand{\sptd}{T^{\smallsetminus}(\spt)}
\newcommand{\sptdhat}{T^{\smallsetminus}(\spthat)}

\let\originalleft\left
\let\originalright\right
\renewcommand{\left}{\mathopen{}\mathclose\bgroup\originalleft}
\renewcommand{\right}{\aftergroup\egroup\originalright}

\usepackage{ upgreek }

\def\DD{\mathbb{D}}
\def\EE{\mathbb{E}}

\def\NN{\mathbb{N}}

\def\RR{\mathbb{R}}
\def\SS{\mathbb{S}}
\def\TT{\mathbb{T}}

\def\ZZ{\mathbb{Z}}

\def\cC{\mathcal{C}}

\def\cE{\mathcal{E}}
\def\cF{\mathcal{F}}
\def\cG{\mathcal{G}}

\def\cI{\mathcal{I}}

\def\cK{\mathcal{K}}
\def\cL{\mathcal{L}}
\def\cM{\mathcal{M}}
\def\cN{\mathcal{N}}

\def\cP{\mathcal{P}}
\def\cQ{\mathcal{Q}}

\def\cS{\mathcal{S}}
\def\cT{\mathcal{T}}

\def\cV{\mathcal{V}}

\newcommand{\E}[1]{\mathbbm{E}\left[#1\right]}

\DeclareMathOperator{\Var}{Var}

\newcommand{\sm}{\setminus}

\def\npe{\cN\cP\cM}

\DeclareMathOperator{\area}{area}

\DeclareMathOperator{\supp}{supp}

\DeclareMathOperator{\Tr}{Tr}
\DeclareMathOperator{\tr}{tr}

\DeclareMathOperator{\mle}{\mathfrak{M}}

\def\<{\langle} \def\>{\rangle}

\usepackage{crossreftools}

\makeatletter
\newcommand{\optionaldesc}[2]{%
  \phantomsection
  #1\protected@edef\@currentlabel{#1}\label{#2}%
}
\makeatother

\newcommand{\unitary}{\mathrm{U}}
\newcommand{\SO}{\mathrm{S}\mathrm{O}}
\newcommand{\SU}{\mathrm{S}\mathrm{U}}

\title{Surface sums in two-dimensional large-$N$ lattice Yang--Mills: Cancellations and explicit computations for general loops}
\date{    }
\author{Jacopo Borga\thanks{Department of Mathematics, Massachusetts Institute of Technology, \texttt{\fontfamily{cmr}\selectfont \href{mailto:jborga@mit.edu}{jborga@mit.edu}}} \quad Sky Cao\thanks{Department of Mathematics, Massachusetts Institute of Technology, \texttt{\fontfamily{cmr}\selectfont \href{mailto:skycao@mit.edu}{skycao@mit.edu}}} \quad Jasper Shogren-Knaak\thanks{Courant Institute of Mathematical Sciences, New York University, \texttt{\fontfamily{cmr}\selectfont \href{mailto:jasper.shogren-knaak@cims.nyu.edu}{jasper.shogren-knaak@cims.nyu.edu}}}}

\begin{document}

\maketitle

\begin{abstract}
In the context of two-dimensional large-$N$ lattice Yang--Mills theory, we perform a refined study of the surface sums defined in the companion work~\cite{borga2024surfacesumslatticeyangmills}. In this setting, the surface sums are \textit{a priori} expected to exhibit significant simplifications because two-dimensional Yang--Mills theory is a special model that admits many known exact formulas. Thus, a natural problem is to understand these simplifications directly from the perspective of the surface sums. Towards this goal, we develop a key new tool in the form of a surface exploration algorithm (or ``peeling process''), which, at each step, carefully selects the next edge to explore. Using this algorithm, we manage to find many cancellations in the surface sums, thereby obtaining a detailed understanding of precisely which surfaces remain after cancellation. As a consequence, we obtain many new explicit formulas for Wilson loop expectations of general loops in the large-$N$ limit of lattice Yang--Mills in two dimensions and prove a convergence result for the empirical spectral measure of any simple loop. \end{abstract}

\tableofcontents

\section{Introduction}\label{sec: intro}

\subsection{Overview and informal statement of the main results}\label{subsec:overview}

The use of geometric representations in the analysis of statistical mechanical models has a long history. Classical examples include the high temperature expansion, low temperature expansion, random current representation, Edwards-Sokal coupling of the Ising model (see the surveys \cite{DC2018, DC2023} for more discussion), and the Brydges-Fr\"{o}lich-Spencer random walk representation of lattice spin systems \cite{BFS1982}. More recent examples include random loop models related to the spin $O(N)$ model \cite{BU2020, LT2021a, LT2021}. In all these examples, the geometric representation involves a system of edges (in terms of a bond percolation configuration or a collection of loops and paths), which reflects the fact that the spin systems in consideration are defined by specifying pairwise interactions of neighboring vertices.

For lattice Yang--Mills, instead of vertices interacting along edges, we now have edges interacting along plaquettes. Thus, the associated geometric representations will naturally consist of a system of plaquettes, for instance, a plaquette percolation configuration or a collection of surfaces. Examples of such representations include Wilson's high-temperature expansion \cite{Wilson1974}, various surface representations \cite{tH1974, KOSTOV1984445, Kostov1994}, and plaquette percolation models \cite{DS2025a, DS2025b}.

Recently, there has been a resurgence of interest in surface representations of lattice and continuum Yang--Mills, starting with the ``vanishing string trajectory'' point of view of Chatterjee \cite{chatterjee_rigorous_2019} and further studied by Chatterjee and Jafarov \cite{chatterjee20161n, jafarov2016wilson}. Other examples include a surface representation for two-dimensional continuum Yang--Mills by Park et al.~\cite{park2023wilson}, and what is most relevant for us, the surface model for lattice Yang--Mills of Cao-Park-Sheffield \cite{cao2023random}. 

The main point in introducing these surface representations is to use them to prove new results about both lattice and continuum $d$-dimensional Yang--Mills. A recent example of this is the work \cite{CNS2025}, whose main contribution, in a sense, can be thought of as developing techniques to analyze the surface sums appearing in \cite{cao2023random} in the case of finite $N$ (the parameter that governs the size of the matrices in the gauge group) and $d\geq 2$.

Towards this aim, in our previous paper \cite{borga2024surfacesumslatticeyangmills}, we introduced a surface representation for $d$-dimensional Yang--Mills in the large-$N$ limit. This regime is notable, as it exhibits a relatively simple and tractable surface model. In fact, the main results of the present paper will be obtained by analyzing this surface model.

To go into slightly more detail about our previous paper \cite{borga2024surfacesumslatticeyangmills}, roughly speaking (precise explanations are given in \cref{sec: Wilson loop expectations as surface sum large N}), we showed that in the large-$N$ limit for any dimension $d\geq 2$, when the inverse temperature $\upbeta$ is small, the Wilson loop expectation $\phi_{N,\upbeta}(\ell)$ of any loop $\ell$ -- which are the main quantities of interest in Yang--Mills theory -- converges  to a limiting ``surface sum'', that is, a sum of the form (see \eqref{eq:def-phi_K} below for a precise definition):
\begin{equation}\label{eq:surf-sum}
    \phi(\ell)=\sum_{M\in\cM(\ell)}w_{\upbeta}(M),
\end{equation} 
where $\mathcal{M}(\ell)$ denotes a \emph{infinite} collection of planar maps, sometimes referred to as \textit{surfaces}, and $w_{\upbeta}(M) \in \mathbb{R}$ represents a signed weight associated with $M$. We stress that the limiting result holds only when $\upbeta$ is small, but the surface sum in \eqref{eq:surf-sum} is formally defined for all $\upbeta\in\RR$ and is shown to be convergent for $\upbeta$ sufficiently small.

The above surface sum representation has many attractive features (simple weights, planar maps, cf.\ the discussion after \cite[Theorem 1.6]{borga2024surfacesumslatticeyangmills} of how the large-$N$ surface sum compares to previous surface representations). However, as the weights are still signed and the sum is infinite, to illuminate new facts about Wilson loop expectations, a refined study of the surface sum needs to be performed.

Focusing now on two dimensions, it is well known that two-dimensional Yang--Mills is special, in that many exact formulas are known, and in a sense, the model is completely integrable, see e.g.\  \cite{Driver1989, GKS1989, Sengupta1997,  dahlqvist2016free, lévy2012masterfieldplane,  DHK2017, DHK2017b, Basu:2016dnp, Driver2019, Lemoine2022, DL2023, SZZ2024, DL2025}. Due to this complete integrability, we expect the planar map model from \cite{borga2024surfacesumslatticeyangmills} to simplify greatly in two dimensions, in the sense that there should be a vast number of cancellations between surfaces. That is, one should be able to identify many subsets of surfaces $\cM'$ such that \[\sum_{M\in \cM'}w_{\upbeta}(M)=0,\] 
where we emphasize that the weights $w_{\upbeta}(\cdot)$ are \textit{signed}. Such cancellations are also expected to occur in higher dimensions, albeit in an even more intricate manner, making the two-dimensional setting the natural starting point.
However, even in dimension two, the nature and origin of these cancellations remain quite opaque when viewed directly from the definition of the surface model.

In the present paper, we develop techniques to find these cancellations between surfaces in the case of dimension $d = 2$. As a first example of this, we show that all but finitely many of the surfaces from the surface sums in \eqref{eq:surf-sum} (those from \cite{borga2024surfacesumslatticeyangmills}) cancel each other out. This is our first main result, which we state informally as follows. For the precise statement, see Theorem \ref{thm: erasable loops have one plaquette assignment}.

\begin{thm}[\textsc{informal}]\label{thm:main-1-informal}
Let $d = 2$ and $\upbeta\in \RR$. For any loop $\ell$ the surface sum $\phi(\ell)$  in \eqref{eq:surf-sum} from our previous paper \cite{borga2024surfacesumslatticeyangmills} is actually a finite sum. Moreover, we can explicitly determine which of the summands may be nonzero. 
\end{thm}

We emphasize here that Theorem \ref{thm:main-1-informal} applies to any loop. Beyond just providing an understanding of which terms may possibly contribute to the surface sum, Theorem \ref{thm:main-1-informal} also allows us to explicitly compute the surface sums in many new cases. As an example, we informally state our next main result. For the precise statement, see Theorem \ref{thm: wound simple loops WLE}.

\begin{thm}[\textsc{informal}]\label{thm:main-2-informal}
Let $d = 2$ and $\upbeta\in \RR$. Let $\ell$ be a simple loop and let $a_\ell = \area(\ell)$ (which may be defined as the number of plaquettes enclosed by $\ell$). For $n \in \NN$, let $\ell^n$ be the loop $\ell$ wound around itself $n$ times. Then the surface sum $\phi(\ell^n)$ is explicitly given by
\[ \frac{(-1)^{n+1}}{n}\binom{na_{\ell}-2}{n-1} \upbeta^{n a_\ell} .\]
\end{thm}

Similarly, our final main result provides explicit formulas for the surface sums corresponding to loops with at most three self-crossings. See Table~\ref{table} for the explicit formulas.

\begin{thm}[\textsc{informal}]\label{thm:main-3-informal}
Let $d = 2$ and $\upbeta\in \RR$. For any loop $\ell$ with at most three self-crossings, we can explicitly compute the surface sum $\phi(\ell)$.
\end{thm}

Recalling from the discussion above \eqref{eq:surf-sum} that, for small $\upbeta$, the surface sums $\phi(\ell)$ coincide with the large-$N$ limit of Wilson loop expectations, a direct consequence of our main results in Theorems \ref{thm:main-1-informal}–\ref{thm:main-3-informal} is that, in this regime, we can compute Wilson loop expectations for the large-$N$ lattice Yang--Mills theory in many new cases.

We briefly discuss the proof ideas for our main results, leaving a more detailed discussion to Section \ref{subsec:proof-techniques and ML}. The main new contribution is to design an exploration algorithm (or ``peeling process'') that explores the surface one face at a time. At each step of the algorithm, we have to carefully choose the next edge of the boundary to explore at, in order to preserve various key properties at every step of the process. Eventually, after enough exploration steps, we will have arrived at a simple enough boundary that can be computed explicitly. 

We view our work as a first step in understanding how to find cancellations in the surface models of \cite{cao2023random, borga2024surfacesumslatticeyangmills}. While the cancellations we introduce in this paper provide a detailed description of the surface sum when $d=2$ and $N=\infty$, in order to continue the program initiated in \cite{cao2023random} and apply the surface approach towards Yang--Mills in more general settings, e.g.\  for dimension $d > 2$ and finite $N$, a more refined understanding of cancellations seems to be needed. See e.g.\  \cite[Remark 1.7]{CNS2025} which discusses why a better understanding of cancellations would lead to an improvement on the main result of \cite{CNS2025}. We hope that the techniques developed in the present paper may inform future approaches.

Next, we briefly comment on how our results relate to the existing literature. Previously (cf.\ \cite[Theorem 2.11]{Basu:2016dnp}) it was known that the Wilson loop expectation for any loop is a finite polynomial in $\upbeta$. However, due to the nature of the methods used, the degree and coefficients of the polynomial (which depend on the loop) were not explicitly known. Thus, \cref{thm:main-1-informal} provides a refined understanding of Wilson loop expectations; in particular, it details exactly which summands in the surface sum may be non-zero and their corresponding weights. We will apply this refined understanding to derive our second and third main results, Theorems \ref{thm:main-2-informal} and \ref{thm:main-3-informal}.

Also, \cref{thm:main-2-informal} was previously only known in the cases of a simple loop wrapped once (i.e.\ $\{n=1, a_{\ell}\in \NN\}$) and of a single plaquette wrapped $n\in \NN$ times (i.e.\ $\{n\in \NN, a_{\ell}=1\}$) (cf.\ \cite[Theorem 2.8, Theorem 4.1]{Basu:2016dnp}). We also note that \cref{thm:main-2-informal} will be used to obtain the limiting (as $N \rightarrow \infty$) empirical spectral measure of the random matrix $Q_\ell$, where $Q_{\ell}$ is distributed according to the $\unitary(N)$ lattice Yang--Mills theory (all notation will be defined in Section \ref{subsec: lattice Yang—Mills}). We give a formula for the limiting spectral density of any simple loop in Section \ref{sect:main-res-2}; see, in particular, \eqref{eq: mua}. Previously, only the single plaquette (i.e.\ $a_{\ell}=1$) case was known (cf.\ \cite[Proposition 4.2]{Basu:2016dnp}).

Next, an additional quick comment on proof techniques. \cite{Basu:2016dnp} proves many of their results using free probability in an essential way. In particular, in principle, our second main result, Theorem \ref{thm:main-2-informal}, could have also been obtained using Basu-Ganguly's result on the one-plaquette case \cite[Proposition 4.1]{Basu:2016dnp} and free probability (in particular, the notions of free independence and free convolution). While we are not opposed to using free probability as a key input, we were more interested in trying to prove everything purely in terms of the analysis of surface sums. 

This approach aligns with the overarching philosophy of the paper: to develop a deep understanding of surface sum cancellations, with an eye toward future applications to Yang--Mills theory. Moreover, this method appears to be more broadly applicable; even with free probability, we do not see a clear path to generalize \cite[Proposition 4.1]{Basu:2016dnp} to obtain~\cref{thm:main-3-informal}.

To close this section, we outline the remainder of the paper. We begin by introducing lattice Yang--Mills theory in Section~\ref{subsec: lattice Yang—Mills} and then present the relevant results from \cite{borga2024surfacesumslatticeyangmills} in Section~\ref{sec: Wilson loop expectations as surface sum large N}. Next, we formally state our main results in \cref{sec: Main results} and then introduce the master loop equation and our main proof techniques in \cref{subsec:proof-techniques and ML}. Finally, \cref{sec: intro} comes to a close in \cref{sec: trivial} with some discussion about taking a continuum limit of our model and future directions. \cref{sec: expectations in 2D} is devoted to proving that the large-$N$ surface sum has a finite number of non-zero terms (i.e.\ proving \cref{thm:main-1-informal}). Then, in \cref{sec: height plaquette assignment coefficent} we present how to explicitly compute the surface sum for simple loops wound around $n$ times and loops with three or fewer self-crossings (i.e.\ prove \cref{thm:main-2-informal} and \cref{thm:main-3-informal}). \cref{sec: height plaquette assignment coefficent} also details the convergence of the empirical spectral distribution for simple loops.

\paragraph{Acknowledgments.}~We thank Scott Sheffield for many helpful discussions. J.B.\ was partially supported by the NSF under Grant No.\ DMS-2441646. S.C.\ was partially supported by the NSF under Grant No.\ DMS-2303165. 

\subsection{Two-dimensional Lattice Yang--Mills theory}\label{subsec: lattice Yang—Mills}

Let $\Lambda$ be some finite subgraph of $\ZZ^2$. We consider the set of oriented nearest neighbor edges in $\Lambda$ which we denote by $E_{\Lambda}$. We say that an edge $e\in E_{\Lambda}$ is positively oriented if the endpoint of the edge is greater than the initial point of the edge in lexicographical ordering. Let $E^+_{\Lambda}$ denote the set of positively oriented edges in $\Lambda$. For an edge $(u,v)=e\in E_{\Lambda}$, we let $e^{-1}=(v,u)$ denote the reverse direction. Whenever we consider a lattice edge, we always assume that such an edge is oriented, unless otherwise specified.

We will denote the group of $N\times N$ unitary matrices by $\unitary(N)$. The $\unitary(N)$-lattice Yang--Mills theory assigns a random matrix from $\unitary(N)$ to each oriented edge in $E_{\Lambda}$. We stipulate that this assignment must have \textbf{edge-reversal symmetry}. That is, if $Q_e$ is the matrix assigned to the oriented edge $e$ then $Q_{e^{-1}}=Q^{-1}_{e}$. 

\begin{remark}
    We defined the $\unitary(N)$-lattice Yang--Mills model; however, since the large-$N$ limit of lattice Yang--Mills is the same for $\unitary(N), \SO(N)$, and $\SU(N)$ (see, for instance, the discussion at the end of \cite[Section 1.2]{borga2024surfacesumslatticeyangmills}) the results in this paper hold for any of the three groups.
\end{remark}

We call an oriented cycle of edges $\ell$ a \textbf{loop}.\footnote{We stress that our definition of loop is different from the one in \cite{chatterjee_rigorous_2019}. Indeed, in the latter, work loops are defined so that they do not contain backtracks (see \eqref{eq:backt} for a definition).}
We call the \textbf{null-loop} the loop with no edges, and denote it by $\emptyset$. For a loop $\ell=e_1e_2\hdots e_n$, we let $Q_{\ell} = Q_{e_1}Q_{e_2}\hdots Q_{e_n}$. We say that $\ell$ is a \textbf{simple loop} if the endpoints of all the $e_i$'s are distinct. We call a \textbf{string} (of cardinality $n$) any multiset of loops $\{\ell_1,\dots,\ell_n\}$ and denote it by $s$. 

Define $\cP_{\Lambda}$ to be the collection of simple loops consisting of four edges (i.e.\ oriented squares). We call such loops \textbf{plaquettes}. We say that a plaquette $p\in \cP_{\Lambda}$ is positively oriented if its leftmost edge is oriented upwards. Let $\cP^+_{\Lambda}$ denote the set of positively oriented plaquettes in $\Lambda$. For $p\in \cP_{\Lambda}$, we let $p^{-1}$ denote the plaquette containing the same edges as $p$ but with opposite orientations. Whenever we consider a plaquette, we always assume that it is oriented, apart from when we explicitly say that we are considering its unoriented version.

We say that a loop $\ell$ has a \textbf{backtrack} if two consecutive edges of $\ell$ correspond to the same edge in opposite orientations. That is, $\ell$ has a backtrack if it is of the form 
\begin{equation}\label{eq:backt}
\ell=\pi_1 \, e  \, e^{-1}  \, \pi_2,
\end{equation} 
where $\pi_1$ and $\pi_2$ are two (possibly empty) paths of edges and $e\in E_{\Lambda}$. Notice for such a loop, we can remove the $e\,e^{-1}$ backtrack and obtain a new loop $\pi_1 \, \pi_2$. We say that a loop $\ell$ is a \textbf{non-backtrack loop} if it has no backtracks. Throughout this work, unless explicitly stated, a loop might have backtracks.
We say that a loop $\ell$ is a \textbf{trivial} loop if, after removing all backtracks from $\ell$, the resulting loop is the null-loop $\emptyset$. In particular, the null-loop is trivial, but the converse is not necessarily true.

For an $N \times N$ matrix $Q$, define the \textbf{normalized trace} to be
\begin{equation*}
\tr(Q) := \frac{1}{N}\Tr(Q),
\end{equation*}
where $\Tr(Q)$ is the sum of the diagonal elements in $Q$, i.e.\ $\Tr(Q) = \sum_{i=1}^N Q_{i,i}$. Let $\cQ = (Q_e)_{e\in E^{+}_{\Lambda}}$ denote a \textbf{matrix configuration}, that is, an assignment of matrices to each positively oriented edge of $\Lambda$. The lattice Yang--Mills measure (with \emph{Wilson} action) is a probability measure for such matrix configurations. In particular, 
\begin{equation}\label{eq: lattice YM measure-2}
\hat{\mu}_{\Lambda,N,\beta} (\cQ):= \hat{Z}_{\Lambda,N, \beta}^{-1}\cdot\left(\prod_{p\in \cP_{\Lambda}}\exp\Big(\beta \cdot \,\Tr(Q_p)\Big)\right)\prod_{e\in E^+_{\Lambda}}dQ_{e},
\end{equation}
where $\hat{Z}_{\Lambda, N, \beta}$ is a normalizing constant\footnote{The normalizing constant $\hat{Z}_{\Lambda, N, \beta}$ is finite because $\unitary(N)$ is a compact Lie group.} (to make $\hat{\mu}_{\Lambda,N,\beta}$ a probability measure), $\beta\in \RR$ is a parameter (often called the inverse temperature), and each $dQ_e$ denotes the Haar measure on $\unitary(N)$. We highlight that in \eqref{eq: lattice YM measure-2} we are considering both positively and negatively oriented plaquettes in $\cP_{\Lambda}$.

Since the goal of our paper is to consider the large-$N$ limit, we prefer to use the following more convenient rescaling, replacing $\beta$ by $\upbeta N$,
\begin{equation}\label{eq: lattice YM measure}
\mu_{\Lambda,N,\upbeta} (\cQ):= Z_{\Lambda,N, \upbeta}^{-1}\cdot\left(\prod_{p\in \cP_{\Lambda}}\exp\Big(\upbeta N\cdot\Tr(Q_p)\Big)\right)\prod_{e\in E^+_{\Lambda}}dQ_{e}.
\end{equation}

\begin{remark}\label{remark: beta factor of two}
In previous work, for instance \cite{chatterjee_rigorous_2019, chatterjee20161n,Basu:2016dnp}, $\beta$ in \eqref{eq: lattice YM measure-2} is simply replaced by $\beta N$ in \eqref{eq: lattice YM measure}, without any distinction between $\beta$ and $\upbeta$.
Moreover, the $\beta$ and $\upbeta$ terms appearing throughout this paper differ from those in previous work by a factor of $2$. That is, where we have $\beta$ or $\upbeta$, previous works would have $\beta/2$. This is because we are considering both positively and negatively oriented plaquettes. 
\end{remark}

The primary quantities of interest in lattice Yang--Mills are the Wilson loop observables. These observables are defined in terms of a matrix configuration $\cQ$ and a string $s$.  In particular, we define \textbf{Wilson loop observables} as (note the normalized trace)
\begin{equation}\label{eq: Wilson loop observables}
W_s(\cQ) :=\prod_{\ell\in s}\tr(Q_{\ell}).
\end{equation}
Wilson loop observables are invariant under backtrack erasure, that is, $W_{\pi_1 \, e  \, e^{-1}  \, \pi_2}(\cQ)=W_{\pi_1 \pi_2}(\cQ)$. Moreover, the Wilson loop observable for the null-loop is defined to be $1$ for any matrix configuration $\cQ$, that is, $W_{\emptyset}(\cQ)=1$. 

One of the fundamental questions of Yang--Mills theory is to understand the expectation of Wilson loop observables with respect to the lattice Yang--Mills measure; see the survey~\cite{chatterjee2018yangmillsprobabilists} for further discussion. We denote this expectation by
\begin{align*}
\phi_{\Lambda, N, \upbeta}(s) := \EE_{\mu_{\Lambda,N,\upbeta}}[W_s(\cQ)].
\end{align*}
The rest of this paper is devoted to understanding the expectation of Wilson loop observables in the specific case when $N$ tends to infinity.

\begin{remark}
As in \cite{borga2024surfacesumslatticeyangmills} and \cite{cao2023random}, we define the Wilson loop observable with respect to the normalized trace. While this choice contrasts with some previous works, for instance \cite{chatterjee_rigorous_2019, chatterjee20161n,Basu:2016dnp}, this scaling will be natural in the large-$N$ limit.
\end{remark}

\subsection{Large-$N$ Wilson loop expectations as surface sums}\label{sec: Wilson loop expectations as surface sum large N}

In this section, we introduce the relevant results from our companion paper \cite{borga2024surfacesumslatticeyangmills}, restricted to the two-dimensional case, and present some additional related notation needed for this work. We call a function $K:\cP_{\ZZ^2}\to\ZZ_{\geq 0}$ a \textbf{plaquette assignment}.

\begin{thm}[{\cite[Theorem 3.4]{borga2024surfacesumslatticeyangmills}}]\label{thm: sum over surfaces representation in 't hooft limit}
There exists a number $\upbeta_0>0$ such that the following is true. Let $\Lambda_1\subseteq\Lambda_2\subseteq\hdots$ be any sequence of finite subsets of the lattice $\ZZ^2$ such that $\ZZ^2 = \cup_{N=1}^{\infty}\Lambda_N$. If $|\upbeta|\leq \upbeta_0$, then for any string $s=\{\ell_1,\dots,\ell_n\}$,
	\begin{align}\label{eq:factor}
		\lim_{N\to\infty}\phi_{\Lambda_N,N,\upbeta}(s) =
		\prod_{i=1}^n\phi(\ell_i),
	\end{align}
	where
	\begin{equation}\label{eq:def-phi_K}
		\phi(\ell)=\sum_{K:\cP_{\ZZ^2}\to\ZZ_{\geq 0}} \phi^K(\ell), \quad\text{with}\quad \phi^K(\ell)=\sum_{M\in\npe(\ell,K)}\upbeta^{\area(K)}w_{\infty}(M).
	\end{equation}
	Here, $\area(K)= \sum_{p\in \cP}K(p)$, the set $\npe(\ell,K)$ is a certain collection of connected planar maps with the disk topology and $w_{\infty}(M)$ is a specific (positive or negative) weight associated with each map $M$.\footnote{We do not need the precise (and somewhat lengthy) definitions of the set $\npe(\ell, K)$ and the weights $w_{\infty}(M)$ in this paper; interested readers can refer to the original paper for details. See also the discussion around \eqref{eq:mle-2} for further explanation of why the precise definitions of the set $\npe(\ell, K)$ and the weights $w_{\infty}(M)$ are not needed here.}

The infinite sum $\phi(\ell)$ is only over finite plaquette assignments $K:\cP_{\ZZ^2}\to\ZZ_{\geq 0}$, i.e.\ plaquette assignments such that $\sum_{p\in\cP_{\ZZ^2}}K(p)<\infty$. Moreover, the infinite sum $\phi(\ell)$ is absolutely convergent.\footnote{The sum $\phi^K(\ell)$ is a finite sum for all $K$ and $\ell$.}
\end{thm}

For the rest of this paper, we will be working only on the entire lattice $\ZZ^2$ instead of a finite subgraph of $\ZZ^2$. Thus, to simplify notation, we set 
\[\cP=\cP_{\ZZ^2}, \quad \cP^+=\cP_{\ZZ^2}^+, \quad E= E_{\ZZ^2} \quad \text{and} \quad E^+ = E^+_{\ZZ^2}.\]

It will be important for the result in this paper to recall how $\phi^K(\ell)$ is defined for some ``base'' cases. To do this, we recall some terminology from \cite{borga2024surfacesumslatticeyangmills}. For a loop $\ell$ and plaquette assignment $K$, let $n_e(\ell)$ denote the number of times the oriented edge $e\in E$ appears in $\ell$ and $n_e(\ell,K)$ denote the number of times $e$ appears in $(\ell,K)$, in the sense that 
\begin{align}\label{defn:ne}
n_e(\ell) := \text{\# copies of $e$ in $\ell$}\quad \text{and}\quad n_e(\ell,K) := \text{\# copies of $e$ in $\ell$} + \sum_{p\in \cP(e)}K(p),
\end{align}
where $\cP(e)$ is the collection of plaquettes in $\cP$ containing $e$ as one of the four boundary edges (with the correct orientation). We say that the pair $(\ell,K)$, sometimes called a \textbf{loop plaquette assignment}, is \textbf{balanced} if 
\begin{equation}\label{eq:jobfowerbf}
n_e(\ell,K) = n_{e^{-1}}(\ell,K),\quad\text{for all $e\in E^+$}.
\end{equation}
With this, we recall how $\phi^K(\ell)$ is defined in the following ``base'' cases:
\begin{enumerate}
    \item If $(\ell,K)$ is not balanced, then $\phi^K(\ell)=0$;
    \item If $\ell$ is a non-trivial loop and $K=0$ (i.e.\ $K(p)=0$ for all $p\in \cP$), then $\phi^K(\ell) = 0$; 
    \item\label{def: phi k empty = 0} If $K\neq 0$, then $\phi^K(\emptyset)=0$; 
    \item\label{def: phi k empty = 1} If $K=0$, then $\phi^K(\emptyset)=1$.
\end{enumerate}
We refer the reader to the discussion below~\cite[Equation (3.2)]{borga2024surfacesumslatticeyangmills} for some more explanations on why these turn out to be the correct definitions.

Lastly, we often write $\phi^K(\ell)$ as (recall \eqref{eq:def-phi_K}),
\begin{align*}
    \phi^K(\ell)=c(\ell,K)\upbeta^{\area(K)},
\end{align*}
where 
\begin{equation}\label{eq:coefficents}
    c(\ell,K)\coloneqq\sum_{M\in\npe(\ell,K)}w_{\infty}(M).
\end{equation}This is an important formulation as many of our results below will concern the value of the coefficients $c(\ell,K)$.

\subsection{Main results}\label{sec: Main results}

\subsubsection{Preliminary definitions: distance, height and canonical plaquette assignments for loops}\label{sec:prelim-def}

To state our main results, we need to introduce some terminology. First, we define two important notions, the height (\cref{def: height}) and distance (\cref{def: lattice distance}) of a loop. We say that two plaquettes are \textbf{adjacent} if they share exactly one unoriented edge.

\begin{defn}\label{def: height}
	Given a loop $\ell$, we define its \textbf{height} $h_{\ell}:\cP \to \ZZ$ as follows (notice $\ell$ must be contained in $B= [-N,N]^2$ for some $N\in \NN$):
    \begin{enumerate}
		\item $h_{\ell}(p)=0$ for $p\in \cP^+\cap B^c$.
		\item If for $p\in \cP^+$, we have that $h_{\ell}(p) = x$, then we define $h_{\ell}(q)$ for $q\in \cP^+$ adjacent to $p$ as follows: If $\ell$ crosses the oriented line from the center of $p$ to the center of $q$ from right to left $n\geq 0$ times and crosses the line from the left to the right $m\geq 0$ times, set $h_{\ell}(q) = x+n-m$.
	\end{enumerate}
    For $p\in\cP\sm \cP^+$, we define $h_{\ell}(p) = h_{\ell}(p^{-1})$.
    For a string $s=\{\ell_1,\dots,\ell_n\}$, we define $h_s(p) = \sum_{i\in [n]}h_{\ell_i}(p)$. 
\end{defn}

See the left-hand side of~\cref{fig: loop regions} for an example.

\begin{remark}
The height is well defined because for every vertex $v$ on the lattice, the loop $\ell$ must enter and exit $v$ the same number of times (since it is a loop). Thus, the sum of the height over the four plaquettes containing $v$ is $0$ and so the order in which we construct the height $h_{\ell}$ using item 2 above is not relevant. We also remark that one can view the height as a discrete notion of the winding number of a loop around a point.
\end{remark}

Next, we introduce the notion of distance of a loop, after defining some preliminary terminology. Given two adjacent plaquettes $p,q$, let $\mathfrak{e}(p,q)$ denote the unoriented edge shared by $p$ and $q$. For a loop $\ell$ and an unoriented lattice edge $\mathfrak{e}$, let $n_{\mathfrak{e}}(\ell)$ denote the number of copies of $\mathfrak{e}$ in $\ell$ in either orientation, that is, 
\[n_{\mathfrak{e}}(\ell) \coloneqq n_{e}(\ell) + n_{e^{-1}}(\ell),\] 
for $e$ one of the orientations of $\mathfrak{e}$. 

A \textbf{path of plaquettes} is a (possibly infinite) collection of positively oriented plaquettes $\{p_i\}_i$ such that no plaquette appears in $\{p_i\}_i$ more than once and $p_{i}$ is adjacent to $p_{i+1}$ for all $i$. For any plaquette $p\in \cP^+$, let $\mathfrak{P}(p)$ denote the set of all infinite paths of plaquettes $\{p_i\}_{i}$ starting at $p$, i.e.\ with $p_1=p$. 

\begin{defn}\label{def: lattice distance}
        For a loop $\ell$, define its \textbf{distance} (to infinity) $d_{\ell}: \cP \to \ZZ_{\geq 0}$ as follows:
        \begin{equation}\label{eq: dist def}
            d_{\ell}(p) = \min_{\{p_i\}_{i}\in \mathfrak{P}(p)}\,\,\,\sum_{i=1}^{\infty}n_{\mathfrak{e}(p_i,p_{i+1})}(\ell), \quad\text{for all }p\in\cP^{+}.
        \end{equation}
        For $p\in\cP\sm \cP^+$, we define $d_{\ell}(p) = d_{\ell}(p^{-1})$. For a string $s=\{\ell_1,\dots,\ell_n\}$, we define $d_s(p) = \sum_{i\in [n]}d_{\ell_i}(p)$ for all $p\in \cP$. 
        
        We will often refer to an infinite path of plaquettes $\{p_i\}_{i}\in \mathfrak{P}(p)$ that achieves the minimum in \eqref{eq: dist def} as a \textbf{distance-achieving path}.
\end{defn}

See again the left-hand side of~\cref{fig: loop regions} for an example.

\medskip

Another important notion will be that of regions. Note that each plaquette $p$ (which is an oriented loop) naturally identifies a closed $(1 \times 1)$-subsquare of the plane. A \textbf{plaquette-region} is a collection of  plaquettes such that the corresponding subset of the plane is connected. We take the convention that each plaquette-region $R$ contains both orientations of each plaquette contained in $R$. The \textbf{area} of a plaquette-region $R$, denoted by $\area(R)$, is equal to half of the number of plaquettes it contains.

In a few cases, when we refer to a plaquette-region, we actually mean the corresponding connected subset of the plane (but this will always be clear from the context). 

Given a loop~$\ell$, the \textbf{plaquette-regions of $\ell$}  are the maximal plaquette-regions separated by $\ell$, i.e.\ the maximal collection of plaquettes corresponding to a maximal connected subset in the complement of $\ell$ in the plane; see the right-hand side of~\cref{fig: loop regions}. Note that if $p$ and $q$ are in the same plaquette-region of $\ell$, then $h_{\ell}(p)= h_{\ell}(q)$ and $d_{\ell}(p)= d_{\ell}(q)$. We refer to the unique plaquette-region of $\ell$ with distance zero as the \textbf{exterior plaquette-region} and to all other plaquette-regions as \textbf{interior plaquette-regions}. We also refer to the complement of the exterior plaquette-region (i.e.\ the union of the interior regions) as the \textbf{support} of $\ell$, which we denote by $\supp(\ell)$. Finally, we set
\begin{equation}\label{eq:area-supp}
    \area(\ell) \coloneqq \area(\supp(\ell)).
\end{equation}

\begin{figure}[ht!]
\begin{center}
	\includegraphics[width=\textwidth]{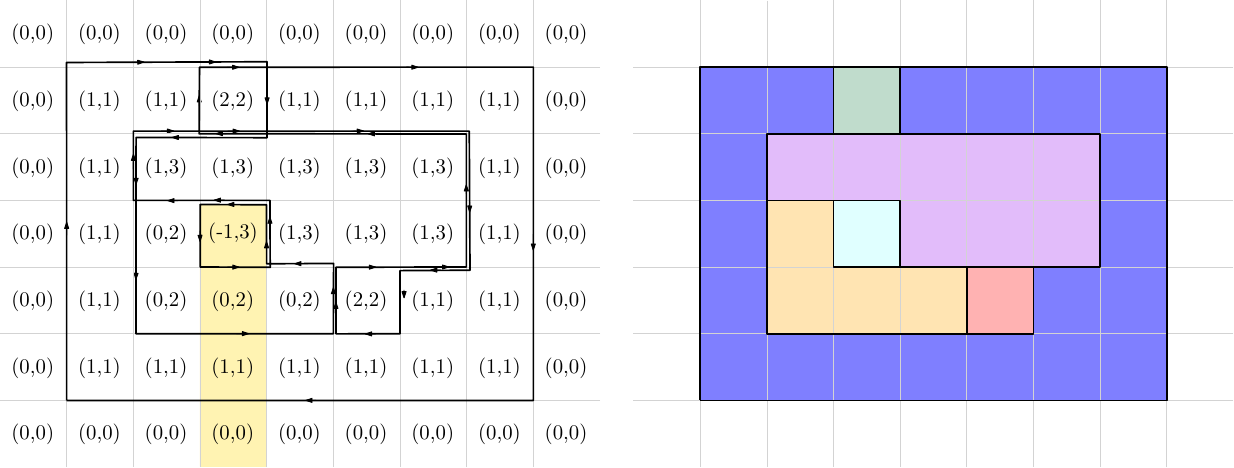}  
	\caption{\label{fig: loop regions}
    \textbf{Left:} A loop $\ell$ of $\ZZ^2$ drawn with non-overlapping edges for graphical convenience. Each plaquette $p$ is labeled with a pair of integers $(h_\ell(p),d_\ell(p))$, where $h_\ell$ is the height and $d_\ell$ is the distance of $\ell$ introduced in Definitions~\ref{def: height}~and~\ref{def: lattice distance}, respectively. The golden path of plaquettes starting at the plaquette labeled by $(-1,3)$ is a distance achieving path for that plaquette. Conversely, the path of plaquettes starting at the same plaquette and going to the right is not a distance achieving path as it crosses five edges.
    \textbf{Right:} The same loop $\ell$ of $\ZZ^2$ is now drawn with overlapping edges. Each interior plaquette is marked with a color denoting the plaquette-region it belongs to. All the remaining plaquettes, uncolored in the figure, form the exterior plaquette-region. Following \cref{defn:collection-plaquette-ass}, each $K \in \cK_{\ell}$ equals $K_{\ell}$ in the blue, red, and green plaquette-regions, as the height coincides with the distance there. In contrast, in the orange, light cyan, and purple plaquette-regions, the height disagrees with the distance, in particular $\frac{d_{\ell}(p)-|h_{\ell}(p)|}{2}=1$ in all these plaquette-regions. Thus, in these plaquette-regions, $K$ equals $K_{\ell}$ plus a possible “layer of positively and negatively oriented plaquettes'' for each plaquette-region.
    }
\end{center}
\vspace{-3ex}
\end{figure}

The final ingredient we need to state our main result is a correspondence between loops and some special collections of plaquette assignments. We associate with a loop $\ell$ its \textbf{height plaquette assignment} $K_{\ell}$, defined as follows:
\begin{align}\label{defn:master-plaq-ass}
	\forall p\in \cP^+,\qquad\left(K_{\ell}(p), K_{\ell}(p^{-1})\right) := \begin{cases}
		(0,|h_{\ell}(p)|) & \text{if $h_{\ell}(p)\geq 0$},\\ (|h_{\ell}(p)|,0)& \text{if $h_{\ell}(p)< 0$}.
	\end{cases}
\end{align}
Note that $(\ell, K_{\ell})$ is balanced. From this special plaquette assignment, we associate a special collection of plaquette assignments to $\ell$.

\begin{defn}\label{defn:collection-plaquette-ass}
Given a non-backtrack loop $\ell$, the \textbf{canonical collection} $\cK_{\ell}$ of plaquette assignments associated with $\ell$ is the collection of finite plaquette assignments of the form $K = K_{\ell} + K'$ where:
\begin{enumerate}
    \item\label{Kl p1} $K'$ is balanced, that is,  $K'(p) = K'(p^{-1})$ for all $p\in \cP$;
    \item\label{Kl p2} $K'$ is constant in each region of $\ell$, that is,  for each plaquette-region $R$ of $\ell$, \[K'(p)= K'(q),\quad \text{ for all } p,q\in R\cap \cP^+;\] 
    \item\label{Kl p3} $K'$ is bounded by $\tfrac{1}{2}(d_{\ell}-|h_{\ell}|)$, that is,
    $K'(p)\leq \frac{d_{\ell}(p)-|h_{\ell}(p)|}{2}$, for all $p\in\cP^+$.
\end{enumerate}
If $\ell$ has backtracks and $\ell'$ is the non-backtrack version of $\ell$, then $\cK_{\ell}\coloneqq\cK_{\ell'}$.
\end{defn}

Read the caption of the right-hand side of~\cref{fig: loop regions} for an example.

\begin{remark}\label{rem:card-set-cK}
Recall that $d_{\ell}$ and $h_{\ell}$ are constant in each region $R$ of $\ell$ and so, taking one plaquette $p$ of $R$, we can set $d_{\ell}(R) := d_{\ell}(p)$ and $h_{\ell}(R) := h_{\ell}(p)$.
With this notation, we observe that $\cK_{\ell}$ consists of all plaquette assignments obtained from $K_{\ell}$ by adding, one by one,  at most 
\[\text{$\tfrac{d_{\ell}(R)-|h_{\ell}(R)|}{2}$ ``layers'' of positively and negatively oriented plaquettes on each plaquette-region $R$ of $\ell$.}\] 
In particular, if $R$ is the exterior plaquette-region, $d_{\ell}(R)-|h_{\ell}(R)|=0$, and so every $K\in K_{\ell}$ is finite and supported inside the loop $\ell$. Moreover, $\cK_{\ell}$ has cardinality 
\[\prod_{R}\left(\tfrac{d_{\ell}(R)-|h_{\ell}(R)|}{2} + 1\right),\] 
where the product ranges over all the regions $R$ of $\ell$. In particular, the cardinality of $\cK_{\ell}$ only depends on the isotopy class of the loop $\ell$ (see the discussion at the end of this section and \cref{table} for more details).

As a consequence, for any loop $\ell$ such that $d_{\ell}=|h_{\ell}|$, we have that $\cK_{\ell}=\{K_{\ell}\}$. Note that this is the case for simple loops. One can actually show that the loops $\ell$ such that $d_{\ell}=|h_{\ell}|$ are exactly the loops that winds around each point always in the same direction (but we omit the details here since this is not necessary for our arguments). 
\end{remark}

\subsubsection{Main result 1: Wilson loop expectations as a finite explicit sum over canonical plaquette assignments}\label{sect:main-res-1}

We now state our first main result (recall the statement of \cref{thm: sum over surfaces representation in 't hooft limit} for the notation).

\begin{thm}\label{thm: erasable loops have one plaquette assignment}
    Let $\upbeta\in \RR$. For any non-trivial loop $\ell$, \begin{equation}\label{eq: WLE sum}
        \phi(\ell) = \sum_{K\in \cK_{\ell}}c(\ell, K)\cdot\upbeta^{\area(K)},
    \end{equation}
    where we recall that:
    \begin{itemize}
        \item $\cK_{\ell}$ is the canonical collection $\cK_{\ell}$ of plaquette assignments associated with $\ell$ introduced in \cref{defn:collection-plaquette-ass};
        \item $c(\ell, K) = \sum_{M\in \npe(\ell,K)}w_{\infty}(M)$ are the coefficients (independent of $\upbeta$) introduced in \eqref{eq:coefficents};
        \item $\area(K)= \sum_{p\in \cP}K(p)$.
    \end{itemize}
    \noindent More specifically, we have that 
    \[\phi^K(\ell)=0, \quad\text{for all } K\notin \cK_{\ell}.\]
\end{thm}

Note that \eqref{eq: WLE sum} is a finite sum (in contrast to the infinite sum in \eqref{eq:def-phi_K}). We also emphasize that our result makes no assumptions on the loop $\ell$. 
In the special case when $\ell$ satisfies $d_{\ell}(p) = |h_{\ell}(p)|$ for all $p \in \cP$ (for instance, when $\ell$ is simple, or in some other cases), the sum in \eqref{eq: WLE sum} simplifies, thanks to \cref{rem:card-set-cK}, to
\begin{equation*}
        \phi(\ell) = c(\ell, K_{\ell})\cdot\upbeta^{\area(K_{\ell})}.
\end{equation*}

We conclude this section by showing that the situation in higher dimensions is more involved.

\begin{remark}\label{rem: Infinite contributing plaquette assignments in dimension three}

    The following discussion assumes some familiarity with the surface sum perspective developed in \cite{borga2024surfacesumslatticeyangmills}, which we have not included in this paper, as this remark and \cref{rmk: Surface cancellations} are the only places where it is required.

    Let $p$ be a plaquette of $\ZZ^3$. We claim that there are an infinite number of plaquette assignments $K$ such that $\phi^K(p)\neq 0$. To see this, consider any cuboid of $\ZZ^3$ containing the plaquette $p$ on its boundary; some examples are shown in \cref{fig: cuboid}. Let $\cC$ denote the infinite set of such cuboids. 

    For each cuboid $C \in \cC$, let $K_C$ denote the unique plaquette assignment such that:
    \begin{itemize}
        \item $(p, K_C)$ is balanced;
        \item for each plaquette $q$ on the boundary of $C$, exactly one of $K(q)$ or $K(q^{-1})$ equals 1, and the other equals 0;
        \item $K(q) = 0$ for all other plaquettes $q$ of $\ZZ^3$.
    \end{itemize} 
     
    Lastly, for each $C\in \cC$, we claim that $c(p,K_C) = 1$. Indeed, note that each lattice edge in $(p,K_C)$ appears twice, once in each orientation. Thus, there is only one map in $\npe(p,K_C)$, that is, the map having a blue 2-gon between each two copies of the same edge. In particular, this map has $w_{\infty}$-weight $1$. 
\end{remark}

\begin{figure}[ht!]
    \begin{center}
            \includegraphics[width=.4\textwidth]{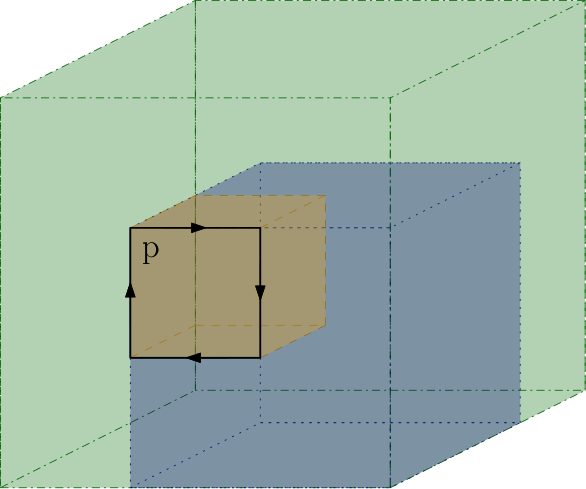} 
    	\caption{\label{fig: cuboid}An oriented plaquette $p$ in dimension three (shown in black) along with some cuboids (orange, blue, and green) that contain $p$.}
    \end{center}
    \vspace{-3ex}
\end{figure}

\subsubsection{Main result 2: Convergence for the empirical spectral measure of simple loops to an explicit limiting density}\label{sect:main-res-2}

Thanks to \cref{thm: erasable loops have one plaquette assignment}, to explicitly compute $\phi(\ell)$, it is enough to compute the coefficients $c(\ell,K)$  for all $K\in \cK_{\ell}$. The next result (and the one in \cref{sect:main-res-3}) illustrates that such computations are often tractable.

\begin{thm}\label{thm: wound simple loops WLE}
    Let $\upbeta\in \RR$. Suppose $\ell$ is a simple loop and set $a_{\ell}=\area(\ell)$. For $n\in \NN$, let $\ell^n$ denote the loop $\ell$ wound around itself $n$ times. Then,
    \begin{equation}\label{eq: cn(a)}
        c(\ell^n, K_{\ell^n}) = \frac{(-1)^{n+1}}{n}\binom{na_{\ell}-2}{n-1},
    \end{equation}
    where, when $a_{\ell}=1$, the above expression should be interpreted as $c(\ell^n, K_{\ell^n}) =\mathds{1}_{n=1}$.
    In particular, since $\cK_{\ell^n} = \{K_{\ell^n}\}$ thanks to \cref{rem:card-set-cK}, we deduce that
    \begin{align}\label{eq: cn(a)_2}
        \phi(\ell^n) = \frac{(-1)^{n+1}}{n}\binom{na_{\ell}-2}{n-1}\upbeta^{na_{\ell}}.
    \end{align}
\end{thm}
When $n=1$, we recover the result $\phi(\ell) =\upbeta^{a_{\ell}}$ for simple loops previously established in \cite[Theorem 2.8]{Basu:2016dnp}. We stress the interesting fact that $\phi(\ell^n)$ only depends on the area of $\ell$ and not on the exact shape of the loop.

\begin{remark}\label{rmk:cont-parallel}
It is interesting to note the similarity between our formula in \eqref{eq: cn(a)_2} and the corresponding expressions for the two-dimensional large-$N$ continuum Yang--Mills theory introduced in \cite{lévy2012masterfieldplane}.
In particular, for a simple continuum loop $\cL$ of (Euclidean) area $\alpha_{\cL}$, let $\cL^n$ be the loop winding $n$-times around $\cL$.  The Wilson loop expectation of $\cL^n$ in L\'evy's two-dimensional large-$N$ continuum Yang--Mills theory is given by (see for instance \cite[Theorem 2.3]{dahlqvist2016free})
    \begin{equation}\label{eq:levy}
        \varphi(\cL^n) =\left(\sum_{k=0}^{n-1}\frac{(-\alpha_{\cL})^k}{k!}n^{k-1}\binom{n}{k+1}\right)\mathrm{e}^{-(n\alpha_{\cL})/2}.
    \end{equation}
Note that replacing $\upbeta$ by $\mathrm{e}^{-1/2}$ in \eqref{eq: cn(a)_2} and identifying $a_\ell$ with $\alpha_\cL$, we get that \eqref{eq: cn(a)_2} and \eqref{eq:levy} are identical for $n=1$ and $n=2$, and equal at the first order as $a_\ell\to\infty$ for all $n\geq 3$. Further discussion of this analogy will be provided in \cref{remark: similar to levy} and \cref{sec: trivial}.
\end{remark}

\cref{thm: wound simple loops WLE} allows us to determine the limiting spectral density for simple loops. 
For $n,a\in \NN$, set 
\begin{equation}\label{eq:defn-cna}
    c_n(a) :=  \frac{(-1)^{n+1}}{n}\binom{na-2}{n-1},
\end{equation}
i.e.\ the coefficients appearing in \eqref{eq: cn(a)}.
We also introduce for all $a\in\NN$, $|\upbeta|\leq 1/2$,  the function
\begin{equation}\label{eq: mua}
        f_{a,\upbeta}(x) = \frac{1}{2\pi}\left(1 +2 \sum_{n=1}^{\infty}c_n(a) \cos(nx)\upbeta^{na}\right), \qquad x\in[0,2\pi).
\end{equation}
The sum above can be explicitly computed for small values of $a$, for instance see \eqref{eq:ieuvbfdowehbd}, \eqref{eq:ieuvbfdowehbd2} and \eqref{eq:ieuvbfdowehbd3} below. In particular, in these cases it is pointwise convergent for all $\upbeta\in\RR$.
The next lemma shows that $f_{a, \upbeta}$ is well-defined under the condition $|\upbeta| \leq 1/2$. 
    
\begin{lem}\label{lem:cov-series}
For all $a\in\NN$ and\footnote{The threshold $|\upbeta|\leq 1/2$ captures the largest possible regime of validity for \cref{cor: limiting spectral density}; see \cref{remark:critical}.} $|\upbeta|\leq 1/2$, the sum in \eqref{eq: mua} is absolutely convergent for all $x\in[0,2\pi)$.
\end{lem}

Recall that for a loop $\ell=e_1e_2\hdots e_n$, we denoted by $Q_{\ell}$ the product $ Q_{e_1}Q_{e_2}\hdots Q_{e_n}$ of matrices along $\ell$ in the $\unitary(N)$-lattice Yang--Mills model.

\begin{cor}\label{cor: limiting spectral density}
    Let $\upbeta_0>0$ be as in \cref{thm: sum over surfaces representation in 't hooft limit} and assume that $|\upbeta|\leq \upbeta_0\wedge 1/2$. Suppose $\ell$ is a simple loop and set $a_{\ell}=\area(\ell)$. Then $f_{a_\ell, \upbeta}$ is a probability density (in particular, it is non-negative\footnote{While the non-negativity of $f_{a_{\ell},\upbeta}$ is only shown for $|\upbeta|\leq \upbeta_0 \wedge 1/2$, we believe it should be true up to $|\upbeta|\leq 1/2$, but we will not pursue this direction in the present paper.}), and moreover, the following holds. Let $\mu_{\ell,N}$ denote the empirical spectral distribution of $Q_{\ell}$, that is,
    \[\mu_{\ell,N}(x)=\frac{1}{N}\sum_{i=1}^N \delta_{\lambda_{\ell,N}^i}(x),\] 
    where $(\lambda_{\ell,N}^i)_i$ are the complex-valued unitary eigenvalues of the unitary matrix $Q_{\ell}$ and $ \delta_y(\cdot)$ is the Dirac delta function at $y$.
    Then, 
    \begin{align*}
        \mu_{\ell,N} \xrightarrow[N\to\infty]{} \mu_{a_\ell},\qquad\text{weakly in probability,}
    \end{align*}
     where $\mu_{a_\ell}$ is a probability measure on the unit circle with density $f_{a_{\ell},\upbeta}(x)$  for all $x\in[0,2\pi)$ (under the standard parameterization of the unit circle).
\end{cor}

For instance, if $a=1$, i.e.\ $\ell$ is just a single plaquette, then we recover \cite[Proposition 4.2]{Basu:2016dnp}, which gives the limiting spectral density

\begin{equation}\label{eq:ieuvbfdowehbd}
        f_{1,\upbeta}(x) = \frac{1}{2\pi}\Big(1 + 2\upbeta\cos(x)\Big),
\end{equation}
while for $a=2$, we obtain
\begin{equation}\label{eq:ieuvbfdowehbd2}
        f_{2,\upbeta}(x) = \frac{1}{2\pi}\sqrt{\frac{1+4\upbeta^2\cos(x)+\sqrt{1+8\upbeta^2\cos(x)+16\upbeta^4}}{2}},
\end{equation}
and for $a=3$, we obtain
\begin{equation}\label{eq:ieuvbfdowehbd3}
        f_{3,\upbeta}(x) = 
        \frac{-1 + 2 \cosh\left(\tfrac{2}{3} \mathrm{arcsinh}\left(\tfrac{3 \sqrt{3}}{2} \upbeta^{3/2} \mathrm{e}^{-\frac{ix}{2}}\right)\right) + 
 2 \cosh\left(\tfrac{2}{3} \mathrm{arcsinh}\left(\tfrac{3 \sqrt{3}}{2} \upbeta^{3/2} \mathrm{e}^{\frac{ix}{2}}\right)\right)}{6 \pi}.
\end{equation}
These densities, and some others, are shown in \cref{fig-limiting-spectral-density}.

\begin{figure}[ht!]
\begin{center}
        \includegraphics[width=.24\textwidth]{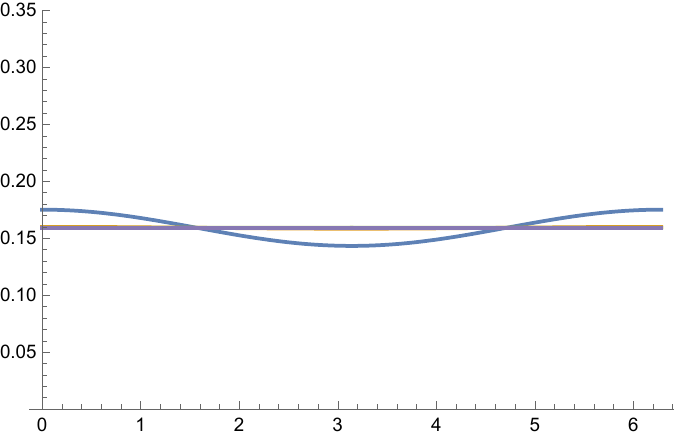} 
        \includegraphics[width=.24\textwidth]{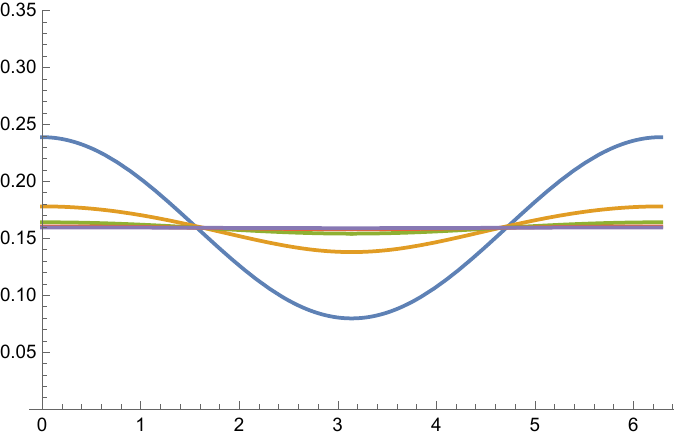} 
        \includegraphics[width=.24\textwidth]{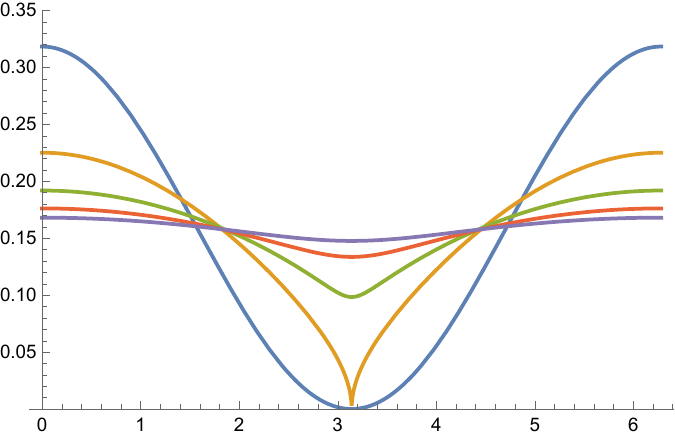} 
	\includegraphics[width=.24\textwidth]{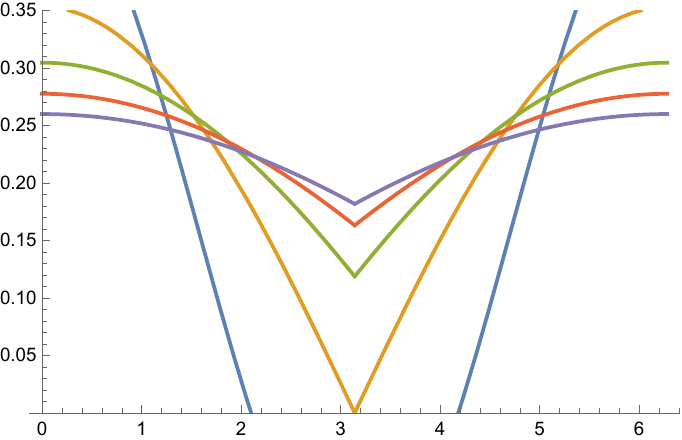}  
	\caption{\label{fig-limiting-spectral-density} 
        In each diagram, the densities $f_{1,\upbeta}(x), f_{2,\upbeta}(x), f_{3,\upbeta}(x), f_{4,\upbeta}(x), f_{5,\upbeta}(x)$ for $x \in [0, 2\pi)$ are shown in blue, orange, green, red, and purple, respectively. From left to right, the values of $\upbeta$ are $0.05$, $0.25$, $0.5$ and $1$, respectively. Note that $\upbeta=1$ is outside the regime of validity of \cref{lem:cov-series} and \cref{cor: limiting spectral density}, but will be useful for the discussion in \cref{sec: trivial}.}
\end{center}
\vspace{-3ex}
\end{figure}

\begin{remark}\label{remark:critical}
    \cref{thm: sum over surfaces representation in 't hooft limit} was proved in \cite{borga2024surfacesumslatticeyangmills} only for $\upbeta$ small. Here, we proved in \cref{thm: wound simple loops WLE} that, in dimension two, the Wilson loop expectation $\phi(p)$ of a plaquette $p$ is exactly $\upbeta$. Since we have that  $|\phi_{\Lambda_N,N,\upbeta}(p)|  \leq 1$ for all $N\geq 1$ and $\upbeta\in \mathbb R$, Theorem~\ref{thm: sum over surfaces representation in 't hooft limit} can hold only if $\upbeta\leq 1$. Thus, it would be interesting to understand the value $\upbeta^*$ at which the conclusion of Theorem \ref{thm: sum over surfaces representation in 't hooft limit} first fails to hold.

    In~\cref{cor: limiting spectral density}, we also proved that for $\upbeta$ small and a plaquette $p$, the empirical spectral distribution of $Q_{p}$  converges in probability as $N\to\infty$ to a deterministic measure on the circle with density (under the standard parameterization of the unit circle) 
    \begin{equation}\label{eq: muaone}
        f_{1,\upbeta}(x) = \frac{1}{2\pi}\Big(1 + 2\upbeta\cos(x)\Big).
    \end{equation}
    Note that the above function can be a density only if $\upbeta\leq 1/2$ (otherwise, the function $f_{1,\upbeta}(x)$ is negative around $x=\pi$).
    For this reason (the paper \cite{HIAI200071} seems to suggest the same conclusion), we suspect, at least in dimension two, that  $\upbeta^* = \frac{1}{2}$, and it would be nice to confirm whether this is true.

    Nevertheless, if one is only interested in loops of large size (for instance, this is the case for the purpose of understanding scaling limits (see \cref{sec: trivial})) we do not exclude that the critical value $\upbeta^*$ for the validity of \cref{thm: sum over surfaces representation in 't hooft limit} might be 1.
\end{remark}

\subsubsection{Main result 3: Explicit Wilson loop expectations for all loops with three or fewer self-crossings}\label{sect:main-res-3}

In fact, Wilson loop expectations for more general loops can be explicitly computed. For instance, we are able to compute the Wilson loop expectations for all loops with three or fewer self-crossings.

Let $\cL : [0,1] \to \mathbb{R}^2$ be a time-parametrized continuum loop of the plane $ \mathbb{R}^2$.
A point $x \in \mathbb{R}^2$ is called a \textbf{self-crossing point} of the loop $\cL$ if there exist distinct times $t_1, t_2 \in [0,1]$ such that $\cL(t_1) = \cL(t_2) = x$, and the corresponding tangent vectors are not parallel: $\cL'(t_1) \notin \text{span}\{\cL'(t_2)\}.$ Moreover, we assume that $\cL(t)\neq x$ for all $t\in[0,1]\setminus\{t_1, t_2\}$, i.e.\ a self-crossing point is a double point.
In words, the loop $\cL$ passes through the point $x$ at exactly two different times and crosses itself. We say that a loop has $n$ self-crossing points if it has $n$ \emph{distinct} self-crossing points and all the other points of the loop are simple, i.e.\ only visited once.

Given a lattice loop $\ell$ (possibly with self-crossings) and a continuum loop $\cL$ (possibly with self-crossings) we say that $\ell$ is (ambient) \textbf{isotopic} to $\cL$ if  there exists a continuous family of homeomorphisms
\[
\Phi_t : \mathbb{R}^2 \to \mathbb{R}^2, \quad t \in [0,1],
\]
such that:
\begin{enumerate}
    \item Each $ \Phi_t $ is a homeomorphism.
    \item $ \Phi_0 = \mathrm{id}_{\mathbb{R}^2}$ (the identity map).
    \item $ \ell = \Phi_1 \circ \cL $.
\end{enumerate}
Note that this kind of isotopy preserves the embedded structure of the loop, including the number of self-crossings and the number of regions.

\cref{table} shows, in the left column, all continuum loops with three or fewer self-crossings (up to isotopy), in the middle column, the cardinality of the canonical collection $\cK_{\ell}$ of plaquette assignments for $\ell$ any lattice loop isotopic to the continuum loop in the same row, and in the right column, the Wilson loop expectation for such a lattice loop. More precisely, the right part of each row of the table states that any lattice loop isotopic to the continuum loop shown on the left, and whose plaquette-regions have areas matching the labels of the corresponding continuum regions, has Wilson loop expectation as indicated in the right column.

\afterpage{%
\begin{center}
    \begin{longtable}{| >{\centering\arraybackslash} m{45pt} | >{\centering\arraybackslash}m{30pt} | > {\centering\arraybackslash} m{0.8\textwidth} |}
     \hline
     $\cL$ & $\#\cK_{\ell}$ & $\phi(\ell)$\\ 
     \specialrule{.1em}{.05em}{.05em} 
     \includegraphics[height=37pt]{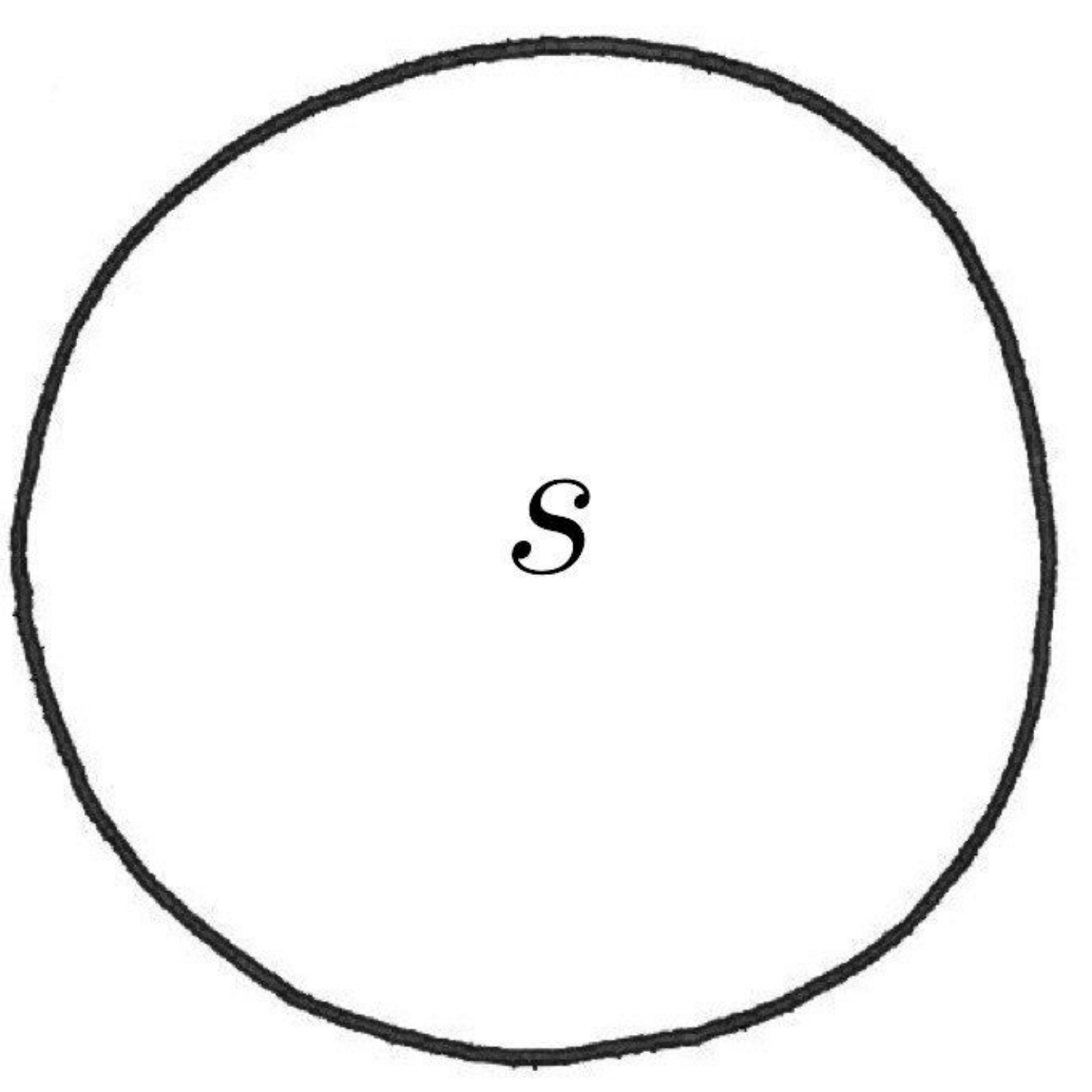} & 1 &  $\upbeta^s$ \\
     \hline
     \includegraphics[height=37pt]{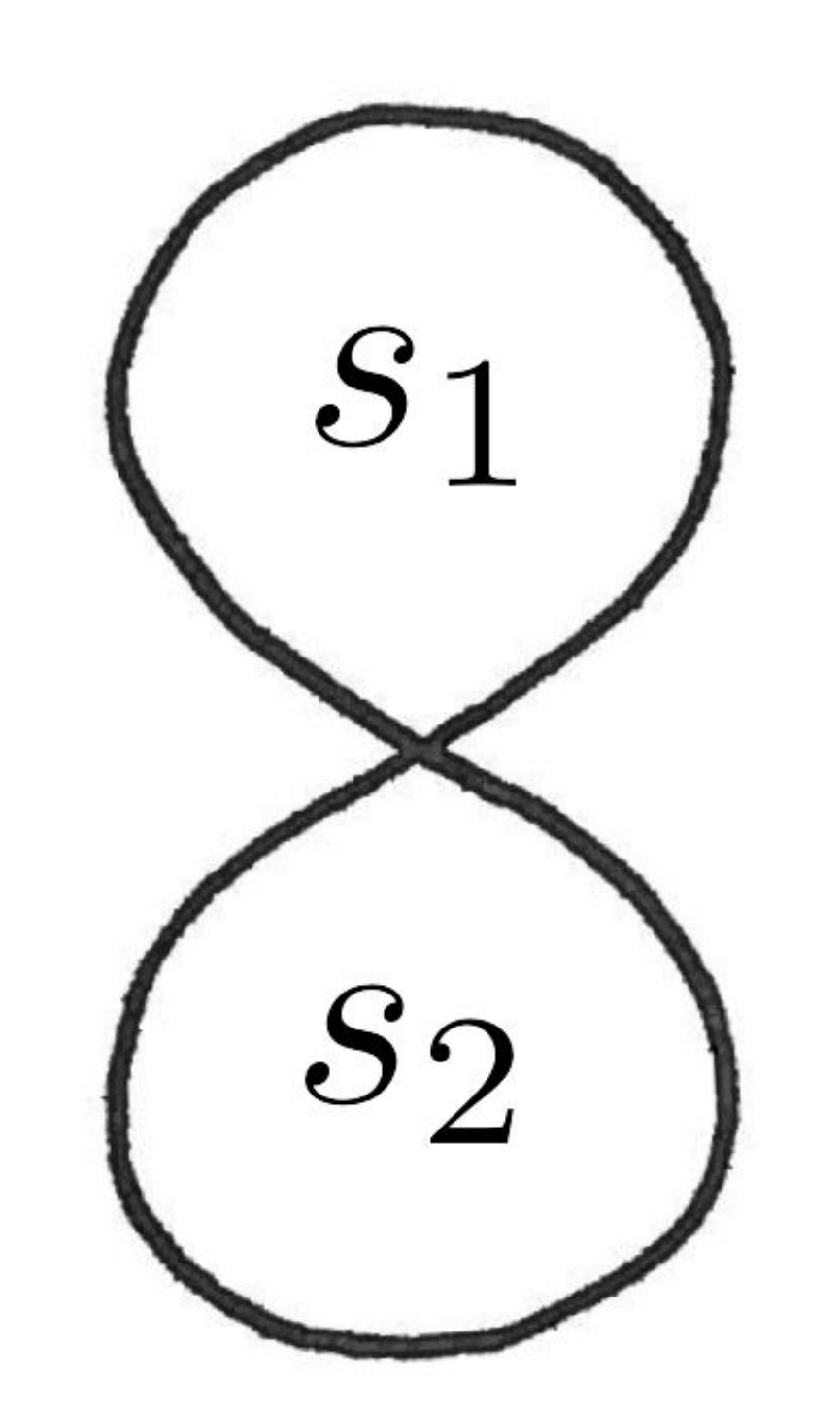} & 1 & $\upbeta^{s_1+s_2}$ \\
     \hline
     \includegraphics[height=37pt]{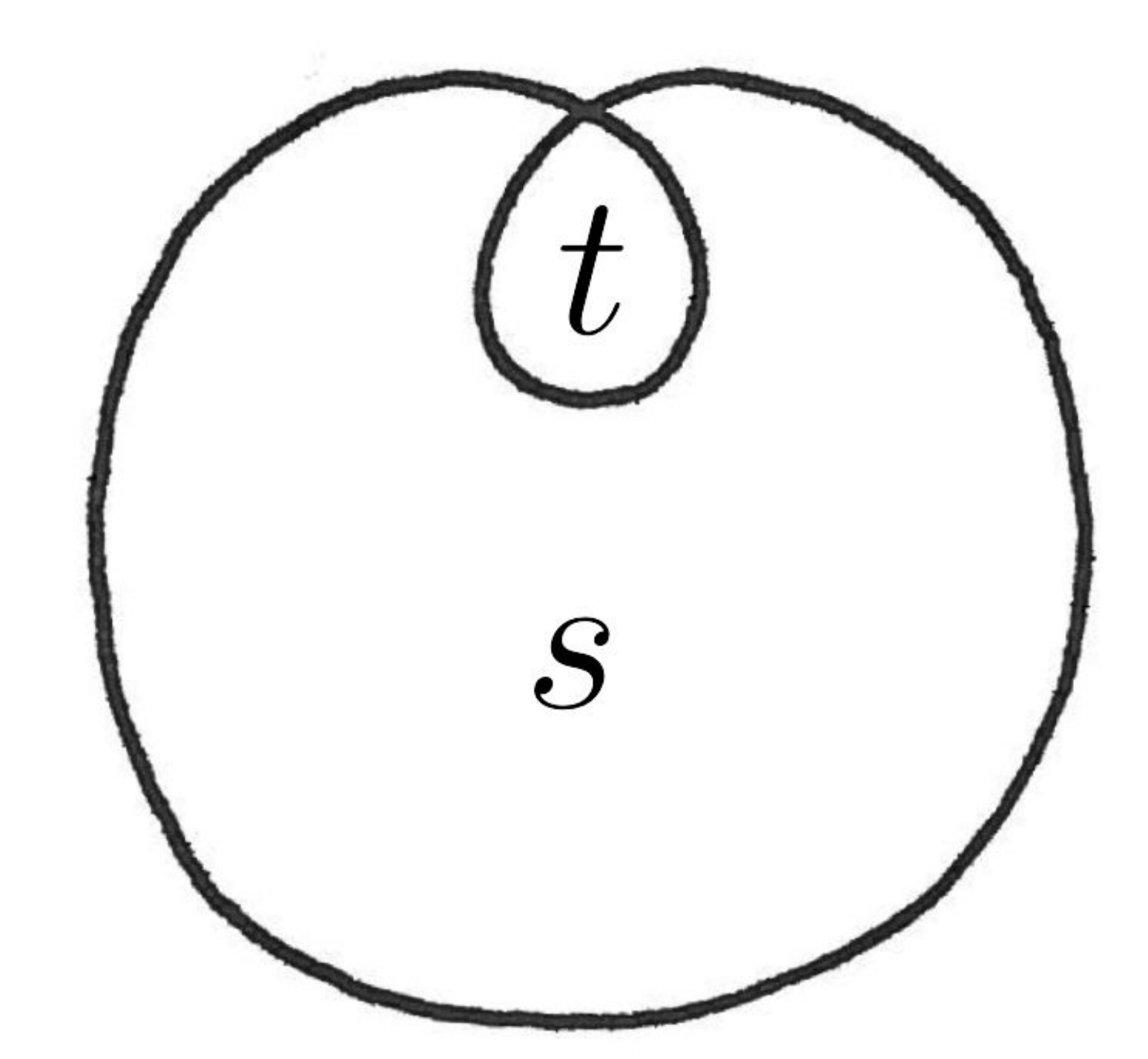} & 1 & $(1-t)\upbeta^{s+2t}$ \\
     \hline
     \includegraphics[height=37pt]{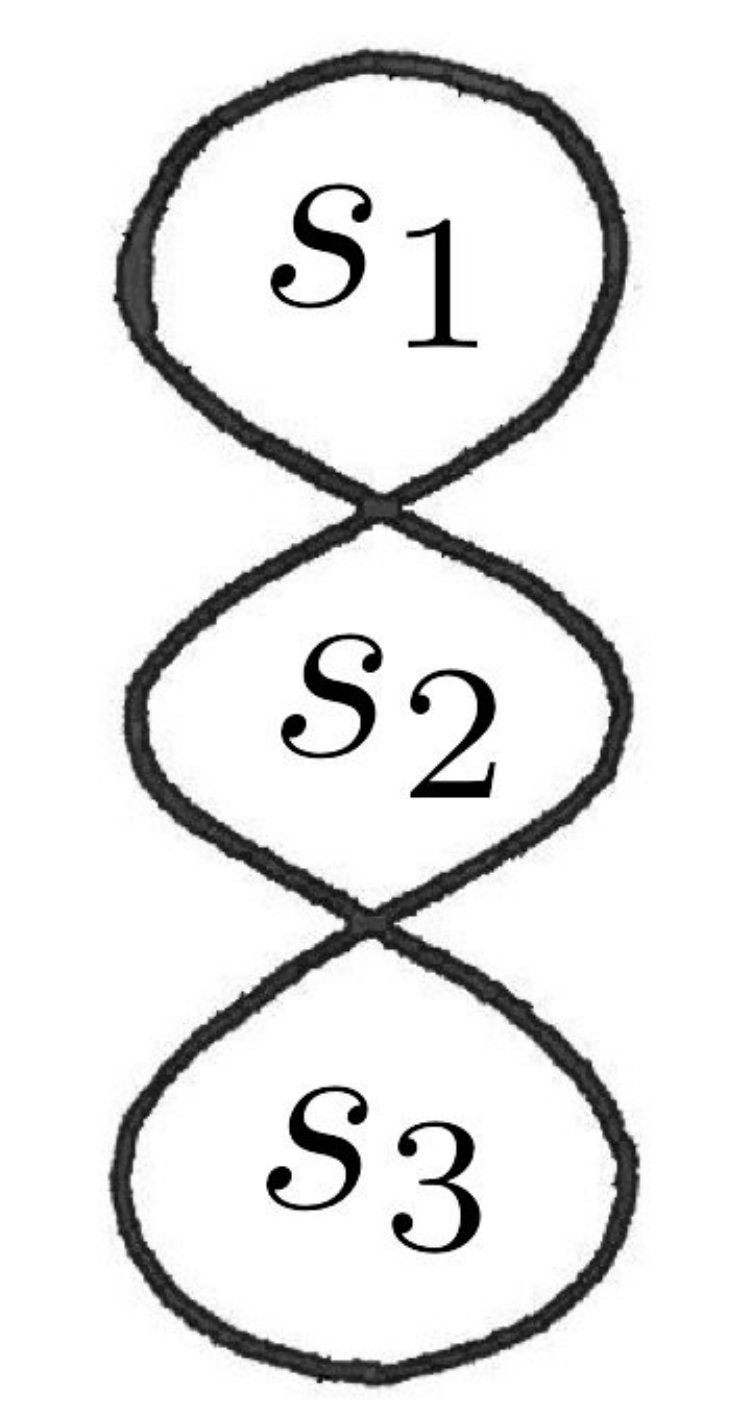} & 1 & $\upbeta^{s_1+s_2+s_3}$ \\
     \hline
     \includegraphics[height=37pt]{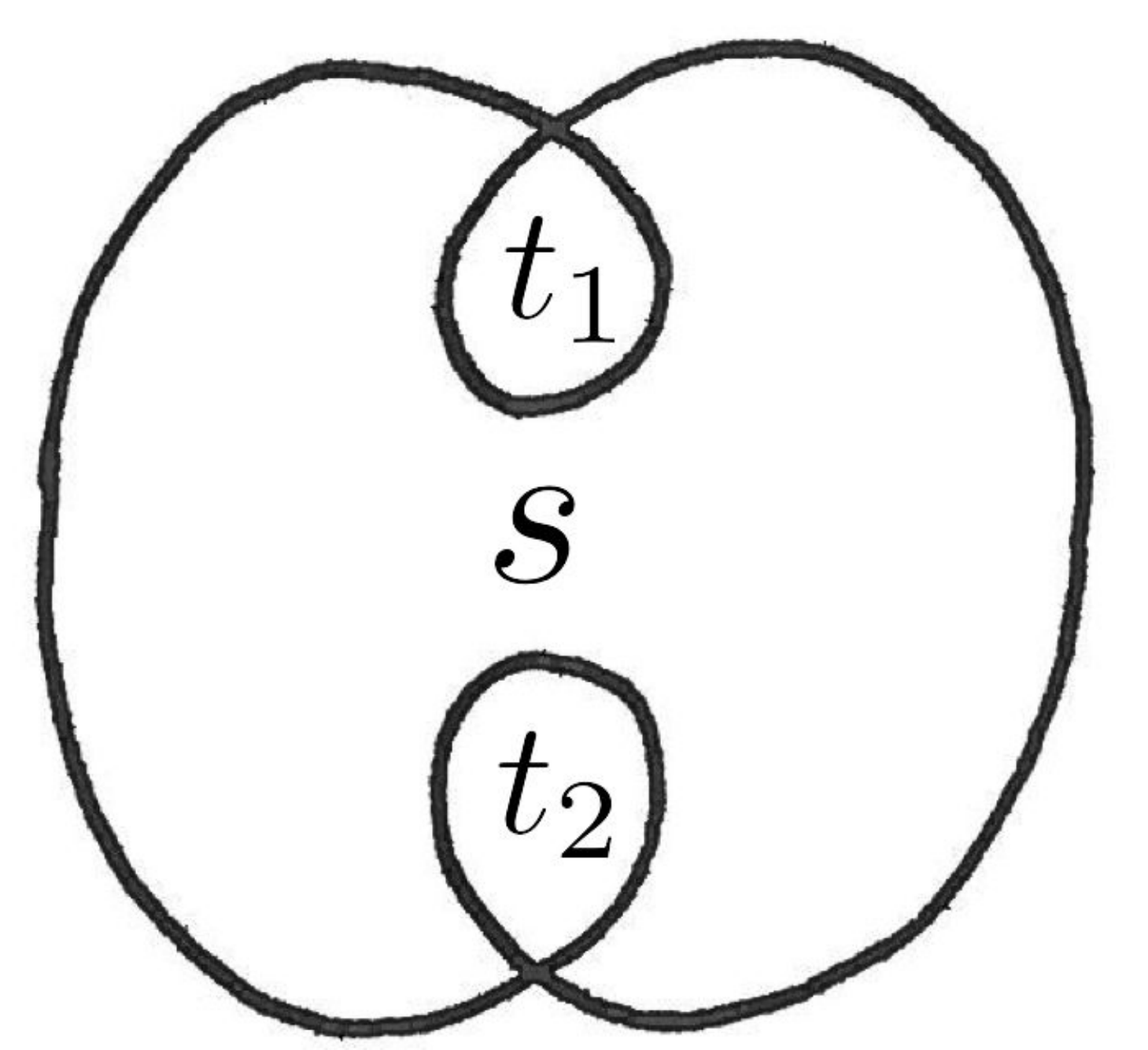} & 1 & $(1-t_1)(1-t_2)\upbeta^{s+2t_1+2t_2}$ \\
     \hline
     \includegraphics[height=37pt]{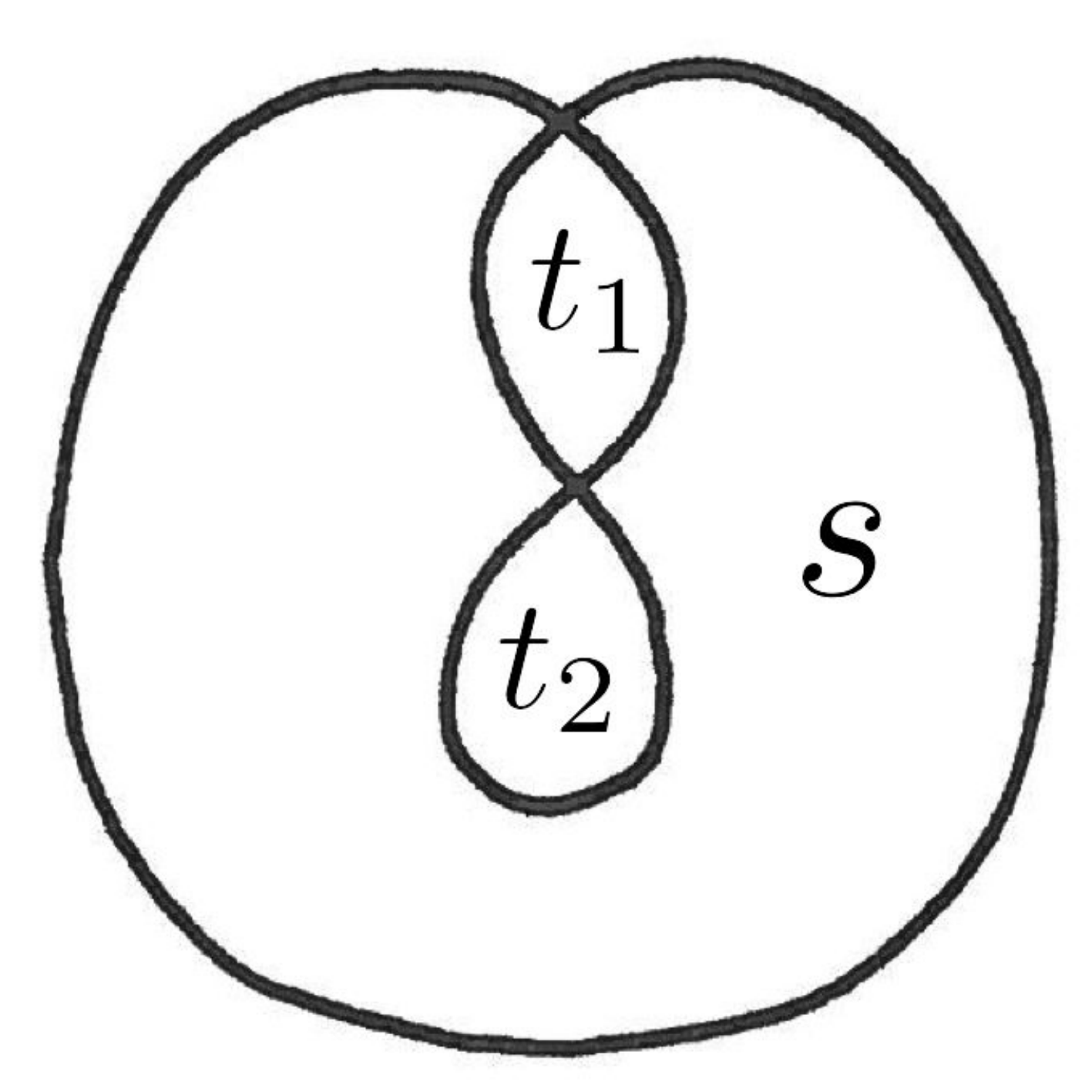} & 2 & $\left(1-t_1\upbeta^{2t_2}\right) \upbeta^{s+2t_1}$ \\
     \hline
     \includegraphics[height=37pt]{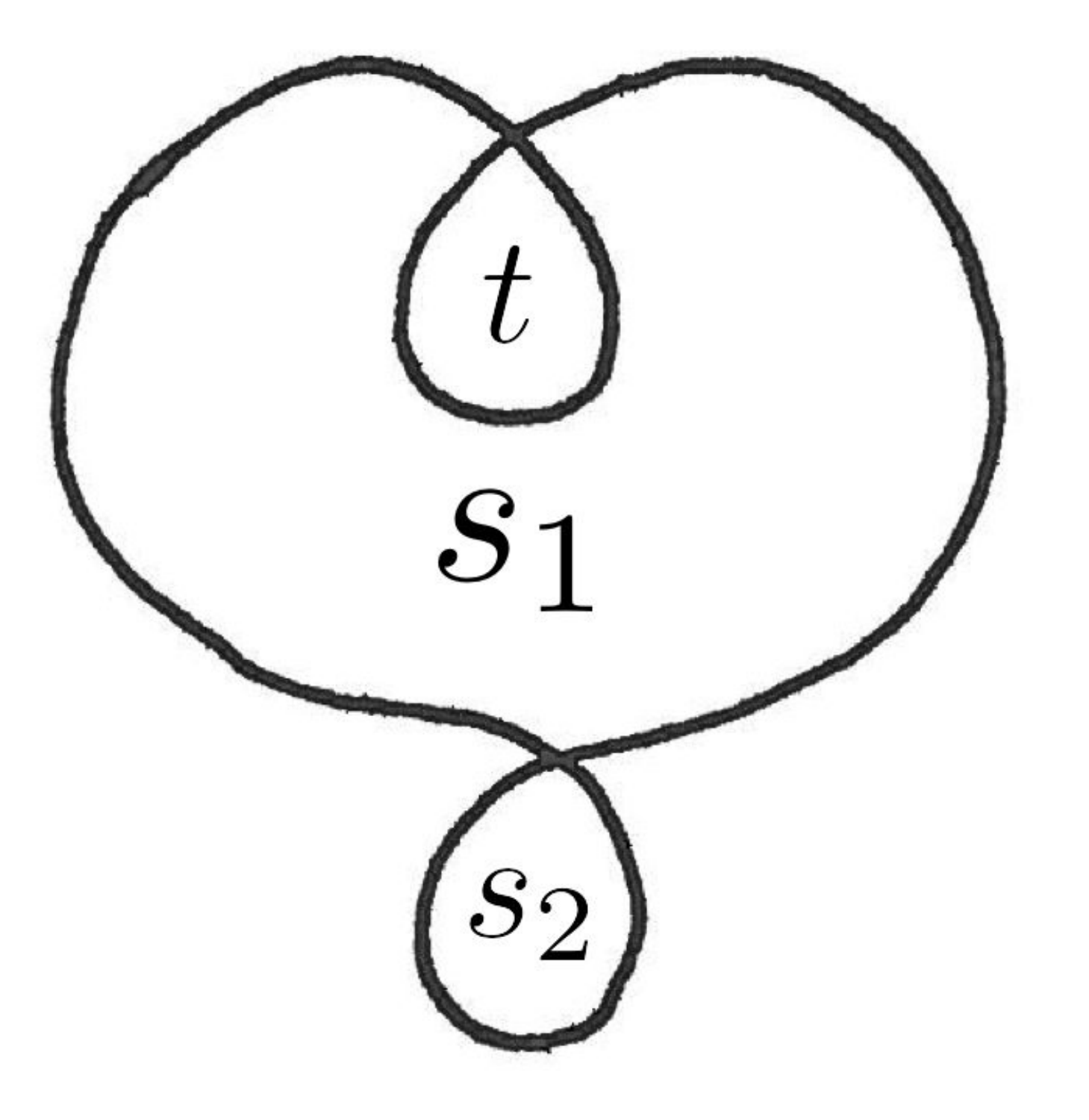} & 1 & $(1-t_1)\upbeta^{s_1+s_2+2t}$ \\
     \hline
     \includegraphics[height=37pt]{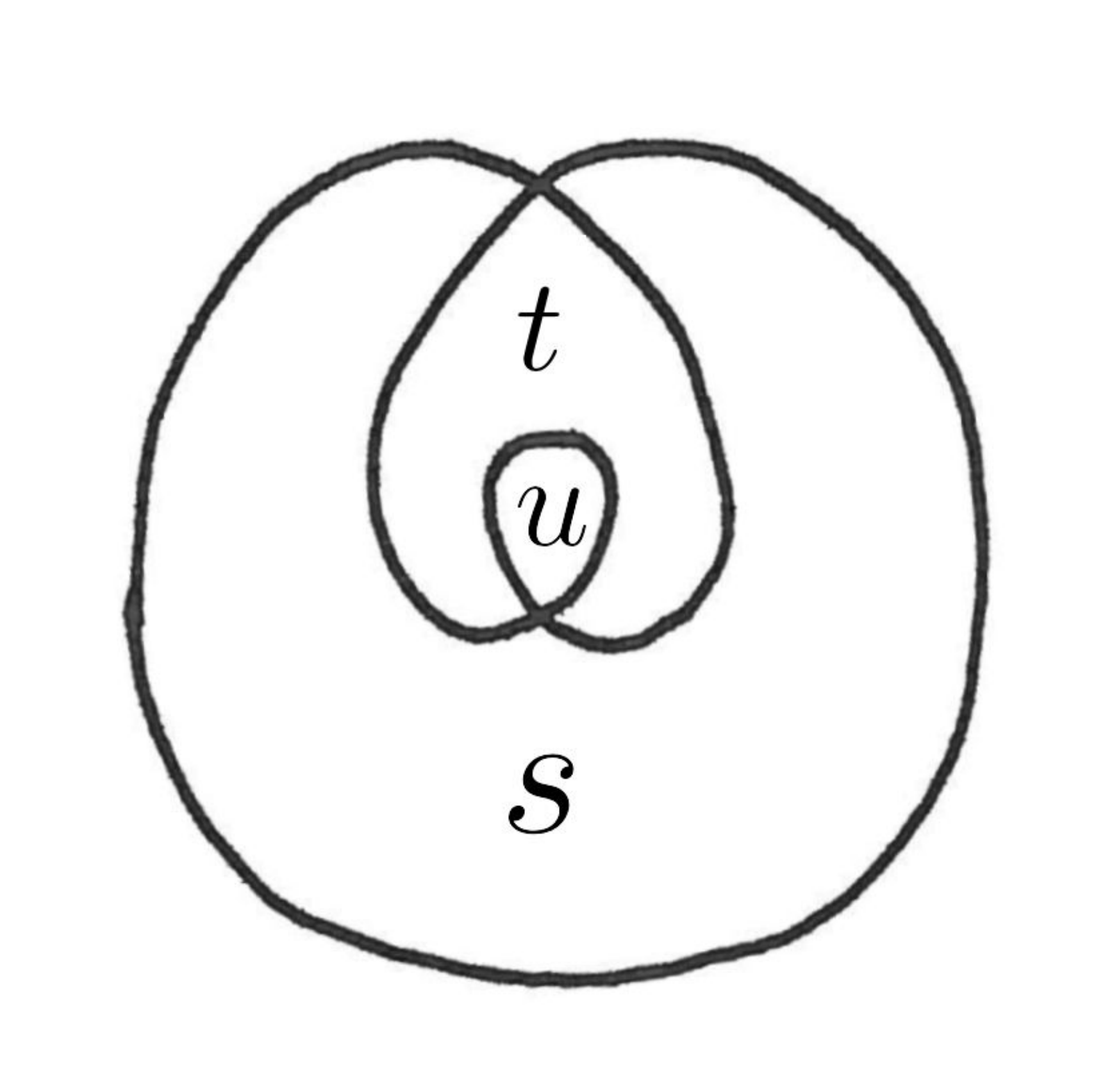} & 1 & $\left[\frac{1}{6}(3u-3)(3u-2)-t(1-u)\right] \upbeta^{s+2t+3u}$ \\
     \hline
     \includegraphics[height=37pt]{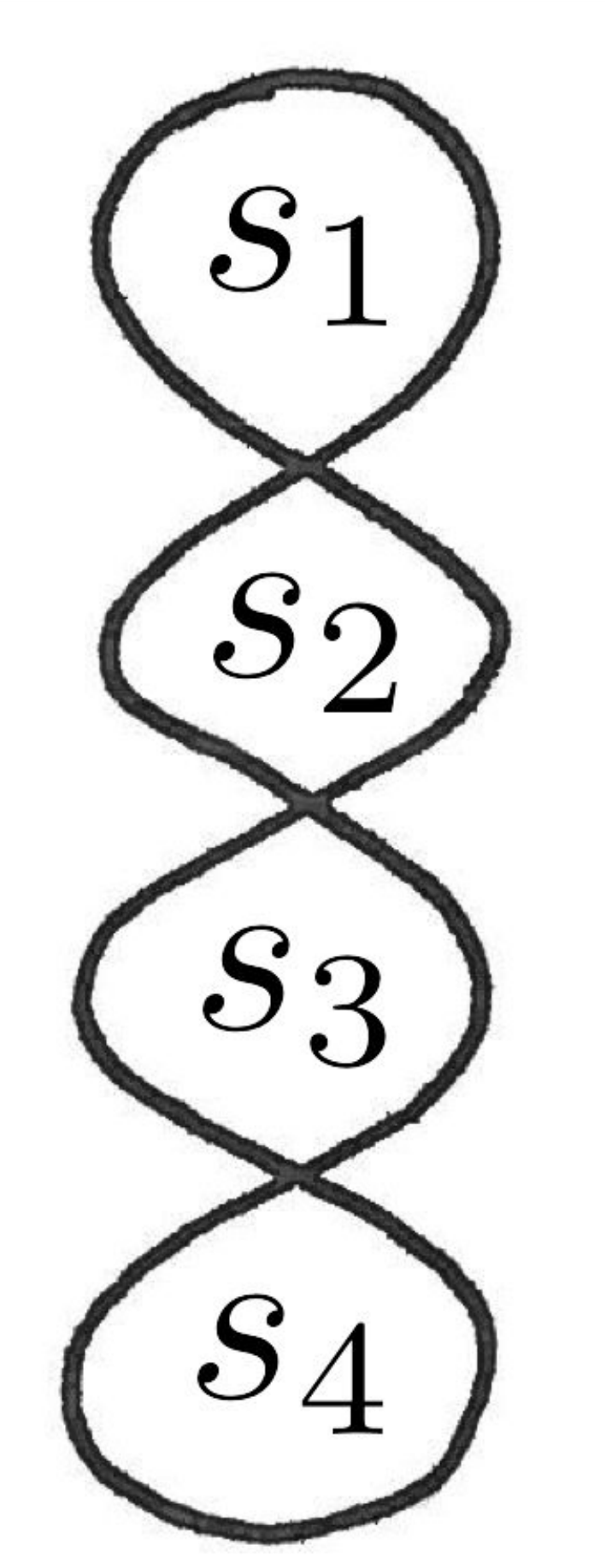} & 1 & $\upbeta^{s_1+s_2+s_3+s_4}$ \\
     \hline
     \includegraphics[height=37pt]{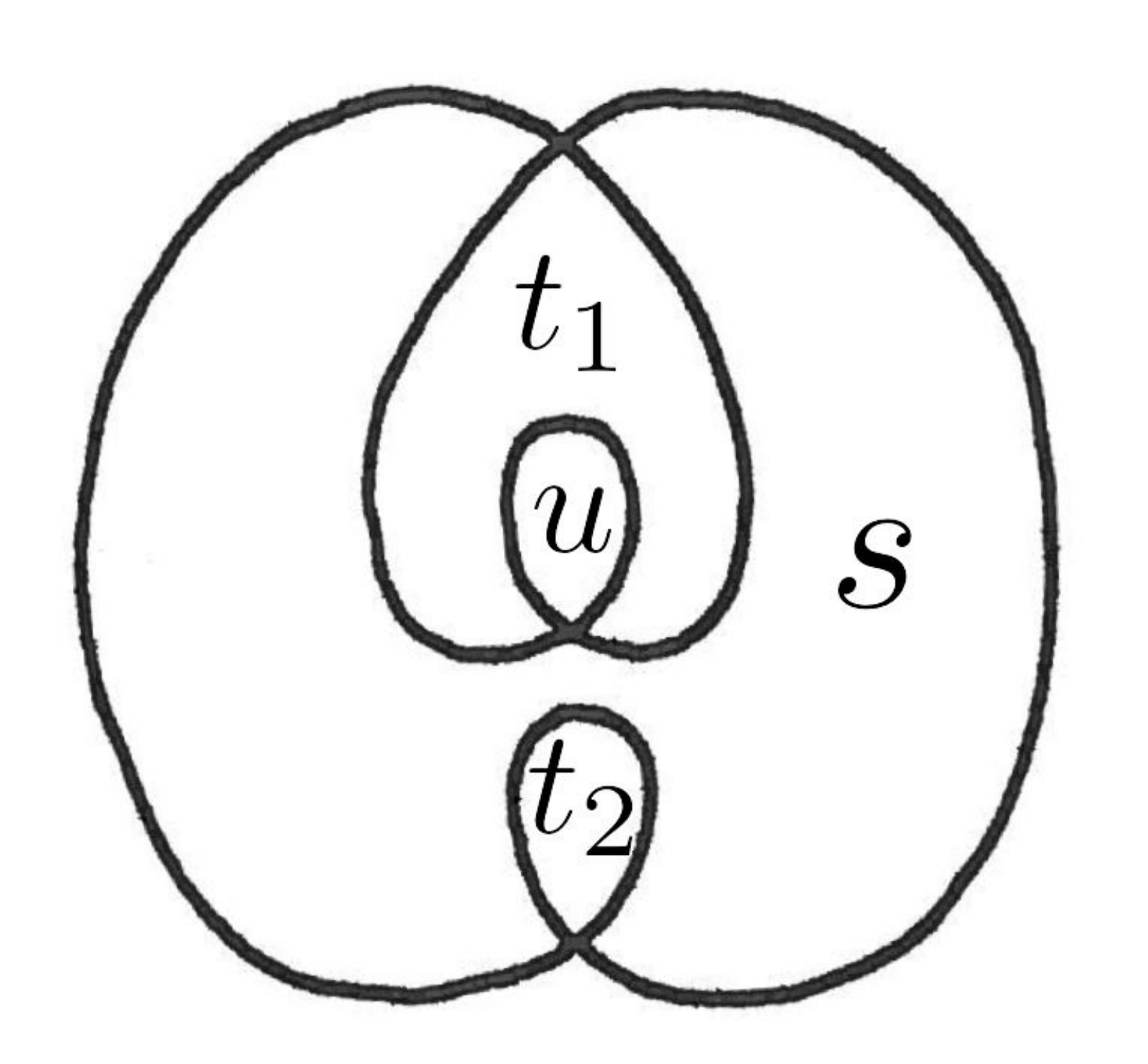} & 1 & $(1-t_2)\left[\frac{1}{6}(3u-3)(3u-2)-t_1(1-u)\right]\upbeta^{s+2t_1+2t_2+3u}$ \\
     \hline
     \includegraphics[height=37pt]{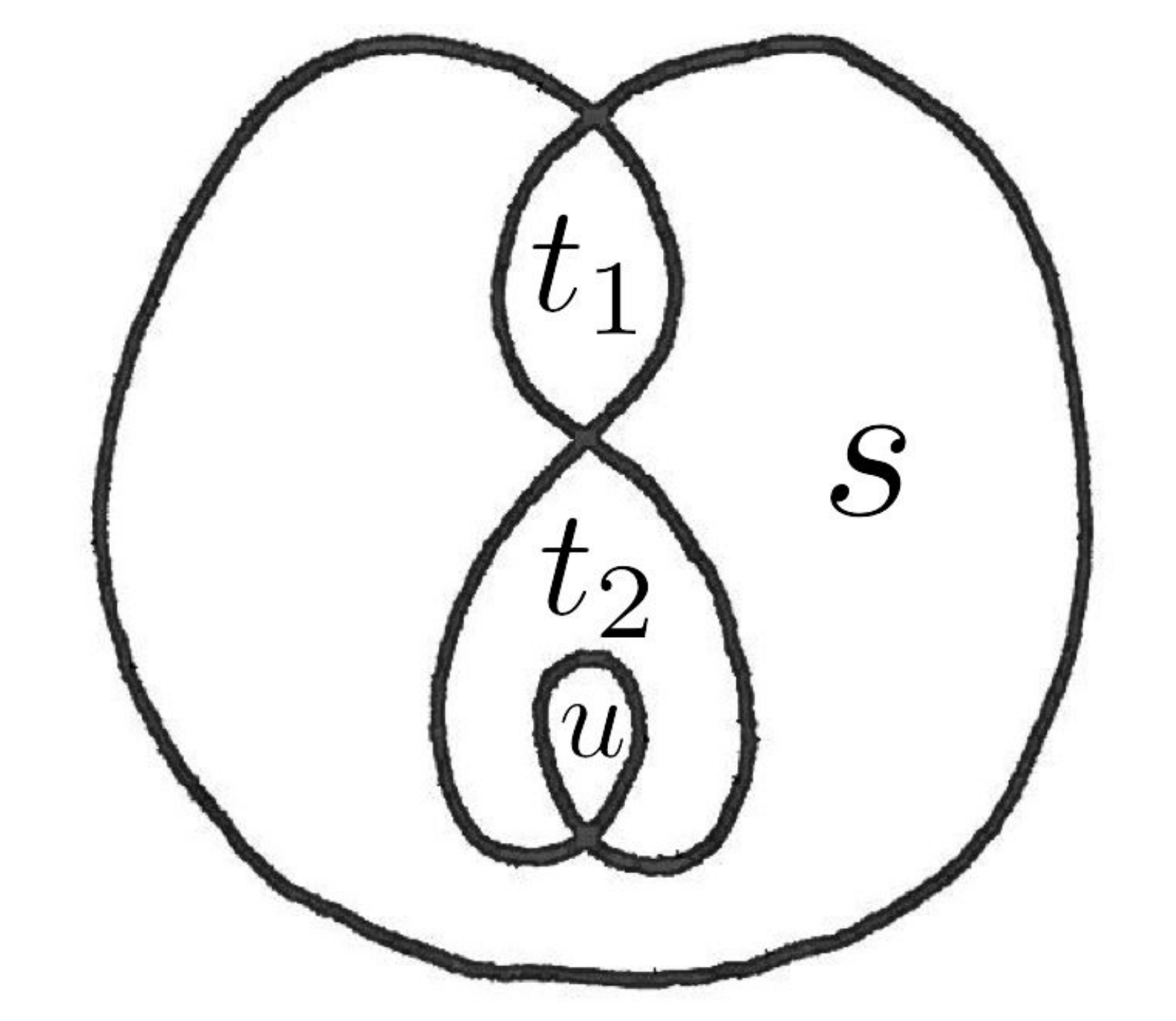} & 4 & $\left(1-t_1(1-u)\upbeta^{2t_2+2u}\right) \upbeta^{s+2t_1+u}$ \\
     \hline
     \includegraphics[height=37pt]{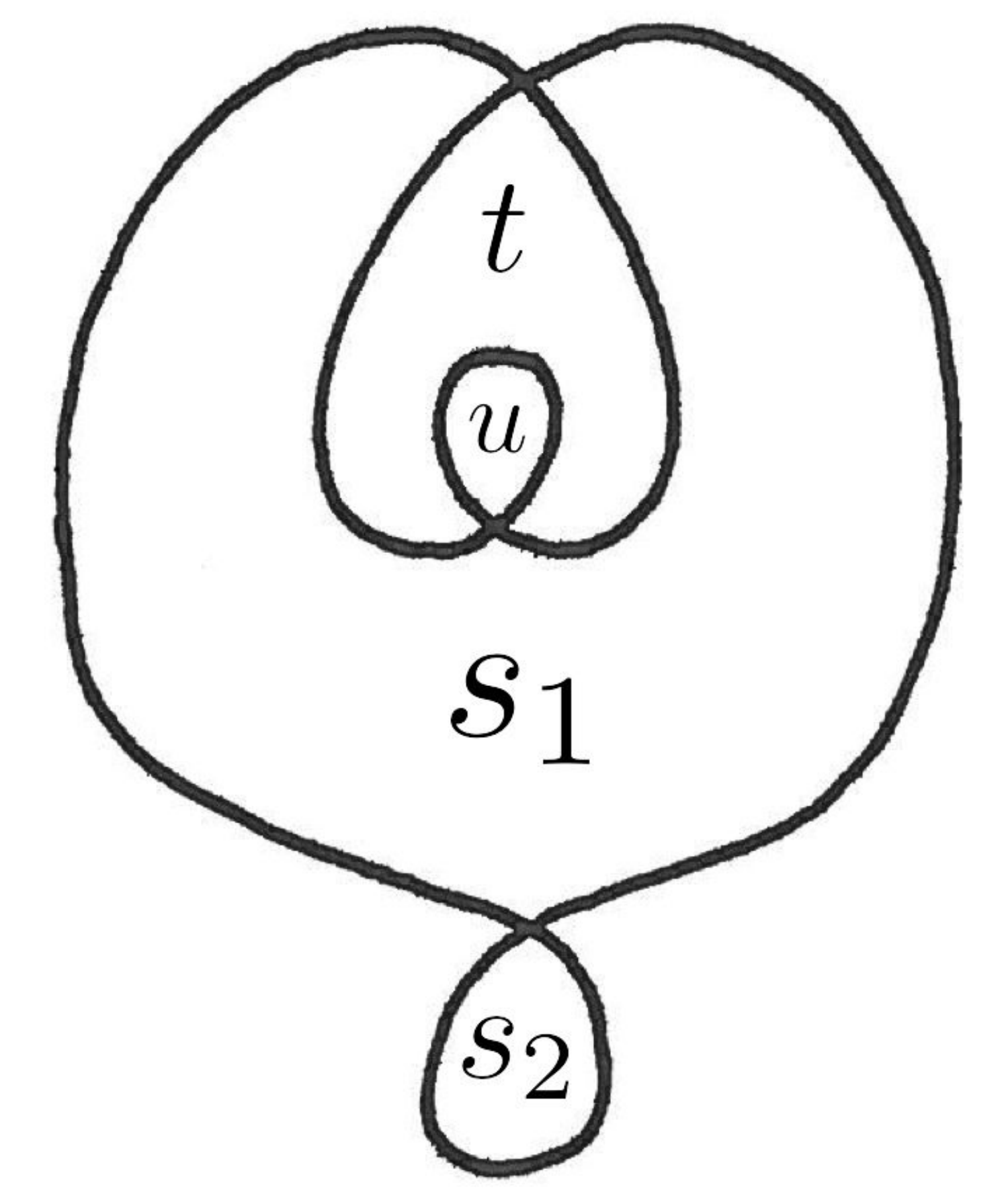} & 1 & $\left[\frac{1}{6}(3u-3)(3u-2)-t(1-u)\right]\upbeta^{s_1+s_2+2t+3u}$ \\
     \hline
     \includegraphics[height=37pt]{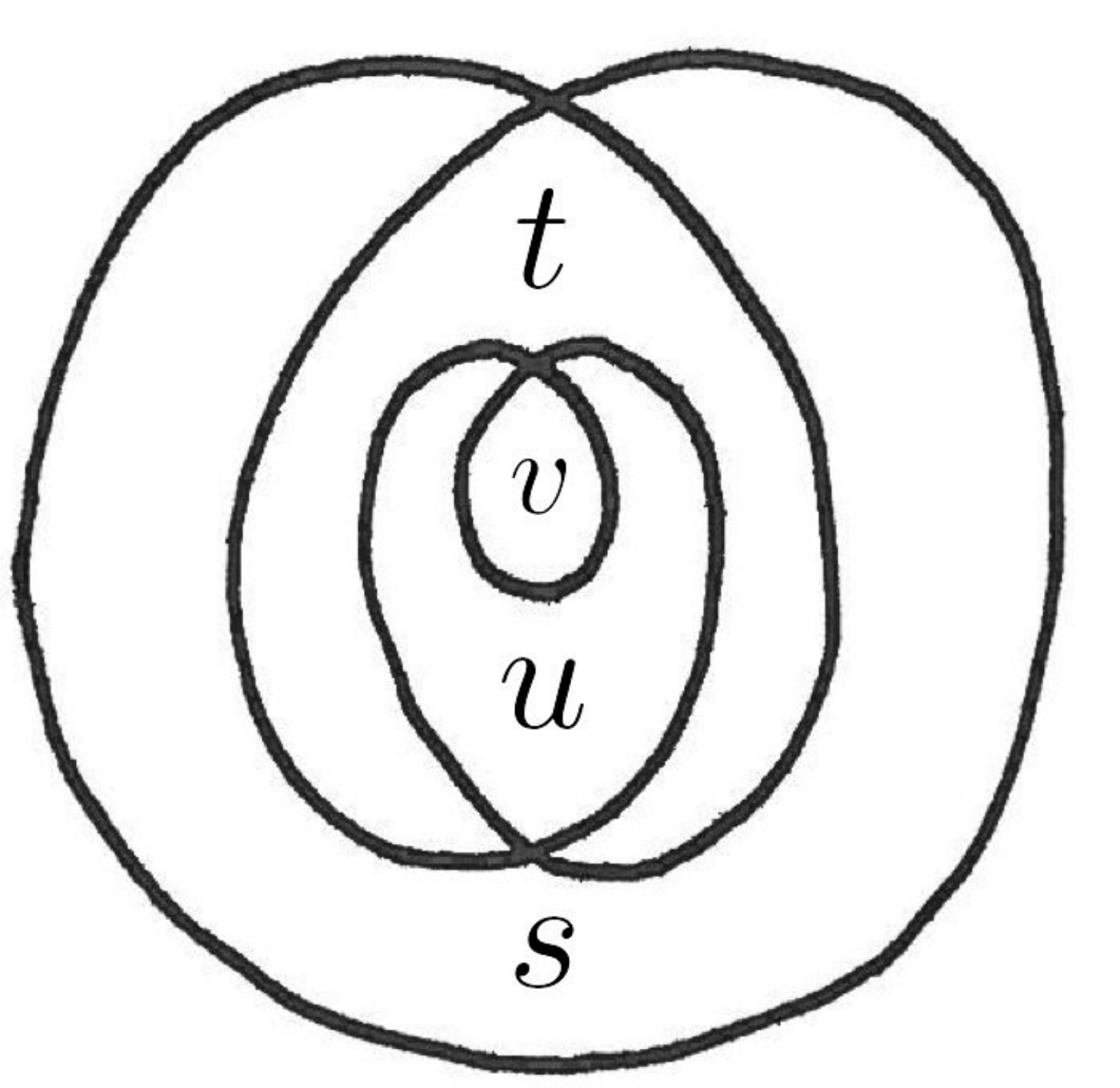} & 1 & \raisebox{0.8cm}{$\left(1-\frac{5}{2}u+\frac{3}{2}u^2 -\frac{13}{3}v+\frac{13}{2}uv-\frac{3}{2}u^2v\right.$}
    \raisebox{0.1cm}{\hspace{-5cm}$\left.+6v^2-4uv^2-\frac{8}{3}v^3 - t\left[\frac{1}{6}(3v-3)(3v-2) - u(1-v)\right]\right) \upbeta^{s+2t+3u+4v}$} \\
     \hline
     \includegraphics[height=37pt]{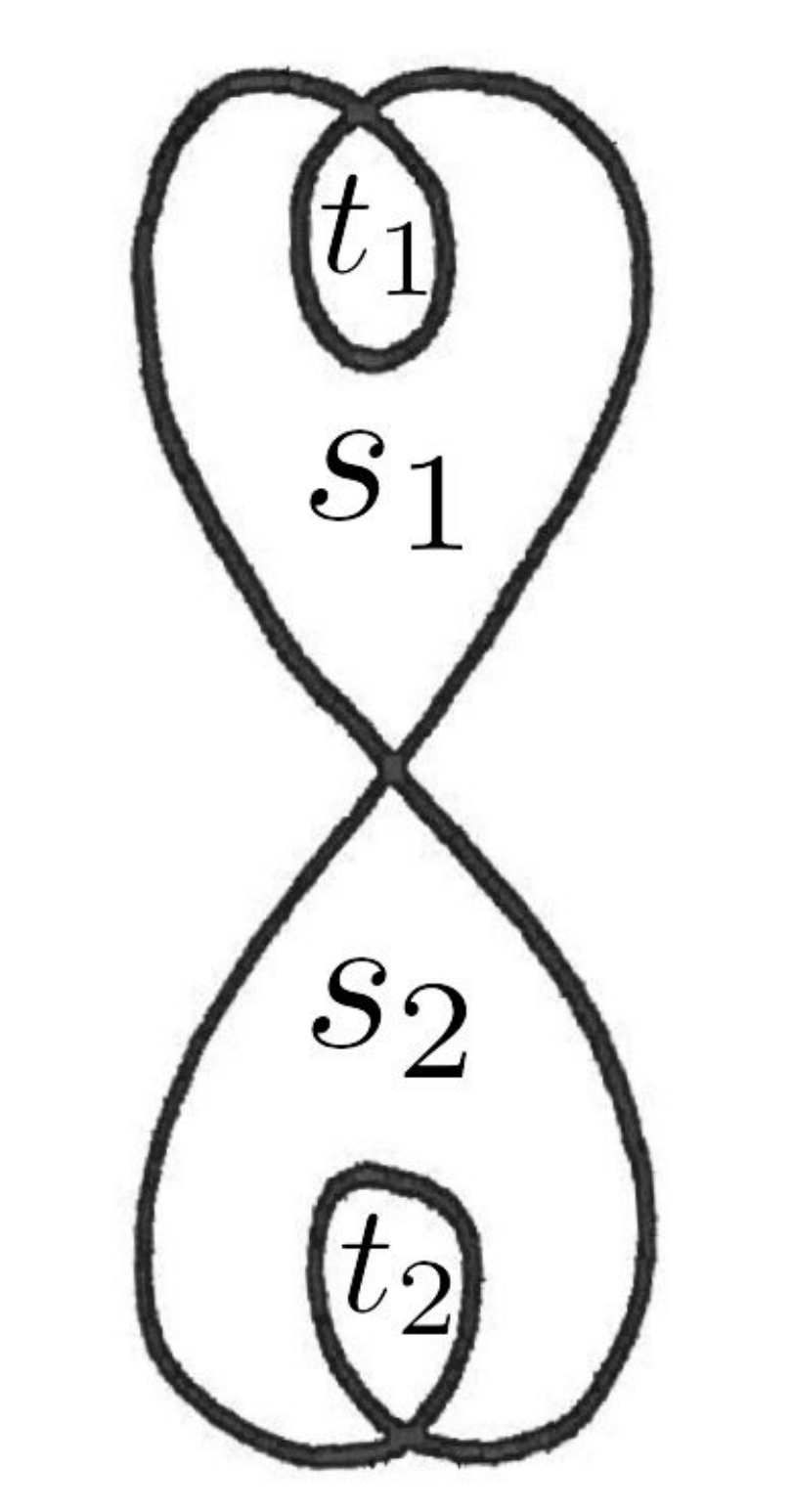} & 1 & $(1-t_1)(1-t_2)\upbeta^{s_1+s_2+2t_1+2t_2}$ \\
     \hline
     \includegraphics[height=37pt]{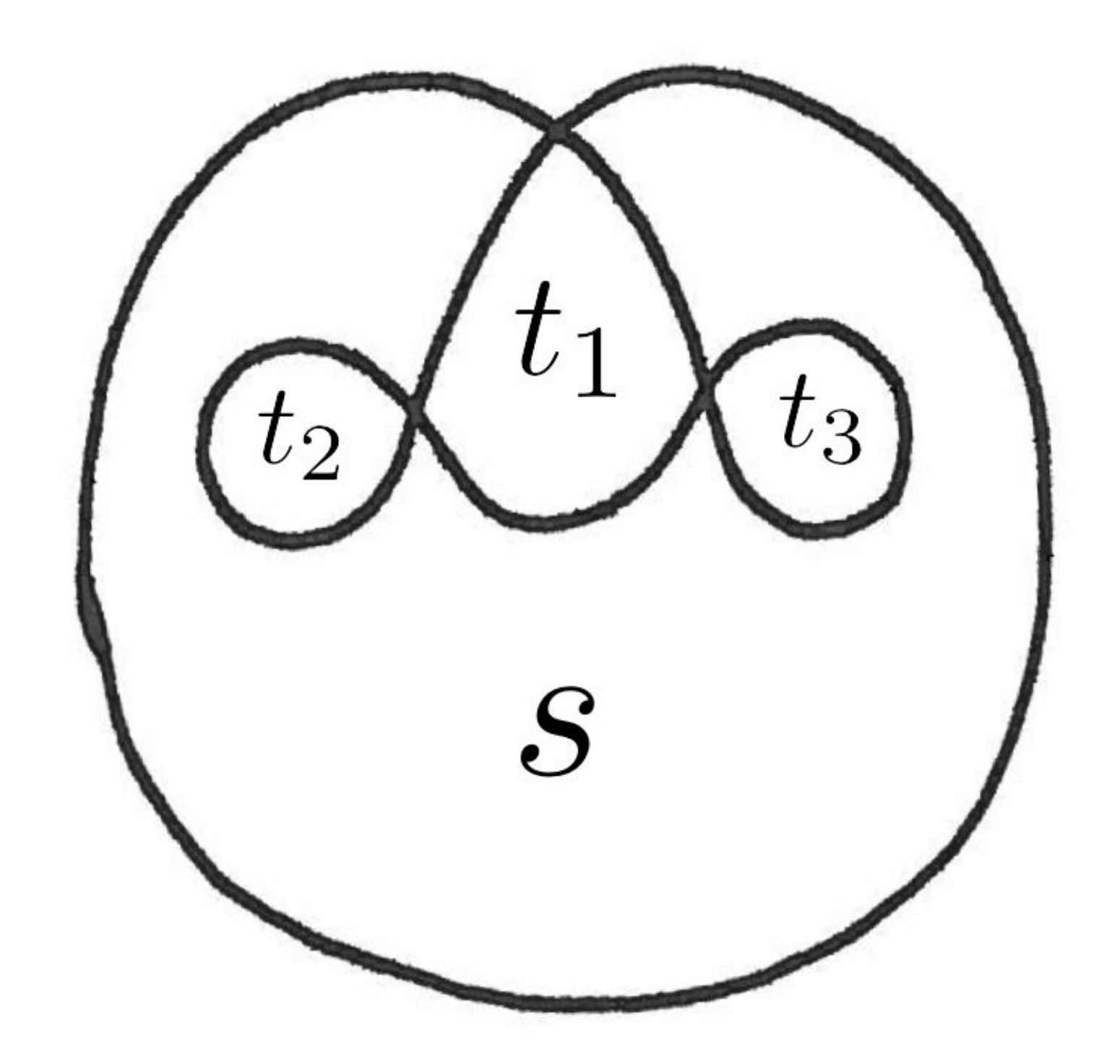} & 4 & $\left(\upbeta^{2t_2}+\upbeta^{2t_3}-(1+t_1)\upbeta^{2t_2+2t_3}\right) \upbeta^{s+2t_1}$ \\
     \hline
     \includegraphics[height=37pt]{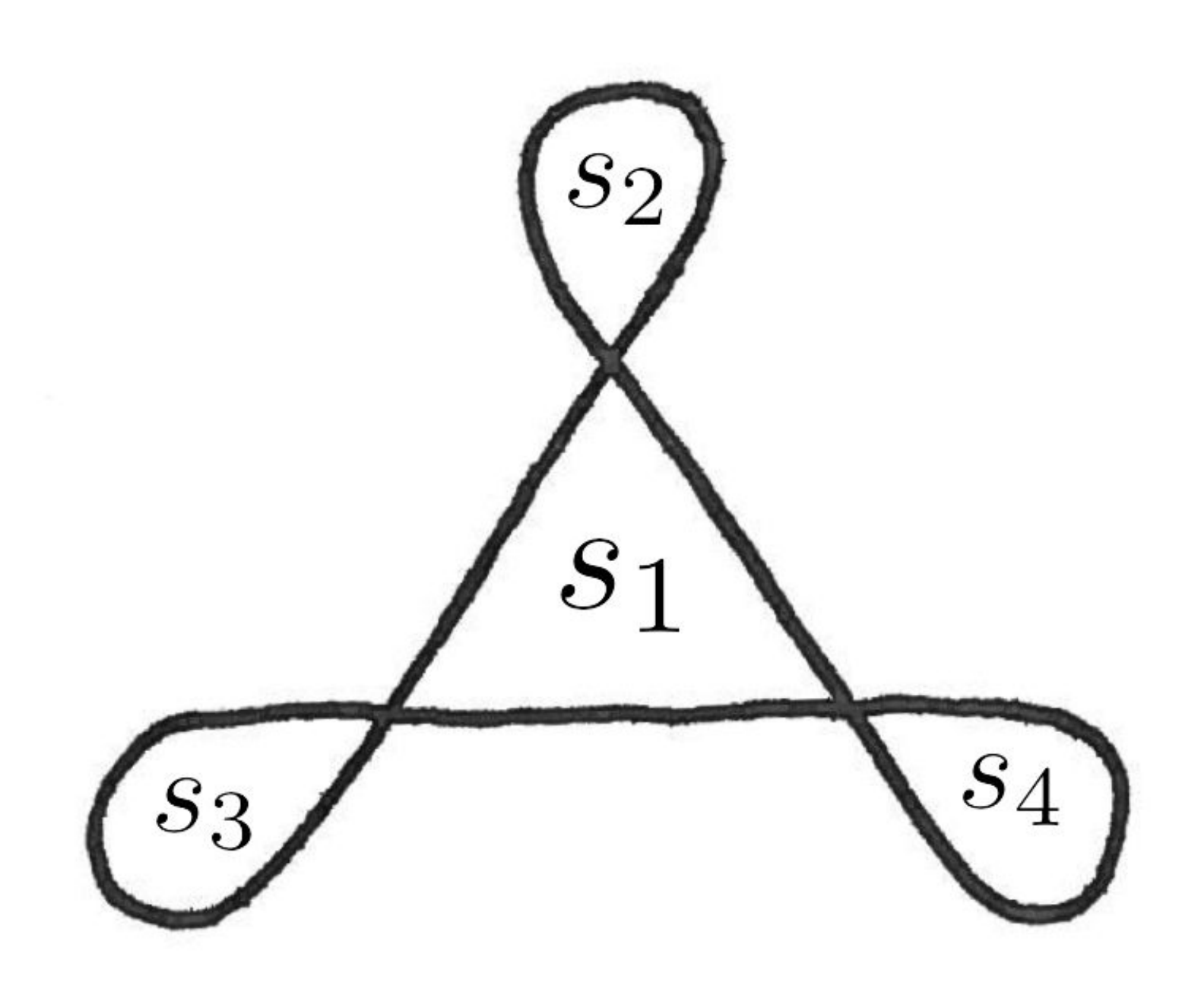} & 1 & $\upbeta^{s_1+s_2+s_3+s_4}$ \\
     \hline
     \includegraphics[height=37pt]{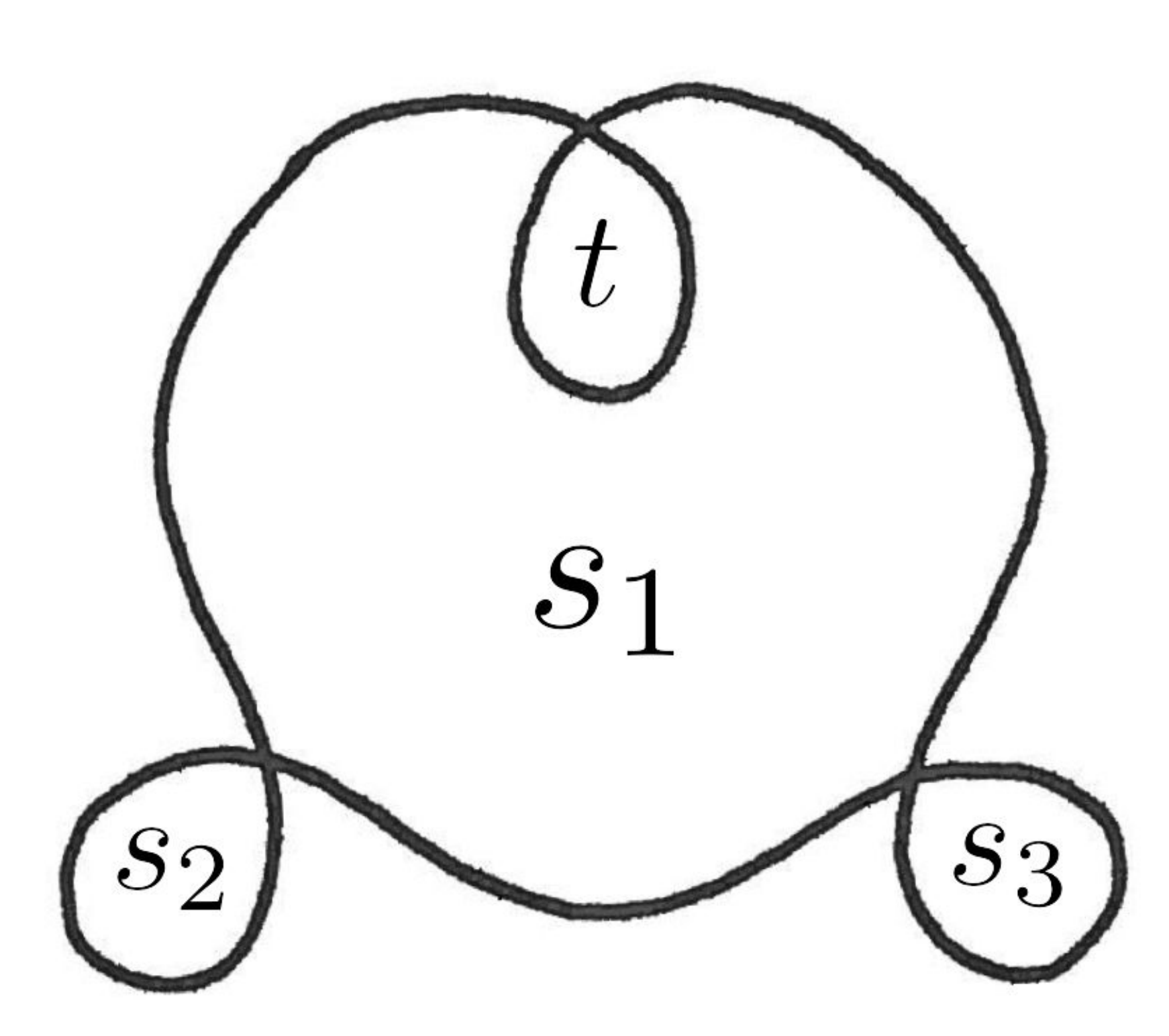} & 1 & $(1-t)\upbeta^{s_1+s_2+s_3}$ \\
     \hline
     \includegraphics[height=37pt]{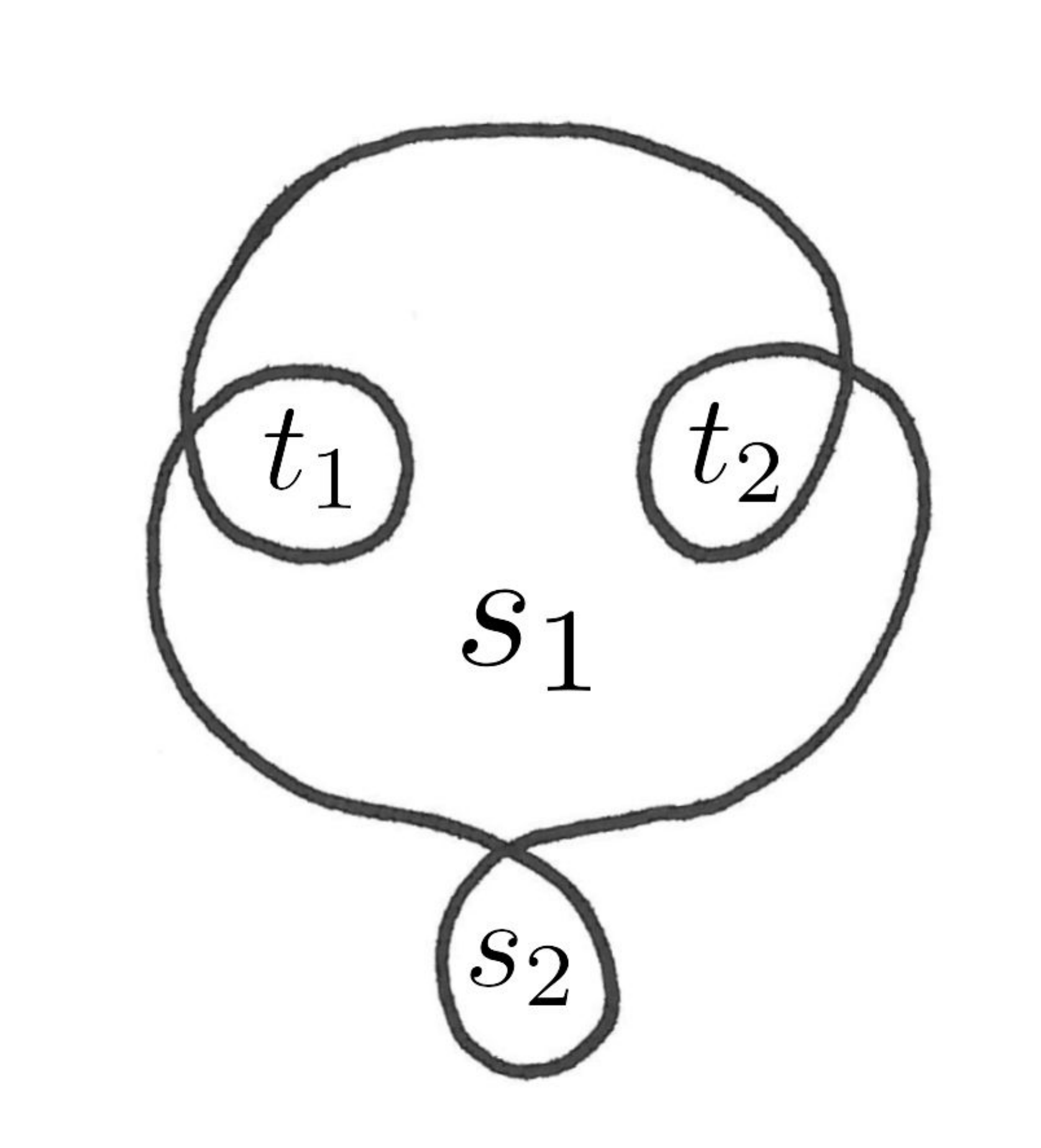} & 1 & $(1-t_1)(1-t_2)\upbeta^{s_1+s_2+2t_1+2t_2}$ \\
     \hline
     \includegraphics[height=37pt]{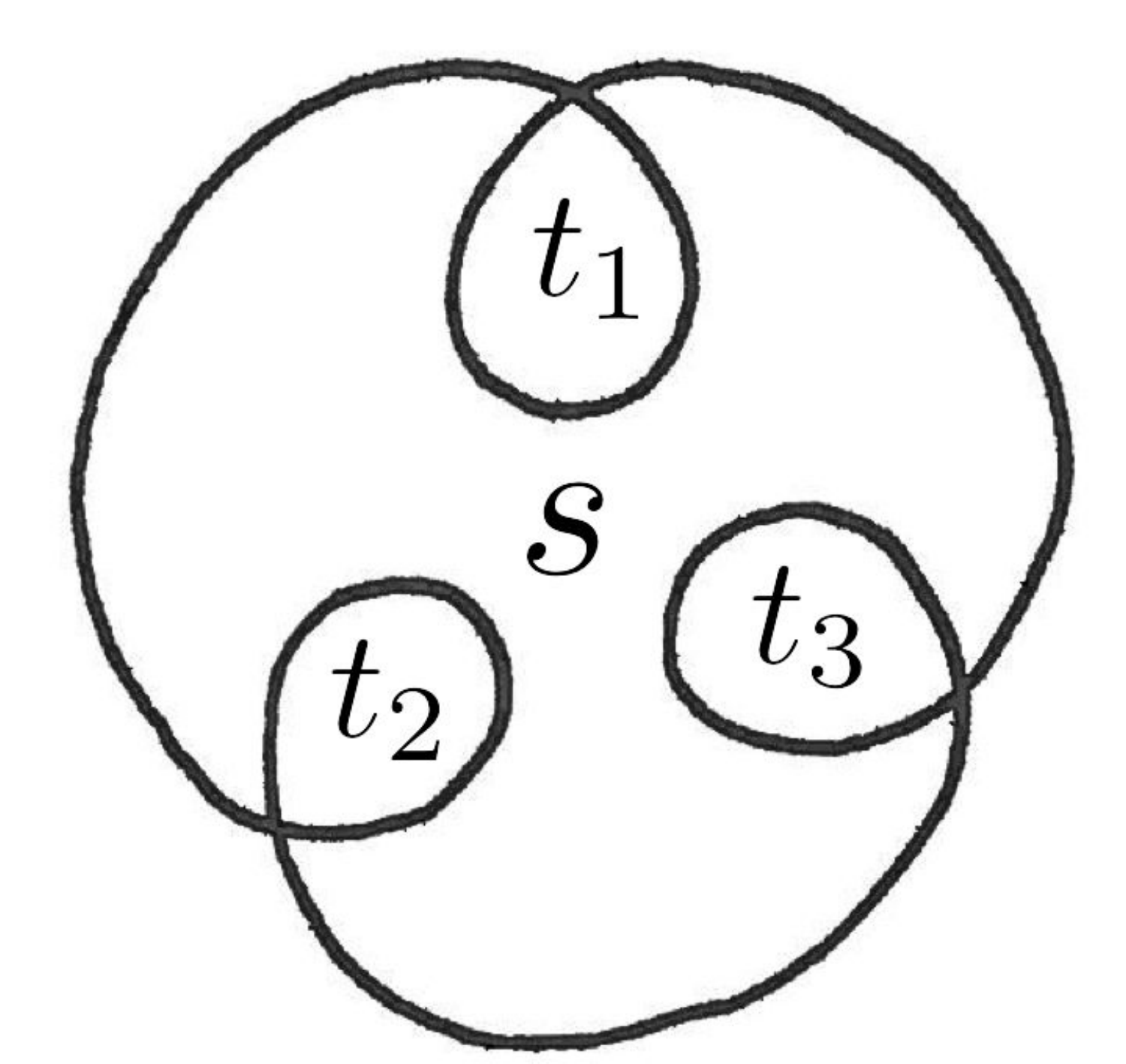} & 1 & $(1-t_1)(1-t_2)(1-t_3)\upbeta^{s+2t_1+2t_2+2t_3}$ \\
     \hline
     \includegraphics[height=37pt]{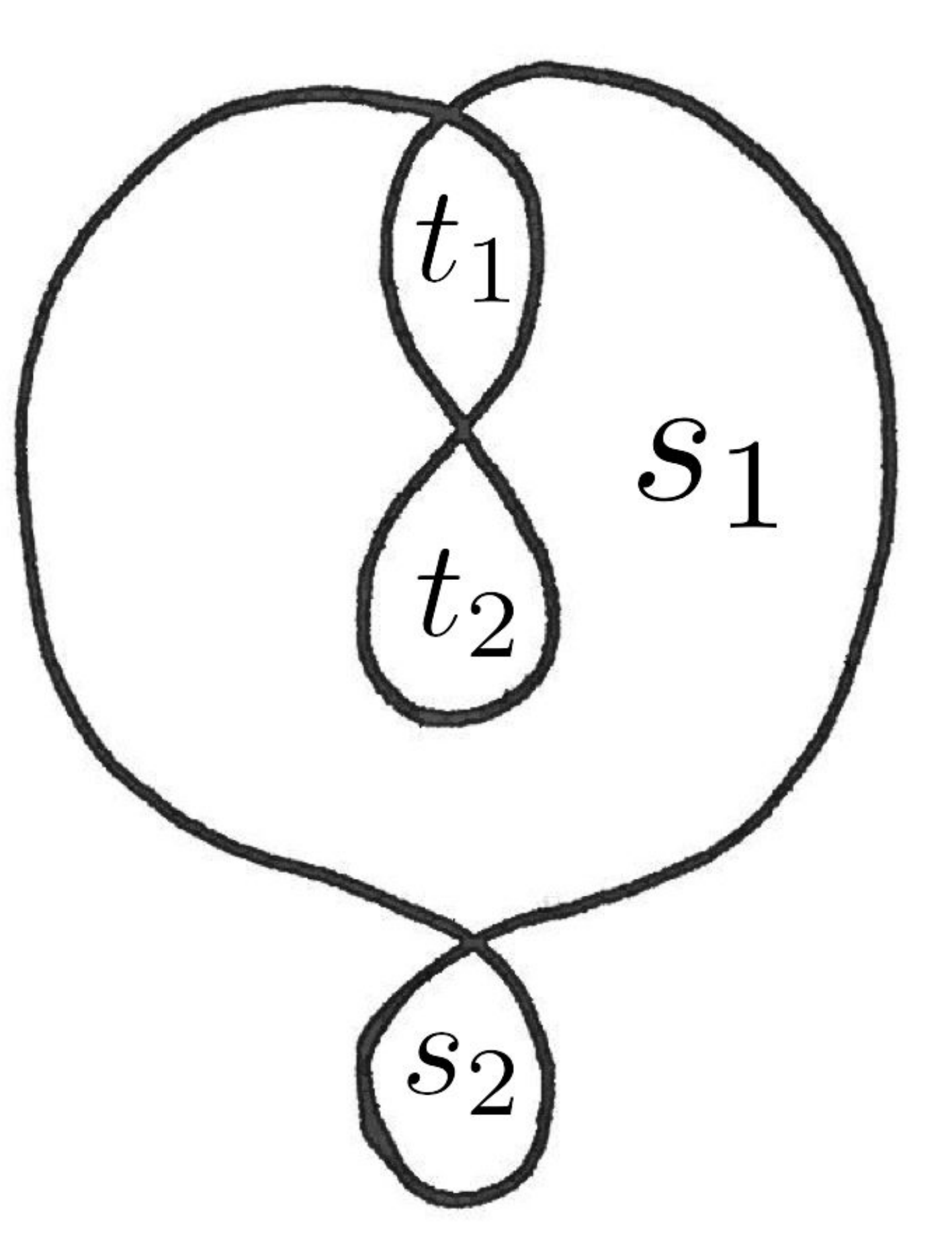} & 2 & $\left(1-t_1\upbeta^{2t_2}\right)\upbeta^{s_1+s_2+2t_1}$ \\
     \hline
     \includegraphics[height=37pt]{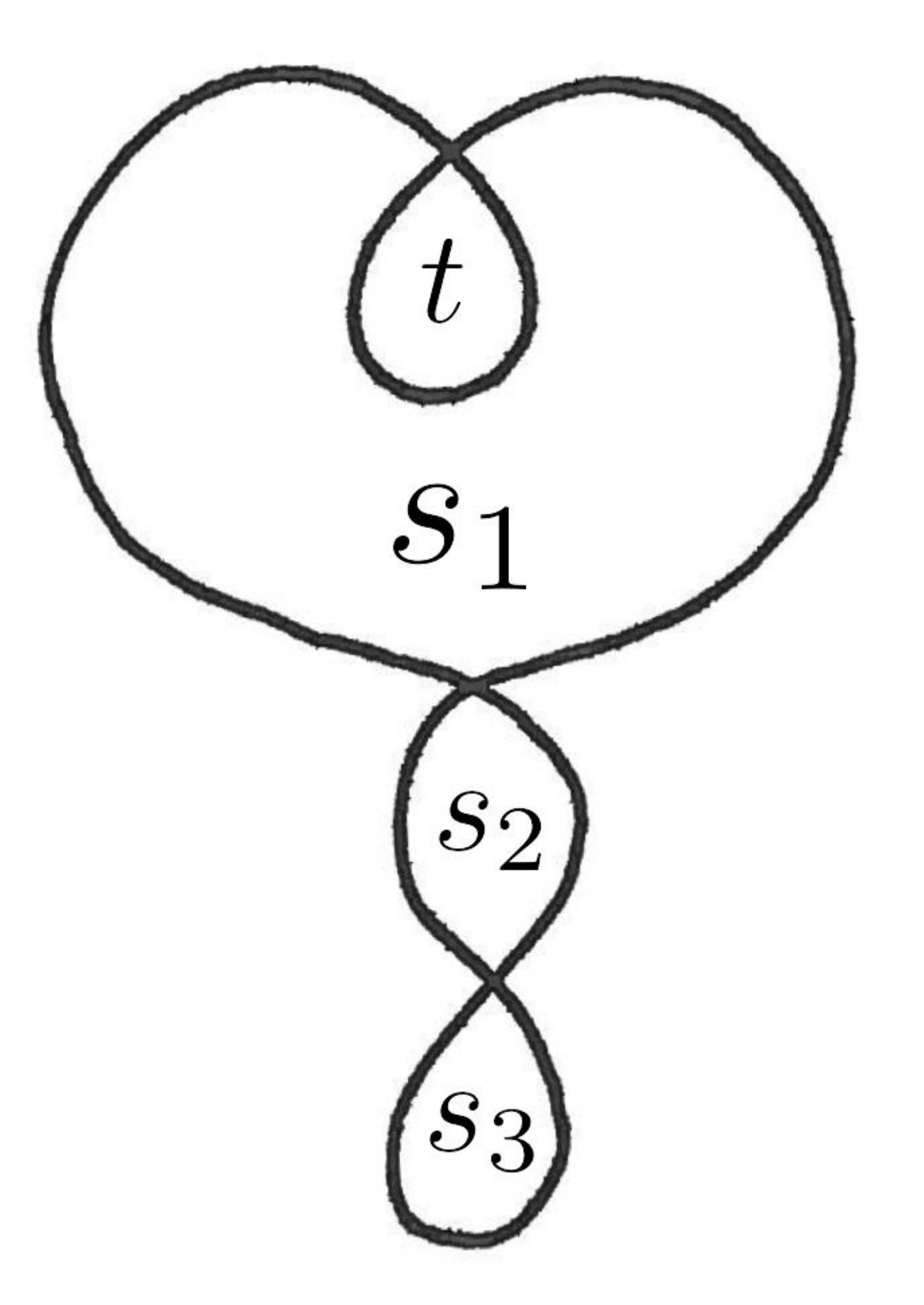} & 1 & $(1-t)\upbeta^{s_1+s_2+s_3+2t}$ \\
     \hline
     \includegraphics[height=37pt]{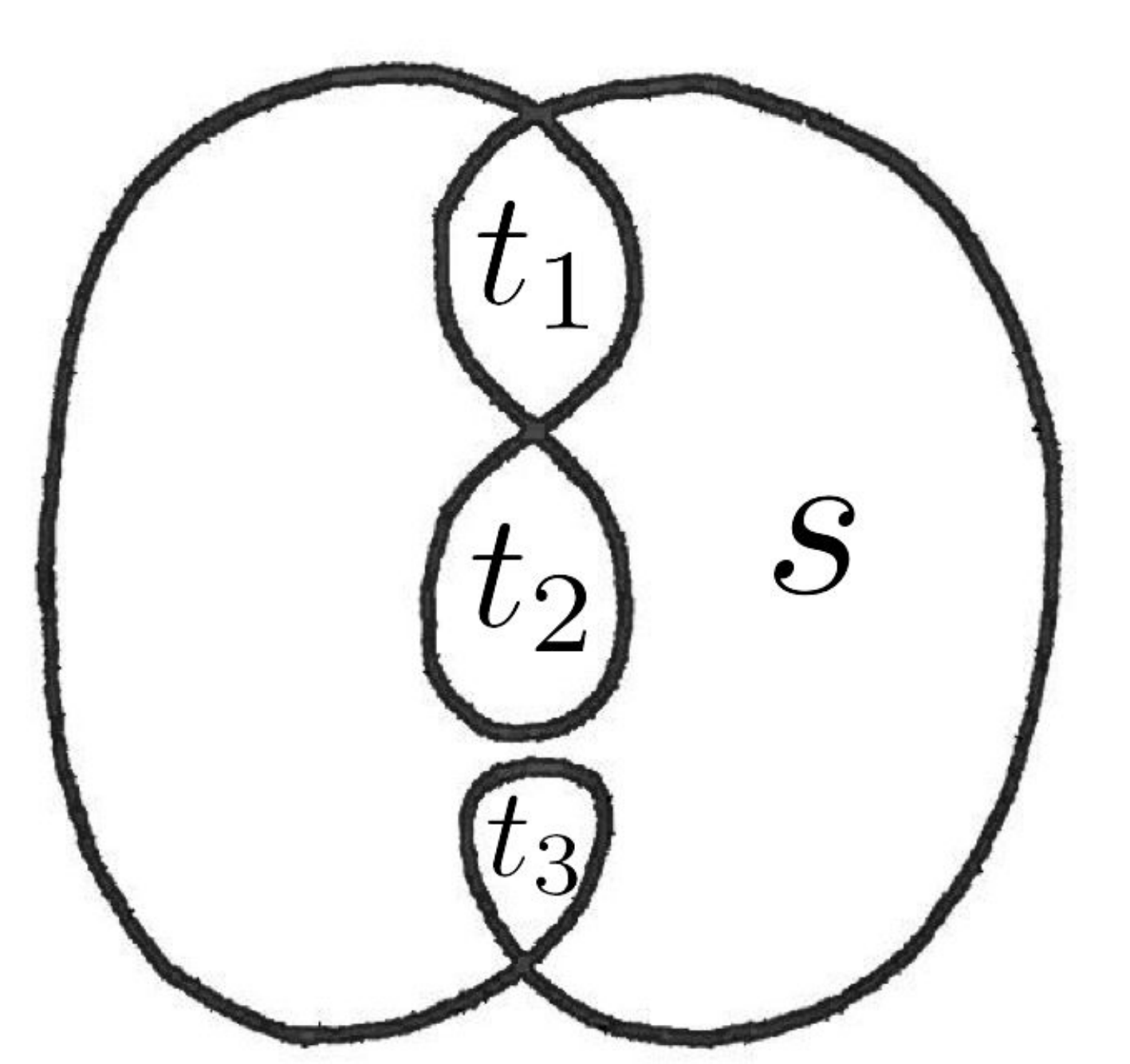} & 2 & $(1-t_3)\left(1-t_1\upbeta^{2t_2}\right) \upbeta^{s+2t_1+2t_3}$ \\
     \hline
     \includegraphics[height=37pt]{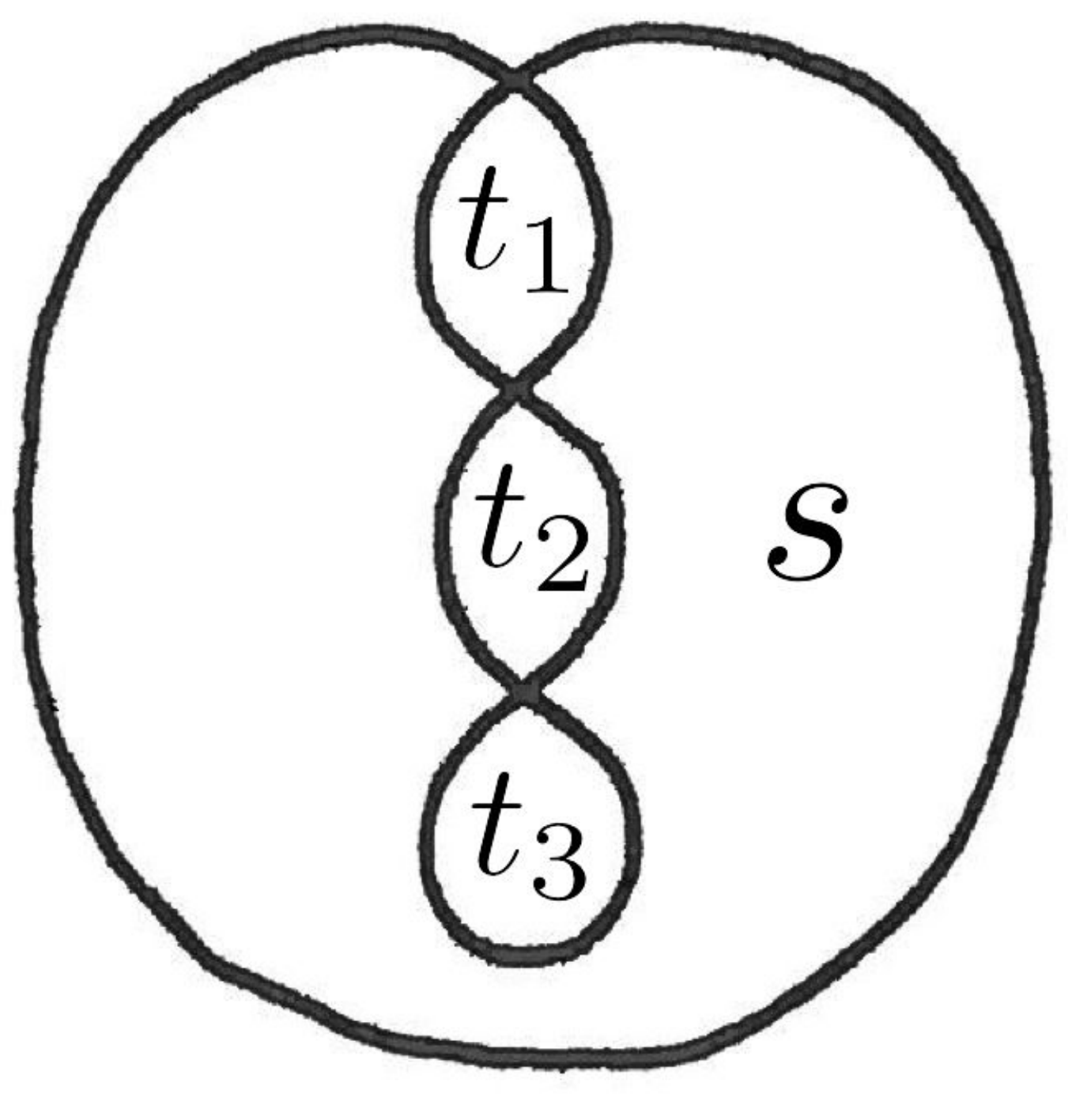} & 2 & $\left((1-t_3)-t_1\upbeta^{2t_2}\right)\upbeta^{s+2t_1+2t_3}$ \\
     \hline
     \includegraphics[height=37pt]{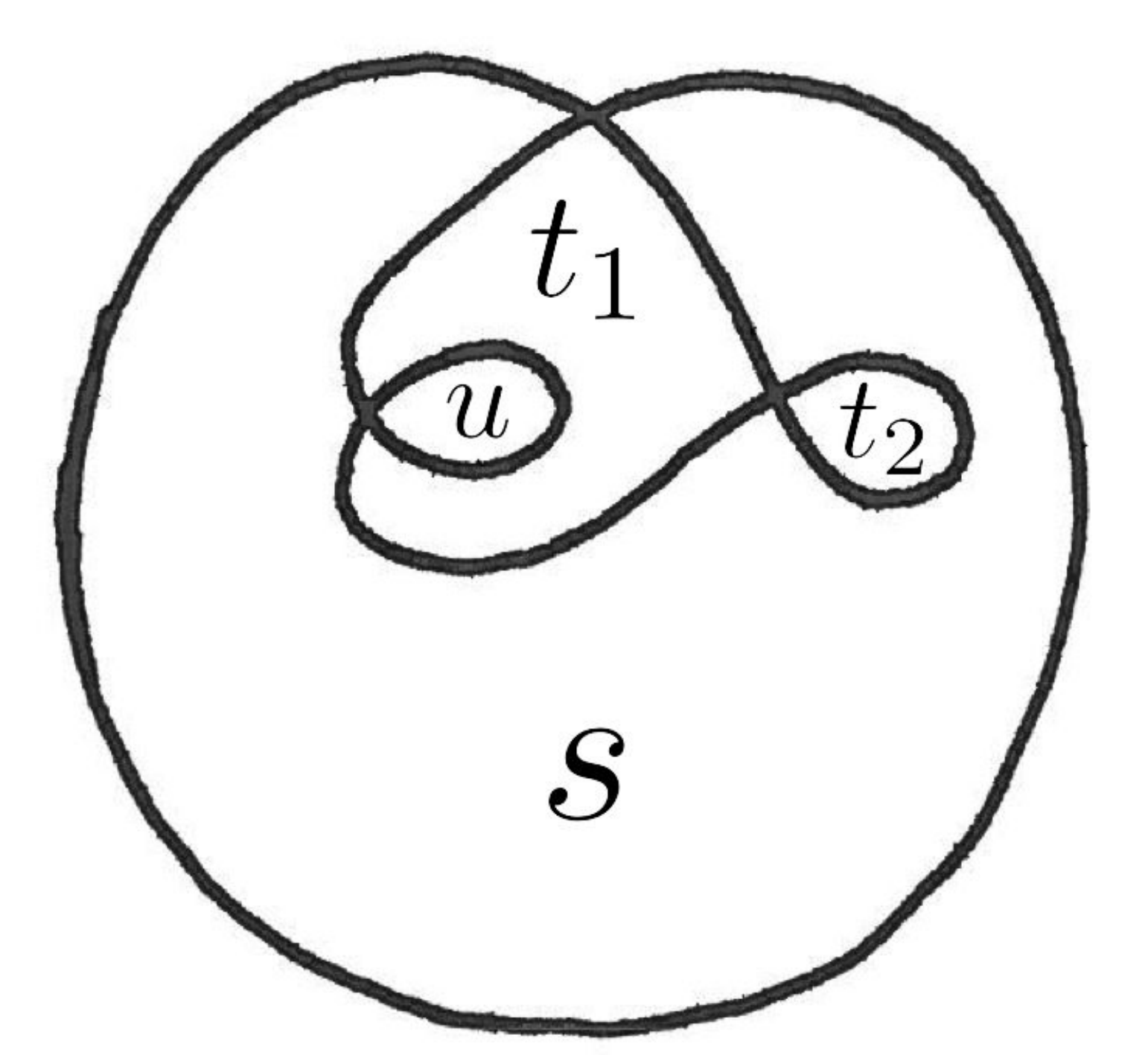} & 2 & $\left((1-u)+\upbeta^{2t_2}\left[\frac{1}{6}(3u-3)(3u-2)-(1+t_1)(1-u)\right]\right) \upbeta^{s+2t_1+3u}$ \\
    \hline
     \includegraphics[height=37pt]{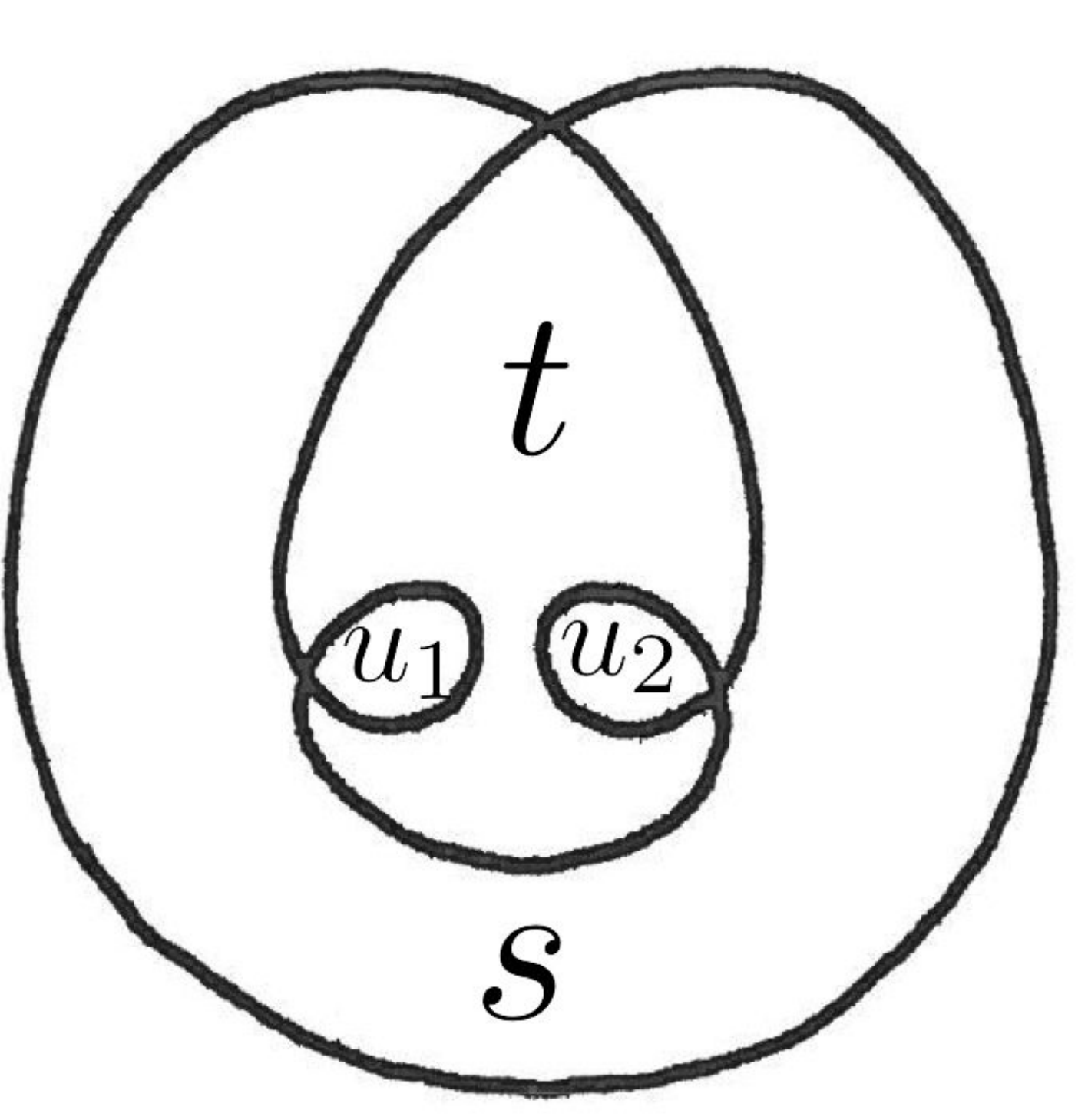} & 1 &
         \raisebox{0.8cm}{$\left[\frac{1}{6}(3(u_1+u_2)-3)(3(u_1+u_2)-2)\right.$}
         \raisebox{0.1cm}{\hspace{-5cm}$\left.-t(1-u_1)(1-u_2) +u_1u_2\left(1-\frac{3}{2}(u_1+u_2)\right)\right] \upbeta^{s+2t+3u_1+3u_2}$} \\
     \hline
     \includegraphics[height=37pt]{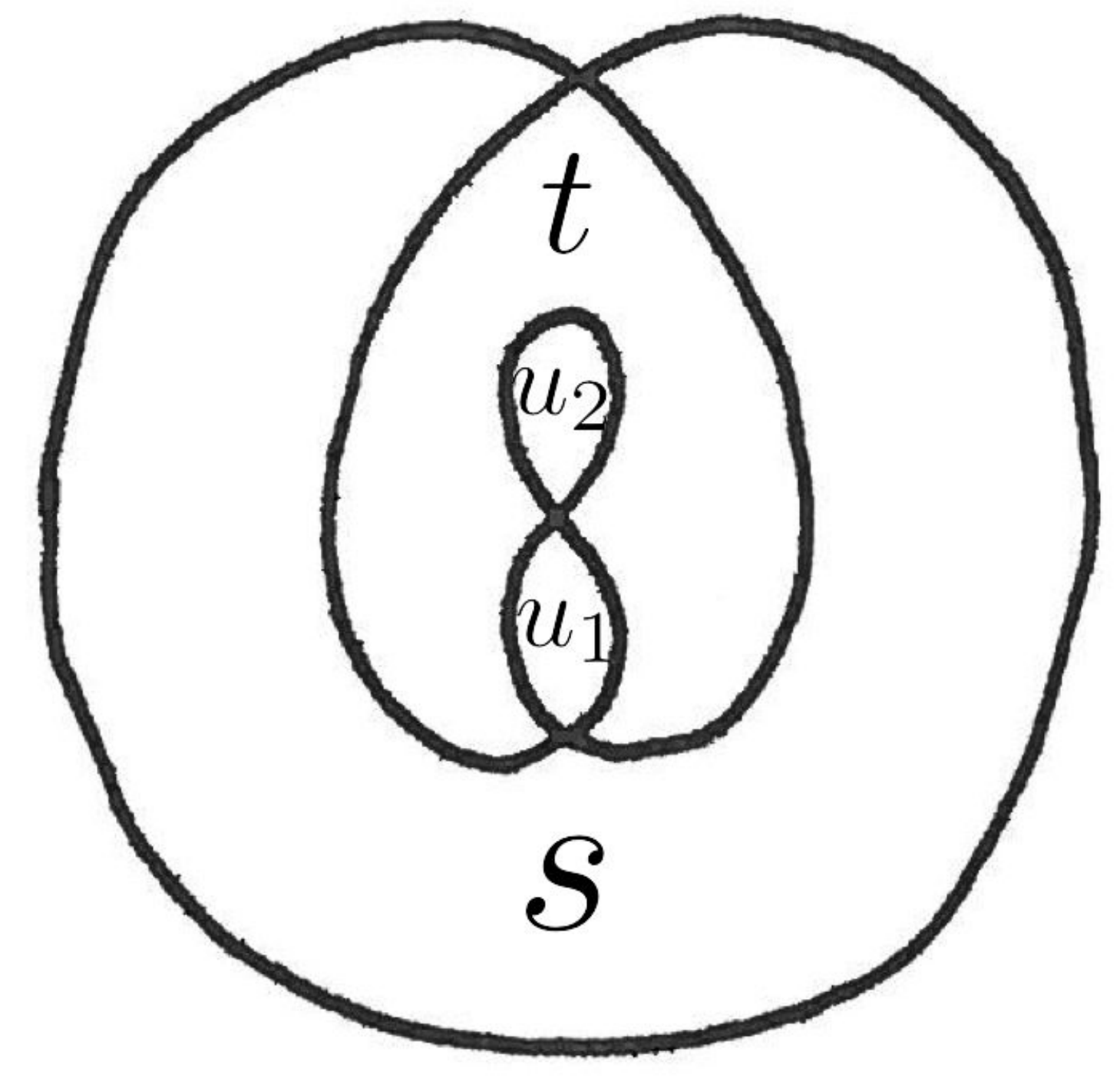} & 2 & $\left(1-2u_1-t+\upbeta^{2u_2}\left[u_1\left(t+u_2-\frac{1}{2}\right)+\frac{3}{2}u_1^2\right]\right) \upbeta^{s+2t+3u_1+3u_2}$ \\
     \hline
     \includegraphics[height=37pt]{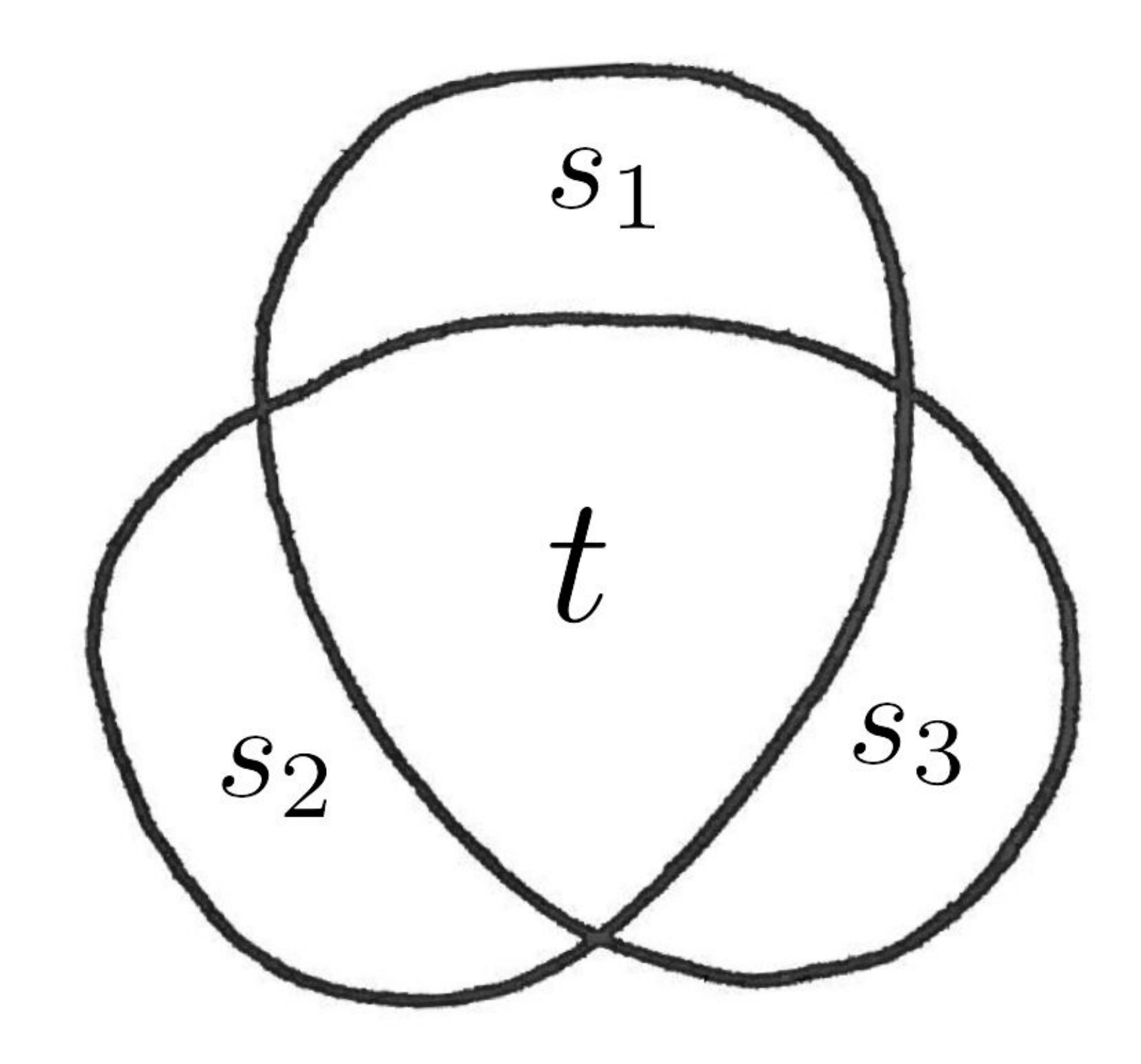} & 1 & $(1-t)\upbeta^{s_1+s_2+s_3+2t}$ \\
     \hline
     \includegraphics[height=37pt]{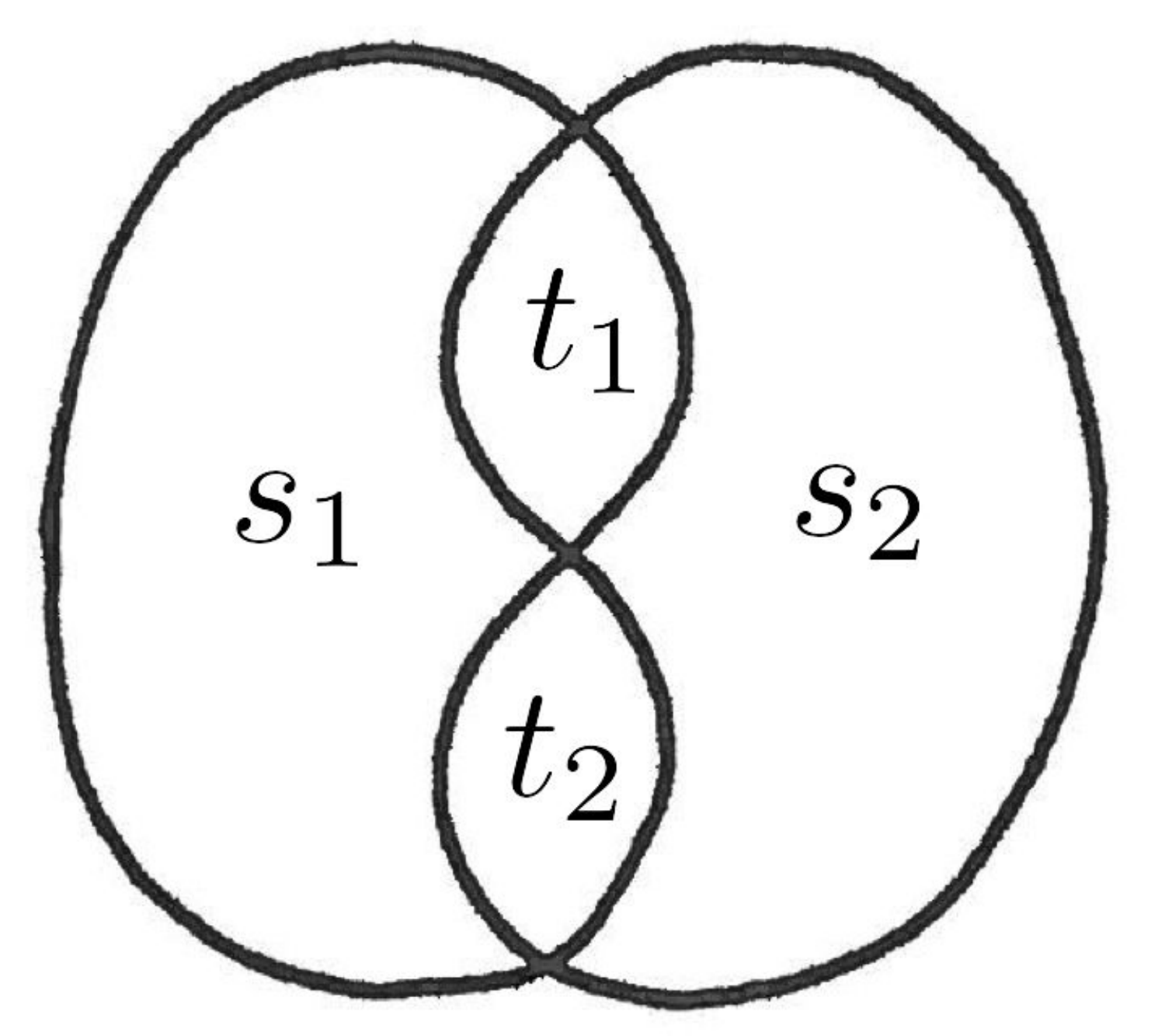} & 4 & $\left(\upbeta^{2t_1}+\upbeta^{2t_2}-\upbeta^{2t_1+2t_2}\right) \upbeta^{s_1+s_2}$ \\ 
     \hline
     \caption{\label{table} Table of Wilson loop expectations for loops with three or fewer self-crossings. The left column displays a continuum loop $\cL$. The middle column shows the cardinality of $\cK_{\ell}$ for any $\ell$ isotopic to $\cL$ on the same row. The right column displays the Wilson loop expectation for such an $\ell$ whose plaquette-regions have areas matching the labels of the corresponding continuum regions. (The continuum loop figures are reproduced from \cite{lévy2012masterfieldplane}.)}
    \end{longtable}
\end{center}
}

\begin{remark}\label{remark: similar to levy}
    The formulas presented in \cref{table} are very similar to those appearing in \cite[Appendix B]{lévy2012masterfieldplane} and \cite[Table in Section 6.2]{park2023wilson}\footnote{Note that this table only presents formulas for the leading term (in the limit where all the areas of the regions tend to infinity). However, if we instead consider the formulas given in the table in \cite[Section 6.1]{park2023wilson} for finite $N$, and then take the limit $N\to\infty$, we recover exactly the expressions found in \cite[Appendix B]{lévy2012masterfieldplane}.} for the large-$N$ continuum Yang--Mills theory. Indeed, the polynomial coefficients of many expressions are exactly the same as in \cite[Appendix B]{lévy2012masterfieldplane}. However, for loops having a region around which the loop is wound three or more times (see, for instance, the $8$-th loop from the top of the table), we get slightly different polynomial coefficients (but with the same leading order in the limit where all the areas of the regions tend to infinity). This discrepancy is a consequence of what we already observed in \cref{rmk:cont-parallel}.
\end{remark}

\begin{remark}
    Many additional loops can be computed using the same procedure detailed in \cref{sec: Wilson loop expectation constant for erasable loops with three or fewer self-crossings}, as was done for those in \cref{table}. Notably, \cref{table} does not include any lattice loop that has an edge appearing twice or more (as this loop would correspond to a continuum loop with an infinite number of non-simple points), but the Wilson loop expectation for these loops can be similarly explicitly computed. We have decided to restrict ourselves to three or fewer self-crossings purely to simplify the presentation and enhance the clarity of the main ideas.
\end{remark}

\subsection{Proof techniques and the master loop equation}\label{subsec:proof-techniques and ML}

In this section, we first introduce the master loop equation, our main technical tool (see \cref{sec: the master loop equation}). Then, in \cref{sec: proof tech}, we give an overview of our key proof techniques, in particular, the surface exploration algorithm.

\subsubsection{The master loop equation}\label{sec: the master loop equation}

The main tool we will use throughout this paper is the master loop equation in~\cref{thm: fixed K 't Hooft master loop equation for surface sum}, a recursive relation satisfied by $\phi^K(\ell)$. One may think of this equation as giving the first step in a surface exploration process. Thus, we may interpret the repeated application of this equation as a surface exploration procedure. As briefly mentioned in Section \ref{subsec:overview}, the key step towards our main results is to design a suitable exploration procedure. Said in terms of the master loop equation, we will need to carefully choose how to apply the master loop equation at each step of the process. These points are elaborated on in \cref{sec: proof tech}.

To introduce the master loop equation, we first need to define two loop operations. Let $s=\{\ell_1,\dots,\ell_n\}$ be a string and fix $\mathbf{e}$ to be one specific copy of the (oriented) edge $e$ appearing in $s$.\footnote{Throughout this paper, we will always use bolded edges to denote a specific copy of the lattice edge (i.e.\ $\mathbf{e}$ is a specific copy of the lattice edge $e$).} Assume that $\mathbf{e}$ is contained in the loop $\ell_i$.

\medskip

\noindent\underline{\textbf{Splittings}}. First, we define splittings at $\mathbf{e}$, an operation that splits a loop into two loops. Let $\mathbf{e}'$ denote another copy of $e$ in $\ell_i$. If $\ell_i$ has the form $\pi_1\, \mathbf{e} \,\pi_2 \,\mathbf{e}' \,\pi_3$, where $\pi_i$ is a path of edges, then we say that \[\mathsf{S}_{\mathbf{e},\mathbf{e}'}(\ell_i)=\{\pi_1 \, \mathbf{e} \, \pi_3 \, , \, \pi_2 \, \mathbf{e}'\}\] 
is a \textbf{positive splitting} of $\ell_i$ at $\mathbf{e}$. 
We let $\SS_+(\mathbf{e},s)$ denote the multiset of strings that can be obtained from $s$ by a positive splitting of $s$ at $\mathbf{e}$. See the left-hand side of Figure~\ref{fig-operations} for an example.

Similarly, let $\mathbf{e}^{-1}$ be any specific copy of the edge $e^{-1}$ in $s$. If $\ell_i$ has the form $\pi_1 \, \mathbf{e} \, \pi_2 \, \mathbf{e}^{-1} \, \pi_3$, then we say that 
\[
\mathsf{S}_{\mathbf{e},\mathbf{e}^{-1}}(\ell_i)=\{\pi_1 \, \pi_3  \, ,  \, \pi_2\}
\] 
is a \textbf{negative splitting} of $\ell_i$ at $\mathbf{e}$. 
We let $\SS_-(\mathbf{e},s)$ denote the multiset of strings that can be obtained from $s$ by a negative splitting of $s$ at $\mathbf{e}$. See the left-hand side of Figure~\ref{fig-operations} for an example.

\medskip

\noindent\underline{\textbf{Deformations}}. Next, we define deformations at $\mathbf{e}$, an operation that combines a loop and a plaquette. Suppose that $\ell_i=\mathbf{e}\pi_1$, where $\pi_1$ is a path of edges. Let  $p=\mathbf{e'}\pi_2$ be a plaquette, where $\mathbf{e}'$ is another copy of the same lattice edge $e$ and $\pi_2=e_1e_2e_3$ is a path of edges.  We define the \textbf{positive deformation} of $\ell_i$ with $p$ at $\mathbf{e}$ to be\begin{equation*}
\ell_i\oplus_{\mathbf{e}}p = \mathbf{e} \pi_2 \mathbf{e}' \pi_1.
\end{equation*}
More generally, we define the sets of positive deformations at $\mathbf{e}$ by \begin{align*}
\DD_+(\mathbf{e},s) = \{\{\ell_1,\dots,\ell_{i-1},\ell_i\oplus_{\mathbf{e}}p,\ell_{i+1},\dots,\ell_n\}:p\in \cP(e)\},
\end{align*}where we recall that $\cP(e)$ is the collection of plaquettes in $\cP$ containing $e$ as one of the four boundary edges (with the correct orientation). See the right-hand side of Figure~\ref{fig-operations} for an example.

Similarly, let  $p=\mathbf{e^{-1}}\pi_2$ be a plaquette, where $\mathbf{e^{-1}}$ is a copy of the edge $e^{-1}$ and  $\pi_2=e_4e_5e_6$ is a path of edges. We define the \textbf{negative deformation} of $\ell_i$ with $p$ at $\mathbf{e}$ to be
\begin{equation*}
\ell_i\ominus_{\mathbf{e}}p = \pi_2 \pi_1.
\end{equation*}
More generally, we define the sets of negative deformations at $\mathbf{e}$ by \begin{align*}
\DD_-(\mathbf{e},s) = \{\{\ell_1,\dots,\ell_{i-1},\ell_i\ominus_{\mathbf{e}}q, \ell_{i+1},\dots,\ell_n\}:q\in \cP(e^{-1})\}.
\end{align*}
See the right-hand side of Figure~\ref{fig-operations} for an example.

\medskip

Note that (cf.\  the bottom part of Figure~\ref{fig-operations}) the negative splitting and negative deformation operations might introduce backtracks to the new loops (even if the original loop was a non-backtrack loop). We stress that if this is the case, we do \emph{not} remove the backtracks when performing these operations (unless explicitly stated otherwise); this convention will play an important role later in the paper.

\medskip

\begin{figure}[ht!]
\begin{center}
	\includegraphics[width=.49\textwidth]{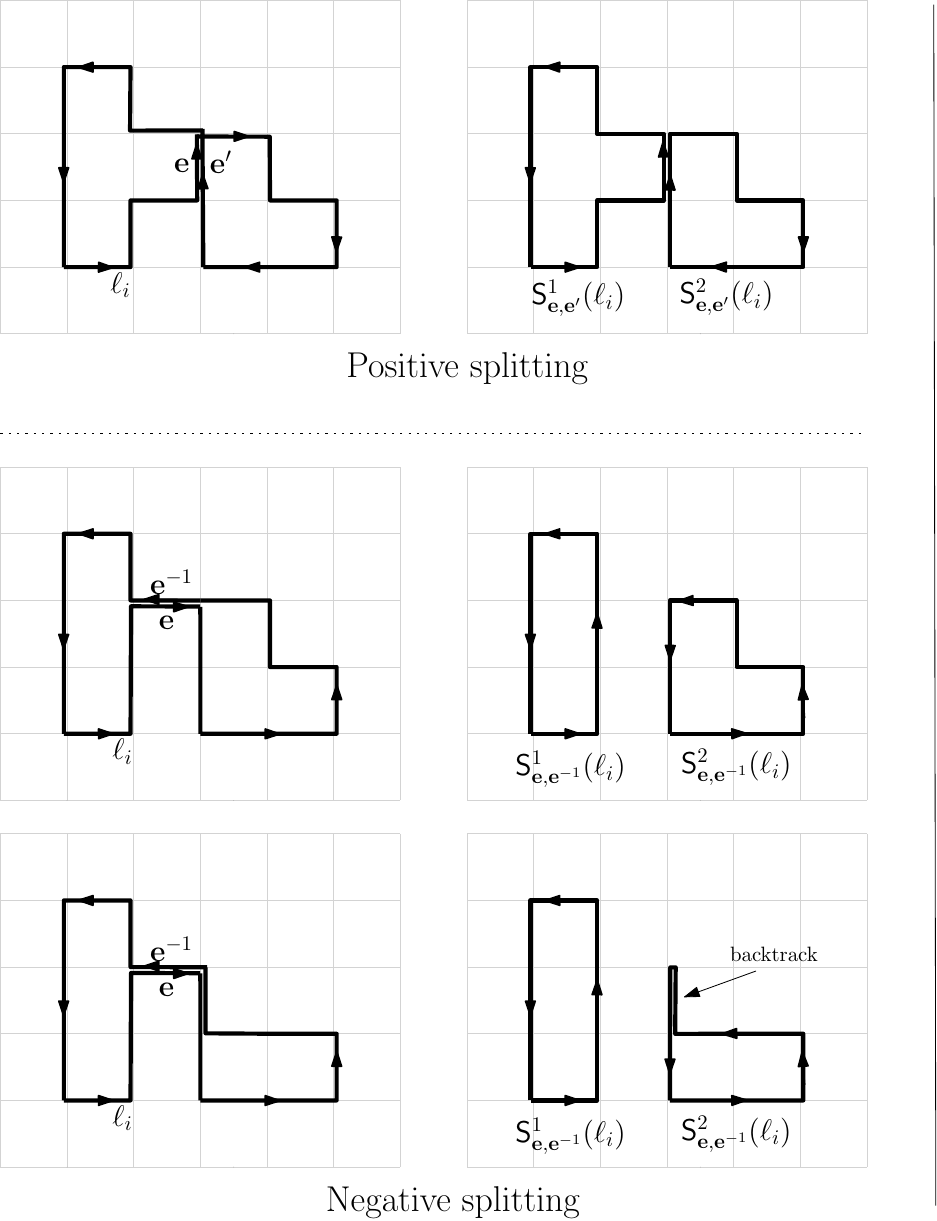}
        \includegraphics[width=.49\textwidth]{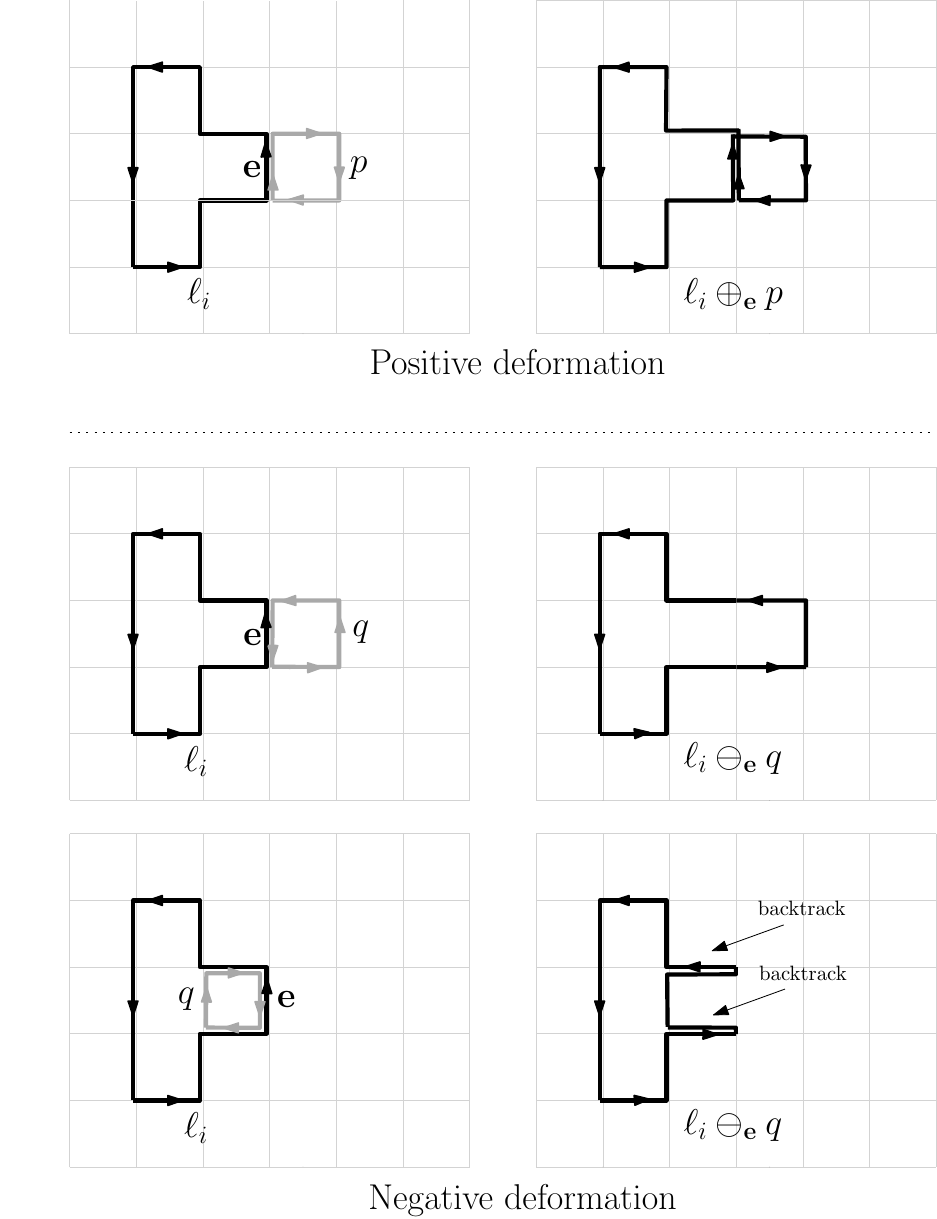} 
	\caption{\label{fig-operations} \textbf{Left:} An example of a positive splitting and two examples of negative splittings (the bottom one introduces a backtrack). \textbf{Right:} An example of a positive deformation and two examples of a negative deformations (the bottom one introduces two backtracks).}
\end{center}
\vspace{-3ex}
\end{figure}

With these operations defined, we can now introduce the master loop equation, which is the most important tool for this paper.

\begin{thm}[{\cite[Theorem 3.6]{borga2024surfacesumslatticeyangmills}}]\label{thm: fixed K 't Hooft master loop equation for surface sum} 
	Fix a non-null-loop $\ell$ and a finite plaquette assignment  $K:\cP\to\ZZ_{\geq 0}$. Let $\mathbf{e}$ be a specific edge of the loop $\ell$. Suppose that the  edge $\mathbf{e}$ is a copy of the lattice edge $e\in E_{\ZZ^2}$. Then
	\begin{align}\label{eq:mle-main}
		\phi^K(\ell) =&\sum_{\{\ell_1,\ell_2\}\in \SS_-(\mathbf{e},\ell)}
		\sum_{K_1+K_2=K}
		\phi^{K_1}(\ell_1)\phi^{K_2}(\ell_2) - \sum_{\{\ell_1,\ell_2\}\in \SS_+(\mathbf{e},\ell)}
		\sum_{K_1+K_2=K}\phi^{K_1}(\ell_1)\phi^{K_2}(\ell_2) \notag\\
		&+ \upbeta
		\sum_{p\in \cP(e^{-1},K)}
		\phi^{K\sm p}(\ell \ominus_{\mathbf{e}}p)
		-\upbeta\sum_{q\in \cP(e,K)}
		\phi^{K\sm q}(\ell \oplus_{\mathbf{e}}q),
	\end{align}
	where $\cP(e,K)$ denotes the collection of plaquettes in $\cP$ containing $e$ as one of the four boundary edges (with the correct orientation) such that $K(p)\geq 1$, and $K\sm p$ denotes the plaquette assignment $K'$ such that $K'(p)=K(p)-1$ and $K'(q)=K(q)$ for all $q\neq p$.
	
	We also used the (more compact) notation $\sum_{K_1+K_2=K}$ to indicate the sum $\sum_{\substack{K_1,K_2:\\K_1+K_2=K}}$, where $K_1+K_2=K$ means that $K_1(p) + K_2(p)=K(p)$ for all $p\in \cP$.
\end{thm}

\begin{remark}[\textsc{single location fixed plaquette assignment master loop equation}]\label{remark:single-location-mle}
We point out that the master loop equation in Theorem \ref{thm: fixed K 't Hooft master loop equation for surface sum} is more general than the ones previously presented in the literature as it is both for a fixed location in the loop and for a fixed plaquette assignment $K$. This seemingly subtle difference will be crucial for our proofs. See Section \ref{sec: proof tech} for more discussion.
\end{remark}

The master loop equation also gives us a useful recursion satisfied by the coefficients $c(\ell,K)$ defined in \eqref{eq:coefficents}. That is, the master loop equation in \eqref{eq:mle-main} gives the following relation:
\begin{align}\label{eq:mle-2}
		c(\ell,K) =&\sum_{\{\ell_1,\ell_2\}\in \SS_-(\mathbf{e},\ell)}
		\sum_{K_1+K_2=K}
		c(\ell_1,K_1)\cdot c(\ell_2,K_2) - \sum_{\{\ell_1,\ell_2\}\in \SS_+(\mathbf{e},\ell)}
		\sum_{K_1+K_2=K}
        c(\ell_1,K_1)\cdot c(\ell_2,K_2) \notag\\
		&+
		\sum_{p\in \cP(e^{-1},K)}
            c(\ell \ominus_{\mathbf{e}}p,K\sm p) 
		-\sum_{q\in \cP_{\ZZ^2}(e,K)}
        c(\ell \oplus_{\mathbf{e}}q,K\sm q).
\end{align}
In this paper, all the coefficients $c(\ell, K)$ will be computed using the master loop equation in \eqref{eq:mle-2} -- together with the  ``base'' cases coming from the discussion below \eqref{eq:jobfowerbf} -- and not using the specific surface sum expression in \eqref{eq:coefficents}. Nevertheless, we stress that the master loop equation \eqref{eq:mle-main} was obtained in \cite{borga2024surfacesumslatticeyangmills} largely using a surface sum point of view (and an associated peeling exploration of planar surfaces).

\subsubsection{Proof techniques}\label{sec: proof tech}

In this section, we discuss the central proof techniques used to prove our main results. As already mentioned in previous sections, the key point is to design a new surface exploration algorithm, i.e.\ ``peeling process'', which gradually explores the surfaces in a controlled manner. As is well known in the literature on random planar maps \cite{Curien2023}, such a peeling process can be encoded by a sequence of loops, which tracks the boundary of the surface as it is gradually deformed. As described in \cite[Chapter 4]{Curien2023}, for random planar maps, conditioning on the first step of the peeling process gives Tutte's equation, which is a fundamental recurrence relation that leads to many enumeration results about planar maps. 

For the surface sums in the present paper, the analog of Tutte's equation is precisely the master loop equation  of~\cite[Theorem 3.6]{borga2024surfacesumslatticeyangmills}, stated in this paper as~\cref{thm: fixed K 't Hooft master loop equation for surface sum}. In particular, this master loop equation roughly corresponds to conditioning on the first step of a peeling process which selects an edge on the boundary and reveals the face on the other side of the boundary.\footnote{There is an additional ``pinching'' step that is needed in \cite{borga2024surfacesumslatticeyangmills}, but we refrain from going more into the details here.} This ``single step'' recursion is one of the main contributions of \cite{borga2024surfacesumslatticeyangmills}.

It will also be important for us to not only track how the boundary of the surface evolves during the exploration process but also how the remaining unknown region evolves. In particular, this corresponds to tracking how the loop and plaquette assignment jointly evolve. Since the master loop equation of \cref{thm: fixed K 't Hooft master loop equation for surface sum} is stated for a fixed plaquette assignment, it naturally tracks both quantities. Note that the previous master loop equations in the literature \cite{chatterjee_rigorous_2019, jafarov2016wilson} are not stated for a fixed plaquette assignment, and thus would not allow us to define the peeling process that we describe shortly.

In the present paper, we develop a ``multi-step'' version of \cref{thm: fixed K 't Hooft master loop equation for surface sum}, which will be the key tool used to obtain our main results. Roughly, we seek to shrink the loop ``plaquette-by-plaquette'', until the loop disappears (and ensure that the resulting plaquette assignment is well behaved, which we do not wish to elaborate on here). But a single step of the master loop equation may not achieve this, and in fact, may grow the loop by a plaquette. Thus, naturally, we can try to take as many steps in the master loop equation as needed until the loop has indeed been shrunk by a plaquette. This can be thought of as a ``macro step'' (note for general loops some work has to be done to precisely define what ``shrunk by a plaquette'' means and thus how to define a macro step; see \cref{sec: Removing a plaquette}). The main point is to ensure that after such a macro step, the loop has indeed been shrunk, and moreover has not become too wild (we also need to have some control on the resulting plaquette assignment which again we do not wish to elaborate on). In other words, we have to ensure that the macro steps shrink the loop in a ``controlled manner''. The key technical contributions of the current paper are in identifying the precise notion of ``macro step'' and ``controlled'' and the verification that such ``controlled macro steps'' are even possible. Assuming this can be done, then by repeatedly applying these macro steps, we will be able to shrink the loop until it disappears, or more generally, until it is simple enough such that we may explicitly calculate the corresponding surface sum.

We highlight that in order to verify that such macro steps are possible, it is crucial that we fix a plaquette assignment (and thus constrain the unknown region of the surface) and track its evolution (i.e.\ have a fixed plaquette assignment master loop equation). This is because if the plaquette assignment is allowed to be picked freely at each step of the evolution, then the loop can grow forever (as there is always some plaquette assignment that can be used to make a surface with the larger loop). We view this ability to fix a plaquette assignment as one of the benefits of the surface perspective. Indeed, the master loop equation derived from the string-trajectory perspective (recall Remark \ref{remark:single-location-mle}) corresponds to an exploration where the plaquette assignment can be picked freely at each step. As an aside, we also note this previous less-general version of the master loop equation also corresponds to conditioning on a ``more averaged'' process which at each step, randomly chooses an edge among a (possibly large) collection of edges, and then explores at the chosen edge. 

We note that planarity is heavily used to define ``macro steps'' and ``control'' the loop under such evolution (i.e.\ in higher dimensions ``shrinking'' the loop is much more complicated and harder to track). We also mention that planarity is used in our algorithm of repeatedly applying macro steps to make the loop vanish.

We remark that the most technical parts of our arguments arise because our exploration process applies to arbitrary loops. If we restrict to the special case where the loop is supported on a single plaquette, which is the case treated in \cite[Section 4]{Basu:2016dnp}, then our exploration process specializes to something which is reminiscent of the algorithm developed in the cited reference. Thus, in some ways our surface exploration may be thought of as a generalization of the algorithm in \cite{Basu:2016dnp} to arbitrary loops. 

We conclude this section by discussing where the surface cancellations appear in this whole process.

\begin{remark}[\textsc{surface cancellations}]\label{rmk: Surface cancellations}
    The following discussion assumes some familiarity with the surface sum perspective developed in \cite{borga2024surfacesumslatticeyangmills}. 

    The master loop equation roughly says that instead of considering all the surfaces that can be constructed from some loop $\ell$ and plaquette assignment $K$ we can equivalently consider all the surfaces that can be constructed from loops obtained from $\ell$ by a loop operation and their corresponding plaquette assignment. In particular, as explained in \cite[Section 5.2]{borga2024surfacesumslatticeyangmills} the surfaces constructed from loops obtained from $\ell$ by a loop operation correspond to grouping the original surfaces constructed from $\ell$.\footnote{Note that the surfaces corresponding to loop operations do not correspond to all the surfaces of the original loop but this does not matter as the remaining surfaces cancel out (cf.\ the master cancellation lemma \cite[Lemma 5.25]{borga2024surfacesumslatticeyangmills}).} That is, a step of the master loop equation roughly corresponds to grouping surfaces.
    
   If we iterate this grouping in such a way that in the end we obtain groups of surfaces for which we can explicitly calculate the corresponding surface sums, then we can use this to compute the original surface sum. For example, if we end up in a setting where we know each resulting group of surfaces sums to zero, then we know that the original surface sum is zero. 

    A priori, the only collections of surfaces for which we know the explicit surface sum value are the ``base'' cases of $\phi^K(\ell)$ introduced in \cref{sec: Wilson loop expectations as surface sum large N}. Nonetheless, building from these cases we will be able to find more collections of surfaces that we know must cancel. Ultimately, once enough cases are covered, we may reduce computing general loops to these known cases.

    Thus, much of the rest of the paper is about grouping surfaces in such a way so that we know everything cancels out. We will achieve this through studying the master loop equation.
\end{remark}

\subsection{Triviality of the scaling limit and future research directions}\label{sec: trivial}

    One reason to consider the large-$N$ limit of lattice Yang--Mills theory (in any dimension $d\geq 2$) is the belief that, when the lattice mesh $\delta$ is sent to zero and $\upbeta$ is properly rescaled, it should lead to the construction of a $d$-dimensional large-$N$ \emph{continuum} Yang--Mills theory. Such a theory has only been constructed in dimension $d=2$ by L\'evy~\cite{lévy2012masterfieldplane}, but without considering the scaling limit of the lattice Yang--Mills model analyzed in this paper (see below for more details).

    Before investigating the approach outlined above in any dimension $d>2$, one would like to first double check (we will refer to this as the “two-dimensional test”) that the two-dimensional large-$N$ limit of lattice Yang--Mills theory (when the lattice mesh $\delta$ is sent to zero and $\upbeta$ is properly rescaled)  behaves  as L\'evy's two-dimensional large-$N$ continuum Yang--Mills theory~\cite{lévy2012masterfieldplane}.  In particular, this theory associates each (reasonably regular) continuum loop $\cL$ of the plane with a (non-trivial) real number, called the Wilson loop expectation of $\cL$; see \cite[Appendix B]{lévy2012masterfieldplane}.

    Before carrying out this “two-dimensional test”, we make an important observation.
    
    \begin{obs}\label{obs-formulas-make-sense}
       Recall that \cref{thm: sum over surfaces representation in 't hooft limit} guarantees that for any loop $\ell$ of $\ZZ^2$,
        \begin{equation}\label{eq:wedrfiuvbfwerui}
            \lim_{N\to\infty}\phi_{\Lambda_N,N,\upbeta}(\ell) =
		\phi(\ell),
        \end{equation}
        under the important assumption that $\upbeta$ is small enough. Moreover, as discussed in \cref{remark:critical}, we know that the above limit cannot hold for $\upbeta > 1$.
        
        Nevertheless, all the formulas that we computed in this paper for the quantity $\phi(\ell)$ make perfect sense for all values of $\upbeta\in\RR$.
    \end{obs}

    Given \cref{obs-formulas-make-sense}, a natural question --related to the ``two-dimensional test'' -- is whether one can send the lattice mesh $\delta$ to zero, while rescaling $\upbeta$ accordingly (not necessarily keeping $\upbeta$ small), in such a way that our formulas for $\phi(\ell)$ converge to those in~\cite[Appendix B]{lévy2012masterfieldplane}.

    We will argue, at least heuristically, that the answer is negative because any rescaling of $\upbeta$ leads to \emph{trivial} limiting formulas, and then discuss the potential sources of this discrepancy.

    \medskip

    Recall from \cref{rmk:cont-parallel} that for a simple continuum loop $\cL_\alpha$ of (Euclidean) area $\alpha$, the Wilson loop expectation of $\cL_\alpha^n$ in L\'evy's two-dimensional large-$N$ continuum Yang--Mills theory is given by 
    \begin{equation}\label{eq:levy2}
        \varphi(\cL_\alpha^n) =\left(\sum_{k=0}^{n-1}\frac{(-\alpha)^k}{k!}n^{k-1}\binom{n}{k+1}\right)\mathrm{e}^{-(n\alpha)/2}.
    \end{equation}
    On the other hand, if $\ell_a$ denotes a simple loop on $\ZZ^2$ of area $a$ (for instance, $\ell_a$ could simply be a rectangle of sides $s_1 \times s_2$), then, thanks to \cref{thm: wound simple loops WLE}, we know that for all $n\geq 1$,
    \begin{align}\label{eq:ref-exprss}
        \phi(\ell_a^n) = \frac{(-1)^{n+1}}{n}\binom{na-2}{n-1}\upbeta^{na}.
    \end{align}
    Now, if we want to approximate a continuum loop $\cL_\alpha$ with a lattice loop $\ell_a$, we should take the lattice mesh $\delta$ to zero (but keeping the loop $\ell_a$ of macroscopic size, for instance, in the rectangle case, by replacing $s_1 \times s_2$ by $\lfloor s_1/\delta \rfloor \times \lfloor s_2/\delta \rfloor$). Notice, when doing this, the area (i.e.\ number of plaquettes contained in the lattice loop) of the lattice loop  must be approximately $a(\delta)\sim \alpha/ \delta^2$. In particular, it must go to infinity as $\delta$ goes to zero. Thus, when taking the lattice mesh $\delta$ to zero, only the leading order term in $a=a(\delta)$ of the polynomial factor in \eqref{eq:ref-exprss} should contribute. That is, for all $n\geq 1$,
    \begin{align}\label{eq:ref-exprss2}
        \phi\left(\ell_{a(\delta)}^n\right) \sim C_n {(\alpha/ \delta^2)}^{n-1}\upbeta^{n{\alpha/ \delta^2}} \qquad\text{as }\delta\to 0,
    \end{align}
    where $C_n$ is some constant depending on $n$. Next, in order to recover the behavior of the continuum limit in \cite{lévy2012masterfieldplane} (recall \eqref{eq:levy2}),  we should scale $\upbeta$ in terms of $\delta$ so that, as $\delta\to 0$, we recover the exponential factor $\mathrm{e}^{-(n\alpha)/2}$. For instance, one could rescale $\upbeta$ as $\mathrm{e}^{-\delta^2/2}$. But, if $\upbeta$ has this rescaling, then the polynomial coefficient in $a/ \delta^2$ in \eqref{eq:ref-exprss2} dominates and  $\phi(\ell_{a(\delta)}^n)$ tends to infinity. Clearly, this is not what we need to get formulas similar to the ones in \eqref{eq:levy2}. Therefore, $\upbeta$ must also be scaled to compensate for the polynomial coefficient in $a/ \delta^2$ in \eqref{eq:ref-exprss2}, getting that $\upbeta$ should be rescaled, for instance, as follows 
    \begin{align*}
        \upbeta = \mathrm{e}^{-\delta^2/2}\delta^{c\delta^2},
    \end{align*}
    for some constant $c\in\RR$ (in particular, $\upbeta\sim 1$ as $\delta\to 0$). Substituting this in \eqref{eq:ref-exprss2},
    we get that, for all $n\geq 1$,
    \begin{align}\label{eq:ref-exprss3}
         \phi\left(\ell_{a(\delta)}^n\right) \sim C_n (\alpha^{n-1}\mathrm{e}^{-n\alpha/2})\delta^{cn\alpha-2(n-1)} \qquad\text{as }\delta\to 0.
    \end{align}
    To get a non-trivial continuum limit as $\delta\to 0$, we should have that the limit of the right-hand side of \eqref{eq:ref-exprss3} is finite for all $n\geq 1$. This is only possible if $cn\alpha-2(n-1)\geq 0$. As this should hold for all $n\geq 1$, we get that $c\geq 2/\alpha$. But if we plug such a $c$ in \eqref{eq:ref-exprss3}, we will get that for all $n\geq 1$,
    \begin{align*}
        \phi\left(\ell_{a(\delta)}^n\right) \sim  0 \qquad\text{as }\delta\to 0.
    \end{align*} 
    One can turn the above heuristic argument into a proof (but we skip these details here) of the following fact: Any rescaling of $\upbeta$ as a function of $\delta$ that yields a finite limit of $\phi(\ell_{a(\delta)}^n)$ as $\delta \to 0$ for all $n \geq 1$ must result in the trivial limit, equal to zero for all $n \geq 1$. Note that the latter fact also suggests that any rescaling of $\upbeta$ as a function of $\delta$ which gives a non-trivial limit for the empirical spectral distribution of $Q_{\ell}$ must result in the uniform measure on the unit circle (which has all the non-trivial moments equal to zero); note that this is consistent with the right-hand side of \cref{fig-limiting-spectral-density}.

    \medskip

    To continue our discussion, we need to first clarify one additional point: the finite-$N$ lattice Yang--Mills model introduced at the beginning of this paper in  \eqref{eq: lattice YM measure-2} has been defined in terms of the so-called \emph{Wilson action}. In the literature, there are  other finite-$N$ lattice Yang--Mills models, for instance, defined in terms of the so-called \emph{Villain or Manton action}; see \cite[Section 2.3.1]{chevyrev2023invariant}.

    With this we now continue the previous discussion. In this section, we argued (at least heuristically) that, for the finite-$N$ lattice Yang--Mills model with the Wilson action, sending the lattice mesh $\delta$ to zero \emph{after} taking the large-$N$ limit, and using our formulas for all values of $\upbeta\in\RR$, can only result in a \emph{trivial} scaling limit. In \cite[ Corollary 2.26]{chevyrev2023invariant} it has been shown that the scaling limit of the finite-$N$ lattice Yang--Mills model (with Wilson, Villain, or Manton action) converges  after sending the lattice mesh $\delta$ to zero to the two-dimensional finite-$N$ continuum Yang--Mills theory of~\cite{levy2006discrete, chevyrev2019yang}. Moreover, \cite{lévy2012masterfieldplane} also shows that such a two-dimensional finite-$N$ continuum Yang--Mills theory converges, when $N$ tends to infinity, to the two-dimensional large-$N$ continuum Yang--Mills theory. Hence, this section can be interpreted as saying that the $\delta$ to zero and the $N$ to infinity limits for the lattice Yang--Mills model with Wilson action do not commute (at least when $\beta=\upbeta N$ and $\upbeta$ is small).

    A few natural questions or possible future directions of research are the following ones:
    \begin{enumerate}
        \item Is it possible that one needs to take the two limits in $N$ and $\delta$ simultaneously in order to recover a non-trivial scaling limit for the two-dimensional lattice Yang--Mills theory with the Wilson action?
        
        \item Recall from \cref{remark:critical} that \cref{thm: sum over surfaces representation in 't hooft limit} cannot hold for $\upbeta>1$. Hence, when $\upbeta>1$, one should find some new limiting  Wilson loop expectations $\phi'(\ell)$.
        
        Is it possible that such new  limiting Wilson loop expectations  $\phi'(\ell)$ yields new formulas that admit a rescaling of $\upbeta$ in $\delta$ leading to non-trivial limits?
        
        \item It would be interesting to construct a surface-sum representation for the large-$N$ limit of Wilson loop expectations in any dimension using the Villain action, instead of the Wilson action of \eqref{eq: lattice YM measure-2} used in this paper. Indeed, \cite{lévy2012masterfieldplane} shows that the two-dimensional large-$N$ lattice Yang--Mills model with Villain action converges, as $\delta$ tends to zero, to the two-dimensional large-$N$ continuum Yang--Mills theory. Hence, the two-dimensional large-$N$ lattice Yang--Mills model with Villain action passes the “two-dimensional test” discussed at the beginning of this section, and it might be a more promising model for investigating the construction of a $d$-dimensional large-$N$ continuum Yang--Mills theory for $d > 2$. Our hope is that such a surface-sum representation for the large-$N$ limit of Wilson loop expectations in any dimension, built using the Villain action, may provide new tools for studying this model—just as it has been the case for this and our previous paper \cite{borga2024surfacesumslatticeyangmills} in the case of the Wilson action.

        \item The explicit computations performed throughout this paper heavily rely on the fact that we are in two dimensions. That being said, can any of the ideas or techniques developed in this paper be applied to higher dimensions? One thing to note is that in dimension three, as shown in \cref{rem: Infinite contributing plaquette assignments in dimension three}, the loop $\ell$ consisting of a single plaquette $p$ already admits an \emph{infinite} number of plaquette assignments that yield a non-zero surface sum (though these sums may be positive or negative,\footnote{Note that all the surface sums considered in \cref{rem: Infinite contributing plaquette assignments in dimension three} have positive sign, but that surface sums are not all the non-zero surface sums.} leading to more intricate cancellations). Thus, in order to understand the higher-dimensional case, we believe a finer understanding of such cancellations would be required.

        \item We showed that, in the large-$N$ limit, all but very few plaquette assignments contribute to the surface sum. Can this phenomenon already be observed in the finite-$N$ surface sum for large values of $N$ in the regime $\beta=\upbeta N$? More generally, could the complete description of the large-$N$ limit detailed in this paper inform our understanding of the finite-$N$ case in any meaningful way? If so, then this would help in extending the main result of \cite{CNS2025} to a larger regime of parameters.
    \end{enumerate}

\section{Contributions from plaquette assignments outiside the canonical collection are trivial}\label{sec: expectations in 2D}

In this section, we prove \cref{thm: erasable loops have one plaquette assignment}. The proof will follow by repeatedly applying, in a suitable way, the master loop equation (\cref{thm: fixed K 't Hooft master loop equation for surface sum}) until we can write $\phi^K(\ell)$ as a sum of terms that are all zero. 

More precisely, the proof of \cref{thm: erasable loops have one plaquette assignment} will proceed as follows: First, in Section~\ref{sec: Removing a plaquette}, we introduce a process that allows us to ``remove'' a plaquette from a loop. Then, in \cref{sec: Loops supported on a plaquette}, we show that each of the three conditions appearing in the definition of the canonical collection of plaquette assignments (\cref{defn:collection-plaquette-ass}) is necessary to have a non-zero surface sum of certain loops of small size (having support on one or two plaquettes).  Finally, in \cref{sec: general loops}, to finish the proof of \cref{thm: erasable loops have one plaquette assignment}, we repeatedly apply the ``removing'' procedure to shrink general loops to those that we know how to handle from the previous section.

\begin{remark}\label{remark: loops}
    We remark that if $\ell$ is such that $d_{\ell}(p) = |h_{\ell}(p)|$ for all $p \in \cP$, then the proof can be significantly simplified (although we will only present the general version).
    Indeed, as explained in \cref{rem:card-set-cK}, for such loops we have $\cK_{\ell} = \{K_{\ell}\}$, so it suffices to show that $\phi^K(\ell) = 0$ for all $K \neq K_{\ell}$; a much simpler task.
\end{remark}

\subsection{Removing a plaquette from a loop by pushing its edges}\label{sec: Removing a plaquette}

In this section, we describe a procedure for applying the master loop equation to ``remove'' a plaquette from a loop~$\ell$ (see Sections~\ref{sec: Master loop equation operator}~and~\ref{sec: removing a plaquette}). We also describe how certain key quantities—such as the height, the plaquette assignment, and the loop’s edges—evolve during this process (see~\cref{sec: removing a plaquette-pushing}) by observing that ``removing'' a plaquette amounts to ``pushing'' certain edges. The terminology of ``removal'' and ``pushing'' will be clarified in Observations~\ref{obs:remove_plaq}~and~\ref{obs: move edges}, respectively.

Before that, we introduce the following notation and recall a fact from \cite{borga2024surfacesumslatticeyangmills}. Both will be useful in the subsequent sections. Given a string $s=\{\ell_1,\dots,\ell_n\}$ and a multiset of plaquette assignments $\multK=\{K_1,\dots, K_n\}$ , we denote a pairing\footnote{Note that both $s$ and $\multK$ are unordered, so such pairings need to be specified. Nevertheless, they will often be clear from the context and not explicitly stated.} of $s$ and $\multK$  by
\begin{align*}
    (s,\multK) = \{(\ell_1,K_1),\dots,(\ell_n,K_n)\},
\end{align*} 
and call it a \textbf{string plaquette assignment}. Note that $(s,\multK)$ is a multiset of pairs $(\ell_i,K_i)$.
We say that a string plaquette assignment is balanced, if $(\ell_i,K_i)$ is balanced for all $i\in[n]$. We also introduce the compact notation
\begin{align}\label{eq: fixed K factor}
    \phi^{\multK}(s) := \prod_{i=1}^n\phi^{K_i}(\ell_i).
\end{align}

Finally, we recall from \cite{borga2024surfacesumslatticeyangmills} that the Wilson loop expectation $\phi^K(\ell)$ for a fixed plaquette assignment $K$ introduced in \cref{eq:def-phi_K} is invariant under removing backtracks of $\ell$.

\begin{lem}[{\cite[Lemma 4.2]{borga2024surfacesumslatticeyangmills}}]\label{lemma: backtrack cancellations}
Suppose that $\ell=\pi_1 \, e \,e^{-1} \, \pi_2$ is a loop with backtrack $e \, e^{-1}$ and $(\ell,K)$ is a balanced loop plaquette assignment. Then, 
\begin{equation*}
	\phi^K(\pi_1 \, e \, e^{-1} \, \pi_2) = \phi^K(\pi_1 \, \pi_2).
\end{equation*}
\end{lem}

\subsubsection{The master loop operator and trajectory trees}\label{sec: Master loop equation operator}

We start by introducing certain trees that provide a convenient way to keep track of all the string plaquette assignments generated by repeatedly applying the master loop equation (\cref{thm: fixed K 't Hooft master loop equation for surface sum}). Recall from \cref{sec: Wilson loop expectations as surface sum large N} that we do \emph{not} remove backtracks when performing loop operations.

When we apply the master loop equation  at an edge~$\mathbf{e}$ of a loop~$\ell$ with a finite plaquette assignment~$K$, we obtain a new collection of string plaquette assignments. We define the \textbf{master loop operator} (ML-operator) to be the operator that takes as input a balanced loop plaquette assignment~$(\ell, K)$ and a specific edge~$\mathbf{e}$ of~$\ell$, and outputs all balanced string plaquette assignments that appear on the right-hand side of the master loop equation. That is (cf.\ \cref{fig:master-loop-operator}),
\begin{align*}
    \mle_{\mathbf{e}}(\ell,K) &\coloneqq \Big\{\{(\ell_1,K_1),(\ell_2K_2)\}: \{\ell_1,\ell_2\}\in \SS_+(\mathbf{e},\ell), K_1+K_2=K, (\ell_i,K_i)\text{ is balanced for $i=1,2$}\Big\}\\
    &\,\,\,
    \cup\Big\{\{(\ell_1,K_1),(\ell_2K_2)\}: \{\ell_1,\ell_2\}\in \SS_-(\mathbf{e},\ell), K_1+K_2=K, (\ell_i,K_i)\text{ is balanced for $i=1,2$}\Big\}\\
    &\,\,\,
    \cup \Big\{\{(\ell\oplus_{\mathbf{e}} p, K\sm\{p\})\}: p\in \cP, e\in p, K(p)\geq 1\Big\} \\
    &\,\,\,
    \cup \Big\{\{(\ell\ominus_{\mathbf{e}} p,K\sm\{p\})\}: p\in \cP, e^{-1}\in p, K(p)\geq 1\Big\},
\end{align*}
where the sets on the right-hand side are (possibly empty) multisets of string plaquette assignments. In particular, note that in the last two cases (corresponding to deformations), $\mle_{\mathbf{e}}$ outputs string plaquette assignments of cardinality one. In particular, there can be multiple \emph{identical} string plaquette assignments or multiple string plaquette assignments having the same string, but with different plaquette assignments corresponding to each loop (cf.\ \cref{fig:master-loop-operator}).

\begin{figure}[ht!]
\begin{center} 
    \includegraphics[width=\textwidth]{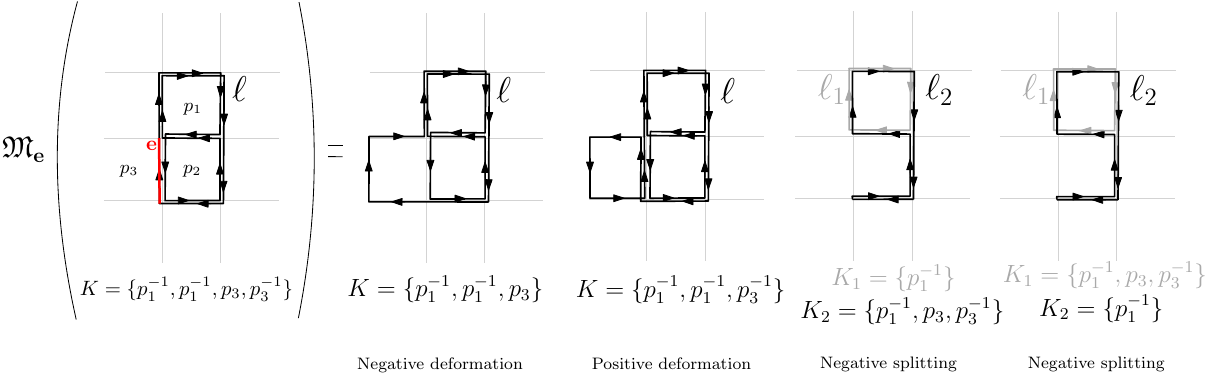} 
	\caption{\label{fig:master-loop-operator}A balanced loop plaquette assignment $(\ell, K)$ is shown in parentheses, with a specific edge~$\mathbf{e}$ of~$\ell$ highlighted in red. On the right, the four string plaquette assignments produced by $\mle_{\mathbf{e}}(\ell, K)$ are displayed. The first results from a negative deformation with $p_3^{-1}$, and the second from a positive deformation with $p_3$. The third and fourth arise from a negative splitting (note that there are two possible choices for $K_1$ and $K_2$ such that the resulting string plaquette assignment remains balanced). Note that we do not remove the backtracks formed in the loop $\ell_2$.}
\end{center}
\vspace{-3ex}
\end{figure}

We similarly define the ML-operator for a balanced string plaquette assignment 
\[
(s,\multK)=\{(\ell_1,K_1),\dots,(\ell_n,K_n)\}
\] 
and a specific edge $\mathbf{e}$ of a loop $\ell_i$ in the string $s$, by applying the ML-operator $\mle_{\mathbf{e}}$ to the loop plaquette assignment $(\ell_i,K_i)$, while leaving the rest of the string identical. That is, 
\begin{multline*}
  \mle_{\mathbf{e}}(s,\multK)=\mle_{\mathbf{e}}\left(\{(\ell_1,K_1),\dots,(\ell_n,K_n)\}\right) \\
  \coloneqq \bigcup \;\;\bigg\{\left\{(\ell_1,K_1),\dots(\ell_{i-1},K_{i-1}), (\ell_{i,1},K_{i,1}),\dots,(\ell_{i,m},K_{i,m}), (\ell_{i+1},K_{i+1}),\dots,(\ell_n,K_n)\right\}\bigg\},
\end{multline*}
where the union ranges over $\{(\ell_{i,1},K_{i,1}),\dots,(\ell_{i,m},K_{i,m})\}\in  \mle_{\mathbf{e}}(\ell_i,K_i)$ (note that $m$ can only take the value 1 or 2) and the sets on the right-hand side are (possibly empty) multisets of string plaquette assignments.

\begin{remark}\label{rem:balanced}
    We emphasize that $\mle_{\mathbf{e}}(\ell,K)$ (thus also $\mle_{\mathbf{e}}(s,\multK)$) outputs balanced string plaquette assignments $(s',\multK')=\{(\ell'_1,K'_1),\dots,(\ell'_n,K'_n)\}$, i.e.\ each component $(\ell'_i,K'_i)$ is balanced.
\end{remark}

Recall that $\phi^K(\ell) = 0$ if $(\ell, K)$ is not balanced. 
Therefore, the ML-operator indeed outputs all string plaquette assignments that potentially contribute to the right-hand side of the master loop equation.

If we pick an edge $\mathbf{e}'$ on one of the string plaquette assignments $(s',\multK')$ outputted by the ML-operator $\mle_{\mathbf{e}}(s,\multK)$, we can apply the ML-operator $\mle_{\mathbf{e}'}$ again (note now at $\mathbf{e}'$), obtaining a new collection of string plaquette assignments $\mle_{\mathbf{e}'}(s',\multK')$.

Iterating these operations, we can naturally form certain unordered\footnote{Unordered means that the children of a node are not ordered from left to right in a specific way.} rooted trees where each vertex in the tree is labeled by one of the outputs of a ML-operator applied to the label of its parent vertex. 
More precisely, starting from a loop plaquette assignment $(\ell,K)$ we build a \textbf{trajectory tree $\cT(\ell,K)$ generated from $(\ell,K)$} as follows (cf.\ \cref{fig:master-loop-operator-tree}): 
\begin{itemize}
    \item We start by adding the root of $\cT(\ell,K)$ and label it with the string plaquette assignment $(\ell,K)$. Then, we have two options:
    \begin{itemize}
        \item Either, we declare the root to be a final leaf;
        \item Or we choose an edge $\mathbf{e}$ of $\ell$, and we add $\#\mle_{\mathbf{e}}(\ell,K)$ new children leaves to the root and label each new leaf with one of the string plaquette assignments in $\mle_{\mathbf{e}}(\ell,K)$.
    \end{itemize}
    \item Then, iteratively, for each leaf $l$ of $\cT(\ell,K)$, that has not yet been declared a final leaf, labeled with a string plaquette assignment $(s,\multK)$, if $s$ is empty then we declare $l$ to be a final leaf; otherwise, if $s$ is non-empty, we have two options:
    \begin{itemize}
        \item Either, we declare $l$ to be a final leaf;
        \item Or, we choose an edge $\mathbf{e}$ of $s$, and we add $\#\mle_{\mathbf{e}}(s,\multK)$ new leaves to $l$ and label each new leaf with one of the string plaquette assignments in $\mle_{\mathbf{e}}(s,\multK)$. 
    \end{itemize}
    \item Finally, for each final leaf $l$ labeled by the string plaquette assignment $\{(\ell_1,K_1),\dots, (\ell_m,K_m)\}$ relabel the leaf $l$ with the string plaquette assignment obtained from $\{(\ell_1,K_1),\dots, (\ell_m,K_m)\}$ by removing all the backtracks from $\ell_i$ for all $i\in[m]$.
\end{itemize}
We stress that $\cT(\ell,K)$ depends both on when we declare a leaf to be a final leaf and, for each internal vertex, on the choice we made when selecting the edge $\mathbf{e}$ for the ML-operator $\mle_{\mathbf{e}}$ to use. We also note that the somewhat odd convention of not removing backtracks until the end of a trajectory tree will be quite useful for some later proofs (see for instance \cref{lemma: MLE preserves K and height} and \cref{lem: MLE for edge loops}).

Often, when the starting loop plaquette assignment is unambiguous, we will simply refer to such a tree as a \textbf{trajectory tree}. We will also often refer to a vertex of a tree $\cT(\ell,K)$ by directly referring to its label.

\begin{figure}[ht!]
\begin{center} 
    \includegraphics[width=\textwidth]{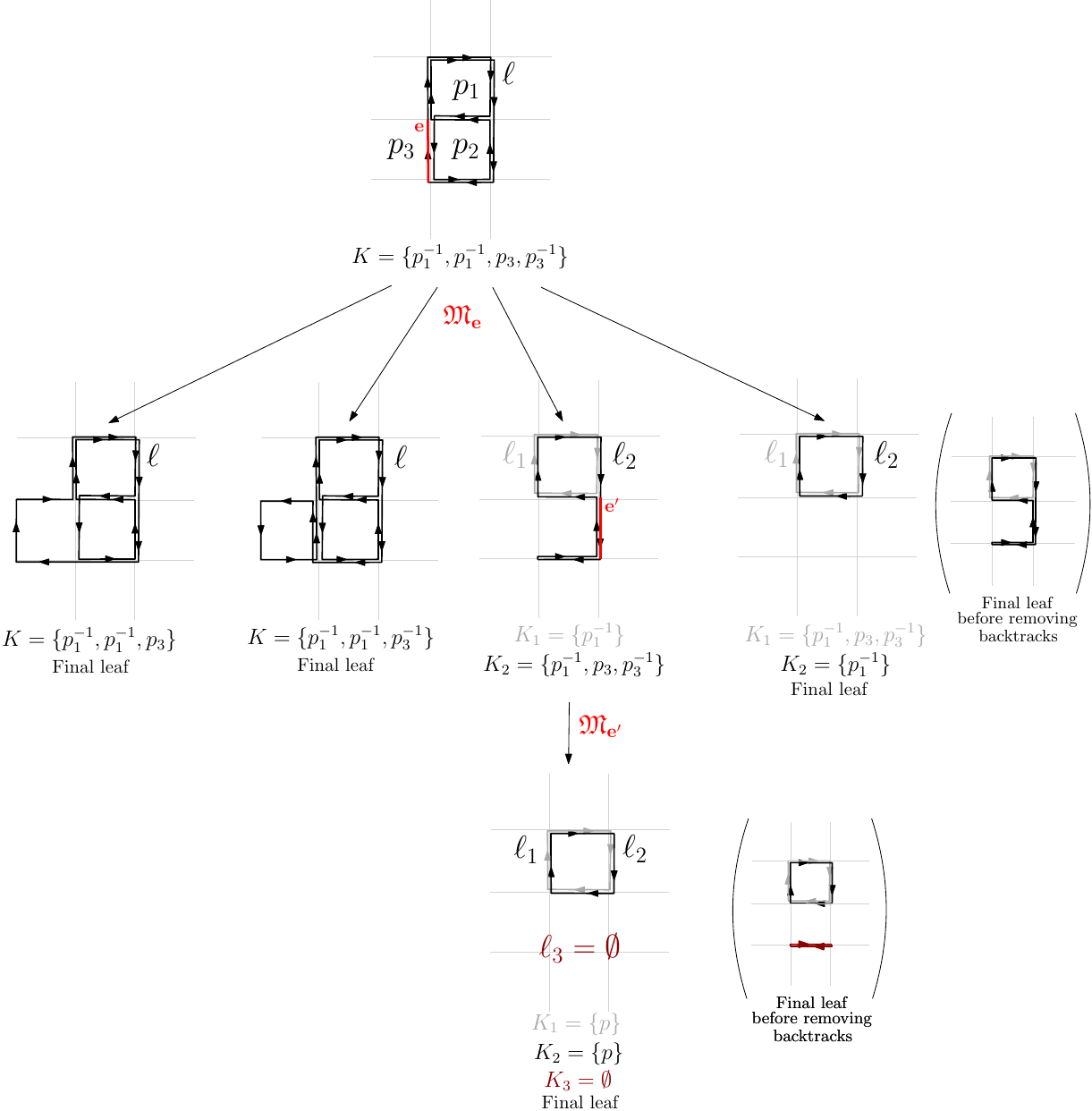} 
	\caption{\label{fig:master-loop-operator-tree} The trajectory tree $\cT(\ell, K)$ generated from the loop plaquette assignment $(\ell, K)$ shown at the top. The first generation of the tree is obtained by applying the ML-operator $\mle_{\mathbf{e}}$, as in \cref{fig:master-loop-operator}. Three out of four leaves are declared final leaves. For the fourth leaf, the ML-operator $\mle_{\mathbf{e}'}$ is applied, producing a new leaf, which is then declared a final leaf. Backtracks are removed from all final leaves; the loops shown in parentheses represent the loops before backtrack removal.}
\end{center}
\vspace{-3ex}
\end{figure}

Let $\cT=\cT(\ell,K)$ be a trajectory tree generated from $(\ell,K)$ and let $\cV_1$ denote the vertices at distance one from the root, i.e.\ the children of the root. Since applying the ML-operator $\mle_{\mathbf{e}}$ for some edge $\mathbf{e}$ in $\ell$ outputs the contributing terms on the right-hand side of the master loop equation (\cref{thm: fixed K 't Hooft master loop equation for surface sum}) and $\phi^{K_i}(\ell_i)$ is invariant under removing backtracks (\cref{lemma: backtrack cancellations}; this is only necessary if one of the children of the root is a final leaf), we have that
\begin{align*}
    \phi^K(\ell) = \sum_{(s,\cK)=\{(\ell_1,K_1),\dots(\ell_m,K_m)\}\in \cV_1} \upbeta_{s,\cK} \cdot \prod_{i=1}^m\phi^{K_i}(\ell_i)\stackrel{\eqref{eq: fixed K factor}}{=}\sum_{(s,\multK)\in \cV_1}\upbeta_{s,\cK}\cdot\phi^{\multK}(s),
\end{align*}
where we note that here $m$ is either 1 or 2, and $\upbeta_{s,\cK}$ is a prefactor (irrelevant for our purposes in this section) of the form $(-1)^x \upbeta^{y}$ for some $x, y \in \{0,1\}$, arising from the right-hand side of \eqref{eq:mle-main}. 

More generally, using the fact that the ML-operator acts on a single loop of each string, we get, again from the master loop equation (\cref{thm: fixed K 't Hooft master loop equation for surface sum}) and the invariance under removing backtracks (\cref{lemma: backtrack cancellations}), the following useful observation. 

\begin{obs}\label{obs: tree WLE relation} 
    For a trajectory tree $\cT$, let $\cL(\cT)$ denote the leaves of $\cT$. If $\cT=\cT(\ell,K)$ is a trajectory tree generated from $(\ell,K)$ then
    \begin{align*}
    \phi^K(\ell) = \sum_{(s, \cK)=\{(\ell_1,K_1),\dots(\ell_m,K_m)\}\in \cL(\cT)} \upbeta_{s,\cK} \cdot \prod_{i=1}^m\phi^{K_i}(\ell_i) \stackrel{\eqref{eq: fixed K factor}}{=} \sum_{(s,\multK)\in \cL(\cT)} \upbeta_{s,\cK} \cdot \phi^{\multK}(s),
\end{align*}
where $\upbeta_{s,\cK}$ is a prefactor (irrelevant for our purposes in this section) of the form $(-1)^x \upbeta^{y}$ for some $x \in \{0,1\}$ and $y\geq 0$, which depends on all the labels appearing along the path from the root of $\cT$ to the leaf $(s,\cK)$.
\end{obs}

\subsubsection{Removing a plaquette from a loop using trajectory trees}\label{sec: removing a plaquette}

Now, we detail a procedure that, given a string plaquette assignment $(\ell,K)$, constructs a trajectory tree $\cT(\ell,K)$ with all leaves labeled by strings which no longer contain a marked plaquette in the support of any of the loops forming the string. To make this precise, we first introduce some terminology and notation. 

Given a (possibly infinite) collection of plaquettes $\{p_i\}_i$, let $\cE(\{p_i\}_i)$ denote the collection of edges in $\{p_i\}_i$ in both orientations. For a (possibly infinite) path of plaquettes $\mathfrak{p}= \{p_i\}_i$, defined above \cref{def: lattice distance}, let 
\begin{equation}\label{eq:bewivfov}
    \cI(\mathfrak{p}) = \{e\in E: \mathfrak{e}(p_i,p_{i+1})\text{ corresponds to $e$}\}_{i}
\end{equation}
be the sequence of edges (in both orientations) in the interior of $\mathfrak{p}$.

We say that an edge $e$ of $\ell$ is \textbf{connected to infinity} if there is an infinite path of plaquettes $\mathfrak{p}= \{p_i\}_{i=1}^{\infty}$ such that $\mathfrak{e}(p_1,p_2)$ is the unoriented edge corresponding to $e$ and $\cI(\mathfrak{p})$ does not contain any edge of $\ell$ other than $e^{\pm 1}$ (cf.\ top-left of \cref{fig: traj tree}). 

Given any infinite sequence of edges $\{e_i\}_{i=1}^{\infty}$ such that no edge appears more than once (in either orientation), we define the \textbf{outermost edge of $\ell$ in $\{e_i\}_{i=1}^{\infty}$} to be the edge $e_j$ with the highest index such that $e_j$ is in $\ell$, and $e_k \notin \ell$ for all $k > j$. Such an index $j$ is well defined because $\ell$ contains only finitely many edges, and each edge appears at most once (in either orientation) in the sequence $\{e_i\}_{i=1}^{\infty}$. 

In what follows, we write ``$e$ is an edge contained in $\ell$'' to refer to a lattice edge $e$ appearing in $\ell$, potentially multiple times, but without referring to any specific copy of $e$ in $\ell$.

\begin{defn}\label{defn:tr-tree-gen}
Let $(\ell, K)$ be a balanced loop plaquette assignment and let $e$ be an edge contained in $\ell$ that is connected to infinity by a path of plaquettes $\mathfrak{p}$. We define the \textbf{trajectory tree $\cT_{e,\mathfrak{p}}(\ell, K)$ generated from $(\ell, K, e, \mathfrak{p})$} to be the trajectory tree constructed from $(\ell, K)$ by (cf.\ \cref{fig: traj tree}) :
\begin{itemize}
    \item Declaring a leaf $l$ to be a final leaf only if the string plaquette assignment $\{(\ell_1, K_1), \dots, (\ell_n, K_n)\}$ corresponding to $l$ satisfies the following property:  
    \begin{equation}\label{eq:noplins}
        \text{$\ell_i\cap \cI(\mathfrak{p})=\emptyset,$ for all $i \in [n]$.}
    \end{equation}
    Note the root cannot be a final leaf as $e$ is contained in $\ell$.

    \item For any non-final leaf, $l= \{(\ell_1,K_1),\dots,(\ell_n,K_n)\}$, arbitrarily select one of the loop plaquette assignments in $l$ such that the corresponding loop contains an edge in $\cI(\mathfrak{p})$ (note such a loop plaquette assingment must exist for any non-final leaf by the construction of final leaves above). Then apply the the ML-operator to any copy of the outermost edge of $\ell_i$ in $\cI(\mathfrak{p})$.

    \item Finally, for all final leaves, remove all backtracks as required by the definition of trajectory trees generated from a loop plaquette assignment (\cref{defn:tr-tree-gen}).
\end{itemize}
\end{defn}

We will prove in \cref{lem:finite-trees} that $\cT_{e,\mathfrak{p}}(\ell,K)$ is a finite tree. Note that the trajectory tree $\cT_{e,\mathfrak{p}}(\ell,K)$ will depend on the choices of the arbitrary copy of the outermost edge selected for the location where to perform the MLE operator (i.e.\ selecting different copies would produce different trees). However, this arbitrary choice will be irrelevant for our results. Thus, we choose to slightly abuse terminology by referring to one possible such trajectory tree as the trajectory tree $\cT_{e,\mathfrak{p}}(\ell, K)$ generated from $(\ell, K, e, \mathfrak{p})$.

\begin{figure}[ht!]
\begin{center}
	\includegraphics[width=0.65\textwidth]{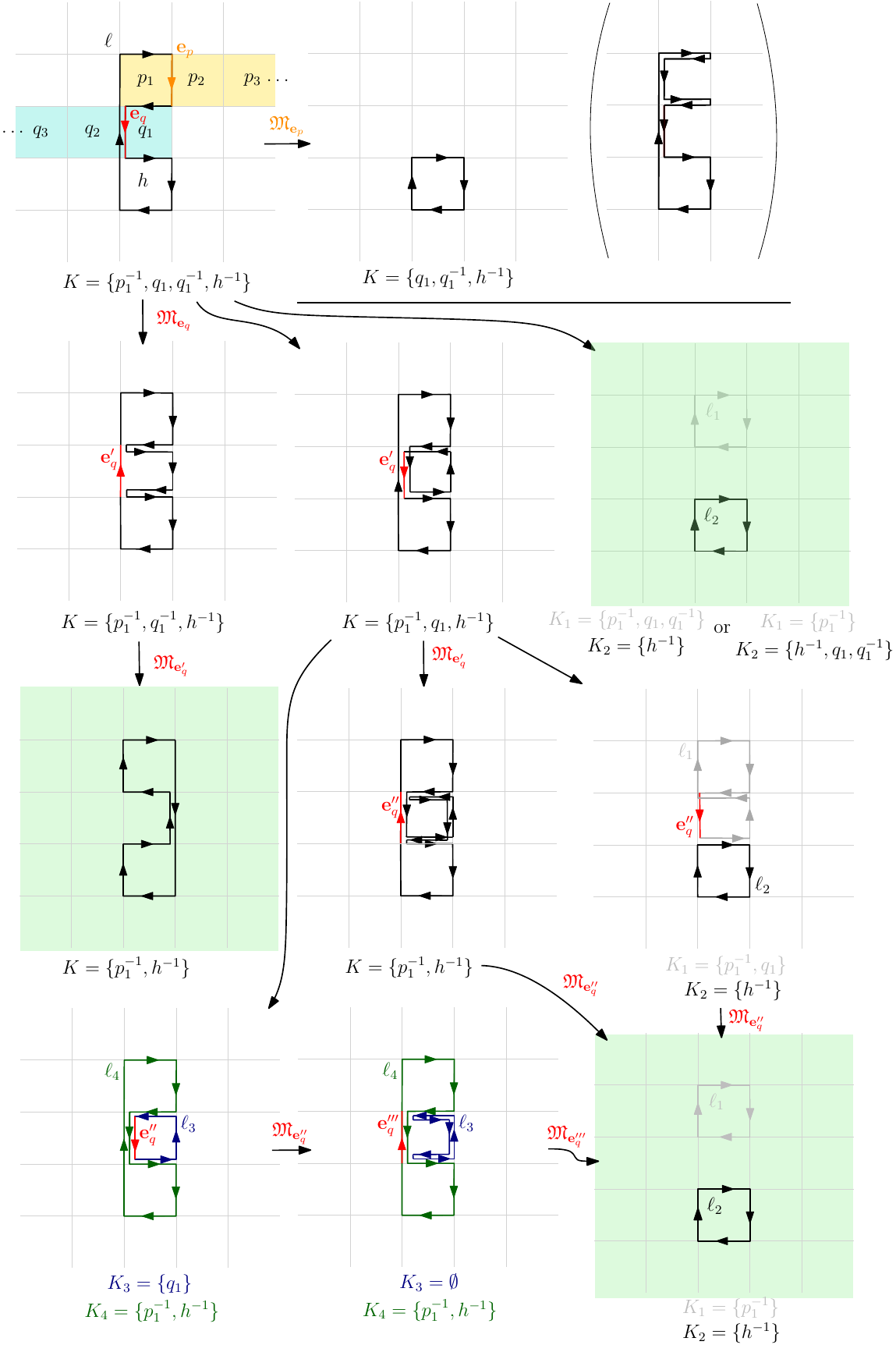}  
	\caption{\label{fig: traj tree} \textbf{Top-left:} A balanced loop plaquette assignment $(\ell,K)$ with two edges, $e_p$ and $e_q$, connected to infinity by $\mathfrak{p}=\{p_i\}_{i=1}^{\infty}$ (yellow) and $\mathfrak{q}=\{q_i\}_{i=1}^{\infty}$ (blue), respectively. A specific copy of each of these edges is highlighted in orange and red respectively. Note $p_1$ is in the support of $\ell$ but $q_1$ is not. \textbf{Top row:} The top row of the figure is the trajectory tree $\cT_{e_p,\mathfrak{p}}(\ell,K)$ generated from $(\ell,K,e_p,\mathfrak{p})$. The backtracks are removed from the final leaf and the loop shown in the parentheses is the final leaf loop before backtrack removal. \textbf{Rest of figure:} The rest of the figure is the trajectory tree $\cT_{e_q,\mathfrak{q}}(\ell,K)$ generated from $(\ell,K,e_q,\mathfrak{q})$. All final leaves (with a green background) are drawn after backtrack removal. To save space, if a splitting is such that multiple string plaquette assignments can be created with the same string but different plaquette assignments, we have drawn the string once and denoted the possible different plaquette assignments below the strings (cf.\ the right string plaquette assignments in the second row). Also, final leaves labeled with the same string plaquette assignment are only drawn once with multiple arrows pointing to them. For example, the string plaquette assignment in the bottom right corner corresponds to three leaves of $\cT_{e_q,\mathfrak{q}}(\ell,K)$.}
\end{center}
\vspace{-3ex}
\end{figure}

\begin{obs}\label{obs:remove_plaq}
    Recall from \cref{fig: traj tree} that, given a path $\mathfrak{p}$ connecting the edge $e$ of a loop $\ell$ to infinity, the first plaquette $p_1$ of $\mathfrak{p}$ may or may not lie in the support of the loop $\ell$. In either case, thanks to~\eqref{eq:noplins}, each loop appearing in a string labeling a (final) leaf of the tree $\cT_{e,\mathfrak{p}}(\ell, K)$ will not contain the plaquette $p_1$ in its support. This is why we sometimes say that the trajectory tree $\cT_{e,\mathfrak{p}}(\ell, K)$ “removes” the plaquette $p_1$ from the loop $\ell$.
\end{obs}

\begin{lem}\label{lem:finite-trees}
    Assume $(\ell,K)$ is balanced and $e$ is an edge contained in $\ell$ connected to infinity by a path of plaquettes $\mathfrak{p}$. Then, $\cT_{e,\mathfrak{p}}(\ell,K)$ is finite.
\end{lem}

\begin{proof} 
We start with some preliminary notation and facts. Let $(\Bar{\ell},\Bar{K})$ be any loop plaquette assignment. For an edge $\Bar{e}\in E$, let $n_{\Bar{e}}(\Bar{K})$ denote the number of copies of $\Bar{e}$ in $\Bar{K}$ (i.e., recalling \eqref{defn:ne}, $n_{\Bar{e}}(\Bar{K}) = \sum_{p\in \cP(\Bar{e})}\Bar{K}(p) = n_{\Bar{e}}(\Bar{\ell},\Bar{K})-n_{\Bar{e}}(\Bar{\ell})$). Also, set
\begin{align*}
    n_{\cI(\mathfrak{p})}(\Bar{K}) := \sum_{\Bar{e} \in \cI(\mathfrak{p})} n_{\Bar{e}}(\Bar{K}),
\end{align*}
where $\mathfrak{p}$ is as in the lemma statement, and similarly define $n_{\cI(\mathfrak{p})}(\Bar{\ell},\Bar{K})$ and $n_{\cI(\mathfrak{p})}(\Bar{\ell})$. 

We now assume that $(\Bar{\ell},\Bar{K})$ is such that $\Bar{\ell}$ contains at least one edge of $\cI(\mathfrak{p})$.
For the following facts, even if not explicitly stated, each operation is performed at the outermost edge of ${\Bar{\ell}}$ in $\cI(\mathfrak{p})$.

If $\{(\Bar{\ell}_1,\Bar{K}_1),(\Bar{\ell}_2,\Bar{K}_2)\}$ is a (positive or negative) splitting of ${(\Bar{\ell}, \Bar{K})}$, then we note that for $i=1,2$,
\begin{equation}\label{eq: split edge decrease}
    n_{\cI(\mathfrak{p})}(\Bar{\ell}_i)< n_{\cI(\mathfrak{p})}({\Bar{\ell}})\quad\text{and}\quad n_{\cI(\mathfrak{p})}(\Bar{K}_i)\leq n_{\cI(\mathfrak{p})}({\Bar{K}}).
\end{equation} 
Next, if $(\Bar{\ell}_1,\Bar{K}_1)$ is obtained from ${(\Bar{\ell},\Bar{K})}$ by a positive deformation, then \begin{equation}\label{eq: pos def edge decrease}
    n_{\cI(\mathfrak{p})}(\Bar{\ell}_1)\leq n_{\cI(\mathfrak{p})}({\Bar{\ell}})+2\quad\text{and}\quad n_{\cI(\mathfrak{p})}(\Bar{K}_1)< n_{\cI(\mathfrak{p})}({\Bar{K}}).
\end{equation}
Lastly, if $(\Bar{\ell}_1,\Bar{K}_1)$ is obtained from ${(\Bar{\ell},\Bar{K})}$ by a negative deformation, then \begin{equation}\label{eq: neg def edge decrease}
    n_{\cI(\mathfrak{p})}(\Bar{\ell}_1)\leq n_{\cI(\mathfrak{p})}({\Bar{\ell}})\quad\text{and}\quad n_{\cI(\mathfrak{p})}(\Bar{K}_1)< n_{\cI(\mathfrak{p})}({\Bar{K}}).
\end{equation}

Now, for the sake of contradiction, assume that $\cT_{e,\mathfrak{p}}(\ell,K)$ is infinite. Then, there must be an infinite sequence of loop plaquette assignments $\{(\ell^i,K^i)\}_{i=1}^{\infty}$ such that 
\begin{itemize}
    \item $(\ell^1,K^1) = (\ell,K)$;
    \item for all $i\geq 1$, $n_{\cI(\mathfrak{p})}(\ell^i)\neq 0$ and $(\ell^{i+1},K^{i+1})$ is one of the loop plaquette assignments in $\mle_{\mathbf{e}^i}(\ell^{i},K^{i})$, where $\mathbf{e}^i$ is one copy of the outermost edge of $\ell^i$ in $\cI(\mathfrak{p})$.
\end{itemize}  
But if such a sequence existed, then it must be that at least one of the following two cases holds:
\begin{enumerate}[(i)]
\item For infinitely many $i$, $(\ell^{i+1}, K^{i+1})$ is obtained from $(\ell^i, K^i)$ via a (positive or negative) deformation.
\item For infinitely many $i$, $(\ell^{i+1}, K^{i+1})$ is obtained from $(\ell^i, K^i)$ via a (positive or negative) splitting. 
\end{enumerate}
However, since $n_{\cI(\mathfrak{p})}(K^i)$ must be at least one to perform a (positive or negative) deformation and \eqref{eq: split edge decrease}, \eqref{eq: pos def edge decrease}, and \eqref{eq: neg def edge decrease} tell us that the sequence $(n_{\cI(\mathfrak{p})}(K^i))_i$ can't increase under splittings and strictly decreases under deformations, we deduce that only case $(ii)$ is possible. But this is also not possible. Indeed, if it were then there would be some $k\in \NN$ such that for all $j\geq k$, $\ell^{j+1}$ was obtained from $\ell^j$ via a splitting. Further, \eqref{eq: split edge decrease}, \eqref{eq: pos def edge decrease}, \eqref{eq: neg def edge decrease} and the fact that $n_{\cI(\mathfrak{p})}(\ell)$ is finite give us that $n_{\cI(\mathfrak{p})}(\ell^k)$ is finite. But, since to perform a (positive or negative) splitting, $n_{\cI(\mathfrak{p})}(\ell^i)$ must be at least two, and  \eqref{eq: split edge decrease} tells us that $n_{\cI(\mathfrak{p})}(\cdot)$ strictly decreases under splittings, this clearly gives that only a finite number of splitting terms can occur after $k$, contradicting the definition of $k$. Thus, we must have that $\cT_{e,\mathfrak{p}}(\ell,K)$ is finite as desired.
\end{proof}

\subsubsection{Removing a plaquette from a loop by pushing the edges}\label{sec: removing a plaquette-pushing}

The procedure detailed in the previous section tells us we can ``remove'' a plaquette from a loop by considering trajectory trees generated from $(\ell,K,e,\mathfrak{p})$, but we will need finer information on the leaves of these trees. In particular, the rest of this section will be devoted to proving the following result, which details how the height, plaquette assignments, and edges of the leaves compare to those of the root. To state the lemma, we need some new notation. If $(s,\multK) = \{(\ell_1,K_1),\dots,(\ell_m,K_m)\}$ is a string plaquette assignment and $e\in E$, let
\begin{equation}\label{eq:defn-KK}
    \multK(q):=\sum_{i\in [m]}K_i(q)\quad \text{and}\quad n_e(s):=\sum_{i\in [m]}n_e(\ell_i).
\end{equation}

\begin{prop}\label{prop: remove plaquette from loop} 
       Suppose that $(\ell, K)$ is a balanced loop plaquette assignment, and let $e$ be an edge contained in $\ell$. Assume that $e$ is connected to infinity via a path of plaquettes $\mathfrak{p} = \{p_i\}_{i=1}^{\infty}$, where the first plaquette is given by $p_1 = e\, e_1\, e_2 \, e_3$. Then, each string plaquette assignment $(s,\multK)= \{(\ell_1, K_1), \dots, (\ell_m , K_m)\}$ labeling a leaf of $\cT_{e,\mathfrak{p}}(\ell,K)$  satisfies the following properties:
        \begin{enumerate}
            \item\label{p1 remove plaquette} $\supp(s)\cap\mathfrak{p} = \emptyset$;
            \item\label{p2 remove plaquette} $\multK(q) = K(q)$ for all $q\in \cP\sm \mathfrak{p}$;
            \item\label{p4 remove plaquette} $h_{s}(q)= h_{\ell}(q)$ for all $q\in \cP\sm \{p_{1}^{\pm 1}\}$;
            \item\label{p5 remove plaquette} $n_{\Tilde{e}}(s) \leq n_{\Tilde{e}}(\ell)$ for all $\Tilde{e}\in E\sm \{e_1^{\pm 1}, e_2^{\pm 1}, e_3^{\pm 1}\}$; 
            \item\label{p6 remove plaquette} $n_{\Tilde{e}}(s) + n_{\Tilde{e}^{-1}}(s)\leq \Big(n_{\Tilde{e}}(\ell) + n_{\Tilde{e}^{-1}}(\ell)\Big) + \Big(n_e(\ell) + n_{e^{-1}}(\ell)\Big)$ for all $\Tilde{e}\in \{e_1, e_2, e_3\}$.
        \end{enumerate}
\end{prop}

\begin{remark}
Property~\ref{p6 remove plaquette} at first may seem surprising, since it does not involve the plaquette assignment $K$. Thus one might worry about the case where $K$ is very large, and the string trajectory includes a lot of deformations at $e$, which would add to the number of copies of $\tilde{e}$ for $\tilde{e} \in \{e_1, e_2, e_3\}$ for intermediate strings in the trajectory, i.e.\ for strings labeling interior nodes in the exploration tree. The point is that such deformations would also greatly increase the number of copies of $e$, and these copies would all have to be erased before stopping the exploration process. In the process of erasing all these copies of $e$, the additional copies of $e_i$, $i \in [3]$ will also have to be canceled out (note such canceling may only occur when we erase all backtracks for strings labeling final leaves).
\end{remark}

The first three properties in \cref{prop: remove plaquette from loop} will immediately follow from the next lemma. For an edge $e\in E$, let $\cP^{\pm}(e)$ denote the four plaquettes that contain $e$ in either orientation.

\begin{lem}\label{lemma: MLE preserves K and height}
     Suppose that $(\ell,K)$ is balanced, $\mathbf{e}$ is a specific copy of a lattice edge  $e$ contained in $\ell$, and $\mathfrak{p}$ connects $e$ to infinity. Then, all $(s,\multK)\in \mle_{\mathbf{e}}(\ell,K)$ satisfy the following properties :
     \begin{enumerate}
        \item\label{item:MLE1} $\multK(q) = K(q)$ for all $q\in \cP\sm \cP^{\pm}(e)$;
        \item\label{item:MLE2} $h_s(q) = h_{\ell}(q)$ for all $q\in \cP\sm \cP^{\pm}(e)$.
    \end{enumerate}
\end{lem}

\begin{proof}
    We start by noting that in general to compute the height (\cref{def: height}) of a plaquette $\Bar{p}$ with respect to a loop $\Bar{\ell}$ one can proceed as follows: First, find an infinite path of plaquettes $\Bar{\mathfrak{p}} = \{\Bar{p}_i\}_{i=1}^{\infty}$ starting at $\Bar{p}$. Then, there must be some $j\in \NN$ such that for all $i\geq j$, the plaquette $\Bar{p}_i$ is not contained in $B= [-N,N]^2$, where $N\in \NN$ is large enough that $\Bar{\ell}$ is contained in $B$. Thus, by the definition of the height of a plaquette (cf.\ \cref{def: height}) we know that $h_{\Bar{\ell}}(\Bar{p}_i)=0$ for all $i\geq j$. Finally, to compute $h_{\Bar{\ell}}(\Bar{p})$ simply go from $\Bar{p}_j$ to $\Bar{p}_1 =\Bar{p}$ following $\Bar{\mathfrak{p}}$ and keeping track of the crossings of $\Bar{\ell}$, as detailed in the definition of the height.

    Now, to prove the result, we look at the four outcomes $(s,\multK)$ of $\mle_{\mathbf{e}}(\ell,K)$ separately:

    \medskip
    
    \noindent\underline{Deformations:} Suppose that $s= \{\ell'\} = \{\ell\oplus_{\mathbf{e}}r\}$ for $r\in \cP^{\pm}(e)$ such that $K(r)\geq 1$. This gives us that  \begin{equation}\label{eq: K'h}
        \multK(q) =\begin{cases} K(q) & \text{if $q\in \cP\sm \{r\}$} \\ K(r)-1 & \text{if $q=r$.}
        \end{cases}
    \end{equation}
    With this, \cref{item:MLE1} is obvious. To prove \cref{item:MLE2}, note that $n_{e}(\ell') = n_{e}(\ell)$ for all $e\in E\sm \cE(r)$ (recall we do not remove backtracks). Now, for any plaquette $q\in \cP\sm \cP^{\pm}(e)$ there is clearly an infinite path of plaquettes $\mathfrak{p}$ that starts at $q$ such that $\mathfrak{p}\cap \cP^{\pm}(e)=\emptyset$. Thus, as $n_{e}(\ell') = n_{e}(\ell)$ for all $e\in \cI(\mathfrak{p})$, we get that $h_{\ell}(q)=h_{\ell'}(q)$ as desired. The negative deformations case can be proven similarly. 

    \medskip
    
    \noindent\underline{Splittings:} Suppose that $s= \{\ell_1,\ell_2\} \in \SS_-(\mathbf{e},\ell)$ and $\multK = \{K_1,K_2\}$ such that $K_1+K_2=K$. Then, by the definition of $\multK$, we get that $\multK(q) = K(q)$ for all $q\in \cP$, proving \cref{item:MLE1}. To prove \cref{item:MLE2}, note that
    the edges in $s$ are the same as the edges in $\ell$ with one fewer copy of $e$ and $e^{-1}$ (recall we do not remove backtracks). Then, the reasoning used in the deformation case gives the invariance of the height outside of $\cP^{\pm}(e)$. The positive splitting case can be proven similarly.
    \end{proof}

 Properties~\ref{p1 remove plaquette}, \ref{p2 remove plaquette} and \ref{p4 remove plaquette} of \cref{prop: remove plaquette from loop} are immediate consequences of \cref{lemma: MLE preserves K and height}. 

\begin{proof}[Proof of the first three properties of \cref{prop: remove plaquette from loop}]
    First, note that removing backtracks cannot add edges, change the plaquette assignment, or change the height of a plaquette. The invariance of the height under removing backtracks simply follows from the fact that the height counts the difference between the number of copies of each orientation of an edge, and removing a backtrack removes one copy of each orientation. Thus, if the first three properties of \cref{prop: remove plaquette from loop} hold for a leaf of $\cT_{e,\mathfrak{p}}(\ell, K)$ before removing backtracks, they also hold after removing backtracks. Throughout this proof, even if not explicitly stated, we will work with the leaves of $\cT_{e,\mathfrak{p}}(\ell, K)$ before removing backtracks.

    By the definition of $\cT_{e,\mathfrak{p}}(\ell, K)$, we know that $s$ does not contain any of the edges in $\cI(\mathfrak{p})$, thus establishing Property~\ref{p1 remove plaquette}.
    Property~\ref{p2 remove plaquette} simply follows from \cref{lemma: MLE preserves K and height} and the construction of $\cT_{e,\mathfrak{p}}(\ell,K)$.
    Lastly, to verify Property~\ref{p4 remove plaquette}, observe that \cref{lemma: MLE preserves K and height}, together with the construction of $\cT_{e,\mathfrak{p}}(\ell, K)$, implies that $h_{s}(q) = h_{\ell}(q)$ for all $q \in \cP \setminus \mathfrak{p}$.
    Now, since $\supp(s)\cap\mathfrak{p} = \emptyset$, it follows that $h_s(p) = 0 = h_\ell(p)$ for all $p \in \mathfrak{p}\sm \{p_1^{\pm 1}\}$, which establishes Property~\ref{p4 remove plaquette}.
\end{proof}

The rest of this section is devoted to proving Properties~\ref{p5 remove plaquette}~and~\ref{p6 remove plaquette} of \cref{prop: remove plaquette from loop}.
Our goal is to show that, relative to $\ell$, the string $s$ introduces no new edges except possibly among $\{e_1^{\pm1},e_2^{\pm1},e_3^{\pm1}\}$ (this is Property~\ref{p5 remove plaquette}), and even there at most $n_{e}(\ell) + n_{e^{-1}}(\ell)$ edges are added (this is Property~\ref{p6 remove plaquette}).

The main tool we will use is the following lemma, which details how the structure of a loop is modified by the ML-operator. 
We first introduce some terminology. 
Let $\ell$ be a fixed loop.  Choose a plaquette $p$ and an infinite path of plaquettes $\mathfrak{p} = \{p_i\}_{i=1}^\infty$
starting at $p$, such that $\ell$ shares at least one edge with $\mathcal{I}(\mathfrak{p})$.  Then $\ell$ admits a (non-unique) decomposition of the form
\begin{equation}\label{eq:decomp-loop}
  \ell = \pi_1\,\pi_1^{\mathfrak{p}}\,\pi_2\,\pi_2^{\mathfrak{p}}\,\cdots\,\pi_n\,\pi_n^{\mathfrak{p}},
\end{equation}
where $n\geq 1$, each $\pi_i$ is a (possibly empty) sub-path of $\ell$ consisting entirely of edges in $E\setminus\mathcal{I}(\mathfrak{p})$, and each $\pi_i^{\mathfrak{p}}$ is a (possibly empty\footnote{The reason why we allow both the $\pi_i$ and the $\pi_i^{\mathfrak{p}}$ to be empty will be clearer later; see for instance~\cref{lem: MLE for edge loops}.}) sub-path of $\ell$ contained in $\mathcal{E}(\mathfrak{p})$.  We call any such decomposition a $\mathfrak{p}$-\textbf{decomposition} of $\ell$ (cf.\ \cref{fig: loop decomp}).

\begin{figure}[ht!]
\begin{center}
	\includegraphics[width=.49\textwidth]{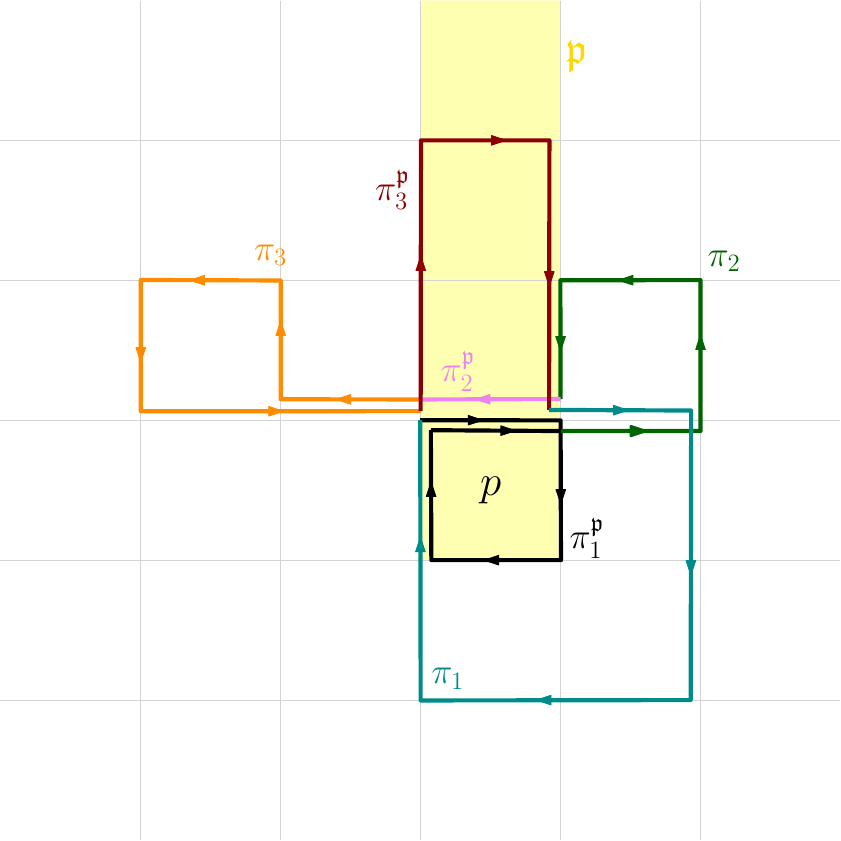}  
	\caption{\label{fig: loop decomp} Let $p$ be the labeled plaquette above and $\mathfrak{p}$ the vertical path of plaquettes above $p$ starting from $p$ (which is highlighted in yellow). A $\mathfrak{p}$-decomposition for the loop is shown. Note that this is not the only possible $\mathfrak{p}$-decomposition of the loop.}
\end{center}
\vspace{-3ex}
\end{figure}

\begin{lem}\label{lem: MLE for edge loops}
   Let $(\ell,K)$ be a balanced loop plaquette assignment and $p$ be a plaquette. Let also $\mathfrak{p} = \{p_i\}_{i=1}^{\infty}$ be an infinite path of plaquettes starting at $p$ such that $\ell$ shares at least one edge with $\mathcal{I}(\mathfrak{p})$. Let $e$ denote the outermost edge of $\ell$ in  $\cI(\mathfrak{p})$ and $\mathbf{e}$ denote a specific copy of $e$ in $\ell$.
   Consider a $\mathfrak{p}$-decomposition of $\ell$ such that $\mathbf{e}$ is in $\pi_s^{\mathfrak{p}}$ for some $s\in [n]$, that is
   \begin{equation*}
        \ell = \pi_1\,\pi_1^{\mathfrak{p}}\,\cdots \pi_s\,\pi_s^{\mathfrak{p}}\,\pi_{s+1}\,\cdots\,\pi_n\,\pi_n^{\mathfrak{p}},\qquad\text{with } \pi^{\mathfrak{p}}_s = \pi^{\mathfrak{p}}_{s,1} \, \mathbf{e} \, \pi^{\mathfrak{p}}_{s,2}.
   \end{equation*}
   Then, the loops obtained from $\ell$ by performing a deformation or a splitting\footnote{Recall that we do not remove backtracks formed by these operations.} at $\mathbf{e}$ satisfy the following properties:

   \begin{itemize}
       \item  If $\ell'$ is obtained from $\ell$ by a (positive or negative) deformation at $\mathbf{e}$, then $\ell'$ has a $\mathfrak{p}$-decomposition of the form
       \begin{equation}\label{eq: ell'}
        \ell'=\pi_1\,\pi^{\mathfrak{p}}_1\,\hdots\, \pi_{s-1}\,\pi^{\mathfrak{p}}_{s-1}\,\pi_s\,\Tilde{\pi}^{\mathfrak{p}}_s\,\pi_{s+1}\,\pi^{\mathfrak{p}}_{s+1}\,\hdots\,\pi_{n}\,\pi^{\mathfrak{p}}_n,
   \end{equation}
   where $\Tilde{\pi}_s^{\mathfrak{p}}$ is a (possibly empty) sequence of edges in $\cE(\mathfrak{p})$.
   
    \item If a (positive or negative) splitting is performed at $\mathbf{e}$ with an edge $\mathbf{e}'$ in $\pi_s^{\mathfrak{p}}$, then the resulting string $s=\{\ell_1, \ell^{\mathfrak{p}}_2\}$ is such that: 
    \begin{itemize}
        \item $\ell_1$ has a $\mathfrak{p}$-decomposition as in \eqref{eq: ell'};
        \item  $\ell^{\mathfrak{p}}_2$ is a loop only consisting of edges in $\cE(\mathfrak{p})$.
    \end{itemize}

    \item If a (positive or negative) splitting is performed at $\mathbf{e}$ with an edge $\mathbf{e}'$ in $\pi_t^{\mathfrak{p}}$ for $t\neq s$ (say $t>s$), then the resulting string $s=\{\ell_1,\ell_2\}$is such that: 
    \begin{itemize}
       \item $\ell_1$ has a $\mathfrak{p}$-decomposition of the form $\ell_1 =  \pi_1\,\pi_{1}^{\mathfrak{p}}\,\hdots\,\pi_s \,\Tilde{\pi}_s^{\mathfrak{p}}\,\pi_{t+1}\,\pi^{\mathfrak{p}}_{t+1}\,\hdots \pi_{n}\,\pi^{\mathfrak{p}}_n$;
       \item $\ell_2$ has a $\mathfrak{p}$-decomposition of the form $ \ell_2 = \pi_{s+1}\,\pi_{s+1}^{\mathfrak{p}}\hdots \pi_t \,\Tilde{\pi}_t^{\mathfrak{p}}$;
    \end{itemize}
    where $\Tilde{\pi}_s^{\mathfrak{p}}$ and $\Tilde{\pi}_t^{\mathfrak{p}}$ are two (possibly empty) sequences of edges in $\cE(\mathfrak{p})$.
   \end{itemize}
   In particular, every string $s\in\mle_{\mathbf{e}}(\ell,K)$ is composed of edges that either already lie in $\ell$ or belong to $\mathcal{E}(\mathfrak{p})$.
\end{lem}

\begin{proof}
    We proceed by considering each operation independently.
    
    If a positive deformation is performed at $\mathbf{e}$ with a plaquette $q=e_1e_2e_3e$, then
    \begin{align*}
        \ell' = \ell\oplus_{\mathbf{e}}q = \pi_1\,\pi^{\mathfrak{p}}_1\,\hdots\,\pi_{s-1}\,\pi^{\mathfrak{p}}_{s-1}
        \pi_s\,\pi^{\mathfrak{p}}_{s,1}\,\mathbf{e} \,e_1e_2e_3e \,\pi^{\mathfrak{p}}_{s,2}\,
        \pi_{s+1}\,\pi^{\mathfrak{p}}_{s+1}\,\hdots \,\pi_{n}\,\pi^{\mathfrak{p}}_n.
    \end{align*}
   Thus, writing $\Tilde{\pi}_s^{\mathfrak{p}} = \pi^{\mathfrak{p}}_{s,1}\,\mathbf{e}\,e_1e_2e_3e\,\pi^{\mathfrak{p}}_{s,2}$, we get that $\ell'$ has the desired form in \eqref{eq: ell'}. The negative deformation follows from the exact same argument.

    Next assume that a splitting is performed. For brevity, we will only give the details for the positive splitting case, as the negative splitting case is similar. If $\mathbf{e}'\in \pi_s^{\mathfrak{p}}$, we can assume without loss of generality that $\mathbf{e}'$ is after $\mathbf{e}$. Then, writing $\pi_{s,2}^{\mathfrak{p}} = \pi_{s,2,1}^{\mathfrak{p}}\,\mathbf{e}'\,\pi^{\mathfrak{p}}_{s,2,2}$, we have that 
    \begin{align*}
        \ell = \pi_1\,\pi^{\mathfrak{p}}_1\hdots \pi_{s-1}\,\pi^{\mathfrak{p}}_{s-1}\,\pi_s\,\pi_{s,1}^{\mathfrak{p}}\,\mathbf{e}\,\pi_{s,2,1}^{\mathfrak{p}}\,\mathbf{e}'\,\pi^{\mathfrak{p}}_{s,2,2}\,\pi_{s+1}\,\pi^{\mathfrak{p}}_{s+1}\hdots \pi_{n}\,\pi^{\mathfrak{p}}_n.
    \end{align*}
    and so, $s$ has the form $\{\ell_1,\ell^{\mathfrak{p}}_2\}$, where
    \begin{align*}
        \ell_1 =  \pi_1\,\pi^{\mathfrak{p}}_1\hdots \pi_{s-1}\,\pi^{\mathfrak{p}}_{s-1}\,\pi_s\,\Tilde{\pi}_s^{\mathfrak{p}}\,\pi_{s+1}\,\pi^{\mathfrak{p}}_{s+1}\hdots \pi_{n}\,\pi^{\mathfrak{p}}_n \text{ with $\Tilde{\pi}_s^{\mathfrak{p}}= \pi_{s,1}^{\mathfrak{p}}\,\mathbf{e}\,\pi^{\mathfrak{p}}_{s,2,2}$,}
        \qquad \text{and}\qquad \ell^{\mathfrak{p}}_2 = \pi_{s,2,1}^{\mathfrak{p}}\,\mathbf{e}',
    \end{align*}
    as we wanted.
    
    Finally, if $\mathbf{e}'\in \pi_t^{\mathfrak{p}}$ for $s<t\leq n$, then, writing $\pi_t^{\mathfrak{p}} = \pi_{t,1}^{\mathfrak{p}}\mathbf{e}'\pi_{t,2}^{\mathfrak{p}}$, we have that
    \begin{align*}
        \ell = \pi_1\,\pi_{1}^{\mathfrak{p}}\,\hdots\,\pi_s\,\pi_{s,1}^{\mathfrak{p}}\,\mathbf{e}\,\pi_{s,2}^{\mathfrak{p}}\,\pi_{s+1}\,\pi^{\mathfrak{p}}_{s+1}\, \,\hdots\pi_{t}\,\pi_{t,1}^{\mathfrak{p}}\,\mathbf{e}'\,\pi^{\mathfrak{p}}_{t,2}\,\pi_{t+1}\,\pi^{\mathfrak{p}}_{t+1}\,\hdots \pi_{n}\,\pi^{\mathfrak{p}}_n,
    \end{align*}
    and so, $s$ has the form  $\{\ell_1,\ell_2\}$, where
    \begin{align*}
        &\ell_1 =  \pi_1\,\pi_{1}^{\mathfrak{p}}\,\hdots\,\pi_s \,\Tilde{\pi}_s^{\mathfrak{p}}\,\pi_{t+1}\,\pi^{\mathfrak{p}}_{t+1}\,\hdots \pi_{n}\,\pi^{\mathfrak{p}}_n,\qquad
        \text{with $\Tilde{\pi}_s^{\mathfrak{p}} = \pi_{s,1}^{\mathfrak{p}}\,\mathbf{e} \,\pi_{t,2}^{\mathfrak{p}}$,}\\
        &\ell_2 = \pi_{s+1}\,\pi_{s+1}^{\mathfrak{p}}\hdots \pi_t \,\Tilde{\pi}_t^{\mathfrak{p}},
        \qquad\qquad\qquad\qquad\,\,\,\,\text{with $\Tilde{\pi}_t^{\mathfrak{p}} =\pi_{t,1}^{\mathfrak{p}}\,\mathbf{e}'\,\pi_{s,2}^{\mathfrak{p}}$},
    \end{align*}
      and so we have the desired decomposition. 
    \end{proof}

The next observation is the last preliminary ingredient we need to complete the proof of Properties~\ref{p5 remove plaquette}~and~\ref{p6 remove plaquette} of \cref{prop: remove plaquette from loop}.

\begin{obs}\label{obs: tree equals null}
    Let $\ell$ be a loop, possibly with backtracks, and suppose the subgraph of $\mathbb{Z}^2$ induced by the set of edges in $\ell$ is a tree.  Then, after removing all backtracks from $\ell$, one obtains the null-loop.
\end{obs}

\begin{proof}[Proof of the last two properties of~\cref{prop: remove plaquette from loop}]
    Recall that our goal is to show that, relative to $\ell$, the string $s$ introduces no new edges except possibly among $\{e_1^{\pm1},e_2^{\pm1},e_3^{\pm1}\}$, and even there at most $n_{e}(\ell) + n_{e^{-1}}(\ell)$ edges are added.

    First, observe that for any edge $e\in E\setminus\mathcal{E}(\mathfrak{p})$, Property~\ref{p5 remove plaquette} holds immediately by the final sentence of Lemma~\ref{lem: MLE for edge loops}.  Indeed, in constructing $\mathcal{T}_{e,\mathfrak{p}}(\ell,K)$ via repeated applications of the ML-operator, no new edges outside of $\mathcal{E}(\mathfrak{p})$ can ever appear.
    
    Next, observe that we can write
    \begin{equation}\label{eq:p-dec-ell}
        \ell = \pi_1\,\mathbf{e}_1\,\pi_2\,\mathbf{e}_2\,\cdots\,\pi_n\,\mathbf{e}_n,
    \end{equation}
    where $n = n_e(\ell) + n_{e^{-1}}(\ell)$, exactly $n_e(\ell)$ of the symbols $\mathbf{e}_i$ are equal to $e$, and the remaining $n_{e^{-1}}(\ell)$ are equal to $e^{-1}$. Since $\mathfrak{p} = \{p_i\}_{i=1}^{\infty}$ is a path of plaquettes connecting $e$ to infinity, and so, by definition, $\cI(\mathfrak{p})$ does not contain any edge of $\ell$ other than $e^{\pm 1}$, we have that \eqref{eq:p-dec-ell} is  a $\mathfrak{p}$-decomposition of $\ell$, where the $\pi_i^{\mathfrak{p}}$ in \eqref{eq:decomp-loop} are simply equal to $\mathbf{e}_i$. So, applying \cref{lem: MLE for edge loops}, we get that the strings corresponding to the children of the root of $\mathcal{T}_{e,\mathfrak{p}}(\ell,K)$ all have one of the three forms appearing in the statement of \cref{lem: MLE for edge loops}. 
    
    Noting that \cref{lem: MLE for edge loops} can be iteratively applied again to all the strings associated with the children of the root of $\mathcal{T}_{e,\mathfrak{p}}(\ell,K)$, we deduce that the string labeling any vertex of $\cT_{e,\mathfrak{p}}(\ell,K)$ must have the form
    \begin{align*}
        \Big\{\ell^{\mathfrak{p}}_1,\dots,\ell^{\mathfrak{p}}_r,\ell_1, \dots, \ell_m\Big\},
    \end{align*}
    where each $\ell_i^{\mathfrak{p}}$ is a non-empty loop with all the edges in $\cE(\mathfrak{p})$, and each loop $\ell_i$ has a $\mathfrak{p}$-decomposition of the form
    \begin{align*}
        \ell_i=\pi_{i,1}\,\pi^{\mathfrak{p}}_{i,1}\hdots \pi_{i,n_i}\,\pi^{\mathfrak{p}}_{i,n_i},
    \end{align*}
    such that 
    \begin{enumerate}
        \item\label{item:samepath1} for each index $(i,j)$ there exists a distinct index $u=u(i,j)$ such that $\pi_{i,j}=\pi_u$;
        \item\label{item:samepath2} $\sum_{i=1}^m\sum_{j=1}^{n_i}1 = n$;
        \item\label{item:samepath3} each $\pi_{i,j}^{\mathfrak{p}}$ is a (possibly empty) sequence of edges only consisting of edges in $\cE(\mathfrak{p})$.
    \end{enumerate} 
    In particular, when this string corresponds to a leaf of $\cT_{e,\mathfrak{p}}(\ell,K)$, we additionally know by the definition of $\cT_{e,\mathfrak{p}}(\ell,K)$ that it contains no edges of $\cI(\mathfrak{p})$, and thus each loop $\ell_i^{\mathfrak{p}}$ only contains edges in $\cE(\mathfrak{p})\sm \cI(\mathfrak{p})$. Therefore, by \cref{obs: tree equals null}, we get that  each loop $\ell_i^{\mathfrak{p}}$ is equivalent (after removing all the backtracks) to the null-loop. Thus, the string simply has the form
    \begin{align*}
        \Big\{\ell_1, \dots, \ell_m\Big\}.
    \end{align*}
    Using again the fact that each string corresponding to a leaf of $\cT_{e,\mathfrak{p}}(\ell,K)$ does not contain edges of $\cI(\mathfrak{p})$, we get from Condition~\ref{item:samepath3} above that each $\pi_{i,j}^{\mathfrak{p}}$ must only contain edges from the tree (or, even more precisely, the path) of edges  $\cE(\mathfrak{p})\sm \cI(\mathfrak{p})$. But as each $\pi_{i,j}$ starts and ends at one of the vertices of $e$ (this is a consequence of Condition~\ref{item:samepath1} above and \eqref{eq:p-dec-ell}), we must have that either $\pi_{i,j}^{\mathfrak{p}}$ starts at one vertex of $e$ and ends at the other, or starts and ends at the same vertex of $e$. Using \cref{obs: tree equals null} once again, this gives us that after removing backtracks either $\pi_{i,j}^{\mathfrak{p}}$ is the null-loop (when it starts and ends at the same vertex), or it has either the form $e_1e_2e_3$ or the form $e_3^{-1}e_2^{-1}e_1^{-1}$ (when it starts at one vertex of $e$ and ends at the other). Using this and Condition~\ref{item:samepath2} above, it follows that, relative to $\ell$, the string $s$ introduces no new edges except possibly among 
    $\{e_1^{\pm1},e_2^{\pm1},e_3^{\pm1}\}$, 
    and even there at most $n$ edges are added.
    Thus we have shown Properties~\ref{p5 remove plaquette}~and~\ref{p6 remove plaquette}, finishing the proof of all the five properties.
\end{proof}

\begin{obs}\label{obs: move edges}
    Note that Property~\ref{p5 remove plaquette} of \cref{prop: remove plaquette from loop} can be interpreted as stating that when we ``remove'' a plaquette $p = e\, e_1\, e_2\, e_3$ from a loop $\ell$ via the trajectory tree $\cT_{e,\mathfrak{p}}(\ell, K)$ (recall \cref{obs:remove_plaq}), then in each loop appearing in a string labeling a leaf of $\cT_{e,\mathfrak{p}}(\ell, K)$, all copies of the edge $e$ that were present in $\ell$ have either been removed or ``pushed'' onto the other three edges of $p$ (cf.\ the leaves with a green background in \cref{fig: traj tree}).

    We will often adopt the perspective that $\cT_{e,\mathfrak{p}}(\ell, K)$ ``removes'' the plaquette $p$ from $\ell$ by ``pushing'' the copies of the edge $e$.
\end{obs}

\subsection{Loops of small size}\label{sec: Loops supported on a plaquette}

Using the ability to ``remove'' plaquettes and track the resulting changes in key quantities (Proposition~\ref{prop: remove plaquette from loop}), in this section we show that each of the three conditions appearing in the definition of the canonical collection of plaquette assignments (Definition~\ref{defn:collection-plaquette-ass}) is necessary in order to have a non-zero surface sum for certain loops of small size, i.e.\ loops winding (one or multiple times) around one or two adjacent plaquettes. These facts will form the ``base cases'' for proving \cref{thm: erasable loops have one plaquette assignment}.

In particular, in \cref{sec: loops of size one}, we prove \cref{thm: erasable loops have one plaquette assignment} in the special case of loops winding around one plaquette. Since Condition~\ref{Kl p2} in \cref{defn:collection-plaquette-ass} is trivially satisfied for loops winding around one plaquette, we will show in \cref{sec: loops of size two} that it is necessary for loops winding around two adjecent plaquettes. These two results will serve as base cases, enabling us to ``shrink'' more general loops down to them.

\subsubsection{The case of loops winding around one plaquette}\label{sec: loops of size one}

In this section, we prove \cref{thm: erasable loops have one plaquette assignment} for loops winding around one plaquette. This provides a ``base case'' for showing the necessity of Conditions~\ref{Kl p1}~and~\ref{Kl p3} to have a non-zero surface sum. For ease of reference, we restate \cref{thm: erasable loops have one plaquette assignment} in this special case.

\begin{prop}\label{lem: size one main thm}
    Assume $\ell$ is a non-backtrack loop winding around one plaquette. If $K\notin \cK_{\ell}$, then $\phi^K(\ell) = 0$.
\end{prop}

Any non-backtrack loop $\ell$ winding around one plaquette clearly has the form $\ell=p^n$ for some $p\in \cP$ and $n\in \NN$.  Also, since the distance $d_{p^n}$ is equal to the absolute value of the height $|h_{p^n}|$, we have that $\cK_{p^n}=\{K_{p^n}\}$ thanks to \cref{rem:card-set-cK}. With these observations, we can claim that \cref{lem: size one main thm} immediately follows from the following two propositions.

\begin{prop}\label{prop: phi^K(p) = 0 for K supported outside of p}
    Fix $n\in \NN$ and let $p$ be a positively or negatively oriented plaquette. Suppose $(p^n,K)$ is balanced. If $K$ is such that $K(q)\geq 1$ for some $q\in \cP\sm\{p,p^{-1}\}$, then $\phi^K(p^n)= 0$.
\end{prop} 
\begin{prop}\label{prop: phi^K(p) = 0 for K only supported on p}
    Fix $n\in \NN$ and let $p$ be a positively or negatively oriented plaquette. Suppose $(p^n,K)$ is balanced. If $K$ is such that $K(q) = 0$ for all $q\in \cP\sm\{p,p^{-1}\}$ and $K\neq K_{p^n}$, then $\phi^K(p^n)=0$.
\end{prop}

Proposition~\ref{prop: phi^K(p) = 0 for K supported outside of p} follows from what we have shown in \cref{prop: remove plaquette from loop}.

\begin{proof}[Proof of Proposition~\ref{prop: phi^K(p) = 0 for K supported outside of p}]
    
    Fix a plaquette $q\in \cP\sm \{p,p^{-1}\}$ such that $K(q)\geq 1$. Note that we can always find an edge $e$ of $p$ such that $q$ is not contained in the infinite path of plaquettes $\mathfrak{p}$ starting  at $p$ and going straight to infinity in the direction of the edge $e$.  Note also that each leaf $l$ of the trajectory tree $\cT_{e,\mathfrak{p}}(\ell,K)$ has the following form once backtracks are removed:
    \begin{align*}
        (s,\multK)=\left\{(\emptyset,K_1),\dots,(\emptyset,K_m)\right\}.
    \end{align*}
    Indeed, each loop in $(s,\cK)$ is only supported on the edges of the tree $\cE(p)\sm \{e^{\pm1}\}$ by Property~\ref{p5 remove plaquette} of \cref{prop: remove plaquette from loop} and the definition of $\cT_{e,\mathfrak{p}}(\ell,K)$. Thus, \cref{obs: tree equals null} gives us that each loop is the null-loop. Further, there is some $j\in [m]$ such that $K_j(q)\neq 0$ by Property~\ref{p2 remove plaquette} of \cref{prop: remove plaquette from loop} and the fact that $q\notin \mathfrak{p}$. Thus, \cref{obs: tree WLE relation} and \cref{lemma: backtrack cancellations} give us that
    \begin{align*}
        \phi^K(p^n)= \sum_{(s,\multK) = \{(\emptyset,K_1),\dots,(\emptyset,K_m)\}\in \cL(\cT)}\upbeta_{s,\multK}\prod_{i\in [m]}\phi^{K_i}(\emptyset)=0,
    \end{align*}
    where $\upbeta_{s,\multK}$ is as in \cref{obs: tree WLE relation} and the last equality follows from the fact that $\phi^{K'}(\emptyset) = 0$ for $K'\neq 0$ (recall Property~\ref{def: phi k empty = 0} of $\phi^K(\ell)$ on page~\pageref{def: phi k empty = 0}).
\end{proof}

Now we prove Proposition~\ref{prop: phi^K(p) = 0 for K only supported on p}. To do this, we will again apply the master loop equation to write $\phi^K(p^n)$ as a sum of terms that are zero. Unfortunately, we will no longer be in the setting where this can be done by constructing a trajectory tree  where each leaf has a label of the form $\{(\emptyset,K_1),\dots,(\emptyset,K_m)\}$ with at least one $K_j\neq 0$. Fortunately, there are some other cancellations, described in Lemma~\ref{lemma: wound plaquettes zero weight} below, which will cover the cases when all $K_j$ would be zero. 

\begin{lem}\label{lemma: wound plaquettes zero weight}
    Let $p$ be a positively or negatively oriented plaquette. For $n\geq 0$ define
    \[
    K^n(p^{-1}) = n\quad\text{and}\quad K^n(q) = 0, \quad\text{for all }q\in\cP\sm\{p^{-1}\}.
    \]
    Then for all $n\geq 2$, we have that $\phi^{K^n}(p^n) = 0$.
\end{lem}

\begin{proof}
    To prove the lemma we induct on $n\geq 2$. 
    
    First, we explicitly compute $\phi^{K^1}(p^1)$. To do this, let $e$ denote any edge of $p$. Applying the master loop equation at $e$, we see that a negative deformation with $p^{-1}$ is performed because there are no possible splittings as $p$ does not contain multiple copies of $e$ and doesn't contain any copies of $e^{-1}$, and there are no positive deformations as $K^1$ does not contain any plaquettes with the edge $e$. Thus, the master loop equation gives that\begin{equation}\label{eq: one pocket coefficent}
        \phi^{K^1}(p^1) = \upbeta\phi^{K^0}(\emptyset) =\upbeta,
    \end{equation}
    where for the first equality we used \cref{lemma: backtrack cancellations} and for the last equality we used that $\phi^{K^0}(\emptyset) =1$ (recall Property~\ref{def: phi k empty = 0} of $\phi^K(\ell)$).
    
    With this we can similarly prove the induction basis $n=2$. Applying the master loop equation at one of the copies of $e$ in $p^2$, we see that either a negative deformation with $p^{-1}$ or a positive splitting is performed. This is because there are no possible negative splittings as $p^2$ does not contain $e^{-1}$ and there are no positive deformations as $K^2$ does not contain any plaquettes with the edge $e$. Thus the master loop equation gives that\begin{align*}
        \phi^{K^2}(p^2) = \upbeta\phi^{K^1}(p^1) - \phi^{K^1}(p^1)\phi^{K^1}(p^1)= \upbeta^2-\upbeta^2=0,
    \end{align*}
    where to get the first equality we used \cref{lemma: backtrack cancellations} and to get the second equality we used \eqref{eq: one pocket coefficent}.
    
    Now, for the inductive step, fix $n\geq 3$ and suppose that the claim holds for all $m$ such that $2 \leq m< n$. Then, we apply the master loop equation to some edge corresponding to $e$ in $p^n$. Notice that, by the structure of $K^n$, with an argument similar to the one above, we have that either a positive splitting or a negative deformation is performed at the edge $e$. That is, we get 
    \begin{align*}
        \phi^{K^n}(p^n) = \upbeta\phi^{K^{n-1}}(p^{n-1})- \sum_{\{\ell_1,\ell_2\}\in \SS_+(\mathbf{e},p^n)}
        \sum_{K_1+K_2=K^n}
        \phi^{K_1}(\ell_1)\phi^{K_2}(\ell_2),
    \end{align*}
    where we used again \cref{lemma: backtrack cancellations}.
    Now, notice that by the induction hypothesis, we have that
    \begin{align*}
        \phi^{K^{n-1}}(p^{n-1})=0.
    \end{align*}
    Next, notice that any $\{\ell_1,\ell_2\} \in \SS_+(e,p^n)$ has the form $\{\ell_1,\ell_2\} = \{p^i,p^j\}$ for some $i,j\in \{1,\dots,n-1\}$ such that $i+j=n$. If $\{\ell_1,\ell_2\} = \{p^i,p^j\}$, then $\{K_1,K_2\} = \{K^i,K^j\}$ is the only plaquette assignment such that $(\ell_1,K_1)$ and $(\ell_2,K_2)$ are balanced and $K_1+K_2 = K^n$. By the inductive hypothesis, we get that
    \begin{align*}
        \phi^{K^i}(p^i)\phi^{K^j}(p^j) = 0,
    \end{align*}
    where we used that since $n\geq 3$, we always have that at least one among $i$ and $j$ is greater than or equal to 2. Thus we can conclude that $\phi^{K^n}(p^n) = 0$, finishing the proof of the lemma.
\end{proof}

We can now complete the proof of Proposition~\ref{prop: phi^K(p) = 0 for K only supported on p}.

\begin{proof}[Proof of Proposition~\ref{prop: phi^K(p) = 0 for K only supported on p}]
    
    To show Proposition~\ref{prop: phi^K(p) = 0 for K only supported on p} we prove the following equivalent claim.
    
    \medskip

    \noindent\textbf{Claim.} Let $n, j \geq 1$. Let $K$ be the plaquette assignment such that $K(q) = 0$ for $q \neq p^{\pm}$, $K(p) = j$, and $K(p^{-1}) = n + j$. We have that $\phi^K(p^n) = 0$.
    
    \medskip
    
    To prove this claim, we perform a double induction, first on $n$, then on $j$. To begin, fix $j = 1$. We will show that the claim is true for all $n \geq 1$. We begin with the base case $n = 1$. By the master loop equation, we have that
    \[ \phi^K(p) = \upbeta \phi^{K \sm p^{-1}}(\emptyset) - \upbeta \phi^{K \sm p}(p^2). \]
    Since $K(p) \neq 0$, we have that  $K \setminus p^{-1} \neq 0$, and thus the first term on the right-hand side above is zero. To conclude that the second term is zero, note that, by definition of $K$, the plaquette assignment $K \sm p$ is precisely $K^2$ in Lemma \ref{lemma: wound plaquettes zero weight}. Thus by that lemma, we have that the second term on the right-hand side above is also zero. This proves the base case.
    
    Next, suppose that the claim is true for $j = 1$ and all $1 \leq m \leq n$. Let $K$ be the plaquette assignment corresponding to the case $n + 1$. Again by the master loop equation, we have that
    \[ \phi^K(p^{n+1}) = \upbeta \phi^{K \sm p^{-1}}(p^n) - \upbeta \phi^{K \sm p} (p^{n+2}) - \sum_{\substack{i_1, i_2\geq 1\\
    i_1 + i_2=n+1}}\sum_{K_1+K_2=K}\phi^{K_1}(p^{i_1})\phi^{K_2}(p^{i_2}).\]
    By the inductive assumption, the first term on the right-hand side is zero. By Lemma \ref{lemma: wound plaquettes zero weight}, the second term is zero. For the third and final term, note that for each term in the sum, either $(p^{i_r}, K_r)$ is not balanced for some $r=1,2$, in which case $\phi^{K_1}(p^{i_1}) = 0$, or both $(p^{i_1}, K_1), (p^{i_2}, K_2)$ are balanced, and additionally $K_r(p) = 1$ for some $r = 1, 2$, in which case we may apply the inductive assumption to conclude that $\phi^{i_r}(K_r) = 0$.
    
    We have thus proven that the claim is true for $j = 1$ and all $n \geq 1$. We now use this as the base case to induct on $j$. For the inductive step, assume that the claim is true for all $1 \leq j \leq j_0$ and $n \geq 1$. We proceed to show that the claim is true for $j = j_0 + 1$ and all $n \geq 1$. To do so, we induct on $n$. First, consider the base case $j = j_0 + 1, n = 1$. Let $K$ be the plaquette assignment corresponding to this case. By the master loop equation, we have that
    \[ \phi^K(p) = \upbeta \phi^{K \sm p^{-1}}(\emptyset) - \upbeta \phi^{K \sm p}(p^2).\]
    Since $K(p) = j_0 + 1 \neq 0$, we have that the first term on the right-hand side above is zero. The second term is zero by our inductive assumption, since $(K \sm p)(p) = j_0$. Thus we have proven the claim when $j = j_0 + 1$, $n = 1$.
    
    Next, suppose that the claim is true for  $j = j_0 +1$ and all $1 \leq m \leq n$. Additionally, recall that we have made the ``outer'' inductive assumption that the claim is true for all $1 \leq j \leq j_0$ and $n \geq 1$. Let $K$ be the plaquette assignment corresponding to the case $j_0 + 1, n+1$. Then similar to before, we apply the master loop equation to $\phi^K(p^{n+1})$, and conclude that all terms appearing on the right-hand side are zero by our inductive assumption(s). We omit the details.
\end{proof}

\subsubsection{The case of loops winding around two adjacent plaquettes}\label{sec: loops of size two}

To see the necessity of Condition~\ref{Kl p2} for the plaquette assignments in the canonical collection to have  a non-zero surface sum, we need to consider loops winding around two adjacent plaquettes.

\begin{lem}\label{lem: size two loops}
    Suppose that $\ell$ is a non-backtrack loop winding around the rectangle formed by the union of the two adjacent plaquettes $p,q\in \cP$.
    Assume that $(\ell,K)$ is balanced and $K$ is such that
    \begin{align*}
        K(p)\neq K(q).
    \end{align*}
    Then, $\phi^K(\ell)=0$.
\end{lem}

\begin{proof}
    Since $n_{\mathfrak{e}(p,q)}(\ell)=0$, we must have that $h_{\ell}(p)=h_{\ell}(q)$.
    We assume that $h_{\ell}(p)=h_{\ell}(q)< 0$; if not, replace $p$ and $q$ below by $p^{-1}$ and $q^{-1}$.  
    
    Since $K(p)\neq K(q)$ by assumption and $K_\ell(p) = K_\ell(q)$ by the definition of $K_{\ell}$ in \eqref{defn:master-plaq-ass}, either $K(p)\neq K_{\ell}(p)$ or $K(q)\neq K_{\ell}(q)$. Assume without loss of generality that
    \begin{equation}\label{eq:forcontra}
        K(q)\neq K_{\ell}(q).
    \end{equation}

    Let $\mathfrak{p}=\{p_i\}_{i=1}^{\infty}$ be a straight path of plaquettes starting at $p$ and such that $p_2\neq q$ (so that $q\notin \mathfrak{p}$). Let also $e=\mathfrak{e}(p_1,p_2)$ and consider the trajectory tree $\cT_{e,\mathfrak{p}}(\ell,K)$.  Now, 
    \cref{obs: tree WLE relation} implies that
    \begin{align*}
        \phi^K(\ell) = \sum_{(s,\multK)\in \cL(\cT_{e,\mathfrak{p}}(\ell,K))}\upbeta_{s,\multK}\cdot\phi^{\multK}(s)
    \end{align*}
    where $\upbeta_{s,\multK}$ is as in \cref{obs: tree WLE relation}.
    Hence, if we show that for each $(s,\multK)=\{(\ell_1, K_1), \dots, (\ell_m , K_m)\}$ in the above sum, there exists some $j\in [m]$ such that $\phi^{K_j}(\ell_j) = 0$, then we are done.

    By Property~\ref{p5 remove plaquette} of \cref{prop: remove plaquette from loop} and the definition of $\cT_{e,\mathfrak{p}}(\ell,K)$, we know that each $\ell_i$ is supported only on the edges in $\cE(\{p,q\})\sm \{e^{\pm 1}\}$. Thus, each $\ell_i$ is either the null-loop or a non-backtrack loop (recall that the leaves of a trajectory tree only contain non-backtrack loops by construction) winding around $q$ in one of the two possible directions. 

    Property~\ref{p2 remove plaquette} of \cref{prop: remove plaquette from loop} gives us that $(s,\multK)$ cannot consist solely of loop plaquette assignments of the form $(\emptyset,0)$. Further, as $\phi^0(\emptyset)=1$, by definition (recall Property~\ref{def: phi k empty = 1} of $\phi^K(\cdot)$ on page~\pageref{def: phi k empty = 0}),  we can assume without loss of generality that $(s,\multK)$ does not contain any loop plaquette assignments of the form $(\emptyset,0)$. Thus, $(s,\multK)$ either consists entirely of loops winding around the plaquette $q$ or there exist some $j\in [m]$ such that $\ell_j=\emptyset$ and  $K_j\neq 0$.
    
    If there exists $j\in[m]$ such that $\ell_j=\emptyset$ and $K_j\neq 0$, then $\phi^{K_j}(\ell_j) = 0$ thanks to Property~\ref{def: phi k empty = 0} of $\phi^K(\cdot)$ on page~\pageref{def: phi k empty = 0}; hence in this case we are done.

    We now look at the case when all the loops $\ell_i$ are non-backtrack loops winding around $q$ in one of the two possible directions.
    Let $e'$ be any edge of $q$ not in $p$. Since $\ell$ is a non-backtrack loop winding around  the rectangle formed by the union of $p$ and $q$, we know that $e'$ only appears in one orientation in $\ell$. Hence, by Property~\ref{p5 remove plaquette} of \cref{prop: remove plaquette from loop}, we must have that $e'$ only appears in one orientation for all $\ell_i$ which must be the same orientation as in $\ell$. Hence, since $h_{\ell}(q)<0$, we have $h_{\ell_i}(q)<0$ for all $i\in[m]$.
    
    We now claim that there is some $j\in [m]$ such that $K_j\neq K_{\ell_j}$. Indeed, if not, we get that (for the first equality we use that $K(q)=\multK(q)$ by Property~\ref{p2 remove plaquette} of \cref{prop: remove plaquette from loop}, for the fourth and fifth equalities we use that $h_{\ell_i}(q)< 0$ for all $i\in [m]$ and $h_{\ell}(q)< 0$, and for the last equality we use that $h_s(q)=h_{\ell}(q)$ by Property~\ref{p4 remove plaquette} of \cref{prop: remove plaquette from loop}),
    \begin{align*}
        K(q)-K_{\ell}(q)=\multK(q)-K_{\ell}(q) =  \sum_{i=1}^mK_{i}(q)-  K_{\ell}(q) =  \sum_{i=1}^mK_{\ell_i}(q)-  K_{\ell}(q) &= \sum_{i=1}^m |h_{\ell_i}(q)|  - |h_{\ell}(q)|\\&= \left|\sum_{i=1}^m h_{\ell_i}(q)\right|- |h_{\ell}(q)|=0,
    \end{align*}
    and this would contradict \eqref{eq:forcontra}.
    Now, since $\ell_j$ is a non-backtrack loop winding around $q$ and $K_j\neq K_{\ell_j}$, and so $K_j\notin \cK_{\ell_j}=\{K_{\ell_j}\}$, \cref{lem: size one main thm} gives us that
    \begin{align*}
        \phi^{K_j}(\ell_j) = 0,
    \end{align*}
     as desired.
\end{proof}

\subsection{Loops of general size}\label{sec: general loops}

In this section we prove \cref{thm: erasable loops have one plaquette assignment}. To do this, for all plaquette assignments $K$ that are not in the canonical collection (recall \cref{defn:collection-plaquette-ass}), we will write $\phi^K(\ell)$ as a sum of terms we know are zero. In particular, we will use \cref{prop: remove plaquette from loop} to ``remove'' a plaquette from $\ell$, obtaining certain new string plaquette assignments $(s,\multK)= \{(\ell_1, K_1), \dots, (\ell_m , K_m)\}$, and then show that at least one of the resulting loop plaquette assignments $(\ell_j, K_j)$ is such that $K_j$ is still not in the canonical collection of plaquette assignments for the ``smaller'' loop $\ell_j$.
Iterating this idea, we will reduce the proof of~\cref{thm: erasable loops have one plaquette assignment} to understanding loops of small size, which we already addressed in \cref{sec: Loops supported on a plaquette}.

As we briefly discussed when sketching the main proof techniques back in Section \ref{sec: proof tech}, the main difficulty in carrying out the above strategy lies in ensuring that the loop is shrunk in a ``controlled manner'', so that the resulting loops are well-behaved and the changes to the plaquette assignments are clearly understood. In \cref{sec: spanning pair}, we introduce a procedure for performing this shrinking and detail many of its properties. Then, in \cref{sec: proof of main}, we prove \cref{thm: erasable loops have one plaquette assignment}.

\subsubsection{Shrinking procedure along a tree of paths of plaquettes}\label{sec: spanning pair}

To shrink a general loop, we will use the following two key facts:

\medskip

$(1)$ \cref{obs: tree equals null} tells us that if a loop, possibly with backtracks, is only supported on edges that form a tree, then after removing all its backtracks, it must be the null-loop. As a consequence, we have the following generalization:

\begin{obs}\label{obs: tree equals null-2}
If the subgraph of $\ZZ^2$ induced by the
set of edges of a loop $\ell$ is only supported on the edges of a plaquette (resp.\ edges on the boundary of the rectangle formed by the union of two adjacent plaquettes) and trees extending outward from these edges (cf.\ \cref{fig: remove backtrack}), then, after removing all backtracks, the loop $\ell$ must be either the null-loop or a loop winding around the plaquette in one of the two possible directions (resp.\ winding around the rectangle formed by the union of the two adjacent plaquettes in one of the two possible directions).
\end{obs}

$(2)$ The second fact, recalling \cref{obs: move edges}, is that Property~\ref{p5 remove plaquette} from \cref{prop: remove plaquette from loop} can be interpreted as ``pushing'' edges when ``removing'' plaquettes. Thus, if we can ``remove'' plaquettes from a loop and ``push'' the edges in such a way that the resulting loop is only supported (in the sense of \cref{obs: tree equals null-2}) on the edges of a plaquette (or edges on the boundary of the rectangle formed by the union of two adjacent plaquettes) and trees extending outside from these edges, we will have shrunk a general loop to a loop that has been already studied in \cref{sec: Loops supported on a plaquette}.

\begin{figure}[ht!]
\begin{center}
	\includegraphics[width=.60\textwidth]{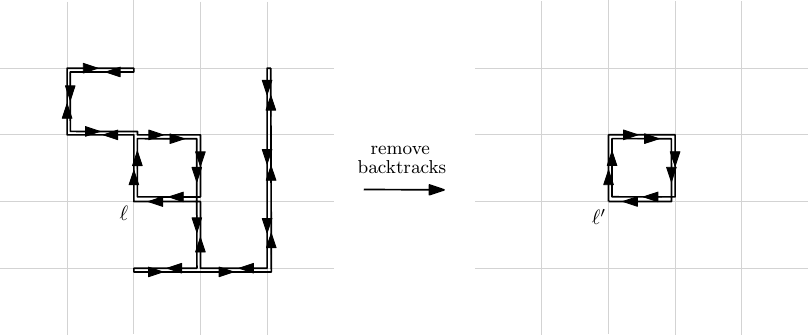}  
	\caption{\label{fig: remove backtrack} The left-hand side shows a loop $\ell$ whose subgraph of $\ZZ^2$ induced by its set of edges is only supported on the edges of a plaquette and trees extending outward from these edges. The right-hand side shows the loop $\ell'$ obtained from $\ell$ after removing all its backtracks.}
\end{center}
\vspace{-3ex}
\end{figure}

\bigskip

In the remainder of this subsection, we introduce some terminology and prove some preliminary facts that will be used later to perform this shrinking procedure.

\medskip

The shrinking procedure we employ will ``follow'' certain trees of paths of plaquettes, which we now introduce (cf.\  \cref{fig: span pair}). Given a plaquette-region $R$ (which we recall by definition contains both positively and negatively oriented plaquettes), we denote by $R^+$ the collection of positively oriented plaquettes in $R$. Let $\cG^*(R^+)$  be the (dual unoriented) graph induced by $R^+$, i.e.\ the graph obtained by adding a vertex at the center of each plaquette in $R^+$ and connecting  vertices corresponding to adjacent plaquettes. A subtree $\tree$ of $\cG^*(R^+)$ is called a \textbf{spanning tree} of $R^+$ if every vertex in $\cG^*(R^+)$ is contained in $\tree$.

We now fix $R^+= \cP^+$, i.e.\ the collection of all positively oriented plaquettes of $\mathbb Z^2$, and simply write $\cG^*$ for $\cG^*(\cP^+)$ . We say that a pair of trees $\spt=\{\tree_1,\tree_2\}$ of $\cG^*$ is a \textbf{spanning pair of trees}, if the trees $\tree_1$ and $\tree_2$ are disjoint, are both infinite, and every vertex in $\cG^*$ is contained in either $\tree_1$ or $\tree_2$. 

Given a spanning pair of trees $\spt=\{\tree_1,\tree_2\}$, we denote by $\sptd$ the oriented graph formed by the subset of the (oriented) edges in $E$ that are \emph{not} crossed by any (dual) edge of $\spt$; see the discussion below \eqref{eq:wefijfb} for more explanations on our choice of notation. Note that $\sptd$ must span all the vertices of $\mathbb{Z}^2$ (otherwise, there would be a loop in $\spt$). Further, as $\tree_1$ and $\tree_2$ are infinite trees we must have that $\sptd$, when interpreted as a graph of unoriented edges, is a (possibly infinite) collection of trees.\footnote{Note the specific spanning pairs of trees we will consider below are such that $\sptd$ only contains one component (i.e.\ $\sptd$ is itself a spanning tree of $\ZZ^2$). This fact is not essential for our proofs so we omit its proof.} To see that each connected component of $\sptd$ must be a tree, note that if a component was not, then the subset of $\cP^+$ it disconnects must completely contain $\tree_1$ or $\tree_2$, which is not possible as $\tree_1$ and $\tree_2$ are both infinite. We refer to $\sptd$ as the  \textbf{Peano forest} of $\spt$. The left-hand side of \cref{fig: span pair} gives an example of a spanning pair of trees and its Peano forest.

Denoting by $E^{\smallsetminus}(\spt)$ the (oriented) edges of $\sptd$, let $E^{\times}(\spt)$ denote the edges in $E\sm E^{\smallsetminus}(\spt)$, i.e.\ the edges in $E$ crossed by the edges of $\spt$. Note that 
\begin{equation}\label{eq:wefijfb}
    E=E^{\smallsetminus}(\spt)\sqcup E^{\times}(\spt),
\end{equation}
and that the $\times$-sign (resp.\ $\smallsetminus$-sign) in the notation denotes the fact that we are referring to edges of $E$ that are crossed (resp.\ not crossed) by dual edges of $\spt$.

\medskip

\begin{figure}[ht!]
\begin{center}
	\includegraphics[width=1\textwidth]{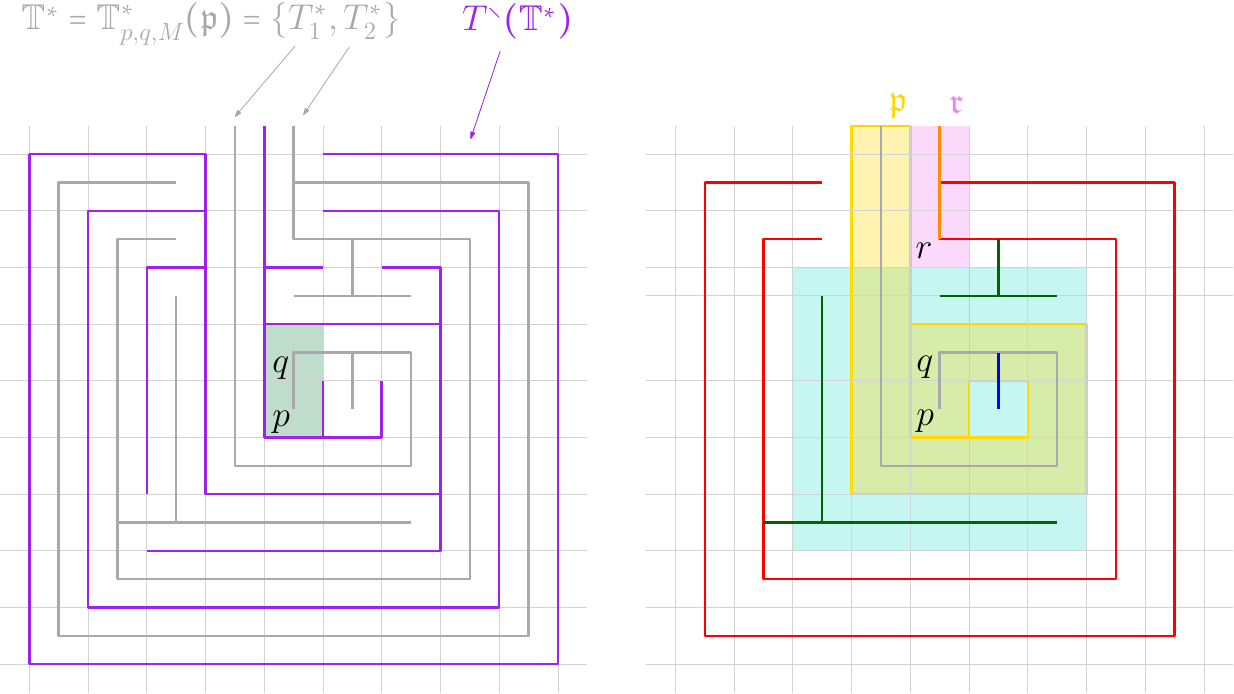}  
	\caption{\label{fig: span pair} A portion of $\mathbb{Z}^2$ is shown in both sides. \textbf{Left:} The two grey trees $\spt=\spt_{p,q,M}(\mathfrak{p})=\{\tree_1,\tree_2\}$ are a spanning pair of trees avoiding $p$ and $q$ generated from the path of plaquettes $\mathfrak{p}$, where $\mathfrak{p}$ is highlighted in gold on the right-hand side of the figure. In purple is shown the corresponding Peano forest $\sptd$ (note that each purple edge is contained twice in $\sptd$, once in each orientation). \textbf{Right:} The box $B_2(p)$ is shown in light blue. To construct the spanning pair of trees $\spt=\spt_{p,q,M}(\mathfrak{p})=\{\tree_1,\tree_2\}$ avoiding $p$ and $q$ generated from $\mathfrak{p}$ shown on the left-hand side of the figure, we first consider the gray path corresponding to $\mathfrak{p}$ and then arbitrarily construct a spanning tree for any plaquette-region cut off by $\mathfrak{p}$ and connect it to $\mathfrak{p}$ (see the blue tree). This gives the tree $T_1^*$. Next, add the orange path $\mathfrak{r}$, i.e.\ the path next of $\mathfrak{p}$ outside of $B_2(p)$, then add the concentric squared red paths starting at the plaquettes of $\mathfrak{r}$, and finally arbitrarily construct a spanning tree for each remaining plaquette-region and connect them to the previous constructed paths. This gives the tree $T_2^*$. This finishes the construction of $T_2^*$ and thus $\spt_{p,q,M}(\mathfrak{p})$.}
\end{center}
\vspace{-3ex}
\end{figure}

Next, we introduce an explicit construction of a specific spanning pair of trees (cf.\  the right-hand side of \cref{fig: span pair}). For any plaquette $p\in \cP^+$ and $k\in \NN$, let $B_k(p)$ denote the square box of positively oriented plaquettes centered at $p$ with width  $2k+1$. For instance, $B_1(p)$ is the square box formed by $p$ and the eight positively oriented plaquettes that share at least one vertex with $p$. Let $\partial B_k(p)$ denote the layer of plaquettes inside $B_k(p)$ touching the boundary of $B_k(p)$.

Now, fix two plaquettes $p,q\in \cP^+$ such that $p$ and $q$ are adjacent and an integer $M\in \NN$. Also, let $\mathfrak{p} = \{p_i\}_{i=1}^{\infty}$ be an infinite path of plaquettes such that $p_1=p$ and $p_2=q$, and once $\mathfrak{p}$ exits $B_M(p)$, it continues as a straight path of plaquettes orthogonal to the boundary of $B_M(p)$ (but we impose no restriction on the shape of the path $\mathfrak{p}$ while it travels inside $B_M(p)$ apart from $p_1=p$ and $p_2=q$). Then construct a spanning pair of trees $\spt_{p,q,M}(\mathfrak{p})=\{\tree_1,\tree_2\}$ as follows. For $\tree_1$:
\begin{enumerate}
    \item Start by including in $\tree_1$ all the edges of $\cG^*$ that cross an interior edge of $\mathfrak{p}$ (cf.\ the grey line on the right-hand side of \cref{fig: span pair}).
    \item For each (if any) finite maximal (positive) plaquette-region of $\cP^+ \sm \mathfrak{p}$, generate in an arbitrary way a spanning tree of that region and connect it (again in an arbitrary way) to $\tree_1\sm\{p,q\}$ (cf.\ the blue tree on the right-hand side of \cref{fig: span pair}).
\end{enumerate}
Then, for $\tree_2$, denoting by $r$ one of the two plaquettes in $\partial B_{M+1}(p)$ adjacent to the first plaquette in $\mathfrak{p}$ outside of $B_M(p)$ and  by $\mathfrak{r}$ the straight path of plaquettes parallel to $\mathfrak{p}$ starting at $r$ (cf.\ the violet path of plaquettes on the right-hand side of \cref{fig: span pair}):
\begin{enumerate}
    \item Start by including in $\tree_2$ all the edges of $\cG^*$ that cross an interior edge of $\mathfrak{r}$ (cf.\ the orange line on the right-hand side of \cref{fig: span pair}).
    \item For all $k\in \NN$, include in $\tree_2$ all the edges of $\cG^*$ that cross an interior edge of the path of plaquettes $\partial B_{M+k}(p) \sm \mathfrak{p}$ (cf.\ the red lines on the right-hand side of \cref{fig: span pair}).
    \item Let $\cS^+$ denote the set of positively oriented plaquettes that have a corresponding dual vertex in $\tree_1\cup\tree_2$. For each (if any) finite maximal (positive) plaquette-region of $\cP^+ \sm \cS^+$, generate in an arbitrary way a spanning tree of that region and connect it (again in an arbitrary way) to $\tree_2$ (cf.\ the green trees on the right-hand side of \cref{fig: span pair}).
\end{enumerate}
We call $\spt_{p,q,M}(\mathfrak{p})=\{\tree_1,\tree_2\}$ \textbf{a spanning pair of trees avoiding\footnote{Here avoiding refers to the fact that we constructed $\tree_1$ so that the dual vertices corresponding to $p$ and $q$ are not branching points of $\tree_1$. This will play a key role later in the proof of \cref{prop: non-zero plaquette assignments three}.} $p$ and $q$ generated from $\mathfrak{p}$} of size $M$.
The left-hand side of \cref{fig: span pair} shows an example of a spanning pair of trees avoiding $p$ and $q$ generated from $\mathfrak{p}$. The right-hand side of \cref{fig: span pair} details its construction step by step.

\medskip

Next, we make a simple observation about the above construction, which we state for future reference.

\begin{obs}\label{obs: edges in dual tree}
    Let $\spt$ be a spanning pair of trees avoiding $p$ and $q$ generated from $\mathfrak{p}$ of size $M$. Then,\begin{enumerate}
        \item $E^{\smallsetminus}(\spt)$ contains all the edges in $\cE(p)$ except the two oriented edges corresponding to $\mathfrak{e}(p,q)$,
        \item $E^{\smallsetminus}(\spt)$ contains all the edges in $\cE(p,q)$ except the two oriented edges corresponding to $\mathfrak{e}(p,q)$ and the two oriented edges corresponding to one of the unoriented edges of $q$ other than $\mathfrak{e}(p,q)$.
    \end{enumerate}
    Thus, the only cycles in $E^{\smallsetminus}(\spt)\cup \cE(p)$ are $p^{\pm 1}$ and the only cycles in $E^{\smallsetminus}(\spt)\cup \cE(p,q)$ are $p^{\pm 1}, q^{\pm 1}$ and the loop winding around the boundary of the rectangle formed by the union of $p$ and $q$ in either orientation.
\end{obs}

\medskip

The idea of our proof will be to shrink a loop in such a way that, for a carefully chosen $\spt_{p,q,M}(\mathfrak{p})$, all the edges in $E^{\times}(\spt_{p,q,M}(\mathfrak{p}))$, except those in $\cE(p)$ (or possibly on the boundary of the rectangle formed by the union of $p$ and $q$), are no longer in the loop. If this is accomplished, we will be in the setting discussed at the beginning of the section, where the loop must simply be the loop around $p$ (or perhaps around the rectangle formed by the union of $p$ and $q$) once backtracks are removed. To perform this shrinking, we need to connect edges to infinity in such a way that the path of plaquettes ``follows'' the carefully chosen $\spt_{p,q,M}(\mathfrak{p})$. We introduce such paths now.

\medskip

Fix any spanning pair of trees $\spt=\spt_{p,q,M}(\mathfrak{p})$ avoiding $p$ and $q$ generated from $\mathfrak{p}$ and of size $M\geq 1$. Let $e$ be any edge in $E^{\times}(\spt)$. Then, by the construction of $\spt$, there are only a finite number of infinite paths $\mathfrak{p}^e=\{p^e_i\}_{i=1}^{\infty}$ of plaquettes such that (cf.\ the left-hand side of~\cref{fig: Tree edges}): 
\begin{equation}\label{eq:tree path}
    e\in \cI(\mathfrak{p}^e), \quad p^e_1 \text{ corresponds to a leaf of }  \spt,  \quad \text{and} \quad \mathfrak{e}(p^e_i,p^e_{i+1})\in E^{\times}(\spt), \forall i\in \NN,
\end{equation}
where here, with $\mathfrak{e}(p^e_i,p^e_{i+1})\in E^{\times}(\spt)$, we actually mean that both oriented edges corresponding to $\mathfrak{e}(p^e_i,p^e_{i+1})$ are in $E^{\times}(\spt)$. We call such a path of plaquettes a \textbf{tree-path of plaquettes containing $e$} and denote by $\mathfrak{P}_{\spt}(e)$ the (finite) set of tree-paths of plaquettes containing $e$.

\medskip

We record a trivial (but important) observation, stated for future reference.

\begin{figure}[ht!]
\begin{center}
    \includegraphics[width=1\textwidth]{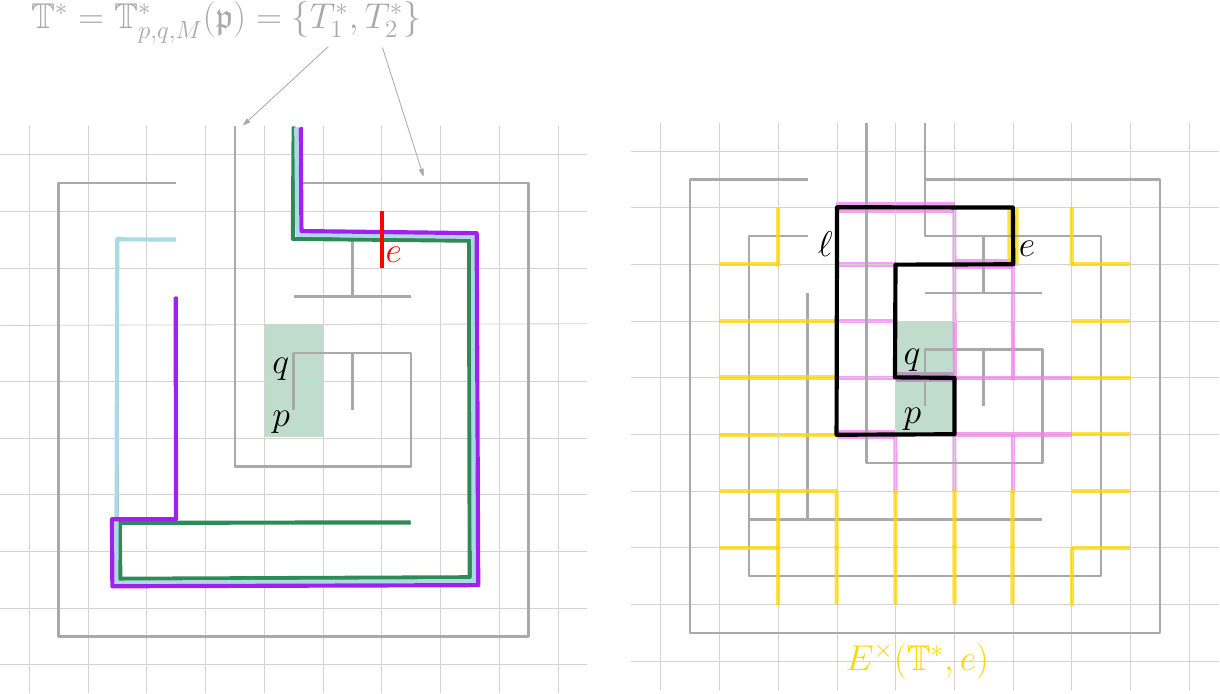}  
	\caption{\label{fig: Tree edges}\textbf{Left:} Given the gray spanning pair of trees $\spt$ avoiding $p$ and $q$ generated from $\mathfrak{p}$ from \cref{fig: span pair}, the three tree-paths of plaquettes in $\mathfrak{P}_{\spt}(e)$ containing the red (oriented) edge $e\in E^{\times}(\spt)$ are highlighted in green, blue, and purple. (The orientation of $e$ is not shown since not important.) \textbf{Right:} Given the black loop $\ell$ and the gray spanning pair of trees avoiding $p$ and $q$ generated from $\mathfrak{p}$ from \cref{fig: span pair}, the edges in $E^{\times}(\spt, e)$ from \eqref{eq:defn-Ee} are colored in yellow ($E^{\times}(\spt, e)$ contains both orientations of each yellow edge). The edges in $E^{\times}(\spt, \ell)$ from \eqref{defn:Ecross} are the yellow and violet edges ($E^{\times}(\spt, \ell)$ contains both orientation of each yellow and violet edge). The edges in $E(\spt, \ell)$ from \eqref{eq:defn-Et} are the yellow and violet edges plus all the oriented edges of $\ZZ^2$ not crossed by any (dual) edge of $\spt$.}
\end{center}
\vspace{-3ex}
\end{figure}

\begin{obs}\label{obs: ppqe}
If $e$ is not in $\mathcal{E}(\{p\})$ (resp. $\mathcal{E}(\{p, q\})$), then $p$ does not contain $e$ (resp. $p, q$ do not contain $e$). In particular, this has the following consequence.
Let $\spt= \{T_1^*,T_2^*\}$ be a spanning pair of trees avoiding $p$ and $q$ generated from $\mathfrak{p}$. Given any edge in $E^{\times}(\spt)$ not in $\cE(\{p\})$ (resp.\ not in $\cE(\{p,q\})$) and any tree-path of plaquettes $\mathfrak{p}^e\in \mathfrak{P}_{\spt}(e)$, the plaquettes in $\mathfrak{p}^e$ that contain $e$ are not $p$ or $p^{-1}$ (resp.\ $p^{\pm 1}$ or $q^{\pm1}$).
\end{obs}

Now,  to shrink a loop $\ell$ to just edges in $E^{\smallsetminus}(\spt)$, we need to keep track of which edges need to be ``removed'' from $\ell$. We now detail such edges. Let $e$ be any edge in $E^{\times}(\spt)$. For each tree-path of plaquettes $\mathfrak{p}^e = \{p^e_i\}_{i=1}^{\infty}\in \mathfrak{P}_{\spt}(e)$ containing $e$, let $j=j(\mathfrak{p}^e)\in \NN$ be such that $\mathfrak{e}(p^e_{j},p^e_{j+1})$ is the unoriented edge corresponding to $e$. With this, for any edge $e\in E^{\times}(\spt)$, let (cf.\ the yellow edges on the right-hand side of \cref{fig: Tree edges})
\begin{align}\label{eq:defn-Ee}
    E^{\times}(\spt, e):= \bigcup_{\mathfrak{p}^e=  \{p^e_i\}_{i=1}^{\infty}\in \mathfrak{P}_{\spt}(e)}\left\{e'\in E: e'\text{ corresponds to }\mathfrak{e}(p^e_i,p^e_{i+1})\right\}_{i=1}^{j(\mathfrak{p}^e)}.
\end{align}
In words, the set $E^\times(\spt, e)$ consists of (both orientations of) $e$ along with (both orientations of) the interior edges in any tree-path containing $e$ which appear before $e$. If we ``push'' $e$ as in \cref{obs: move edges} along some $\mathfrak{p}^e\in \mathfrak{P}_{\spt}(e)$, we might introduce new edges in $E^{\times}(\spt)$. As we will see later, the set $E^{\times}(\spt, e)$ includes all the edges that need to be ``removed'' from $\ell$ (by ``pushing'' $e$) to ensure that $e$ and the edges introduced when ``pushing'' $e$ are only contained in $E^{\smallsetminus}(\spt)$.

Let $\cE(\ell)$ denote the collection of edges containing any edge of the loop $\ell$ in either orientation, i.e.\   
\[\cE(\ell)=\left\{e\in E: e\in \ell \text{ or }e^{-1}\in\ell \right\}.\]
Then we define (cf.\ the yellow and violet edges on the right-hand side of  \cref{fig: Tree edges})
\begin{equation}\label{defn:Ecross}
     E^{\times}(\spt, \ell):=\bigcup_{e\in \cE(\ell) \cap E^{\times}(\spt)} E^{\times}(\spt, e).
\end{equation}
As we will show later, $E^{\times}(\spt, \ell)$ is the set of all the edges that need to be ``removed'' from $\ell$ (by ``pushing'') so that the new loop consists solely of edges in $E^{\smallsetminus}(\spt)$. 

It will be important to keep track of the number of edges that need to be removed from $\ell$, excluding those in $p$ (or $\{p,q\}$). Let
\begin{equation}\label{eq:defn-np}
    n^{\times}_{p}(\spt, \ell):= |E^{\times}(\spt, \ell)\sm \cE(p)|\quad\text{ and }\quad n^{\times}_{p,q}(\spt, \ell):= |E^{\times}(\spt, \ell)\sm \cE(\{p,q\})|,
\end{equation}
both of which we refer to as \textbf{the number of edges to remove}.
Lastly, we state and prove two simple results (Lemmas~\ref{lem: paths agree}~and~\ref{lem: exists an outer edge}) that will be used later.

\begin{lem}\label{lem: paths agree}
    Let $\spt=\spt_{p,q,M}(\mathfrak{p})= \{T_1^*,T_2^*\}$ be a spanning pair of trees avoiding $p$ and $q$ generated from $\mathfrak{p}$ of size $M\geq 1$. Fix $e\in E^{\times}(\spt)$. Then, for any $\mathfrak{p}^e=\{p^e_i\}_{i=1}^{\infty}\in \mathfrak{P}_{\spt}(e)$ and $\mathfrak{q}^e=\{q^e_i\}_{i=1}^{\infty}\in \mathfrak{P}_{\spt}(e)$,
    \begin{align*}
        p^e_{j(\mathfrak{p}^e)+k} = q^e_{j(\mathfrak{q}^e)+k}
    \end{align*}
    for all $k\geq 0$.
\end{lem}

\begin{proof}
    This simply follows from the fact that given any plaquette $z$, there is a unique infinite path of plaquettes $\{z_i\}_{i=1}^{\infty}$ such that $z_1=z$ and $\mathfrak{e}(z_i,z_{i+1})\in E^{\times}(\spt)$ for all $i\geq 1$, by the construction of $\spt$ (where the notation $\mathfrak{e}(z_i,z_{i+1})\in E^{\times}(\spt)$ should be interpreted as in \eqref{eq:tree path}).
\end{proof}



Finally, we state a lemma that will be used when determining which edge to ``remove'' (cf.\ the left-hand side of \cref{fig: Tree exploration example} for an illustrative example). First, we introduce the set of edges 
\begin{equation}\label{eq:defn-Et}
    E(\spt, \ell)\coloneqq E^{\smallsetminus}(\spt) \sqcup  E^{\times}(\spt, \ell).
\end{equation}

\begin{figure}[b!]
\begin{center}
    \includegraphics[width=1\textwidth]{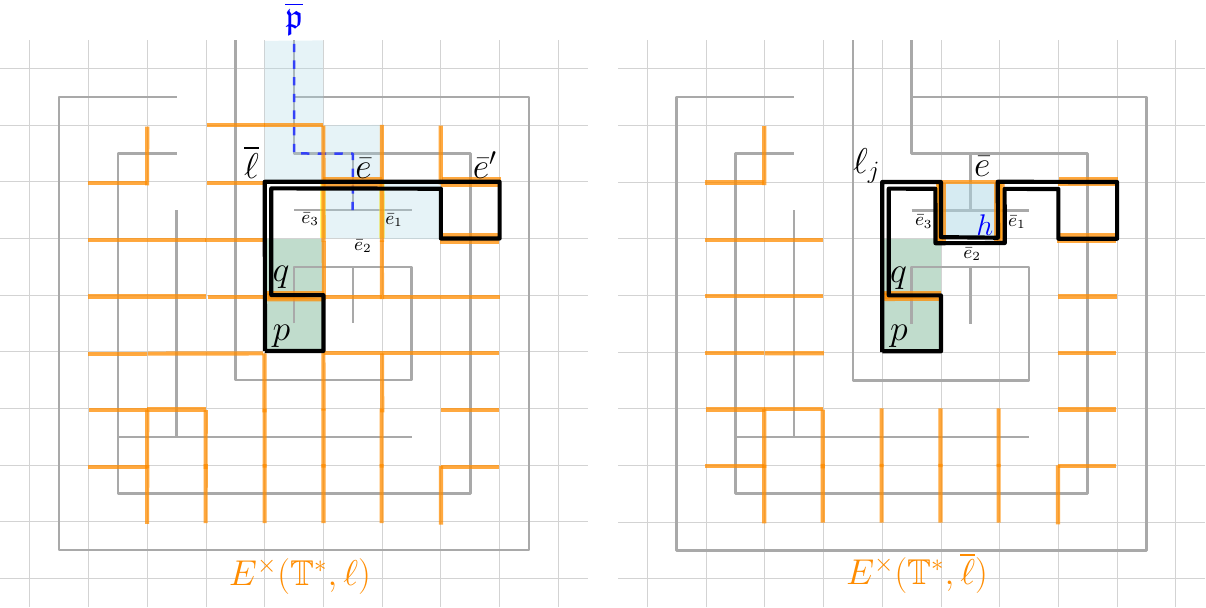}  
	\caption{\label{fig: Tree exploration example}\textbf{Left:} A typical situation in \cref{lem: exists an outer edge}. Given the gray spanning pair of trees $\spt$ avoiding $p$ and $q$ generated from $\mathfrak{p}$ from \cref{fig: span pair} and $\ell$ (not shown here) as on the right-hand side of \cref{fig: Tree edges}, $\Bar{\ell}$ is a loop using only edges of $E(\spt,\ell)=E^{\smallsetminus}(\spt) \sqcup  E^{\times}(\spt, \ell)$, where the edges in $E^{\times}(\spt, \ell)$ are highlighted in orange (and correspond to the union of the violet and yellow edges on the right-hand side of \cref{fig: Tree edges}). Note $\cE(\Bar{\ell})\cap E^{\times}(\spt,\ell)$ contains eight edges (only two in $\cE(p)$), and thus $n^{\times}_{p}(\spt,\ell)\geq 1$. The two edges $\bar{e}$ and $\bar{e}'$ of $\Bar{\ell}$ are such that they are the outermost edge of $\Bar{\ell}$ for all the tree paths of plaquettes corresponding to that edge. Also, for the edge $\Bar{e}$, the set $E^{\times}(\spt,\Bar{e})$ consists of $\Bar{e}$ and the (orange) edges $\Bar{e}_1$ and $\Bar{e}_3$. Notice that if $\{p_i^{\Bar{e}}\}_{i=1}^{\infty}\in \mathfrak{P}_{\spt}(\Bar{e})$ (for example the path of plaquettes highlighted in light blue), and $k$ is such that $\Bar{e}$ corresponds to $ \mathfrak{e}(p_{k}^{\Bar{e}},p_{k+1}^{\Bar{e}})$, then clearly $\Bar{\mathfrak{p}} = \{p^{\Bar{e}}_{i+k-1}\}_{i=1}^{\infty}$ (the path of plaquettes highlighted by dotted blue line) connects $\Bar{e}$ to infinity. \textbf{Right:} A typical situation in \cref{prop: non-zero plaquette assignments one}/\cref{claim: K1 claim}. Given the loop $\Bar{\ell}$, edge $\Bar{e}$, and path of plaquettes $\Bar{\mathfrak{p}}$ from the left-hand side of the figure, $\ell_j$ shows one of the possible strings (a single loop in this case) corresponding to a leaf on the trajectory tree $\cT_{\Bar{e},\Bar{\mathfrak{p}}}(\Bar{\ell},\Bar{K})$ (where $\Bar{K}$ is any plaquette assignment such that $(\Bar{\ell},\Bar{K})$ is balanced and $\Bar{K}$ contains at least one copy each of the plaquettes $h$ and $h^{-1}$ highlighted in blue). Note that $\ell_j$ is constructed from the edges in $E(\spt,\Bar{\ell})\subset E(\spt,\ell)$, where we note that the edges in $E^{\times}(\spt, \ell)$ are highlighted in orange on the left-hand side and the edges in $E^{\times}(\spt, \Bar{\ell})$ are highlighted in orange on the right-hand side. 
    Note that $n^{\times}_{p}(\spt,\ell_j)<n^{\times}_{p}(\spt,\Bar{\ell})$ as $E^\times(\spt,\ell_j)\subset E^\times(\spt,\Bar{\ell})$ and  $\ell_j$ no longer contains $\Bar{e}$ (so also $ E^{\times}(\spt,\ell_j)$ no longer contains $\Bar{e}$).}
\end{center}
\vspace{-3ex}
\end{figure}

\begin{lem}\label{lem: exists an outer edge}
    Let $\spt=\spt_{p,q,M}(\mathfrak{p})= \{T_1^*,T_2^*\}$ be a spanning pair of trees avoiding $p$ and $q$ generated from $\mathfrak{p} =\{p_i\}_{i=1}^{\infty}$ of size $M\geq 1$ and $\ell$ be a loop. Then, for any loop $\Bar{\ell}$ with all its edges in $E(\spt, \ell)$ such that $n^{\times}_{p}(\spt, \Bar{\ell})\geq 1$ (resp.\ $n_{p,q}^{\times}(\spt, \Bar{\ell})\geq 1$), there exists an edge $\Bar{e}$ of $\Bar{\ell}$ such that: 
    \begin{enumerate}
        \item[1.] $\Bar{e}\in E^{\times}(\spt)$;
        \item[2.] $\Bar{e}$ is the outermost edge of $\Bar{\ell}$ in $\mathfrak{p}^{\Bar{e}}$, for all $\mathfrak{p}^{\Bar{e}}\in \mathfrak{P}_{\spt}(\Bar{e})$; 
        \item[3.] $\Bar{e}\notin \cE(p)$ (resp.\ $\Bar{e}\notin \cE(\{p,q\})$).
    \end{enumerate}
    Furthermore, for such an edge $\Bar{e}$, either $\Bar{e}\in \cI(\mathfrak{p})$; in which case, taking $k\in \NN$ such that $\mathfrak{e}(p_k,p_{k+1})$ corresponds to $\Bar{e}$,
    \begin{enumerate}
    \item[{\crtcrossreflabel{4a}[item:3a]}.]  the path $\Bar{\mathfrak{p}} = \{\Bar{p}_i\}_{i=1}^{\infty}:=\{p_{i+k-1}\}_{i=1}^{\infty}$ connects $\Bar{e}$ to infinity;
    \end{enumerate}
    or $\Bar{e}\notin \cI(\mathfrak{p})$; in which case, fixing any $\{p_i^{\Bar{e}}\}_{i=1}^{\infty} =\mathfrak{p}^{\Bar{e}}\in \mathfrak{P}_{\spt}(\Bar{e})$ and taking $k\in \NN$ such that $\mathfrak{e}(p_k^{\Bar{e}},p_{k+1}^{\Bar{e}})$ corresponds to $\Bar{e}$,
    \begin{enumerate}
         \item[{\crtcrossreflabel{4b}[item:3b]}.] 
         the path $\Bar{\mathfrak{p}} = \{\Bar{p}_i\}_{i=1}^{\infty}:=\{p^{\Bar{e}}_{i+k-1}\}_{i=1}^{\infty}$ connects $\Bar{e}$ to infinity. Moreover, in this case $\Bar{p}_1\notin\mathfrak{p}$.
    \end{enumerate}
\end{lem}

\begin{proof}
    We only prove the case when $n^{\times}_{p}(\spt, \Bar{\ell})\geq 1$ as the $n_{p,q}^{\times}(\spt, \Bar{\ell})\geq 1$ case can be proven using the exact same arguments.
    
    As $n^{\times}_{p}(\spt, \Bar{\ell})\geq 1$, there must be some edge $e'$ in $E^{\times}(\spt,\ell)\sm \cE(p)$ in $\Bar{\ell}$. Let $\mathfrak{p}^{e'}\in \mathfrak{P}_{\spt}(e')$. Also, let $\Bar{e}$ denote the outermost edge of $\Bar{\ell}$ in $\mathfrak{p}^{e'}$. As $\Bar{e}$ is the outermost edge of $\Bar{\ell}$ in $\mathfrak{p}^{e'}$, we know that $\Bar{e}\in E^{\times}(\spt)$ by the definition of tree-paths of plaquettes in \eqref{eq:tree path}. Also, clearly $\mathfrak{p}^{e'}\in \mathfrak{P}_{\spt}(\Bar{e})$. Thus $\Bar{e}$ is the outermost edge of $\Bar{\ell}$ for at least one of the paths in $\mathfrak{P}_{\spt}(\Bar{e})$. Then \cref{lem: paths agree} implies that $\Bar{e}$ is the outermost edge for all the paths in $\mathfrak{P}_{\spt}(\Bar{e})$, as desired. Lastly, $\Bar{e}\notin \cE(p)$ because $e'\notin \cE(p)$, $\Bar{e}$ is the outermost edge of $\Bar{\ell}$ in $\mathfrak{p}^{e'}$, and any path in $\mathfrak{P}_{\spt}(e')$ either does not contain any edges in $\cE(p)$ or it contains an edge corresponding to $\mathfrak{e}(p_1,p_2)$ (the only edge in $\cE(p)$ that can be in a path of $\mathfrak{P}_{\spt}(e')$, see e.g.~\cref{fig: Tree edges}) as the first edge of the path, but this edge cannot be the outermost edge of the path. To see this last claim, note that $\mathfrak{e}(p_1, p_2) \in \cE(p)$ and $e' \notin \cE(p)$.

    Now let $\Bar{e}$ be an edge such that it is the outermost edge in $\Bar{\ell}$ for all tree-paths in $\mathfrak{P}_{\spt}(\Bar{e})$. Then, for any $\{p_i^{\Bar{e}}\}_{i=1}^{\infty}= \mathfrak{p}^{\Bar{e}}\in \mathfrak{P}_{\spt}(\Bar{e})$, taking $k$ such that $\mathfrak{e}(p_k^{\Bar{e}},p_{k+1}^{\Bar{e}}) = \Bar{e}$, we clearly have that $\{p^{\Bar{e}}_{j}\}_{j=k}^{\infty}$ connects $\Bar{e}$ to infinity (since $\Bar{e}$ is the outermost edge of $\Bar{\ell}$ in $\mathfrak{p}^{\Bar{e}}$). Thus, if $\Bar{e}\in \cI(\mathfrak{p})$, the fact that $\mathfrak{p}\in \mathfrak{P}_{\spt}(\Bar{e})$ and the above reasoning give us Item~\ref{item:3a}.
    
    On the other hand, if $\Bar{e}\notin \cI(\mathfrak{p})$, the above reasoning gives the first claim in  Item~\ref{item:3b}. The second claim simply follows from the fact that $\{p_i^{\Bar{e}}\}_{i=1}^k\cap\mathfrak{p}=\emptyset$ if $\Bar{e}\notin \cI(\mathfrak{p})$ by construction of the paths of plaquettes in $\mathfrak{P}_{\spt}(\Bar{e})$.
\end{proof}

\subsubsection{Wilson loop expectations as a finite explicit sum over canonical
plaquette assignments}\label{sec: proof of main}

Finally, we prove \cref{thm: erasable loops have one plaquette assignment}. To do this, we will prove Propositions~\ref{prop: non-zero plaquette assignments one},~\ref{prop: non-zero plaquette assignments two}~and~\ref{prop: non-zero plaquette assignments three}, each showing that one of the three conditions appearing in the definition of the canonical collection of plaquette assignments (\cref{defn:collection-plaquette-ass}) is necessary to have a non-zero surface sum.

Each proposition will be proven by shrinking a loop along the tree-paths of plaquettes of a well-chosen spanning pair of trees avoiding $p$ and $q$ generated from $\mathfrak{p}$, for a specific choice of $p,q$, and $\mathfrak{p}$, until the loops are small. 

\begin{prop}[\textsc{necessity of the first condition}]\label{prop: non-zero plaquette assignments one}
    Let $\ell$ be any non-trivial loop. Let $\cK_{\ell}^1$ denote the collection of finite plaquette assignments of the form 
    \begin{equation*}
        K = K_{\ell}+K',\qquad\text{where $K'$ is such that $K'(p)=K'(p^{-1})$ for all $p\in \cP$}.
    \end{equation*}
    If $K\notin \cK_{\ell}^1$, then $\phi^K(\ell)=0$.
\end{prop}

\begin{proof}
    First, \cref{lemma: backtrack cancellations} allows us to assume that $\ell$ has no backtracks (recall from \cref{defn:collection-plaquette-ass} that if $\ell'$ is the loop obtained from a backtrack loop $\ell$ by removing all its backtracks, then $\cK_{\ell'}= \cK_{\ell}$).
    Further, if $\ell$ has no backtracks and  winds around a single plaquette, we have the desired result by \cref{lem: size one main thm}. Also, if $(\ell,K)$ is not balanced, then the result is trivial. Thus, we consider $(\ell,K)$ such that
    \begin{equation}\label{eq:ass-on-pl}
        \text{$\ell$ has no backtracks},\qquad \area(\ell)\geq 2, \qquad K\notin \cK_{\ell}^1,\qquad \text{and} \qquad(\ell,K)  \text{ is balanced}.
    \end{equation}
    Since $K\notin \cK_{\ell}^1$ by assumption and $K_\ell(p)-K_\ell(p^{-1})=h_{\ell}(p)$ for all plaquettes $p\in\cP$ by definition of $K_{\ell}$ in \eqref{defn:master-plaq-ass}, there exists some plaquette $\hat{p}\in \cP$ such that 
    \begin{equation}\label{eq:K-not-h}
        K(\hat{p})-K\big(\hat{p}^{-1}\big) \neq h_{\ell}(\hat{p}).
    \end{equation}
    Let $\hat{\mathfrak{p}}$ be the vertical path of plaquettes starting at $\hat{p}$. Let $\hat{q}$ denote the second plaquette in $\hat{\mathfrak{p}}$. As $\hat{\mathfrak{p}}$ is just a straight line, we can construct a spanning pair of trees $\spthat=\spt_{\hat{p},\hat{q},\hat{M}}(\hat{\mathfrak{p}})$ avoiding $\hat{p}$ and $\hat{q}$ from $\hat{\mathfrak{p}}$ of size (say) $\hat{M}= 1$. Also, note that since $\area(\ell)\geq 2$, we must have that (recall the notation in \eqref{eq:defn-np})
    \begin{equation}\label{eq: more than one edge to remove K1}
        n^{\times}_{\hat{p}}(\spthat,\ell)\geq 1.
    \end{equation}
    The following claim will be our key tool to shrink the loop $\ell$ to smaller loops we know how to deal with (cf.\ the right-hand side of \cref{fig: Tree exploration example} for an illustrative example). It tells us that we can iteratively shrink $\ell$, preserving the condition in \eqref{eq:K-not-h}, until the shrunken loop consists only of edges in $\sptdhat \cup \cE(\hat{p})$. This will imply (thanks to \cref{obs: tree equals null-2} and \cref{obs: edges in dual tree}) that the shrunken loop, after removing backtracks, will only be supported on $\hat{p}$. Recall the set $E(\spt, \ell)$ introduced in  \eqref{eq:defn-Et}.

    \begin{claim}\label{claim: K1 claim}
        Let $\ell$, $\hat{p}$, $\hat{\mathfrak{p}}$ and $\spthat$ be as above. Suppose that $(\Bar{\ell},\Bar{K})$ is a  loop plaquette assignment such that:
        \begin{enumerate}[(i)]
            \item $(\Bar{\ell},\Bar{K})$ is balanced and does not contain backtracks;
            \item all the edges of $\Bar{\ell}$ are in $E(\spthat, \ell)$; 
            \item $n^{\times}_{\hat{p}}(\spthat, \Bar{\ell})\geq 1$;
            \item $\Bar{K}(\hat{p})-\Bar{K}\big(\hat{p}^{-1}\big) \neq h_{\Bar{\ell}}(\hat{p})$.
        \end{enumerate} 
        There exists an edge $\Bar{e}$ of $\Bar{\ell}$ and a path of plaquettes $\Bar{\mathfrak{p}}$ connecting $\Bar{e}$ to infinity such that the trajectory tree $\cT_{\Bar{e},\Bar{\mathfrak{p}}}(\Bar{\ell},\Bar{K})$ generated from $(\Bar{\ell},\Bar{K},\Bar{e},\Bar{\mathfrak{p}})$ from \cref{defn:tr-tree-gen} has the following property: for all  leaves $(s,\multK)=\{(\ell_1,K_2),\dots,(\ell_m,K_m)\}$ of $\cT_{\Bar{e},\Bar{\mathfrak{p}}}(\Bar{\ell},\Bar{K})$, there exists  $j\in [m]$ such that:
        \begin{enumerate}[(i)]
            \item $(\ell_j,K_j)$ is balanced and does not contain backtracks; 
            \item all the edges of $\ell_j$ are in $E(\spthat, \ell)$;
            \item $n^{\times}_{\hat{p}}(\spthat, \ell_j)<n^{\times}_{\hat{p}}(\spthat, \Bar{\ell})$;
            \item $K_j(\hat{p})-K_j\big(\hat{p}^{-1}\big) \neq h_{\ell_j}(\hat{p})$.
        \end{enumerate}
    \end{claim}

    Assuming the claim, we first complete the proof of the proposition.
    Thanks to \eqref{eq:ass-on-pl}, \eqref{eq:K-not-h}, and \eqref{eq: more than one edge to remove K1}, we know that the pair $(\ell, K)$ satisfies the assumptions of \cref{claim: K1 claim}. Let $\Bar{e}$ and $\Bar{\mathfrak{p}}$ be the edge and the path of plaquettes given by the claim (used with $(\Bar{\ell},\Bar{K})=(\ell,K)$). Then, by \cref{obs: tree WLE relation}, we have that
    \begin{equation}\label{eq: K1 reduction}
    \phi^K(\ell) = \sum_{(s,\multK)\in \cL(\cT_{\Bar{e},\Bar{\mathfrak{p}}}(\ell,K))}\upbeta_{s,\multK}\cdot\phi^{\multK}(s) = \sum_{(s,\multK)\in \cL(\cT_{\Bar{e},\Bar{\mathfrak{p}}}(\ell,K))} \upbeta_{s,\multK}\prod_{(\ell_i, K_i) \in (s,\multK)} \phi^{K_i}(\ell_i),
    \end{equation}
    where $\upbeta_{s,\multK}$ is as in \cref{obs: tree WLE relation}. Now, by \cref{claim: K1 claim}, for each $(s,\multK) \in \cL(\cT_{\Bar{e},\Bar{\mathfrak{p}}}(\ell,K))$, there exists $(\ell_j, K_j) \in (s,\multK)$ such that: $(i)$ $(\ell_j, K_j)$ is balanced and does not contain backtracks; $(ii)$ all the edges of $\ell_j$ are in $E(\spthat, \ell)$; $(iii)$~$n^{\times}_{\hat{p}}(\spthat, \ell_j) < n^{\times}_{\hat{p}}(\spthat, \ell)$; and $(iv)$ $K_j(\hat{p})-K_j\big(\hat{p}^{-1}\big) \neq h_{\ell_j}(\hat{p})$.

    Thus, if we can show that $\phi^{K_j}(\ell_j) = 0$, it follows that the entire summand corresponding to $(s,\multK)$ in \eqref{eq: K1 reduction} is zero. Since this reasoning applies to each leaf in $\cT_{\Bar{e},\Bar{\mathfrak{p}}}(\ell,K)$, we reduced the task of proving that $\phi^K(\ell) = 0$ to showing that $\phi^{\tilde{K}}(\tilde{\ell}) = 0$ for all $(\tilde{\ell}, \tilde{K})$ such that $(i)$, $(ii)$, $(iii)$ and $(iv)$ hold. 

    Iterating this reduction, by successive applications of \cref{claim: K1 claim}, we conclude that it suffices to verify that 
    \[\phi^{{\tilde{K}}}({\tilde{\ell}}) = 0,\] 
    for all loop plaquette assignments $({\tilde{\ell}}, {\tilde{K}})$ such that:
    \begin{enumerate}[(i)]
    \item $({\tilde{\ell}}, {\tilde{K}})$ is balanced and does not contain backtracks;
    \item all the edges of $\Tilde{\ell}$ are in $E(\spthat, \ell)$;
    \item $n^{\times}_{\hat{p}}(\spthat, {\tilde{\ell}}) = 0$;
    \item ${\tilde{K}}(\hat{p})-{\tilde{K}}\big(\hat{p}^{-1}\big) \neq h_{{\tilde{\ell}}}(\hat{p})$.
    \end{enumerate}

    Now, assuming the latter four assumptions $(i)$--$(iv)$, observe that since all the edges of $\tilde{\ell}$ are in $E(\spthat, \ell)$ and $n^{\times}_{\hat{p}}(\spthat, {\tilde{\ell}}) = 0$, ${\tilde{\ell}}$ is constructed solely from edges in $\sptdhat \cup \cE(\hat{p})$. Therefore, due to \cref{obs: edges in dual tree} and the fact that $\Tilde{\ell}$ is a non-backtrack loop, \cref{obs: tree equals null-2} guaranties that the loop ${\tilde{\ell}}$ must be either the null-loop or a loop winding around the plaquette $\hat{p}$ (in one of the two possible directions).
    
    If $\tilde{\ell}$ is a loop winding around the plaquette $\hat{p}$, since ${\tilde{K}}(\hat{p})-{\tilde{K}}\big(\hat{p}^{-1}\big) \neq h_{{\tilde{\ell}}}(\hat{p})$, we have that ${\tilde{K}} \notin \cK_{{\tilde{\ell}}}^1$. This allows us to apply \cref{lem: size one main thm} because $\cK_{{\tilde{\ell}}}\subset \cK_{{\tilde{\ell}}}^1$, yielding $\phi^{{\tilde{K}}}({\tilde{\ell}}) = 0$.
    If $\tilde{\ell}$ is the null-loop, we note that the condition ${\tilde{K}} \notin \cK_{{\tilde{\ell}}}^1$ implies that ${\tilde{K}} \neq 0$ on at least one plaquette (otherwise ${\tilde{K}}$ would belong to $\cK_{{\tilde{\ell}}}^1$). Hence, using \cref{obs: tree equals null} and the fact that $\phi^{\dot{K}}(\emptyset) = 0$ for all $\dot{K} \neq 0$ from Property~\ref{def: phi k empty = 0} on page~\pageref{def: phi k empty = 0}, we again get $\phi^{{\tilde{K}}}({\tilde{\ell}}) = 0$.

    Therefore, assuming the validity of the claim, we conclude that $\phi^K(\ell) = 0$, as desired.

    \medskip

    We now prove \cref{claim: K1 claim}. Suppose that $(\Bar{\ell},\Bar{K})$ satisfies all the assumptions of the claim. Then, the first three items of \cref{lem: exists an outer edge} give us the existence of an edge $\Bar{e}$ of $\Bar{\ell}$ such that 
    \begin{equation}\label{eq:fvwievfow}
        \text{$\Bar{e}$ is in $E^{\times}(\spthat)$ but not in $\cE(\hat{p})$, and $\Bar{e}$ is the outermost edge of every tree-path $\mathfrak{q}\in\mathfrak{P}_{\spthat}(\Bar{e})$.}
    \end{equation}
    Further, Items~\ref{item:3a}~and~\ref{item:3b} of \cref{lem: exists an outer edge} give us a path of plaquettes $\Bar{\mathfrak{p}} = \{\Bar{p}_i\}_{i=1}^{\infty}$ connecting $\Bar{e}$ to infinity.
    
    Now, we show that the conclusion of the claim holds for this choice of $\Bar{e}$ and $\Bar{\mathfrak{p}}$. Let 
    $$(s,\multK)=\left\{(\ell_1,K_1),\dots,(\ell_m,K_m)\right\}$$ 
    be a leaf on the trajectory tree generated from $(\Bar{\ell},\Bar{K},\Bar{e},\Bar{\mathfrak{p}})$. 
    
    \medskip
    
    \noindent{\emph{\underline{Proof of (i):}}} Thanks to the definition of the trajectory tree generated from $(\Bar{\ell},\Bar{K},\Bar{e},\Bar{\mathfrak{p}})$ and \cref{rem:balanced}, we have that for all $i\in[m]$,
    \begin{equation}\label{eq:cond1-ok}
        (\ell_i,K_i) \text{ is balanced and does not contain backtracks.} 
    \end{equation}
    
    \medskip
    
    \noindent{\emph{\underline{Proof of (ii):}}} Property~\ref{p5 remove plaquette} of \cref{prop: remove plaquette from loop} gives us that\begin{equation}\label{eq: edge dec}
        n_{e}(s) \leq n_{e}(\Bar{\ell}),\quad\text{ for all }e\in E\sm \{(\Bar{e}_1)^{\pm 1}, (\Bar{e}_2)^{\pm 1}, (\Bar{e}_3)^{\pm 1}\},
    \end{equation}
    where $\Bar{p}_1=\Bar{e}\,\Bar{e}_1\,\Bar{e}_2\,\Bar{e}_3$. Hence, for all $i\in[m]$, the only edges that can be in $\ell_i$ without necessarily being in $\Bar{\ell}$ are $(\Bar{e}_1)^{\pm 1}, (\Bar{e}_2)^{\pm 1}, (\Bar{e}_3)^{\pm 1}$. But $\cE(\Bar{p}_1) \subset   E(\spthat, \Bar{\ell})$ by the definition of $ E(\spthat, \Bar{\ell})$ in \eqref{eq:defn-Et}. Thus, every edge of $\ell_i$ is contained in $E(\spthat, \Bar{\ell}) = E^{\smallsetminus}(\spthat) \sqcup E^{\times}(\spthat,\Bar{\ell})$ for all $i\in [m]$. Thus, to show (ii), it is sufficient to show that (recall from \eqref{defn:Ecross} that 
     $E^{\times}(\spthat, \Bar{\ell}):=\bigcup_{e\in \cE(\Bar{\ell}) \cap E^{\times}(\spthat)} E^{\times}(\spthat, e)$)
    \begin{equation}\label{eq:ewivfvqweuoi}
        E^{\times}(\spthat,\Bar{\ell})\subset E^{\times}(\spthat,\ell).
    \end{equation}
    To see this, suppose that $\Tilde{e}$ is an edge in $\cE(\Bar{\ell})\cap E^{\times}(\spthat)$. Now, since $E(\spthat,\ell)$ contains all the edges of $\Bar{\ell}$ (by assumption (ii)) and contains both orientations of its own edges (by definition), we get that $\cE(\Bar{\ell})\subset E(\spthat,\ell)$. Thus, as $\Tilde{e}\in \cE(\Bar{\ell})\cap E^{\times}(\spthat)$, we get that $\Tilde{e}\in E^{\times}(\spthat,\ell)$. Therefore, by the definition of $E^{\times}(\spthat,\ell)$ in~\eqref{defn:Ecross}, there is some $\Tilde{e}'\in \cE(\ell)\cap E^{\times}(\spthat)$ such that $\Tilde{e}\in E^{\times}(\spthat,\Tilde{e}')$. Now, the definition of $E^{\times}(\spthat,\Tilde{e})$ in~\eqref{eq:defn-Ee} and \cref{lem: paths agree} give us that $E^{\times}(\spthat,\Tilde{e})\subset E^{\times}(\spthat,\Tilde{e}')$. Thus, as this holds for any edge in $\cE(\Bar{\ell})\cap E^{\times}(\spthat)$, we get the desired inclusion in \eqref{eq:ewivfvqweuoi} (recalling once again the definition in~\eqref{defn:Ecross}). So, we conclude that 
    \begin{equation}\label{eq:cond2-ok}
    \text{all the edges of $\ell_i$ are in $E(\spthat, \ell)$ for all $i\in [m]$}.
    \end{equation}

    \medskip
    
    \noindent{\emph{\underline{Proof of (iii):}}}  
    The same proof as above, used to show that $E^{\times}(\spthat,\Bar{\ell})\subset E^{\times}(\spthat,\ell)$, gives us that $E^{\times}(\spthat,\ell_i)\subseteq E^{\times}(\spthat,\Bar{\ell})$ for all $i\in [m]$.
    We claim that for all $i \in [m]$, $E^{\times}(\spthat,\ell_i)$ is strictly contained in $E^{\times}(\spthat,\Bar{\ell})$ because $\Bar{e}\in E^{\times}(\spthat,\Bar{\ell})$ but $\Bar{e}\notin E^{\times}(\spthat,\ell_i)$. 
    The fact that $\Bar{e}\in E^{\times}(\spthat,\Bar{\ell})$ is trivial because $\Bar{e}$ is in $\Bar{\ell}$ and $\bar{e} \in E^\times(\spthat)$ (recall \eqref{eq:fvwievfow}). 
    The fact that 
    \[\Bar{e}\notin E^{\times}(\spthat,\ell_i)=\bigcup_{\Tilde{e}\in \cE(\ell_i) \cap E^{\times}(\spthat)} E^{\times}(\spthat, \Tilde{e})\]
    immediately follows if we show that for all $\Tilde{e}\in\cE(\ell_i)\cap E^\times(\spthat)$,
    \begin{equation}\label{eq:wfeiwefbw}
        \Bar{e}\notin E^{\times}(\spthat, \Tilde{e})=\bigcup_{\mathfrak{p}^{\Tilde{e}}=  \{p^{\Tilde{e}}_i\}_{i=1}^{\infty}\in \mathfrak{P}_{\spthat}(\Tilde{e})}\left\{e'\in E: e'\text{ corresponds to }\mathfrak{e}(p^{\Tilde{e}}_i,p^{\Tilde{e}}_{i+1})\right\}_{i=1}^{j(\mathfrak{p}^{\Tilde{e}})}.
    \end{equation}
    Clearly, \eqref{eq:wfeiwefbw} is true when $\Tilde{e}\in \cE(\ell_i)
    \cap \{(\Bar{e}_1)^{\pm 1}, (\Bar{e}_2)^{\pm 1}, (\Bar{e}_3)^{\pm 1}\}$ as any tree-path of plaquettes $\mathfrak{p}\in \mathfrak{P}_{\spthat}(\Bar{e}_i^{\pm 1})$ must go through $(\Bar{e}_i)^{\pm 1}$ before $\Bar{e}$. To show that the claim is true for the other edges in $\cE(\ell_i)$, suppose that $\Tilde{e}\in \cE(\ell_i)$ but  $ \Tilde{e}\notin\{(\Bar{e}_1)^{\pm 1}, (\Bar{e}_2)^{\pm 1}, (\Bar{e}_3)^{\pm 1}\}$ and, to get a contradiction, assume also that $\Tilde{e}$ is such that $\Bar{e}\in E^{\times}(\spthat, \Tilde{e})$. Note that $\Bar{e}\neq \Tilde{e}$, since $\Bar{e}\notin \cE(\ell_i)$ by the definition of $\cT_{\Bar{e},\Bar{\mathfrak{p}}}(\Bar{\ell},\Bar{K})$, and note also that  $\Tilde{e}\in \Bar{\ell}$, since $n_{\Tilde{e}}(\Bar{\ell})\geq 1$ by \eqref{eq: edge dec} (which we can use with $\Tilde{e}$ because $ \Tilde{e}\notin\{(\Bar{e}_1)^{\pm 1}, (\Bar{e}_2)^{\pm 1}, (\Bar{e}_3)^{\pm 1}\}$). Further, as $\Bar{e}\in E^{\times}(\spthat, \Tilde{e})$ there is some tree path of plaquettes $\mathfrak{p}^{\Tilde{e}}= \{p_i^{\Tilde{e}}\}_{i=1}^{\infty}$ such that $\bar{e}$ corresponds to an edge in $\{\mathfrak{e}(p_i^{\Tilde{e}}, p_{i+1}^{\Tilde{e}})\}_{1\leq i< j(\mathfrak{p}^{\Tilde{e}})}$. Thus, we have that there is a tree-path of plaquettes containing $\Bar{e}$ such that $\Bar{e}$ is not the outermost edge of $\Bar{\ell}$ in the path (because we just argued that $\Bar{e}\neq \Tilde{e}$ and $\Tilde{e}\in\Bar{\ell}$). This contradicts the fact (from \eqref{eq:fvwievfow}) that $\Bar{e}$ is the outermost edge of $\Bar{\ell}$ in every tree-path $\mathfrak{q}\in\mathfrak{P}_{\spthat}(\Bar{e})$, concluding the proof of \eqref{eq:wfeiwefbw}.
    
    Therefore, we get that $\Bar{e}\notin E^{\times}(\spthat,\ell_i)$ and thus the strict containment of $E^{\times}(\spthat,\ell_i)$ in $E^{\times}(\spthat,\Bar{\ell})$.
    This latter fact and the fact that $\Bar{e}$ is not in $\cE(\hat{p})$ (recall \eqref{eq:fvwievfow}) give us the desired property:
    \begin{equation}\label{eq:cond3-ok}
        \text{$n^{\times}_{\hat{p}}(\spthat,\ell_i)<n^{\times}_{\hat{p}}(\spthat,\Bar{\ell})$ for all $i\in [m]$}.
    \end{equation}
    
    \medskip
    
    \noindent{\emph{\underline{Proof of (iv):}}} We claim that there must be some $j\in [m]$ such that $K_j(\hat{p})-K_j\big(\hat{p}^{-1}\big) \neq h_{\ell_j}(\hat{p})$. Indeed, if $K_i(\hat{p})-K_i\big(\hat{p}^{-1}\big) = h_{\ell_i}(\hat{p})$ for all $i\in[m]$, then 
    \begin{multline}\label{eq: k equal h}
        K(\hat{p}) - K(\hat{p}^{-1})= \multK(\hat{p}) - \multK(\hat{p}^{-1}) = \sum_{i\in [m]}\left(K_i(\hat{p})-K_i(\hat{p}^{-1}) \right)= \sum_{i\in [m]}h_{\ell_i}(\hat{p}) =  h_{\Bar{\ell}}(\hat{p}),
    \end{multline}
    where, to obtain the first equality we used that $\hat{p}\notin \Bar{\mathfrak{p}}$ by \cref{obs: ppqe} and Property~\ref{p2 remove plaquette} of \cref{prop: remove plaquette from loop}; for the second equality we used the definition of $\multK(\cdot)$ from \eqref{eq:defn-KK}; for the third equality we used that we are assuming that $K_i(\hat{p})-K_i\big(\hat{p}^{-1}\big) = h_{\ell_i}(\hat{p})$ for all $i\in[m]$; and finally to get the last equality we used Property~\ref{p4 remove plaquette} of \cref{prop: remove plaquette from loop} together with the fact that $\Bar{p}_1\neq (\hat{p})^{\pm 1}$ by \cref{obs: ppqe}. As \eqref{eq: k equal h} contradicts \eqref{eq:K-not-h}, there must be some $j\in [m]$ such that $K_j(\hat{p})-K_j\big(\hat{p}^{-1}\big) \neq h_{\ell_j}(\hat{p})$. 

    \medskip
    
    The four proofs above finish the entire proof of \cref{claim: K1 claim}, and so of \cref{prop: non-zero plaquette assignments one}.
\end{proof}

Next we show that the second condition of the canonical collection of plaquette assignments is necessary to have a non-zero surface sum. 

\begin{prop}[\textsc{necessity of the second condition}]\label{prop: non-zero plaquette assignments two}
    Let $\ell$ be any non-trivial loop and $\cK_{\ell}^2$ denote the collection of finite plaquette assignments such that
    \[\text{if $p,q\in \cP^+$ are in the same plaquette-region of $\ell$, then $K(p)=K(q)$}.\] 
    If $K\notin \cK_{\ell}^2$, then $\phi^K(\ell)=0$.
\end{prop}

\begin{proof}
    First, \cref{lemma: backtrack cancellations} allows us to assume that $\ell$ has no backtracks. Now, if $\ell$ does not have any plaquette-region with at least two plaquettes, then every plaquette assignment is trivially in $\cK_{\ell}^2$. Thus, we assume that $\ell$ has at least one plaquette-region with at least two plaquettes. Further, if $(\ell,K)$ is not balanced, the result is trivial. Thus, we consider $K$ such that
    \begin{equation}\label{eq: K notin K2}
        K\notin K_{\ell}^2\quad \text{and}\quad (\ell,K)\text{ is balanced.}
    \end{equation}
    Since $K\notin  \cK_{\ell}^2$, there exists $\hat{p},\hat{q}\in \cP^+$ in the same plaquette-region of $\ell$ such that \begin{equation}\label{eq: Kp neq Kq}
        K(\hat{p})\neq K(\hat{q}).
    \end{equation}
    Further, by the definition of plaquette-regions, we can assume that $\hat{p}$ and $\hat{q}$ are adjacent and the edge shared by them is not in $\ell$. That is, that
    \begin{equation}\label{eq: ehp hq =0}
        n_{\mathfrak{e}(\hat{p},\hat{q})}(\ell)=0.
    \end{equation}

    Now, let $\hat{\mathfrak{p}}$ be the straight path of plaquettes starting at $\hat{p}$ such that $\hat{q}$ is the second plaquette in $\hat{\mathfrak{p}}$. As $\hat{\mathfrak{p}}$ is just a straight line, we can construct a spanning pair of trees $\spthat=\spt_{\hat{p},\hat{q},\hat{M}}(\hat{\mathfrak{p}})$ avoiding $\hat{p}$ and $\hat{q}$ from $\hat{\mathfrak{p}}$ of size (say) $\hat{M}= 1$. Note that if $n^{\times}_{\hat{p},\hat{q}}(\spthat,\ell)=0$, we have the desired result. This is because if $n^{\times}_{\hat{p},\hat{q}}(\spthat,\ell)=0$, then all the edges of $\ell$ are in $E^{\smallsetminus}(\spthat)\cup\mcl E(\{p,q\})\setminus\{\mathfrak{e}(\hat{p},\hat{q})\}$, thus the fact that  $\ell$ is a non-backtrack loop, \cref{obs: edges in dual tree}, and \cref{obs: tree equals null-2} guarantees that the loop $\ell$ must be either the null-loop or a loop winding around the rectangle formed by the union of $\hat{p}$ and $\hat{q}$ (in one of the two possible directions). Since we assumed that $\ell$ was non-trivial we know it is not the null-loop and thus we must get that we are in the latter case above. But in this case, \cref{lem: size two loops} gives the desired result. Thus, we can assume that
    \begin{equation}\label{eq: more than one edge to remove K1-2}
        n^{\times}_{\hat{p},\hat{q}}(\spthat,\ell)\geq 1.
    \end{equation}

     The following claim will be our key tool to shrink the loop to a smaller loop we know how to deal with. It allows us to iteratively shrink $\ell$, preserving the condition in \eqref{eq: Kp neq Kq} and the fact that $\mathfrak{e}(\hat{p},\hat{q})$ is not in the loop, until the shrunken loop just consists of edges in $E^{\smallsetminus}(\spthat) \cup \cE(\{\hat{p},\hat{q}\})$. This will imply that the shrunken loop, after removing backtracks, will only be supported on the boundary of the rectangle formed by $\hat{p}$ and $\hat{q}$.

    \begin{claim}\label{claim: K2 claim}
         Let $\ell$, $\hat{p}$, $\hat{q}$, $\hat{\mathfrak{p}}$, and $\spthat$ be as above. Suppose that $(\Bar{\ell},\Bar{K})$ is a loop plaquette assignment such that:
        \begin{enumerate}[(i)]
            \item $(\Bar{\ell},\Bar{K})$ is balanced and does not contain backtracks;
            \item all the edges of $\Bar{\ell}$ are contained in $E(\spthat, \ell)$; 
            \item $n^{\times}_{\hat{p},\hat{q}}(\spthat, \Bar{\ell})\geq 1$;
            \item $n_{\mathfrak{e}(\hat{p},\hat{q})}(\Bar{\ell})=0$;
            \item $\Bar{K}(\hat{p})\neq \Bar{K}(\hat{q})$.
        \end{enumerate} 
        There exists an edge $\Bar{e}$ of $\Bar{\ell}$ and a path of plaquettes $\Bar{\mathfrak{p}}$ connecting $\Bar{e}$ to infinity such that the trajectory tree $\cT_{\Bar{e},\Bar{\mathfrak{p}}}(\Bar{\ell},\Bar{K})$ generated from $(\Bar{\ell},\Bar{K},\Bar{e},\Bar{\mathfrak{p}})$ from \cref{defn:tr-tree-gen} has the following property: for all leaves $(s,\multK)=\{(\ell_1,K_2),\dots,(\ell_m,K_m)\}$ of $\cT_{\Bar{e},\Bar{\mathfrak{p}}}(\Bar{\ell},\Bar{K})$, there exists a $j\in [m]$ such that:
        \begin{enumerate}[(i)]
            \item $(\ell_j,K_j)$ is balanced and does not contain backtracks; 
            \item all the edges of $\ell_j$ are contained in $E(\spthat, \ell)$;  
            \item $n^{\times}_{\hat{p},\hat{q}}(\spthat, \ell_j)<n^{\times}_{\hat{p},\hat{q}}(\spthat, \Bar{\ell})$;
            \item $n_{\mathfrak{e}(\hat{p},\hat{q})}(\ell_j)=0$;
            \item $K_j(\hat{p})\neq K_j(\hat{q})$.
        \end{enumerate}
    \end{claim}

    First, assuming the validity of \cref{claim: K2 claim}, we complete the proof of the proposition. As before, by successive applications of \cref{claim: K2 claim}, we obtain that to prove that $\phi^K(\ell)=0$, it suffices to verify that $\phi^{\Tilde{K}}(\Tilde{\ell}) = 0$ for all loop plaquette assignments $(\Tilde{\ell}, \Tilde{K})$ such that 
    \begin{enumerate}[(i)]
         \item $(\Tilde{\ell}, \Tilde{K})$ is balanced and does not contain backtracks;
         \item all the edges of $\Tilde{\ell}$ are in $E(\spthat, \ell)$;
         \item $n^{\times}_{\hat{p},\hat{q}}(\spthat, \Tilde{\ell}) = 0$; 
         \item $n_{\mathfrak{e}(\hat{p},\hat{q})}(\Tilde{\ell})=0$;
         \item $\Tilde{K}(\hat{p})\neq \Tilde{K}(\hat{q})$.
    \end{enumerate}

    Now, assuming the latter five assumptions $(i)$--$(v)$, the facts that all edges of $\Tilde{\ell}$ are in $E(\spthat, \ell)$, that $n_{\mathfrak{e}(\hat{p},\hat{q})}(\Tilde{\ell})=0$, and that  $n^{\times}_{\hat{p},\hat{q}}(\spthat, \Tilde{\ell}) = 0$ imply that $\Tilde{\ell}$ consists only of edges in $E^{\smallsetminus}(\spthat) \cup \cE(\{\hat{p},\hat{q}\})$ excluding the two oriented edges corresponding to $\mathfrak{e}(\hat{p}, \hat{q})$. Thus, since  $\Tilde{\ell}$ is a non-backtrack loop, \cref{obs: edges in dual tree} and \cref{obs: tree equals null-2} guaranty that the loop ${\tilde{\ell}}$ must be either the null-loop or the loop winding around the rectangle formed by the union of $\hat{p}$ and $\hat{q}$ (in one of the two possible directions).
    
    Now, if $\Tilde{\ell}$ is the loop winding around  the rectangle formed by the union of $\hat{p}$ and $\hat{q}$, then \cref{lem: size two loops} gives us that $\phi^{\Tilde{K}}(\Tilde{\ell}) = 0$. If $\Tilde{\ell}$ is the null-loop, then the condition $\Tilde{K}\notin \cK_{\Tilde{\ell}}^2$ implies that $\Tilde{K}\neq 0$ on at least one plaquette (otherwise $\Tilde{K}$ would belong to $\cK_{\Tilde{\ell}}^2$) and so $\phi^{\Tilde{K}}(\Tilde{\ell})=0$ again.

    Therefore, under the assumption of the claim, we conclude that $\phi^K(\ell) = 0$, as desired.

    \medskip

    Now, we prove \cref{claim: K2 claim}. Suppose that $(\Bar{\ell},\Bar{K})$ satisfies all the assumptions of the claim. Then, the first three items of \cref{lem: exists an outer edge} give us that there is an edge $\Bar{e}$ of $\Bar{\ell}$ such that $\Bar{e}$ is in $E^{\times}(\spthat)$ but not in $\cE(\{\hat{p},\hat{q}\})$, and $\Bar{e}$ is the outermost edge of every tree-path $\mathfrak{q}\in\mathfrak{P}_{\spthat}(\Bar{e})$. Further, Items~\ref{item:3a}~and~\ref{item:3b} of \cref{lem: exists an outer edge} give us a path of plaquettes $\Bar{\mathfrak{p}} = \{\Bar{p}_i\}_{i=1}^{\infty}$ connecting $e$ to infinity.
    
    Now, we show that the conclusion of the claim holds for this choice of $\Bar{e}$ and $\Bar{\mathfrak{p}}$. Let 
    $$(s,\multK)=\left\{(\ell_1,K_1),\dots,(\ell_m,K_m)\right\}$$ 
    be a leaf on the trajectory tree generated from $(\Bar{\ell},\Bar{K},\Bar{e},\Bar{\mathfrak{p}})$. The first three properties in the claim follow exactly with the same proof used for \cref{claim: K1 claim}; hence we omit the details.

    Next, since $\Bar{p}_1\notin\{\hat{p}^{\pm1},\hat{q}^{\pm 1}\}$ by \cref{obs: ppqe}, we get from Property~\ref{p5 remove plaquette} of \cref{prop: remove plaquette from loop} that
    \begin{equation}\label{eq:K2cond4-ok}
        \text{$n_{\mathfrak{e}(\hat{p},\hat{q})}(\ell_i)=0$ for all $i\in [m]$}.
    \end{equation}
    Laslty, we claim that there must be some $j\in [m]$ such that $K_j(\hat{p})\neq K_j(\hat{q})$. Indeed, if all $K_i$ were such that $K_i(\hat{p}) = K_i(\hat{q})$, we get that
    \begin{equation}
        \begin{split}\label{eq: kp-kq=0}
        \Bar{K}(\hat{p})-\Bar{K}(\hat{q}) = \multK(\hat{p})-\multK(\hat{q}) = \sum_{i=1}^m \left( K_i(\hat{p})-K_i(\hat{q})\right) = 0,
        \end{split}
    \end{equation}
    where, to get the first equality we used that  $\hat{p}^{\pm 1}, \hat{q}^{\pm 1}\notin \Bar{\mathfrak{p}}$ by \cref{obs: ppqe} and Property~\ref{p2 remove plaquette} of \cref{prop: remove plaquette from loop}; and to get the last equality we used our assumption that $K_i(\hat{p}) = K_i(\hat{q})$ for all $i\in [m]$. As \eqref{eq: kp-kq=0} contradicts the assumption that $\Bar{K}(\hat{p})\neq \Bar{K}(\hat{q})$, there must be some $j\in [m]$ such that $K_j(\hat{p})\neq K_j(\hat{q})$. This finishes the proof of the claim. 
    \end{proof}

Finally, we show that the third condition of the canonical collection of plaquette assignments is necessary to have a non-zero surface sum. Recall the definition of distance $d_{\ell}(\cdot)$ for a loop $\ell$ from \cref{def: lattice distance}.

\begin{prop}[\textsc{necessity of the third condition}]\label{prop: non-zero plaquette assignments three}
    Let $\ell$ be any non-trivial loop and $\cK_{\ell}^3$ the set of plaquette assignments such that for all $p\in \cP^+$,
    \begin{align*}
        K(p) +K(p^{-1})\leq d_{\ell}(p).
    \end{align*}
    If $K\notin \cK_{\ell}^3$, then $\phi^K(\ell)=0$.
\end{prop}

\begin{proof}
    First, \cref{lemma: backtrack cancellations} allows us to assume that $\ell$ has no backtracks.
    Hence, if $\ell$ has no backtracks and is supported on a single unoriented plaquette, we have the desired result by \cref{lem: size one main thm}. 
    
    Further, if $(\ell,K)$ is not balanced, then the result is trivial. Thus, we consider $(\ell,K)$ such that
    \begin{equation}\label{eq: K notin K3}
        \text{$\ell$ has no backtracks}, \qquad \area(\ell)\geq 2,\qquad K\notin \cK_{\ell}^3,\qquad \text{and} \qquad(\ell,K)  \text{ is balanced}.
    \end{equation}
    Since $K\notin K_{\ell}^3$, there exists some plaquette $\hat{p}\in \cP^+$ such that
    \begin{equation}\label{eq: K bigger d}
    K(\hat{p})+K(\hat{p}^{-1}) >d_{\ell}(\hat{p}).
    \end{equation}
    Let $\hat{M}\in \NN$ be any integer such that all the edges of $\ell$ are in $\cE(B_{\hat{M}}(\hat{p}))$. Also, let $\hat{\mathfrak{p}}=\{\hat{p}_i\}_{i=1}^{\infty}$ denote a distance achieving path for $d_\ell(\hat{p})$ such that $\hat{\mathfrak{p}}$ is a straight line outside of $B_{\hat{M}}(\hat{p})$ (note that such a distance achieving path exists). Let $\hat{q}$ denote the second plaquette in $\hat{\mathfrak{p}}$. With this, let $\spthat=\spt_{\hat{p},\hat{q},\hat{M}}(\hat{\mathfrak{p}})$ be the spanning pair of trees avoiding $\hat{p}$ and $\hat{q}$ from $\hat{\mathfrak{p}}$ of size $\hat{M}$. Since $\hat{\mathfrak{p}}$ is a distance achieving path for $d_\ell(\hat{p})$, we have by defnition that
    \begin{equation}\label{eq: K big sum n}
        K(\hat{p})+K(\hat{p}^{-1})>\sum_{i=1}^{\infty}n_{\mathfrak{e}(\hat{p}_i,\hat{p}_{i+1})}(\ell).
    \end{equation}
    Note that, for any loop plaquette assignment $(\Tilde{\ell}, \Tilde{K})$, knowing that $\Tilde{K}(\hat{p})+\Tilde{K}(\hat{p}^{-1})>\sum_{i=1}^{\infty}n_{\mathfrak{e}(\hat{p}_i,\hat{p}_{i+1})}(\Tilde{\ell})$ implies that $\Tilde{K}(\hat{p})+\Tilde{K}(\hat{p}^{-1}) >d_{\Tilde{\ell}}(\hat{p})$. Also, note that since $\area(\ell)\geq 2$, we must have that\begin{equation}\label{eq: more than one edge to remove K3}
        n^{\times}_{\hat{p}}(\spthat,\ell)\geq 1.
    \end{equation}
    
    The following claim will be our key tool to shrink the loop $\ell$ to a smaller loop we know how to deal with. It tells us that we can iteratively shrink the loop $\ell$ until it just consists of edges in $E^{\smallsetminus}(\spthat)\cup \cE(\hat{p})$. Thus, after removing backtracks, the shrunken loop can only be supported on $\hat{p}$. Moreover, while doing this shrinking, we will preserve the condition in \eqref{eq: K big sum n} (and so also in  \eqref{eq: K bigger d} thanks to the comment in the previous paragraph).

    \begin{claim}\label{claim: K3 claim}
        Let $\ell$, $\hat{p}$, $\hat{\mathfrak{p}}$, and $\spthat$ be as above. Suppose that $(\Bar{\ell},\Bar{K})$ is a loop plaquette assignment such that:
        \begin{enumerate}[(i)]
            \item $(\Bar{\ell},\Bar{K})$ is balanced and does not contain backtracks;
            \item all the edges of $\Bar{\ell}$ are contained in $E(\spthat, \ell)$;
            \item $n^{\times}_{\hat{p}}(\spthat, \ell)\geq 1$;
            \item $\Bar{K}(\hat{p})+\Bar{K}(\hat{p}^{-1})>\sum_{i=1}^{\infty}n_{\mathfrak{e}(\hat{p}_i,\hat{p}_{i+1})}(\Bar{\ell})$.
        \end{enumerate} 
        There exists an edge $\Bar{e}$ of $\Bar{\ell}$ and a path of plaquettes $\Bar{\mathfrak{p}}$ connecting $\Bar{e}$ to infinity such that the trajectory tree $\cT_{\Bar{e},\Bar{\mathfrak{p}}}(\Bar{\ell},\Bar{K})$ generated from $(\Bar{\ell},\Bar{K},\Bar{e},\Bar{\mathfrak{p}})$ from \cref{defn:tr-tree-gen} has the following property: For all leaves $(s,\multK)=\{(\ell_1,K_2),\dots,(\ell_m,K_m)\}$ of $\cT_{\Bar{e},\Bar{\mathfrak{p}}}(\Bar{\ell},\Bar{K})$, there exists $j\in [m]$ such that:
        \begin{enumerate}[(i)]
            \item $(\ell_j,K_j)$ is balanced and does not contain backtracks; 
            \item all the edges of $\ell_j$ are contained in $E(\spthat, \ell)$;
            \item $n^{\times}_{\hat{p}}(\spthat, \ell_j)<n^{\times}_{\hat{p}}(\spthat, \Bar{\ell})$;
            \item $K_j(\hat{p})+K_j(\hat{p}^{-1})>\sum_{i=1}^{\infty}n_{\mathfrak{e}(\hat{p}_i,\hat{p}_{i+1})}(\ell_j)$.
        \end{enumerate}
    \end{claim}

    First, assuming the validity of \cref{claim: K3 claim} we finish the proof of the proposition. 
    As before, by successive applications of \cref{claim: K3 claim}, we obtain that to prove that $\phi^K(\ell)=0$, it suffices to verify that $\phi^{\Tilde{K}}(\Tilde{\ell}) = 0$ for all loop plaquette assignments $(\Tilde{\ell},\Tilde{K})$ such that 
    \begin{enumerate}[(i)]
    \item $(\Tilde{\ell},\Tilde{K})$ is balanced and does not contain backtracks;
    \item all the edges of $\Tilde{\ell}$ are contained in $E(\spthat, \ell)$;
    \item$n^{\times}_{\hat{p}}(\spthat, \Tilde{\ell}) = 0$;
    \item $\Tilde{K}(\hat{p})+\Tilde{K}(\hat{p}^{-1})>\sum_{i=1}^{\infty}n_{\mathfrak{e}(\hat{p}_i,\hat{p}_{i+1})}(\Tilde{\ell})$.
    \end{enumerate}

    Now, assuming the latter four assumptions $(i)$--$(iv)$, $\Tilde{\ell}$ having all its edges in $E(\spthat, \ell)$ and $n^{\times}_{\hat{p}}(\spthat, {\tilde{\ell}}) = 0$ imply that the loop ${\tilde{\ell}}$ must be either the null-loop or a loop winding around the plaquette $\hat{p}$ (in one of the two possible directions).
    If $\tilde{\ell}$ is a loop winding around the plaquette $\hat{p}$, note that $\Tilde{K}\notin\cK_{\Tilde{\ell}}^3$ as $\Tilde{K}(\hat{p})+\Tilde{K}(\hat{p}^{-1})>\sum_{i=1}^{\infty}n_{\mathfrak{e}(\hat{p}_i,\hat{p}_{i+1})}(\Tilde{\ell})\geq d_{\Tilde{\ell}}(\hat{p})$, and thus $\phi^{\Tilde{K}}(\Tilde{\ell})=0$ by \cref{lem: size one main thm}, since $\cK_{\Tilde{\ell}}\subset\cK_{\Tilde{\ell}}^3$.
    If $\tilde{\ell}$ is the null-loop, the condition $\Tilde{K}\notin \cK_{\Tilde{\ell}}^3$ implies that $\Tilde{K}\neq 0$ on at least one plaquette. Hence, $\phi^{\Tilde{K}}(\Tilde{\ell})=0$ as usual.
    
    Therefore, assuming the validity of the claim, we conclude that $\phi^K(\ell) = 0$, as desired.

    \medskip

    Now, we prove \cref{claim: K3 claim}. Assume that $(\Bar{\ell},\Bar{K})$ satisfies all the assumptions of the claim. Then, the first three items of \cref{lem: exists an outer edge} give us that there is an edge $\Bar{e}$ of $\Bar{\ell}$ such that $\Bar{e}$ is in $E^{\times}(\spthat)$ but not in $\cE(\hat{p})$, and $\Bar{e}$ is the outermost edge of every tree-path $\mathfrak{q}\in\mathfrak{P}_{\spthat}(\Bar{e})$. Further, Items~\ref{item:3a}~and~\ref{item:3b} of \cref{lem: exists an outer edge} give us a path of plaquettes $\Bar{\mathfrak{p}} = \{\Bar{p}_i\}_{i=1}^{\infty}$ connecting $e$ to infinity.
    
    Now, we show that the conclusion of the claim holds for this choice of $\Bar{e}$ and $\Bar{\mathfrak{p}}$. Let 
    $$(s,\multK)=\left\{(\ell_1,K_1),\dots,(\ell_m,K_m)\right\}$$ 
    be a leaf on the trajectory tree generated from $(\Bar{\ell},\Bar{K},\Bar{e},\Bar{\mathfrak{p}})$.
    The first three properties in the claim follow exactly with the same proof used for \cref{claim: K1 claim}; hence we omit the details.

    Next, we show that (this will allow us to quickly conclude the proof of the claim later)
    \begin{equation}\label{eq: lj leq d l}
         \sum_{j=1}^m\sum_{i=1}^{\infty}n_{\mathfrak{e}(\hat{p}_i,\hat{p}_{i+1})}(\ell_j)\leq \sum_{i=1}^{\infty}n_{\mathfrak{e}(\hat{p}_i,\hat{p}_{i+1})}(\Bar{\ell}).
    \end{equation}
    Since all the terms are non-negative, we can interchange the summations, giving that
    \begin{equation}\label{eq: sum switch}
        \sum_{j=1}^m\sum_{i=1}^{\infty}n_{\mathfrak{e}(\hat{p}_i,\hat{p}_{i+1})}(\ell_j) = \sum_{i=1}^{\infty}\sum_{j=1}^mn_{\mathfrak{e}(\hat{p}_i,\hat{p}_{i+1})}(\ell_j)= \sum_{i=1}^{\infty}n_{\mathfrak{e}(\hat{p}_i,\hat{p}_{i+1})}(s).
    \end{equation}
    Now we split the proof into two cases. Recall the notation $\cI(\cdot)$ from the beginning of \cref{sec: removing a plaquette}.

    \medskip

    \noindent\underline{\emph{Case 1:}} We assume that $\Bar{e}\in \cI(\hat{\mathfrak{p}}\sm\{\hat{p}\})$. This is the situation of Item~\ref{item:3a} in \cref{lem: exists an outer edge}, and so let $\Bar{\mathfrak{p}} = \{\Bar{p}_i\}_{i=1}^{\infty}:=\{\hat{p}_{i+k-1}\}_{i=1}^{\infty}$. Now, writing $\Bar{p}_1=\hat{p}_{\Bar{k}} = \Bar{e}\Bar{e}_1\Bar{e}_2\Bar{e}_3$, we find that 
    \[
    \mathfrak{e}(\hat{p}_{\Bar{k}-1},\hat{p}_{\Bar{k}}) \in \{(\Bar{e}_1)^{\pm 1}, (\Bar{e}_2)^{\pm 1}, (\Bar{e}_3)^{\pm 1}\},
    \quad
    \mathfrak{e}(\hat{p}_{\Bar{k}},\hat{p}_{\Bar{k}+1})=\Bar{e}
    \quad\text{and}\quad
    \mathfrak{e}(\hat{p}_{\Bar{i}},\hat{p}_{\Bar{i}+1})\notin \cE(\Bar{p}_1), \quad\forall i\notin\{\Bar{k}-1,\Bar{k}\},
    \]
    where, for the last claim, we used the fact that our paths of plaquettes are, by assumption, simple; i.e.\ each plaquette is visited at most once.
    So, we get that
\begin{align*}
       \sum_{j=1}^m\sum_{i=1}^{\infty}n_{\mathfrak{e}(\hat{p}_i,\hat{p}_{i+1})}(\ell_j)
       &\stackrel{\eqref{eq: sum switch}}{=} \bigg(\sum_{i=1}^{\Bar{k}-2}n_{\mathfrak{e}(\hat{p}_i,\hat{p}_{i+1})}(s) \bigg) + n_{\mathfrak{e}(\hat{p}_{\Bar{k}-1},\hat{p}_{\Bar{k}})}(s) + n_{\mathfrak{e}(\hat{p}_{\Bar{k}},\hat{p}_{\Bar{k}+1})}(s) + \sum_{i=\Bar{k}+1}^{\infty}n_{\mathfrak{e}(\hat{p}_i,\hat{p}_{i+1})}(s)\\
       &\,\,\,\leq\,\,\, \bigg(\sum_{i=1}^{\Bar{k}-2}n_{\mathfrak{e}(\hat{p}_i,\hat{p}_{i+1})}(\Bar{\ell}) \bigg) + \left(n_{\mathfrak{e}(\hat{p}_{\Bar{k}-1},\hat{p}_{\Bar{k}})}(\Bar{\ell}) +  n_{\mathfrak{e}(\hat{p}_{\Bar{k}},\hat{p}_{\Bar{k}+1})}(\Bar{\ell})\right) + 0 + \sum_{i=\Bar{k}+1}^{\infty}n_{\mathfrak{e}(\hat{p}_i,\hat{p}_{i+1})}(\Bar{\ell}) \\&\,\,\,\,\,=\,\,\,\,\,
       \sum_{i=1}^{\infty}n_{\mathfrak{e}(\hat{p}_i,\hat{p}_{i+1})}(\Bar{\ell}),
    \end{align*}
    where, for the inequality, we applied Property~\ref{p5 remove plaquette} from \cref{prop: remove plaquette from loop} to the first and last summands (since $\mathfrak{e}(\hat{p}_{\Bar{i}},\hat{p}_{\Bar{i}+1})\notin \cE(\Bar{p}_1)$ for all $i\notin\{\Bar{k}-1,\Bar{k}\}$), Property~\ref{p6 remove plaquette} from \cref{prop: remove plaquette from loop} to the second summand (also using also that $\mathfrak{e}(\hat{p}_{\Bar{k}},\hat{p}_{\Bar{k}+1})=\Bar{e}$), and noted that $n_{\mathfrak{e}(\hat{p}_{\Bar{k}},\hat{p}_{\Bar{k}+1})}(s)=n_{\Bar{e}}(s)=0$ by construction of $\cT_{\Bar{e}, \Bar{\mathfrak{p}}}(\Bar{\ell},\Bar{K})$.
    Thus, \eqref{eq: lj leq d l} holds in this case.

    \medskip

    \noindent\underline{\emph{Case 2:}} We assume that $\Bar{e}\notin \cI(\hat{\mathfrak{p}}\sm\{\hat{p}\})$. This is the situation of Item~\ref{item:3b} in \cref{lem: exists an outer edge}, and so we know that $\Bar{p}_1\notin\hat{\mathfrak{p}}$. Therefore, by Property~\ref{p5 remove plaquette} of \cref{prop: remove plaquette from loop}, we get that 
    \begin{align*}
        \sum_{j=1}^m\sum_{i=1}^{\infty}n_{\mathfrak{e}(\hat{p}_i,\hat{p}_{i+1})}(\ell_j)
        \stackrel{\eqref{eq: sum switch}}{=}
        \sum_{i=1}^{\infty} n_{\mathfrak{e}(\hat{p}_i,\hat{p}_{i+1})}(s) \leq \sum_{i=1}^{\infty} n_{\mathfrak{e}(\hat{p}_i,\hat{p}_{i+1})}(\Bar{\ell}),
    \end{align*} 
    Thus, \eqref{eq: lj leq d l} also holds in this second case.

    \medskip

    Finally, we deduce that there is some $j\in [m]$ such that $K_j(\hat{p})+K_j(\hat{p}^{-1})>\sum_{i=1}^{\infty}n_{\mathfrak{e}(\hat{p}_i,\hat{p}_{i+1})}(\ell_j)$. Suppose not, i.e.\  $K_j(\hat{p})+K_j(\hat{p}^{-1})\leq\sum_{i=1}^{\infty}n_{\mathfrak{e}(\hat{p}_i,\hat{p}_{i+1})}(\ell_j)$ for all $j\in [m]$. Then
    \begin{equation}
    \begin{aligned}\label{eq: k less d}
        \Bar{K}(\hat{p})+\Bar{K}(\hat{p}^{-1}) = \multK(\hat{p}) + \multK(\hat{p}^{-1}) &= \sum_{j=1}^m  K_{j}(\hat{p})+K_{j}(\hat{p}^{-1})\\
        &\leq \sum_{j=1}^m\sum_{i=1}^{\infty}n_{\mathfrak{e}(\hat{p}_i,\hat{p}_{i+1})}(\ell_j)\leq \sum_{i=1}^{\infty} n_{\mathfrak{e}(\hat{p}_i,\hat{p}_{i+1})}(\Bar{\ell}),
    \end{aligned}
    \end{equation}
    where, for the first equality, we used that $(\hat{p})^{\pm 1}\notin \Bar{\mathfrak{p}}$ by \cref{obs: ppqe} and Property~\ref{p2 remove plaquette} of \cref{prop: remove plaquette from loop}; for the first inequality we used our assumption $K_j(\hat{p})+K_j(\hat{p}^{-1})\leq\sum_{i=1}^{\infty}n_{\mathfrak{e}(\hat{p}_i,\hat{p}_{i+1})}(\ell_j)$ for all $j\in [m]$; and  for the last inequality we used \eqref{eq: lj leq d l}. 
    
    As \eqref{eq: k less d} contradicts $\Bar{K}(\hat{p})+\Bar{K}(\hat{p}^{-1})>\sum_{i=1}^{\infty}n_{\mathfrak{e}(\hat{p}_i,\hat{p}_{i+1})}(\Bar{\ell})$, there must be some $j\in [m]$ such that $K_j(\hat{p})+K_j(\hat{p}^{-1})>\sum_{i=1}^{\infty}n_{\mathfrak{e}(\hat{p}_i,\hat{p}_{i+1})}(\ell_j)$. This finishes the proof of the claim.
\end{proof}

    \section{Explicit computations for Wilson loop expectation coefficients and  spectral convergence for simple loops}\label{sec: height plaquette assignment coefficent}

    In this section, we explicitly compute (for several cases of loops) the coefficients $c(\ell, K)$ appearing in the formula for the Wilson loop expectation given by \cref{thm: erasable loops have one plaquette assignment}. In particular, in \cref{sec: Wilson loop expectation constant for wound simple loops}, we prove \cref{thm: wound simple loops WLE}, then, in \cref{sec: spec conv}, we deduce that the empirical spectral distribution of $Q_{\ell}$, for $\ell$ a simple loop, converges weakly to an explicit deterministic measure on the unit circle, proving \cref{cor: limiting spectral density},
    and finally, in \cref{sec: Wilson loop expectation constant for erasable loops with three or fewer self-crossings}, we explain how to compute the Wilson loop expectations appearing in \cref{table}.

    Throughout this section, we always assume loops are non-backtrack loops. In particular, when we apply any loop operation, we assume that the resulting new loops have all backtracks removed (recall~\cref{fig-operations}).

    \subsection{Wilson loop expectation coefficients for simple self-winding loops}\label{sec: Wilson loop expectation constant for wound simple loops}
    
    Let $\ell$ be a simple loop with area $a_{\ell}$. The goal of this section is to prove \cref{thm: wound simple loops WLE} by computing $c(\ell^n,K_{\ell^n})$ for any $n\in \NN$.

    First, a fact that says the property of being a simple loop is closed under certain negative deformations. For any edge $e$ of $\ell$, only one of the two plaquettes containing $e$ in the right orientation (i.e.\ in $\cP(e)$) is contained in $\supp(\ell)$. We denote this plaquette by $p(e)$. We say that an edge $e$ of $\ell$ is \textbf{simple-preserving} for $\ell$ if $\ell\ominus_e p(e)^{-1}$ is a simple loop. The following result guarantees that any simple loop has simple-preserving edges.
    
    \begin{lem}\label{lemma: simple loops are closed under shrinking}
        Let $\ell$ be a simple loop with $\area(\ell)\geq 2$. Then, it contains at least two simple-preserving edges.
    \end{lem}

    \begin{proof}
        First, notice that if a plaquette $p$ in the interior of $\ell$ is such that three of its edges are in $\ell$ then clearly $\ell \ominus_e p(e)$, with $e$ any of the three edges, is a simple loop. Similarly, if exactly two consecutive edges of $p$ are in $\ell$ and $\ell$ does not touch the other vertex of $p$, then $\ell \ominus_e p(e)$, with $e$ either of the two edges, is a simple loop. We will call an interior plaquette good if it is one of the two types above.

        \begin{claim*}
            If $\ell$ is a simple loop with $\area(\ell)\geq 2$, then it contains two good plaquettes.
        \end{claim*}

        Notice the claim clearly implies the lemma. To prove the claim we will induct on the area of $\ell$.

        \medskip

        \noindent\underline{\emph{Base case}:} Suppose that $\area(\ell)=2$. Then we have that $\ell$ is a rectangle made up of two plaquettes and clearly both interior plaquettes are good plaquettes.

        \medskip

        \noindent\underline{\emph{Induction step}:} Suppose that $n\coloneqq\area(\ell)\geq 3$. First, notice that $\ell$ must contain at least two plaquettes whose two consecutive edges are contained in $\ell$. Indeed, if this were not the case, it would be impossible for the loop to close back on itself. If these two plaquettes are good, then we are done. So assume that at least one of them, say $p$, is not good. That is, $\ell$ contains the other vertex of $p$. This means that $\ell$ must have the form detailed on the left-hand side of \cref{fig: SimpleLoopShrinkingProof}. Further, as detailed on the right-hand side of \cref{fig: SimpleLoopShrinkingProof}, notice that removing the two edges on $p$ in $\ell$ and replacing them by the two missing edges on $p$ produces two simple loops $\ell_1$ and $\ell_2$. In particular, we must have that $1\leq \area(\ell_1)<n$ and $1\leq \area(\ell_2)<n$. 
        
        Now, if $\area(\ell_1)=1$ (resp.\ $\area(\ell_2)=1$) then the plaquette corresponding to $\ell_1$ (resp.\ $\ell_2$) must have been a good plaquette for $\ell$. Further, if $\area(\ell_1)\geq 2$ (resp.\ $\area(\ell_2)\geq 2$) then the induction hypothesis gives us that $\ell_1$ (resp.\ $\ell_2$) contains two good plaquettes. Now only one of these good plaquettes can be adjacent to $p$. Thus, at least one of the good plaquettes for $\ell_1$ (resp.\ $\ell_2$) has to be a good plaquette for $\ell$. Thus $\ell_1$ and $\ell_2$ each contain at least one plaquette that is a good plaquette for $\ell$, thus giving the claim.
    \end{proof}

    \begin{figure}[ht!]
		\begin{center}
			\includegraphics[width=.65\textwidth]{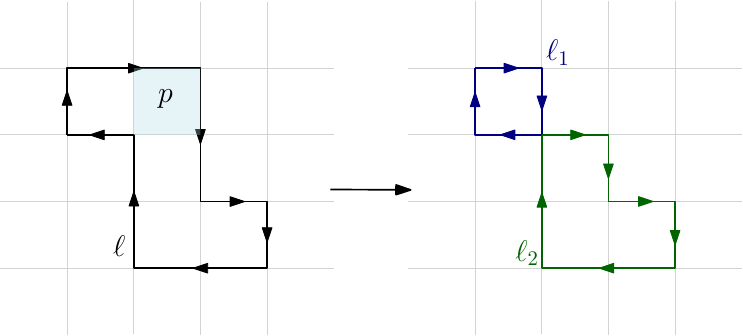}  
			\caption{\label{fig: SimpleLoopShrinkingProof} Any loop $\ell$ must contain at least two plaquettes whose two consecutive edges are contained in $\ell$. The left-hand side of the figure shows a loop $\ell$ where such a plaquette $p$ is not a good plaquette. The right-hand side illustrates how to construct two simple loops $\ell_1$ and $\ell_2$ from $\ell$ by replacing the edges of $p$ in $\ell$ by those that are not in $\ell$.}
		\end{center}
		\vspace{-3ex}
	\end{figure}

    Next, we introduce a special case of the coefficients master loop equation \eqref{eq:mle-2}. To do this, we need some notation. Let $e$ be a simple-preserving edge of $\ell$ and set $p=p(e)^{-1}$. For $0\leq m\leq n$, let \begin{align*}
        \ell^n\ominus_e^m p := \ell^n\ominus_{\mathbf{e}_1}p \ominus_{\mathbf{e}_2}p\hdots \ominus_{\mathbf{e}_m}p,  
    \end{align*}
    where the $\mathbf{e}_i$s are $m$ distinct copies of the edge $e$ in $\ell^n$, and $\ell^n\ominus_e^0 p :=\ell^n$. Further, for a general loop $\ell'$, edge $e'$ contained in $\ell'$, $p' := p(e')^{-1}$, and $0\leq m\leq n$, let $K^{e',\ell'}_{n,m}$ be the plaquette assignment such that
    \begin{equation}\label{eq: K e l n m}
        K^{e',\ell'}_{n,m}(p') = m\quad\text{and}\quad  K^{e',\ell'}_{n,m}(q) = K_{(\ell')^n}(q) \text{ for all $q\in \cP\sm\{p'\}$.}
    \end{equation}
    With this, we give a recursive relation for these coefficients (coming from the master loop equation).

    \begin{lem}\label{lem:coefficent-MLE}
         Suppose $\ell$ is a simple loop, $e$ is a simple-preserving edge of $\ell$, and $p=p(e)^{-1}$. Then, for $1\leq m\leq n$,\begin{equation}\label{eq: coeffienct MLE full}
             c(\ell^n\ominus_e^{n-m} p, K^{e,\ell}_{n,m}) =  c(\ell^n\ominus_e^{n-m+1} p, K^{e,\ell}_{n,m-1})- \sum_{i=1}^{m-1} c(\ell^i, K_{\ell^i}) c(\ell^{n-i}\ominus_e^{n-m} p, K^{e,\ell}_{n-i,m-i}).
         \end{equation}
    \end{lem}
    \begin{proof}
        By definition, $\ell^n\ominus_e^{n-m} p$ contains $m$ copies of $e$. Let $\mathbf{e}$ denote one of these copies. Then, applying the master loop equation in \eqref{eq:mle-2} at $\mathbf{e}$, we get that 
        \begin{align*}
            c(\ell^n\ominus_e^{n-m} p, K^{e,\ell}_{n,m}) =  c(\ell^n\ominus_e^{n-m+1} p, K^{e,\ell}_{n,m-1})- \sum_{i=1}^{m-1} c(\ell^i, K_{\ell^i}) c(\ell^{n-i}\ominus_e^{n-m} p, K^{e,\ell}_{n-i,m-i}),
        \end{align*}
        where we used that only negative deformations and positive splittings are allowed, due to the fact that $e$ only appears in one orientation as $\ell$ is a simple loop, together with the structure of  $\ell^n\ominus_e^{n-m} p$ and $K^e_{n,m}$.
    \end{proof}
    
    Next, we introduce some simplified notation. Let $a_{\ell}$ denote the area of the loop $\ell$. Next, let\begin{align*}
        \Tilde{c}_n(a_{\ell}):= c(\ell^n,K_{\ell^n}).
    \end{align*}Next, fix a simple-preserving edge $e$ of $\ell$ (which exists by \cref{lemma: simple loops are closed under shrinking}) and let $p=p(e)^{-1}$. Then let\begin{align*}
        \Tilde{c}_n(a_{\ell}-1):= c(\ell^n \ominus^n_e p ,K^{e,\ell}_{n,0}) = c(\ell^n_{a-1}, K_{\ell^n_{a_{\ell}-1}}),
    \end{align*}where $\ell_{a_{\ell}-1} = \ell\ominus_e p$. Now as $\ell_{a_{\ell}-1}$ is a simple loop, we can use the same method to define $\ell^n_{a_{\ell}-2}$ and $\Tilde{c}_n(a_{\ell}-2)$. More generally, for $1\leq a\leq a_{\ell}$, we can iterate the above process to define $\ell_{a}^n$ and $\Tilde{c}_n(a)$. Further, for $1< a\leq a_{\ell}$, let $e_a$ denote the simple-preserving edge of $\ell_a$ that is used to define $\ell_{a-1}$ and let $p_a = p(e_a)^{-1}$ (i.e.\ $e_{a_{\ell}} = e$ and $p_{a_{\ell}} = p$). Lastly, for $1\leq a\leq a_{\ell}$ and $0\leq m\leq n$, let\begin{equation}\label{eq: two height WLE coeffiencte}
        \Tilde{c}_{n,m}(a):= c(\ell_a^n\ominus_{e_a}^{n-m} p_a, K^{e_a,\ell_a}_{n,m}).
    \end{equation}Now, we introduce a few facts for future reference. For all of these facts, we assume that $1\leq a\leq a_{\ell}$. First, note that  
    \begin{equation}\label{eq:nn=n}
         \Tilde{c}_{n,n}(a) =  \Tilde{c}_n(a).
    \end{equation}
    Also notice, \begin{equation}\label{eq: two height coef with zero}
         \Tilde{c}_{n,0}(a) =  \Tilde{c}_{n}(a-1),
    \end{equation}
    by definition. Next, note that $\ell_1^n$ must be a plaquette wrapped around $n$ times and thus we know that \begin{equation}\label{eq: wound plaq coef}
         \Tilde{c}_n(1) = \mathbbm{1}_{\{n=1\}},
    \end{equation}
    by \cref{lem: size one main thm}, \cref{lemma: wound plaquettes zero weight}, and \eqref{eq: one pocket coefficent}. Lastly, we rewrite \eqref{eq: coeffienct MLE full} in terms of our new simplified coefficients, for $0\leq m\leq n$,
    \begin{equation}\label{eq: coeffienct MLE}
            \Tilde{c}_{n,m}(a) =  \Tilde{c}_{n,m-1}(a)- \sum_{i=1}^{m-1} \Tilde{c}_{i,i}(a) \Tilde{c}_{n-i,m-i}(a).
    \end{equation}

    Next, we define for all $1\leq a\leq a_{\ell}$ the following two formal power series in the variable $t$:
    \begin{align*}
         \Tilde{C}(a,t)&:= \sum_{n\geq 0} \Tilde{c}_n(a)t^n,\quad \text{where $ \Tilde{c}_0(a):=1$ and $ \Tilde{c}_n(a)$ is as above,} \\C(a,t)&:= \sum_{n\geq 0}c_n(a)t^n,\quad \text{where $c_0(a):=1$ and $c_n(a):= \frac{(-1)^{n+1}}{n}\binom{na-2}{n-1}$ for $n\in\NN$.}
    \end{align*}

    Note that if we can show that for all $1\leq a\leq a_{\ell}$, the formal power series  $C(a,t) $ and $\Tilde{C}(a,t)$ are equal, then we clearly get that 
    \begin{equation}\label{eq:coef-eq}
        \Tilde{c}_n(a)=c_n(a)  \qquad\text{for all } n\geq 0,
    \end{equation}
    proving \cref{thm: wound simple loops WLE} (equality of coefficients is the definition of equality for formal power series). In particular, \eqref{eq:coef-eq} guarantees that the coefficients $\Tilde{c}_n(a)$ do not depend on the shape of the loop $\ell_a$ but only on its area. Thus, \cref{thm: wound simple loops WLE} follows from the following three propositions.

    \begin{prop}\label{prop: unique sol to moment equation}
        For all $a\in \NN$, there exists a unique formal power series $F(a,t)$ that solves the equation \begin{equation}\label{eq: moment eq}
            F(a,t)^a-F(a,t)^{a-1}=t,
        \end{equation}such that $F(a,0)=1$.
    \end{prop}

    \begin{prop}\label{prop: C solves moment equation}
        For all $1\leq a\leq a_{\ell}$, $ \Tilde{C}(a,0) =1$ and $ \Tilde{C}(a,t)$ solves \eqref{eq: moment eq}.
    \end{prop}

    \begin{prop}\label{prop: TilC solves moment equation}
        For all $1\leq a\leq a_{\ell}$, $C(a,0) =1$ and $C(a,t)$ solves \eqref{eq: moment eq}.
    \end{prop}

The rest of this section is devoted to the proof of these three results.

\begin{proof}[Proof of \cref{prop: unique sol to moment equation}]
        Fix any positive integer $a\in\NN$. Define the functional
        \begin{equation*}
            \mathcal{F}_a(z,t)\coloneqq (z+1)^{a}-(z+1)^{a-1}-t.
        \end{equation*}
        We want to apply the formal power series implicit function theorem (see for instance \cite[Proposition 3.1 (a)]{sokal2009ridiculously}). First, note that\begin{align*}
            \cF_a(0,0) = 1^a - 1^{a-1} = 0.
        \end{align*}Next, note that
        \begin{equation*}
            \frac{\partial\mathcal{F}_a}{\partial z}(z,t) = a(z+1)^{a-1} - (a-1)(z+1)^{a-2}.
        \end{equation*}Thus at $(0,0)$ we have that 
        \begin{equation*}
            \frac{\partial\mathcal{F}_a}{\partial z}(0,0)= a\cdot 1^{a-1} - (a-1) \cdot 1^{a-2} = 1 \neq 0.
        \end{equation*}
        Thus, the formal power series implicit function theorem gives us that there is a unique formal power series $G(a,t)$ in the variable $t$ such that $G(a,0)=0$ and
        \begin{align*}
            \cF_a(G(a,t),t) = 0.
        \end{align*}Thus we have that there is a unique formal power series $F(a,t) := G(a,t) + 1$ with constant term equal to one solving \eqref{eq: moment eq} as desired.
    \end{proof}

    \begin{proof}[Proof of \cref{prop: C solves moment equation}]
        We induct on $1\leq a\leq a_{\ell}$. For $a=1$, we get that 
        \begin{align*}
             \Tilde{C}(1,t) = \sum_{n\geq 0} \Tilde{c}_n(1)t^n = 1+t,
        \end{align*}
        where to get the second inequality we used that $ \Tilde{c}_0(a)=1$ by definition, and \eqref{eq: wound plaq coef}. Thus,
        \begin{align*}
             \Tilde{C}(1,t) -  \Tilde{C}(1,t)^0 = 1+t-1=t,
        \end{align*}as desired.

        Fix $1\leq a\leq a_{\ell}$. To prove the inductive step, we introduce another auxiliary generating function. For $m\geq 0$, define
        \begin{align*}
             \Tilde{C}_{m}(a,t) :=\sum_{n\geq 0} \Tilde{c}_{n+m,n}(a)t^n.
        \end{align*}
        Note that $\Tilde{C}_{0}(a,t)= \Tilde{C}(a,t)$ since $\Tilde{c}_{n,n}(a) =  \Tilde{c}_n(a)$ by \eqref{eq:nn=n}.
        \begin{claim}\label{claim: C proof}
            Fix $1\leq a\leq a_{\ell}$. For all $m\geq 0$,
            \begin{equation}\label{eq: Cm eq}
                t \cdot \Tilde{C}_{m+1}(a,t) =  \Tilde{C}(a,t)\Big[ \Tilde{C}_m(a,t)-  \Tilde{c}_{m}(a-1)\Big].
            \end{equation}
        \end{claim}
        Assuming the claim, we first finish the proof of the proposition. We introduce an additional auxiliary generating function. 
        For $z\in\RR$, define
        \begin{align*}
             \Tilde{C}(a,t,z) := \sum_{m\geq 0} \Tilde{C}_m(a,t)z^m.
        \end{align*}
        Notice
        \begin{align*}
            \frac{t}{z}\left( \Tilde{C}(a,t,z)- \Tilde{C}(a,t)\right) &= t\sum_{m\geq 0} \Tilde{C}_{m+1}(a,t)z^m\\&=  \Tilde{C}(a,t)\left[\sum_{m\geq 0} \Tilde{C}_{m}(a,t)z^m-\sum_{m\geq 0} \Tilde{c}_{m}(a-1)z^m\right]\\&=  \Tilde{C}(a,t)\Big[ \Tilde{C}(a,t,z)-  \Tilde{C}(a-1,z)\Big],
        \end{align*}
        where for the second equality we used \eqref{eq: Cm eq}. Now, setting $z=t/ \Tilde{C}(a,t)$ in the above equality, we obtain
    \begin{align*}
             \Tilde{C}(a,t)\cdot  \Tilde{C}\left(a,t,\frac{t}{ \Tilde{C}(a,t)}\right)- \Tilde{C}(a,t)^2 =  \Tilde{C}(a,t)\cdot  \Tilde{C}\left(a,t,\frac{t}{ \Tilde{C}(a,t)}\right)-  \Tilde{C}(a,t)\cdot  \Tilde{C}\left(a-1,\frac{t}{ \Tilde{C}(a,t)}\right)
        \end{align*}
        and thus
        \begin{equation}\label{eq: C area eq}
             \Tilde{C}(a,t) =  \Tilde{C}\left(a-1,\frac{t}{ \Tilde{C}(a,t)}\right).
        \end{equation}
        Therefore, by the inductive hypothesis, we can write
        \begin{align*}
             \Tilde{C}(a,t)^{a-1}-  \Tilde{C}(a,t)^{a-2} =  \Tilde{C}\left(a-1,\frac{t}{ \Tilde{C}(a,t)}\right)^{a-1}-  \Tilde{C}\left(a-1,\frac{t}{ \Tilde{C}(a,t)}\right)^{a-2} = \frac{t}{ \Tilde{C}(a,t)},
        \end{align*}
        where the first equality follows from \eqref{eq: C area eq} and the second equality uses the inductive hypothesis. Rearranging the terms, we get that $\Tilde{C}(a,t)$ solves \eqref{eq: moment eq}, as desired.

        \medskip

        \noindent Finally, we prove \cref{claim: C proof}. \eqref{eq: coeffienct MLE} gives us that, for $n\geq 1$,
        \begin{align*}
             \Tilde{c}_{n+m,n}(a) &=  \Tilde{c}_{n+m,n-1}(a) - \sum_{i=1}^{n-1} \Tilde{c}_{i,i}(a) \Tilde{c}_{n+m-i,n-i}(a)\\&=  \Tilde{c}_{n+m,n-1}(a)- \sum_{i=0}^n \Tilde{c}_{i,i}(a) \Tilde{c}_{n+m-i,n-i}(a) +  \Tilde{c}_{n,n}(a) \Tilde{c}_{m,0}(a) +  \Tilde{c}_{0,0}(a) \Tilde{c}_{n+m,n}(x) \\&=  \Tilde{c}_{n+m,n-1}(a)- \sum_{i=0}^n \Tilde{c}_{i,i}(a) \Tilde{c}_{n+m-i,n-i}(a) +  \Tilde{c}_{n,n}(a) \Tilde{c}_{m,0}(a) +  \Tilde{c}_{n+m,n}(a),
        \end{align*}
        where to get the last equality we used that $ \Tilde{c}_{0,0}(a)= \Tilde{c}_0(a)=1$ by definition. Rearranging terms, we get that
        \begin{align*}
             \Tilde{c}_{n+m,n-1}(a) = \sum_{i=0}^n \Tilde{c}_{i,i}(a) \Tilde{c}_{n+m-i,n-i}(a) -  \Tilde{c}_{n,n}(a) \Tilde{c}_{m,0}(a).
        \end{align*}
        Multiplying by $t^n$ and summing over $n\geq 1$, we get that 
        \begin{align*}
            t\sum_{n\geq 1} \Tilde{c}_{(n-1)+(m+1),n-1}(a)t^{n-1} &= \sum_{n\geq 1}\left(\sum_{i=0}^n \Tilde{c}_{i,i}(a) \Tilde{c}_{n+m-i,n-i}(a)\right)t^n -  \Tilde{c}_{m,0}(a)\sum_{n\geq 1} \Tilde{c}_{n,n}(a)t^n \\
            &= \left(\sum_{n\geq 0} \Tilde{c}_{n,n}(a)t^n\right)\left(\sum_{n\geq 0} \Tilde{c}_{n+m,n}(a)t^n\right) -  \Tilde{c}_{0,0}(a) \Tilde{c}_{m,0}(a)\\
            &\qquad-  \Tilde{c}_{m,0}(a)\sum_{n\geq 1} \Tilde{c}_{n,n}(a)t^n \\
            &=  \Tilde{C}(a,t) \Tilde{C}_m(a,t) -  \Tilde{c}_{m}(a-1) -  \Tilde{c}_{m}(a-1)\sum_{n\geq 1} \Tilde{c}_{n,n}(a)t^n \\
            &=  \Tilde{C}(a,t) \Tilde{C}_m(a,t) -  \Tilde{c}_{m}(a-1) \Tilde{C}(a,t),
        \end{align*}
        where to get the third equality we used \eqref{eq: two height coef with zero} and that $\Tilde{c}_{0,0}(a)= \Tilde{c}_0(a)=1$ by definition. 
        Noting that the left-hand side of the last display is exactly $t \Tilde{C}_{m+1}(a,t)$, we get the claim.
    \end{proof}

    \begin{proof}[Proof of \cref{prop: TilC solves moment equation}]
        We will use the Lagrange inversion theorem to solve \eqref{eq: moment eq}. Let $\cF_{a}(F) \coloneqq F^{a}-F^{a-1}$. Then \eqref{eq: moment eq} is just $\cF_{a}(F) =t$. Notice that\begin{align*}
            \frac{\partial \cF_{a}}{\partial F}(1) = a \cdot 1^{a-1}-(a-1) \cdot 1^{a-2} = 1\neq 0.
        \end{align*}
        Thus applying the Lagrange inversion theorem at $F=1$ gives us that\begin{align*}
            F(a,t) = 1 +\sum_{n=1}^{\infty}g_n(a)\frac{t^n}{n!}
        \end{align*}where\begin{align*}
            g_n(a) &= \lim_{F\to 1}\frac{d^{n-1}}{dF^{n-1}}\left[\left(\frac{F-1}{F^{a}-F^{a-1}}\right)^n\right] \\&= \lim_{F\to 1}\frac{d^{n-1}}{dF^{n-1}}\left(F^{n(1-a)}\right)\\&= (n-na)(n-na-1)\hdots(n-na-n+2)\\&= (-1)^{n+1}(na-2)\hdots(na-n).
        \end{align*}
        Hence,
        \begin{align*}
            F(a,t) = 1 + \sum_{n\geq 1}\left(\frac{(-1)^{n+1}}{n!}\frac{(na-2)!}{(n(a-1)-1)!}\right)t^n = 1 + \sum_{n\geq 1}\left(\frac{(-1)^{n+1}}{n}\binom{na-2}{n-1}\right)t^n = C(a,t).
        \end{align*}Thus $C(a,t)$ solves \eqref{eq: moment eq}, as desired.
    \end{proof}

    \subsection{Convergence of the empirical spectral distribution for simple loops}\label{sec: spec conv}

    In this section, we show that the empirical spectral distribution of $Q_{\ell}$, for $\ell$ a simple loop, converges weakly to a deterministic measure on the unit circle that only depends on the area of $\ell$. In particular, we prove \cref{cor: limiting spectral density}.
    
    First, we give the missing proof of \cref{lem:cov-series}.

    \begin{proof}[Proof of \cref{lem:cov-series}]
        The sum in \eqref{eq: mua} is a finite sum when $a=1$.
        Hence we assume that $a \geq 2$ and $|\upbeta| \leq 1/2$. For the first claim, it suffices to show that the sum
        $\sum_{n=1}^{\infty} \left| c_n(a) \right| \cdot |\upbeta|^{na}$
        converges.
        Using standard asymptotic estimates for binomial coefficients, we have that
        \[
        \binom{na - 2}{n - 1} \sim \frac{1}{\sqrt{2 \pi}}\frac{(a-1)^{1/2 + n - na} a^{-(3/2) + na}}{\sqrt{n}}, \quad \text{as } n\to\infty.
        \]
        Hence, for $n$ large enough, we get that
        \[
        |c_n(a)| \cdot |\upbeta|^{na} \leq \frac{1}{n} \cdot \binom{na - 2}{n - 1} \cdot |\upbeta|^{na} \leq \frac{1}{n} \cdot \frac{(a-1)^{ n } (\frac{a}{a-1})^{na} }{\sqrt{n}} \cdot |\upbeta|^{na} \leq n^{-\frac{3}{2}}\left( (a-1)\left(\frac{a}{a-1}\right)^{a}|\upbeta|^a \right)^{ n }.
        \]
        Since $a \geq 2$ and $|\upbeta| \leq 1/2$, we have that $(a-1) \left(\frac{a}{a-1}\right)^{a}|\upbeta|^a \leq 1$. Hence, the sum converges absolutely.
    \end{proof}

    Recall that $\mu_{\ell_a,N}$ denotes the empirical spectral distribution of $Q_{\ell_a}$, i.e.\ 
    \[\mu_{\ell_a,N}(x)=\frac{1}{N}\sum_{i=1}^N \delta_{\lambda_{\ell_a,N}^i}(x),\] 
    where $\lambda_{\ell_a,N}^i$ are the complex-valued unitary eigenvalues of the unitary matrix $Q_{\ell_a}$. Also, let  $\nu_{\ell_a,N}$ denote the corresponding \emph{expected} empirical spectral measure, that is,
    \begin{align*}
        \nu_{\ell_a,N}(\cdot) = \E{\mu_{\ell_a,N}(\cdot)}.
    \end{align*}
    Equivalently, for any continuous function $f:\SS^1\to\RR$,
    \begin{equation}\label{eq: nu}
        \int f \ d \nu_{\ell_a,N} = \E{\int f\ d\mu_{\ell_a,N}} = \E{\frac{1}{N}\sum_{i=1}^Nf(\lambda_{\ell_a,N}^i)}.
    \end{equation}

\begin{proof}[Proof of \cref{cor: limiting spectral density}]
    We begin with some general considerations. By elementary Fourier theory, we may define the following metric on the space of finite real-valued measures on $\TT$:
    \[ \rho(\mu, \nu) := \sum_{n \in \NN} 2^{-n} \bigg|\int z^n d\mu - \int z^n d\nu\bigg|. \]
    Note that the above sum is absolutely convergent since $\mu, \nu$ are finite measures and $z \mapsto z^n$ is a bounded function on $\TT$. When restricted to probability measures, $\rho$ metrizes the weak convergence. Moreover, the set of probability measures forms a closed subset with respect to this topology. Thus, if $\{\mu_n\}_{n\geq 1}$ is a sequence of probability measures and $\mu$ is a finite real-valued measure such that $\rho(\mu_n, \mu) \rightarrow 0$, then $\mu$ is also a probability measure. Also, we see that in order for $\mu_n$ to converge weakly to $\mu$, it is enough that $\int z^k d\mu_n \rightarrow \int z^k d\mu$ for all $k \in \NN$, as $n \rightarrow \infty$.
    
    Next, as $\nu_{\ell_a,N}$ is a probability measure on the $\TT$, we know that there is some real-valued random variable $\theta_{\ell_a,N}$ such that $\mathrm{e}^{i \theta_{\ell_a,N}}$ has distribution $\nu_{\ell_a,N}$. By \cref{thm: sum over surfaces representation in 't hooft limit} (here is where we use the assumption that $|\upbeta|\leq \upbeta_0$) and \cref{thm: wound simple loops WLE}, we know the asymptotic behavior of the moments of $\nu_{\ell_a,N}$, or equivalently, the asymptotic behavior of the characteristic function of $\theta_{\ell_a,N}$ evaluated at the integer values. That is, for $n\in \NN$, we have that 
    \begin{equation}\label{eq: spec con pos}
        \E{\mathrm{e}^{in\theta_{\ell_a,N}}} = \int z^n\ d\nu_{\ell_a,N} = \E{\frac{1}{N}\sum_{i=1}^N(\lambda_{\ell_a,N}^i)^n} = \E{\tr(Q_{\ell_a}^n)} \xrightarrow[N\to\infty]{}c_n(a)\upbeta^{na},
    \end{equation}
    where we recall that $c_n(a) :=  \frac{(-1)^{n+1}}{n}\binom{na-2}{n-1}$. Further, as the conjugate expected empirical spectral measure is equal to the empirical spectral measure of $\mu_{\ell_a^{-1},N}$ and $\ell_a^{-1}$ is still a simple loop of area $a$, the above reasoning gives that for all $n\in \NN$, \begin{equation}\label{eq: spec con neg}
        \E{\mathrm{e}^{-in\theta_{\ell_a,N}}}= \E{\tr(Q_{\ell_a^{-1}}^n)}\xrightarrow[N\to\infty]{}c_n(a)\upbeta^{na}.
    \end{equation}
    Also,  if $n=0$ we clearly have that \begin{equation}\label{eq: spec con zero}
        \E{\mathrm{e}^{in\theta_{\ell_a,N}}} = 1 \xrightarrow[N\to\infty]{} 1 =c_0(a).
    \end{equation}
    Note, \eqref{eq: spec con pos}, \eqref{eq: spec con neg}, and \eqref{eq: spec con zero} give us the asymptotic behavior of the Fourier coefficients of $\nu_{\ell_\alpha, N}$. Thus, as the measure $\mu_a$  with density (the sum pointwise converges by \cref{lem:cov-series} and the discussion above it)
    \begin{align*}
         f_a(x)
         = 
         \frac{1}{2\pi}\left(1 + \sum_{n\geq 1}2c_n(a)\upbeta^{na}\cos(nx)\right)
         =
         \frac{1}{2\pi} + \frac{1}{2\pi}\sum_{n\geq 1}c_n(a)\upbeta^{na}\left(\mathrm{e}^{inx} + \mathrm{e}^{-inx}\right) 
    \end{align*}
    has the matching Fourier coefficients, we conclude that $\nu_{\ell_a,N}$ converges weakly to $\mu_a$. Further, by the discussion at the beginning of the proof, this gives that $\mu_a$ is a probability measure. Thus, we obtain that $f_{\alpha,\upbeta}$ is non-negative.
    
    To get the claimed weak convergence in probability for $\mu_{\ell_a,N}$, simply note that the factorization as $N\to\infty$ given by \cref{thm: sum over surfaces representation in 't hooft limit} implies that for all $n\in\NN$,
    \begin{align*}
        \Var_{\mu_{\Lambda}, N, \upbeta}(\tr(Q_{\ell_a}^n)) = \phi_{\Lambda, N,\upbeta}(\{\ell_a,\ell_a\}) -\phi_{\Lambda, N,\upbeta}(\ell_a)^2  \xrightarrow{\: N \to \infty \: }   \phi(\ell_a)^2- \phi(\ell_a)^2 = 0.
    \end{align*}
    By Chebyshev's inequality, we have that $\rho(\mu_{\ell_a, N}, \mu_a) \stackrel{p}{\rightarrow} 0$, thus giving the desired weak convergence in probability.
\end{proof}

\subsection{Wilson loop expectations for loops with three or fewer self-crossings}\label{sec: Wilson loop expectation constant for erasable loops with three or fewer self-crossings}

In this section, we explicitly compute the coefficients $c(\ell, K)$, defined in \eqref{eq:coefficents}, which appear in the Wilson loop expectation for lattice loops isotopic to some of the continuum loops listed in \cref{table}. (All the loops not explicitly treated in this section can be analyzed using similar arguments, and so we omit the details.) 
    
We recall that throughout this section, we are assuming that lattice loops are non-backtrack loops. In particular, when we apply any loop operation, we assume that the resulting new loops have all backtracks removed.

Moreover, in this section, given any lattice loop $\ell$ with its corresponding height plaquette assignment $K_{\ell}$ from \eqref{defn:master-plaq-ass}, we set
\[c(\ell) \coloneqq c(\ell, K_{\ell}).\]
If $\ell$ is a lattice loop isotopic to a continuum loop $\cL$, and $R$ is a region of $\cL$, we also denote by $R$ the corresponding plaquette-region of $\ell$.

    \medskip

    The three main tools that we use to compute the Wilson loop expectations appearing in \cref{table} are:
    \begin{itemize}
        \item \cref{thm: erasable loops have one plaquette assignment}, describing exactly what plaquette assignments may yield a non-zero surface sum (and actually we will see that for all lattice loops treated in this section, the plaquette assignments in the canonical collection will all give a non-zero surface sum).
        \item The master loop equation in \cref{thm: fixed K 't Hooft master loop equation for surface sum}; or more precisesly its version stated in \eqref{eq:mle-2}.
        \item \cref{thm: wound simple loops WLE}, computing exactly the coefficents $c(\ell^n)=c(\ell^n, K_{\ell^n})$ for all simple loops $\ell$.
    \end{itemize}

    \subsubsection{Example 1}\label{sect:loop1}

    We start with the first continuum loop $\cL$ in \cref{table}, i.e.\ the one isotopic to a circle. Let $\ell$ be any lattice loop isotopic to $\cL_0$. Then $\ell$ must be a simple loop. Let $s$ be the area of $\ell$. Then, \cref{thm: wound simple loops WLE} gives us that 
    \[\phi(\ell)=\upbeta^s,\]
    as claimed in \cref{table}.

    \subsubsection{Example 2}\label{sect:loop2}

    We now move to the continuum loop $\cL$ in the top left corner of~\cref{fig: Ex1_loop}, labeling each region as shown in the figure. Let $\ell$ be any lattice loop isotopic to $\cL$ (an example is shown on the top right corner of \cref{fig: Ex1_loop}). Because of the topology of $\cL$, note that $\ell$ must have its height equal to its distance. Thus, recalling \cref{rem:card-set-cK}, we must have that the canonical collection of plaquette assignments associated with $\ell$ is just $\cK_{\ell} = \{K_{\ell}\}$. Therefore, thanks to \cref{thm: erasable loops have one plaquette assignment}, to compute $\phi(\ell)$, there is only one coefficient to compute, that is, $c(\ell)$. 

    \begin{figure}[ht!]
		\begin{center}
			\includegraphics[width=.6\textwidth]{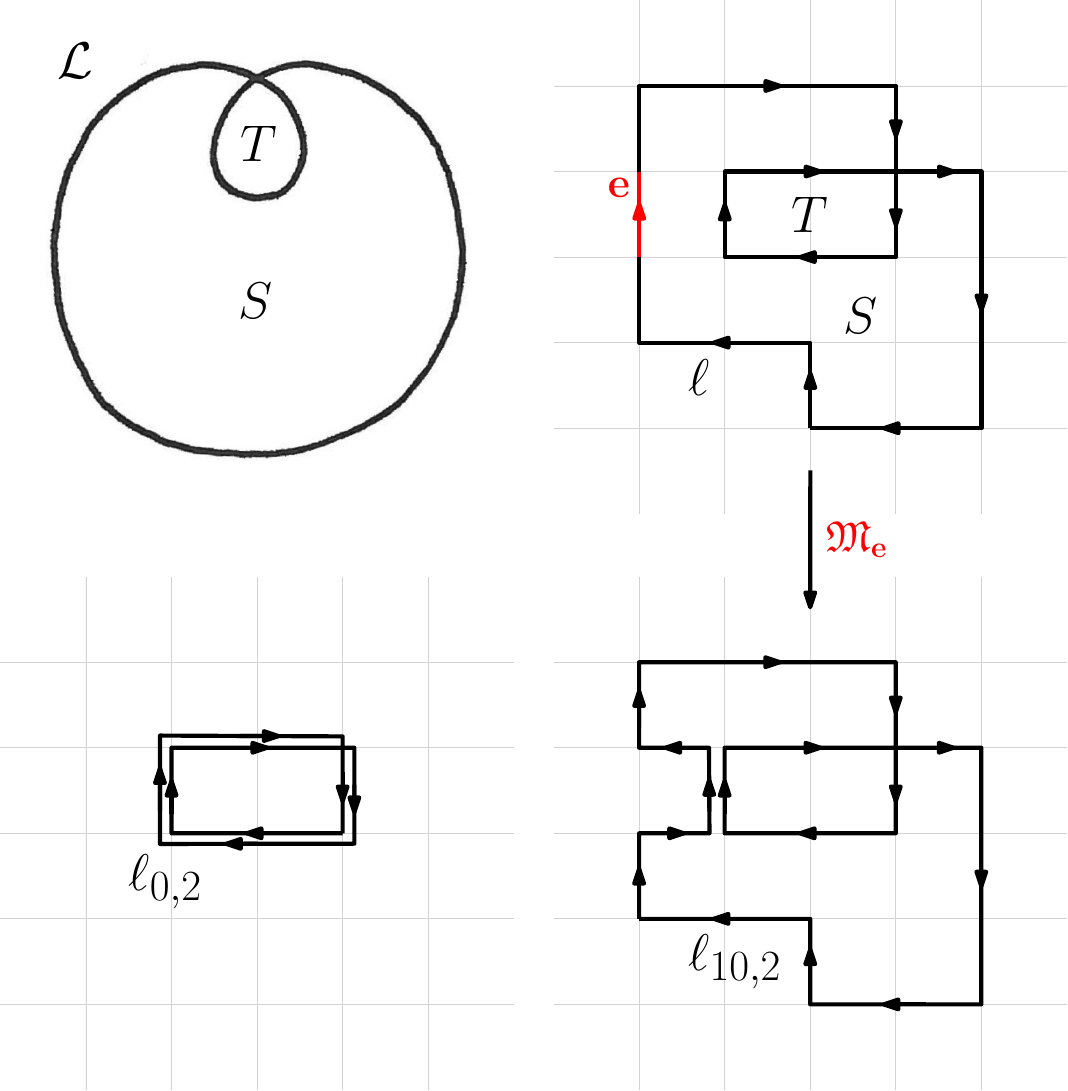}  
			\caption{\label{fig: Ex1_loop} \textbf{Top left:} A continuum loop $\cL$ with two interior regions labeled $S$ and $T$. \textbf{Top right:} A lattice loop $\ell$ isotopic to $\cL$. The plaquette region $S$ contains $s=11$ plaquettes and the plaquette region $T$ contains $t=2$ plaquettes. An edge $\mathbf{e}$ on the boundary of $S$ but not $T$ is highlighted. \textbf{Bottom right:} The loop, $\ell_{10,2}$, obtained from $(\ell,K_{\ell})$ by applying the master loop equation at $\mathbf{e}$.  \textbf{Bottom left:} The loop, $\ell_{0,2}$, obtained by repeatedly applying the master loop equation to remove all of the $S$ region from $\ell$.}
		\end{center}
		\vspace{-3ex}
	\end{figure}
    
    Assume that the plaquette-region $S$ of $\ell$ has area $s$  and the plaquette-region $T$ has area $t$ (cf.\ the top right corner of \cref{fig: Ex1_loop}). 
    Let $\mathbf{e}$ be any edge on the boundary of $S$ but not on the boundary of $T$ (such an edge exists because $\ell$ is isotopic to $\cL$).

    Then, applying the master loop equation at $\mathbf{e}$ must result in a negative
    deformation since this is the only admissible operation. Indeed, it must be the case as the edge $e$ is only contained once in $\ell$ (again, because $\ell$ is isotopic to $\cL$) and $e^{-1}$ is only contained once in the plaquette assignment $K_{\ell}$ (by definition of $K_{\ell}$ in \eqref{defn:master-plaq-ass}; in particular, $e^{-1}$ is contained in the plaquette $p^{-1}$ where $p$ is the unique plaquette in $S$ containing $e$). The bottom right corner of \cref{fig: Ex1_loop} shows an example of this application of the master loop equation. 
    
    Thus, denoting by $\ell_{s-1,t}$ the loop obtained after performing this operation (cf.\ the bottom right corner of \cref{fig: Ex1_loop}), we get, thanks to the master loop equation in \eqref{eq:mle-2}, that
    \[c(\ell) = c(\ell_{s-1,t}).\]
    The above reasoning can be iterated until we remove all the plaquettes in the  plaquette-region $S$. Thus, thanks to \eqref{eq:mle-2}, we get that
    \[c(\ell) = c(\ell_{0,t}),\]
    where $\ell_{0,t}$ denotes the loop that winds twice around the plaquette-region $T$ (cf.\ the bottom left corner of \cref{fig: Ex1_loop}). Since the boundary of the region $T$ in $\cL$ is simple and  $\ell$ is isotopic to $\cL$, it must be that the boundary of the plaquette-region $T$ of $\ell$ is also simple. Hence, $\ell_{0,t}$ is a loop with area $t$ winding around $T$ twice and so, \cref{thm: wound simple loops WLE} gives us that
    \begin{equation*}
        c(\ell)=c(\ell_{0,t}) = 1-t.
    \end{equation*}
    Since $\area(K_\ell)=s+2t$ by definition of $K_\ell$, \cref{thm: erasable loops have one plaquette assignment} allows us to conclude that
    \begin{equation*}
        \phi(\ell)= (1-t)\upbeta^{s+2t},
    \end{equation*}
    as claimed in \cref{table}.

    \subsubsection{Example 3}

    We now move to the continuum loop $\cL$ in the top left of \cref{fig: Ex2_loop}, labeling each region as shown in the figure. Let $\ell$ be any lattice loop isotopic to $\cL$ (an example is shown in the top middle of \cref{fig: Ex2_loop}). For the same reason as for the previously analyzed loop, thanks to \cref{thm: erasable loops have one plaquette assignment}, we only need to compute the coefficient $c(\ell)$ to determine $\phi(\ell)$,

    \begin{figure}[ht!]
		\begin{center}
			\includegraphics[width=.75\textwidth]{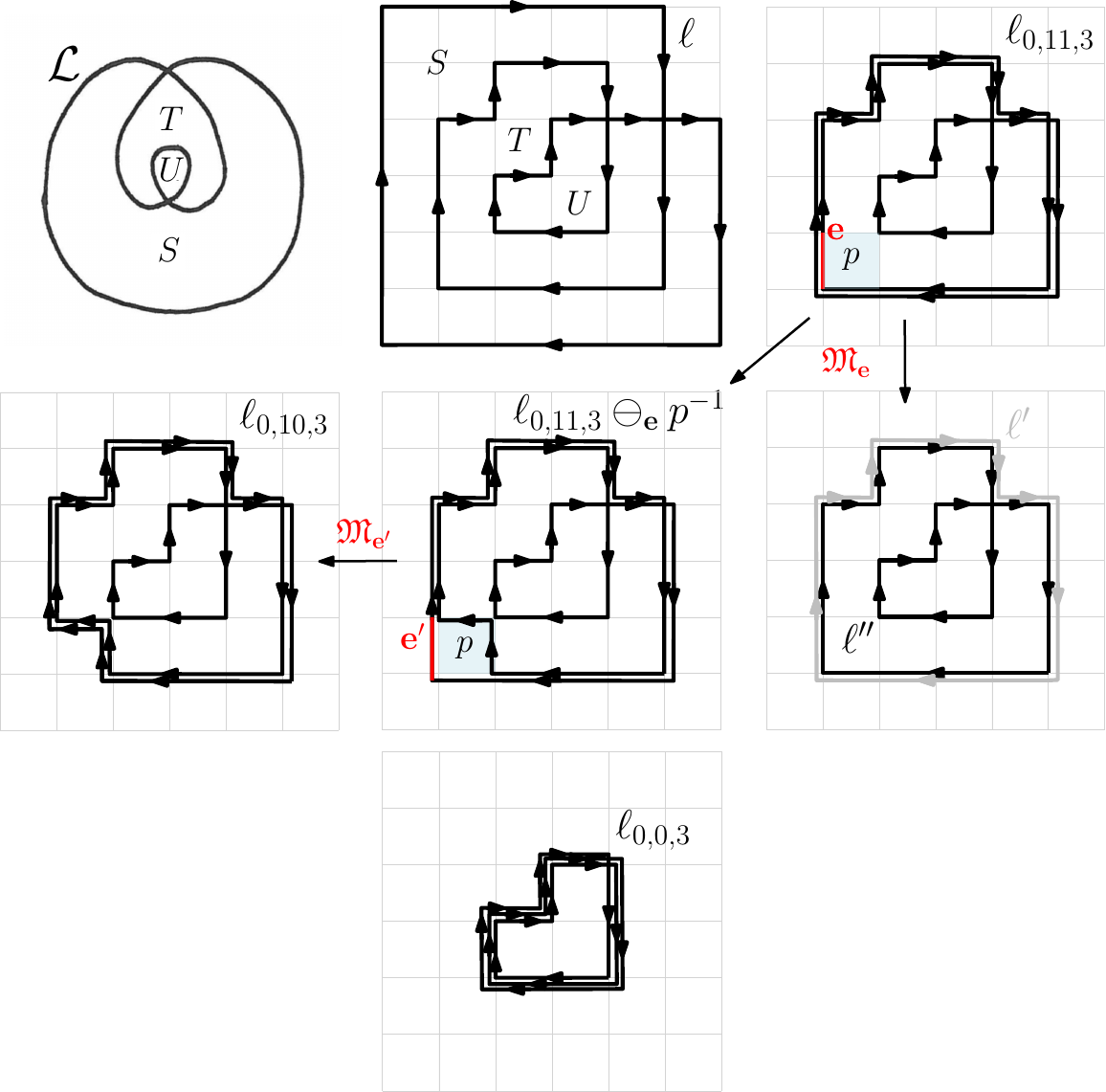}  
			\caption{\label{fig: Ex2_loop} \textbf{Top left:} A continuum loop $\cL$ with three interior regions labeled $S$, $T$, and $U$. \textbf{Top middle:} A lattice loop $\ell$ isotopic to $\cL$. The plaquette region $S$ contains $s=20$ plaquettes, the plaquette region $T$ contains $t=11$ plaquettes, and the plaquette region $U$ contains $u=3$ plaquettes. \textbf{Top right:} the loop, $\ell_{0,11,3}$, obtained from $\ell$ by removing the $S$ region as in \cref{sect:loop2}. An edge $\mathbf{e}$ on the boundary of $T$ but not $U$ is highlighted. \textbf{Middle right:} The string $\{\ell',\ell''\}$ obtained from $(\ell_{0,11,3},K_{\ell_{0,11,3}})$ when a positive splitting is preformed at $\mathbf{e}$. \textbf{Middle middle:} The loop obtained from $(\ell_{0,11,3},K_{\ell_{0,11,3}})$ when a negative deformation is preformed at $\mathbf{e}$. \textbf{Middle left:} The loop obtained from $(\ell_{0,11,3},K_{\ell_{0,11,3}})$ when two consecutive negative deformations are preformed at the two copies of $e$ in $\ell_{0,11,3}$. \textbf{Bottom middle:} The loop $\ell_{0,0,3}$ obtained by repeatedly applying the negative deformations to remove all of the $S$ and $T$ regions from $\ell$.}
		\end{center}
		\vspace{-3ex}
	\end{figure}

    Using the same techniques as for the previously analyzed loop, we get that
    \[c(\ell) = c(\ell_{0,t,u}),\]
    where $\ell_{0,t,u}$ denotes the loop that one obtains after removing all the plaquette-region $S$ from $\ell$ only performing negative deformations; as detailed before.

    Now, our goal is to remove the plaquette-region $T$ from $\ell_{0,t,u}$. We immediately note that, by construction, the loop $\ell_{0,t,u}$ has two copies,  oriented in the same direction, of each lattice edge on the boundary of $T$ that is not part of the boundary of $U$.
    Let $\mathbf{e}$ be any edge of $\ell_{0,t,u}$ on the boundary of  $T$ but not on the boundary of $U$ (such an edge exists because our original loop $\ell$ is isotopic to $\cL$). 
    
    Applying the master loop equation at $\mathbf{e}$ will result in either a negative deformation with the plaquette  $p^{-1}$ of $K_{\ell_{0,t,u}}$ (where $p$ is the plaquette in the support of $\ell(0,t,u)$ that contains $e$; cf.\ the middle middle of \cref{fig: Ex2_loop}) or a positive splitting with the other copy of the edge $e$ in $\ell_{0,t,u}$ (cf.\ the middle right of \cref{fig: Ex2_loop}).

    Let $\ell'$ and $\ell''$ be the two loops obtained through a positive splitting (cf.\ the middle left of \cref{fig: Ex2_loop}). Note if a negative deformation was performed then if we apply the master loop equation again at the unique occurrence of $e$ in the new loop, we can only perform a negative deformation with $p^{-1}$. Let $\ell_{0,t-1,u}$ denote the loop obtained through two consecutive negative deformations. An example of such a loop obtained through two consecutive negative deformations is shown in the middle left of \cref{fig: Ex2_loop}. 
    
    Translating this into the coefficients we are aiming to compute, we get, thanks to the master loop equation in \eqref{eq:mle-2}, that
    \begin{align*}
        c(\ell_{0,t,u})=c(\ell_{0,t-1,u}) - c(\ell')\cdot c(\ell'') = c(\ell_{0,t-1,u}) - (1-u),
    \end{align*}
    where to get the last equality we used that $c(\ell')\cdot c(\ell'')=1\cdot (1-u)$ by our results in  Sections~\ref{sect:loop1}~and~\ref{sect:loop2}. 

    Now, repeating the above procedure to $\ell_{0,t-1,u}$ until all the plaquette-region $T$ is removed, by the same argument we get  that
    \begin{align*}
        c(\ell) = c(\ell_{0,t,u}) = c(\ell_{0,0,u}) - t(1-u),
    \end{align*}
    where $\ell_{0,0,u}$ denotes the loop that winds three times around the plaquette-region $U$ (cf.\ the bottom middle of \cref{fig: Ex2_loop}). Since the boundary of the region $U$ in $\cL$ is simple and we started with a loop $\ell$ isotopic to $\cL$, it must be that the boundary of the plaquette-region $U$ is also simple. Hence, $\ell_{0,0,u}$ is a loop with area $u$ winding around $U$ three times and so, \cref{thm: wound simple loops WLE} gives us that $c(\ell_{0,0,u})=\frac{1}{6}(3u-3)(3u-2)$ and so
    \begin{equation*}\label{eq: cstu}
        c(\ell) = \frac{1}{6}(3u-3)(3u-2) - t(1-u).
    \end{equation*}
    Since $\area(K_\ell)=s+2t+3u$ by definition of $K_\ell$, \cref{thm: erasable loops have one plaquette assignment} allows us to conclude that
    \begin{equation*}
        \phi(\ell)= \left[\frac{1}{6}(3u-3)(3u-2) - t(1-u)\right]\upbeta^{s+2t+3u},
    \end{equation*}
    as claimed in \cref{table}.

    \subsubsection{Example 4}

    Finally, we compute the Wilson loop expectations for lattice loops isotopic to the continuum loop $\cL$ in the top left of \cref{fig: Ex3_loop_2}, an example of which is given in the top middle of \cref{fig: Ex3_loop_2}.

    \begin{figure}[ht!]
		\begin{center}
			\includegraphics[width=.75\textwidth]{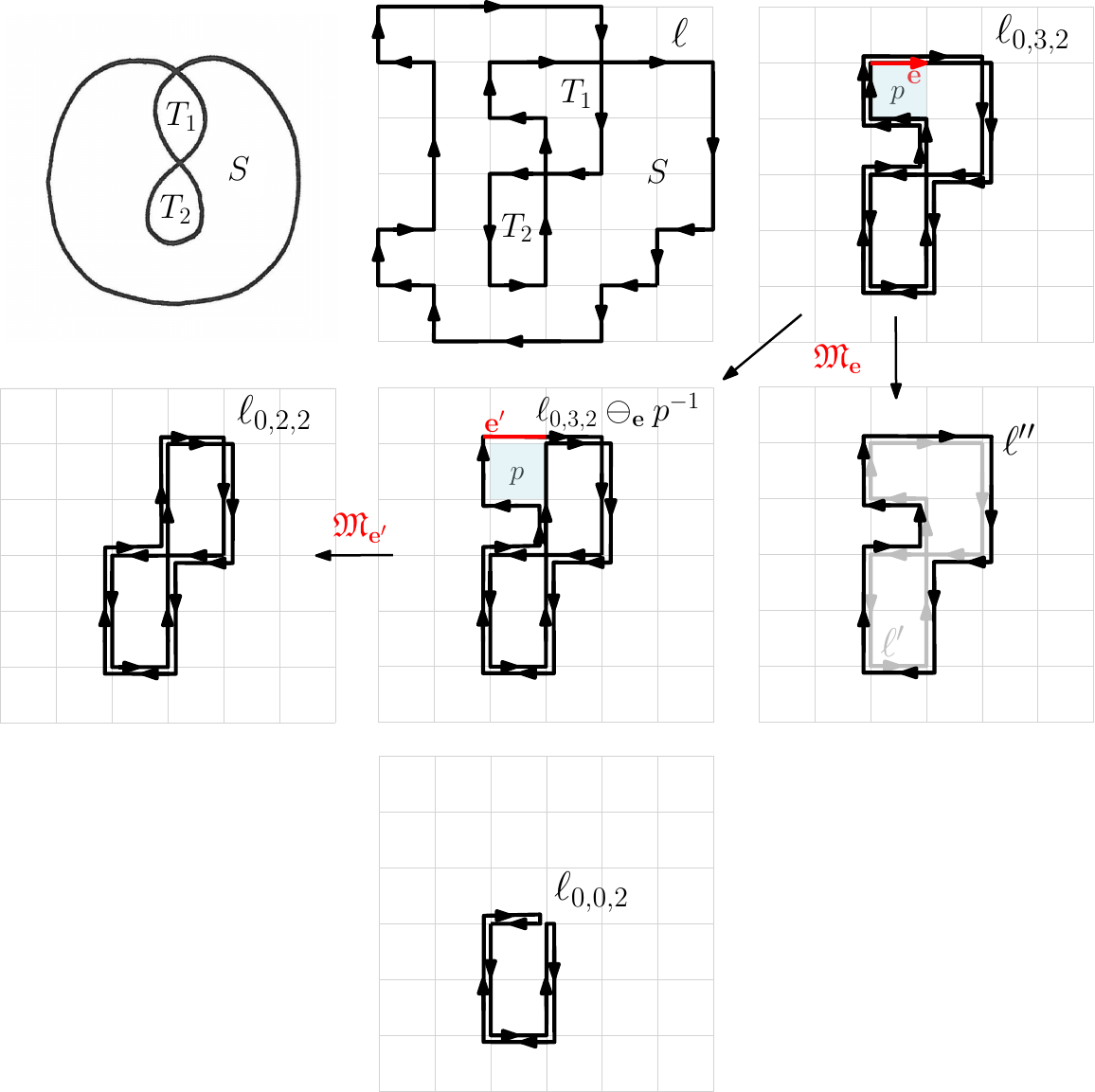}  
			\caption{\label{fig: Ex3_loop_2} \textbf{Top left:} A continuum loop $\cL$ with three interior regions labeled $S$, $T_1$, and $T_2$. \textbf{Top middle:} A lattice loop $\ell$ isotopic to $\cL$. The plaquette region $S$ contains $s=22$ plaquettes, the plaquette region $T_1$ contains $t_1=3$ plaquettes, and the plaquette region $T_2$ contains $t_2=2$ plaquettes. \textbf{Top right:} The loop $\ell_{0,3,2}$ obtained from $\ell$ by removing the $S$ region as in \cref{sect:loop2}. An edge $\mathbf{e}$ on the boundary of $T_1$ but not $T_2$ is highlighted. \textbf{Middle right:} The string, $\{\ell',\ell''\}$, obtained from $(\ell_{0,3,2},K_{\ell_{0,3,2}})$ when a positive splitting is preformed at $\mathbf{e}$. \textbf{Middle middle:} The loop obtained from $(\ell_{0,3,2},K_{\ell_{0,3,2}})$ when a negative deformation is preformed at $\mathbf{e}$. \textbf{Middle left:} The loop obtained from $(\ell_{0,3,2},K_{\ell_{0,3,2}})$ when two consecutive negative deformations are preformed at the two copies of $e$ in $\ell_{0,3,2}$. \textbf{Bottom middle:} The loop $\ell_{0,0,2}$ obtained by repeatedly applying the negative deformations to remove all of the $S$ and $T_1$ regions from $\ell$. Note $\ell_{0,0,2}$ is equivalent to the null-loop once backtracks are removed.}
		\end{center}
		\vspace{-3ex}
	\end{figure}

    First, notice that any lattice loop $\ell'$ isotopic to $\cL'$ and $\ell''$ isotopic to $\cL''$, where $\cL'$ and $\cL''$ are shown on the left-hand side of  \cref{fig: Ex3_loop_1} (the right-hand side of \cref{fig: Ex3_loop_1} shows two examples of such lattice loops), we have that\begin{equation}\label{eq: fig eight}
        c(\ell') = 1\quad\text{and}\quad c(\ell'')=1.
    \end{equation}
    We leave the (simple) proof of this claim to the reader.

    \begin{figure}[ht!]
		\begin{center}
			\includegraphics[width=.5\textwidth]{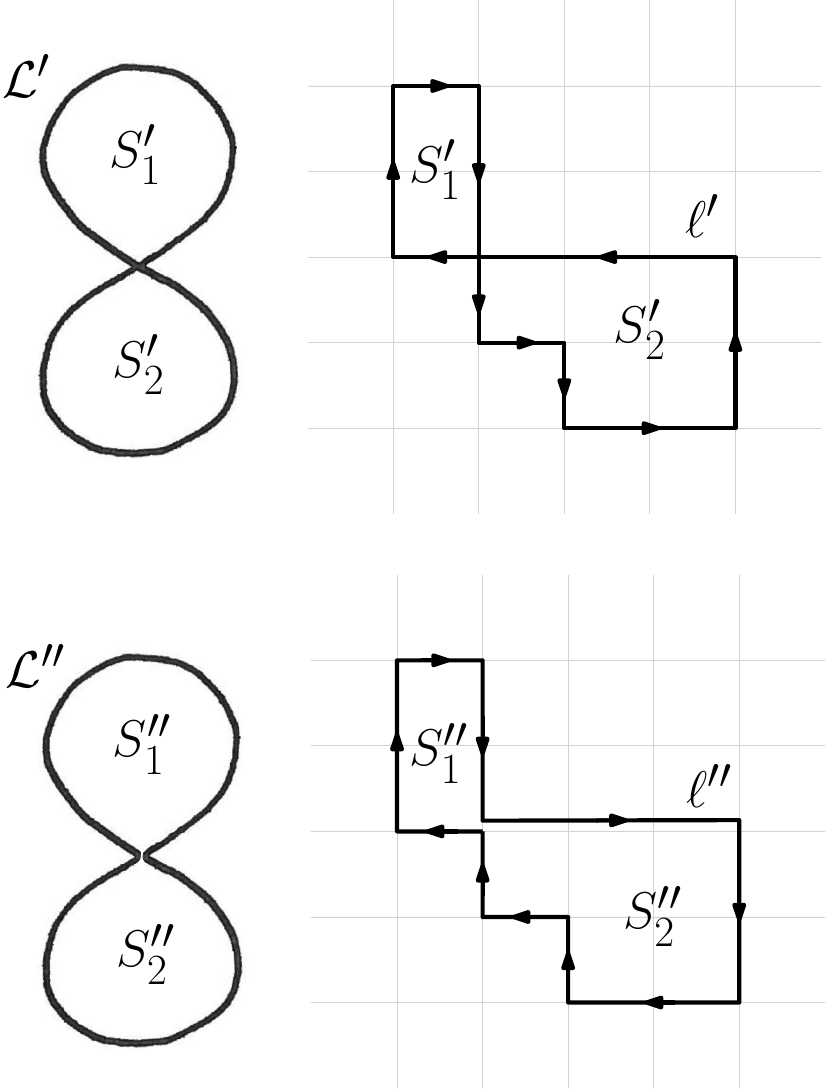}  
			\caption{\label{fig: Ex3_loop_1} \textbf{Top left:} A continuum loop $\cL'$ with one self-crossing and two interior regions labeled $S_1'$ and $S_2'$. \textbf{Top right:} A lattice loop $\ell'$ isotopic to $\cL'$. \textbf{Bottom left:} A continuum loop $\cL''$ with no self-crossings but one double point  and two interior regions labeled $S_1''$ and $S_2''$. \textbf{Bottom right:} A lattice loop $\ell''$ isotopic to $\cL''$. Note the edges of the loop are drawn slightly off the lattice so there is no ambiguity in which order the edges are visited, but, if the edges where drawn on the lattice, there would be one double point corresponding to the unique vertex contained in both plaquette regions, but no self-crossings.}
		\end{center}
		\vspace{-3ex}
	\end{figure}

     Let $\ell$ be any lattice loop isotopic to $\cL$ (cf.\ the top middle of \cref{fig: Ex3_loop_2}). Then, for any plaquette assignment $\Bar{K}$, let\begin{align*}
        \Bar{K}^{T_2}(p) = \begin{cases} \Bar{K}(p) + 1 & \text{if $p$ is in $T_2$,}\\ \Bar{K}(p) & \text{otherwise,}
        \end{cases}
    \end{align*}
    and set $c_{T_2}(\Bar{\ell}) := c(\Bar{\ell}, K_{\Bar{\ell}}^{T_2})$ for any loop $\Bar{\ell}$.
    
    Note that thanks to \cref{rem:card-set-cK}, the canonical collection $\cK_{\ell}$ of plaquette assignments has cardinality two because the $T_2$ region has height $0$ but distance $2$. In particular, 
    \[\cK_{\ell} = \{K_{\ell}, K_{\ell}^{T_2}\}.\] 
    Thus, to compute $\phi(\ell)$ we will need to compute $c(\ell)$ and $c_{T_2}(\ell)$. We will compute $c(\ell)$ and $c_{T_2}(\ell)$ in parallel.

    First, using the same techniques as in the previous examples, we remove the $S$ region. That is, we get that\begin{align*}
        c(\ell) = c(\ell_{0,t_1,t_2})\quad\text{and}\quad c_{T_2}(\ell) = c_{T_2}(\ell_{0,t_1,t_2}),
    \end{align*}where $\ell_{0,t_1,t_2}$ denotes the loop obtained after removing all of the plaquette-region $S$ from $\ell$ while only performing negative deformations; as detailed before. 

    Now, we remove the $T_1$ region. We note that, by construction, the loop $\ell_{0,t_1,t_2}$ has two copies, oriented in the same direction, of each lattice edge on the boundary of $T_1$. Let $\mathbf{e}$ be any boundary edge of $T_2$ not on the boundary of $T_2$ (cf.\ the top right of \cref{fig: Ex3_loop_2}).

    Applying the master loop equation at $\mathbf{e}$ will result in either a negative deformation with the plaquette $p^{-1}$ of $K_{\ell_{0,t_1,t_2}}$ (where $p$ is the plaquette in the support of $\ell_{0,t_1,t_2}$ that contains $e$; cf.\ the middle middle of \cref{fig: Ex3_loop_2}) or in a positive splitting  (cf.\ the middle right of \cref{fig: Ex3_loop_2}). Let $\ell'$ and $\ell''$ be the two loops obtained through a positive splitting.

   Note that if a negative deformation was performed, then, if we apply the master loop equation again at the unique occurrence of $e$ in the new loop, we can only perform a negative deformation with $p^{-1}$. Let $\ell_{0,t_1-1,t_2}$ denote the loop obtained through two consecutive negative deformations. An example of such a loop obtained through two consecutive negative deformations is shown in the middle left of \cref{fig: Ex3_loop_2}. 
   
   Translating this into the coefficients we are aiming to compute, we get that
   \begin{align*}
       c(\ell_{0,t_1,t_2}) = c(\ell_{0,t_1-1,t_2}) - c(\ell', K')\cdot c(\ell'',K'') = c(\ell_{0,t_1-1,t_2}),
   \end{align*}where to get the final equality we used that there is no $K'+K''=K_{\ell_{0,t_1,t_2}}$ such that both $(\ell',  K')$ and $(\ell'',K'')$ are balanced, and
   \begin{align*}
       c_{T_2}(\ell_{0,t_1,t_2}) = c_{T_2}(\ell_{0,t_1-1,t_2}) - c(\ell')c(\ell'') = c_{T_2}(\ell_{0,t_1-1,t_2})-1,
   \end{align*}where to get the final equality we used that $\ell'$ and $\ell''$ are isotopic to $\cL'$ and $\cL''$ respectively and thus we know their coefficients by \eqref{eq: fig eight}. 

   Now, repeating the above procedure to $\ell_{0,t_1-1,t_2}$ until all of the plaquette-region $T_1$ is removed, by the same argument we get that\begin{align*}
       c(\ell_{0,t_1,t_2}) = c(\emptyset) = 1\quad\text{and}\quad c_{T_2}(\ell_{0,t_1,t_2)} = c_{T_2}(\emptyset) - t_1 = -t_1
   \end{align*}where we used that the loop $\ell_{0,0,t_2}$ is equivalent to the null-loop once backtracks are removed (cf.\ the bottom middle of \cref{fig: Ex3_loop_2}) and $0^{T_2}\neq 0$, where $0$ denotes the zero plaquette assignment. Thus, we get that\begin{align*}
       c(\ell) = 1\quad\text{and}\quad c_{T_2}(\ell) = -t_1.
   \end{align*}Finally, since $\area(K_{\ell}) = s+2t_1$ and $\area(K^{T_2}_{\ell}) = s+2t_1+2t_2$, then \cref{thm: erasable loops have one plaquette assignment} allows us to conclude that\begin{align*}
       \phi(\ell) = \upbeta^{s+2t_1} - t_1\upbeta^{s+2t_1+2t_2} = \upbeta^{s+2t_1}(1-t_1\upbeta^{2t_2}),
   \end{align*}as claimed in \cref{table}.

\addcontentsline{toc}{section}{References}
\bibliography{mybib}
\bibliographystyle{hmralphaabbrv}
 
\end{document}